%% file: non-semisimple_etqfts.tex
\documentclass{memo-l}

\usepackage[T1]{fontenc}
\usepackage[utf8x]{inputenc}
\usepackage[english]{babel}
\usepackage{amssymb}
\usepackage{accents}
\usepackage{tensor}
\usepackage{graphicx}
\usepackage{tikz}
\usetikzlibrary{calc,matrix,arrows,decorations.pathmorphing}
\usepackage{calc}
\usepackage[cal=boondox,scr=boondoxo]{mathalfa}
\usepackage[section]{placeins}
\usepackage{scalerel}[2016/12/29]
\numberwithin{figure}{chapter}
\usepackage{hyperref}
\hypersetup{colorlinks=true,pageanchor=false,linkcolor=blue,citecolor=red}
\usepackage[all]{hypcap}

\newtheorem*{theorem*}{Theorem}

\newtheorem{theorem}{Theorem}[chapter]
\newtheorem{proposition}{Proposition}[chapter]
\newtheorem{lemma}{Lemma}[chapter]

\theoremstyle{definition}
\newtheorem{definition}{Definition}[chapter]
\newtheorem{example}{Example}[chapter]

\theoremstyle{remark}
\newtheorem{remark}{Remark}[chapter]

\numberwithin{section}{chapter}
\numberwithin{equation}{chapter}

\makeindex

\newcommand{\N}{\mathbb{N}}
\newcommand{\Z}{\mathbb{Z}}

\newcommand{\R}{\mathbb{R}}
\newcommand{\C}{\mathbb{C}}

\newcommand{\rmd}{\mathrm{d}}

\newcommand{\rmh}{\mathrm{h}}
\newcommand{\rml}{\mathrm{l}}
\newcommand{\rmr}{\mathrm{r}}
\newcommand{\rmt}{\mathrm{t}}
\newcommand{\rmv}{\mathrm{v}}

\newcommand{\rmI}{\mathrm{I}}

\newcommand{\rmV}{\mathrm{V}}

\newcommand{\bbA}{\mathbb{A}}
\newcommand{\bbB}{\mathbb{B}}

\newcommand{\bbD}{\mathbb{D}}

\newcommand{\bbF}{\mathbb{F}}

\newcommand{\bbH}{\mathbb{H}}
\newcommand{\bbI}{\mathbb{I}}
\newcommand{\bbJ}{\mathbb{J}}
\newcommand{\bbM}{\mathbb{M}}

\newcommand{\bbP}{\mathbb{P}}
\newcommand{\bbS}{\mathbb{S}}

\newcommand{\bbV}{\mathbb{V}}
\newcommand{\bbW}{\mathbb{W}}
\newcommand{\bbX}{\mathbb{X}}
\newcommand{\bbY}{\mathbb{Y}}
\DeclareRobustCommand{\bbf}{\mathbin{\text{\raisebox{-0.125pt}{\includegraphics[height=\heightof{$\mathbf{f}$}]{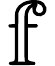}}}}}
\DeclareRobustCommand{\bbeta}{\mathbin{\text{\raisebox{-2.5pt}{\includegraphics[height=\heightof{$\Gamma$}]{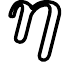}}}}}

\DeclareRobustCommand{\bbDelta}{\mathbin{\text{\includegraphics[height=\heightof{$\mathbf{\Delta}$}]{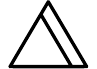}}}}
\DeclareRobustCommand{\bbGamma}{\mathbin{\text{\includegraphics[height=\heightof{$\mathbf{\Gamma}$}]{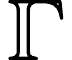}}}}

\DeclareRobustCommand{\bbSigma}{\mathbin{\text{\includegraphics[height=\heightof{$\mathbf{\Sigma}$}]{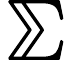}}}}

\newcommand{\bfc}{\mathbf{c}}

\newcommand{\bff}{\mathbf{f}}
\newcommand{\bfh}{\mathbf{h}}
\newcommand{\bfi}{\mathbf{i}}
\newcommand{\bfj}{\mathbf{j}}
\newcommand{\bfk}{\mathbf{k}}
\newcommand{\bfl}{\mathbf{l}}
\newcommand{\bfm}{\mathbf{m}}
\newcommand{\bfn}{\mathbf{n}}

\newcommand{\bfp}{\mathbf{p}}
\newcommand{\bfq}{\mathbf{q}}

\newcommand{\bfx}{\mathbf{x}}
\newcommand{\bfy}{\mathbf{y}}
\newcommand{\bfA}{\mathbf{A}}
\newcommand{\bfB}{\mathbf{B}}
\newcommand{\bfC}{\mathbf{C}}

\newcommand{\bfF}{\mathbf{F}}

\newcommand{\bfH}{\mathbf{H}}
\newcommand{\bfI}{\mathbf{I}}
\newcommand{\bfJ}{\mathbf{J}}
\newcommand{\bfM}{\mathbf{M}}

\newcommand{\bfP}{\mathbf{P}}

\newcommand{\bfT}{\mathbf{T}}

\newcommand{\bfalpha}{\boldsymbol{\alpha}}

\newcommand{\bfepsilon}{\boldsymbol{\varepsilon}}

\newcommand{\bfiota}{\boldsymbol{\iota}}
\newcommand{\bflambda}{\boldsymbol{\lambda}}
\newcommand{\bfmu}{\boldsymbol{\mu}}
\newcommand{\bfrho}{\boldsymbol{\rho}}
\newcommand{\bfsigma}{\boldsymbol{\sigma}}
\newcommand{\bftau}{\boldsymbol{\tau}}

\newcommand{\bfchi}{\boldsymbol{\chi}}
\newcommand{\bfGamma}{\mathbf{\Gamma}}
\newcommand{\bfDelta}{\mathbf{\Delta}}
\newcommand{\bfTheta}{\boldsymbol{\Theta}}
\newcommand{\bfLambda}{\mathbf{\Lambda}}
\newcommand{\bfXi}{\mathbf{\Xi}}
\newcommand{\bfPi}{\boldsymbol{\Pi}}
\newcommand{\bfSigma}{\mathbf{\Sigma}}
\newcommand{\bfOmega}{\boldsymbol{\Omega}}

\newcommand{\calA}{\mathcal{A}}

\newcommand{\calC}{\mathcal{C}}
\newcommand{\calD}{\mathcal{D}}
\newcommand{\calE}{\mathcal{E}}

\newcommand{\calG}{\mathcal{G}}
\newcommand{\calK}{\mathcal{K}}
\newcommand{\calL}{\mathcal{L}}

\newcommand{\calU}{\mathcal{U}}
\newcommand{\calV}{\mathcal{V}}

\newcommand{\scrI}{\mathscr{I}}

\newcommand{\frakg}{\mathfrak{g}}
\newcommand{\frakh}{\mathfrak{h}}
\newcommand{\frakn}{\mathfrak{n}}
\newcommand{\frakS}{\mathfrak{S}}

\newcommand{\bcalC}{\mathbcal{C}}

\newcommand{\id}{\mathrm{id}}

\newcommand{\tr}{\mathrm{tr}}
\newcommand{\lk}{\mathrm{lk}}
\newcommand{\im}{\operatorname{im}}
\newcommand{\ind}{\operatorname{ind}}
\newcommand{\sgn}{\operatorname{sgn}}
\newcommand{\lev}{\smash{\stackrel{\leftarrow}{\mathrm{ev}}}}
\newcommand{\lcoev}{\smash{\stackrel{\longleftarrow}{\mathrm{coev}}}}
\newcommand{\rev}{\smash{\stackrel{\rightarrow}{\mathrm{ev}}}}
\newcommand{\rcoev}{\smash{\stackrel{\longrightarrow}{\mathrm{coev}}}}
\DeclareMathOperator{\Exists}{\exists}
\DeclareMathOperator{\Forall}{\forall}
\DeclareMathOperator{\ttimes}{\tilde{\times}}
\DeclareMathOperator{\tsqcup}{\tilde{\sqcup}}
\DeclareMathOperator{\cappr}{\smallfrown}
\DeclareMathOperator{\cuppr}{\smallsmile}

\DeclareMathOperator{\sqtimes}{\scaleobj{0.8}{\boxtimes}}
\DeclareMathOperator{\csqtimes}{\hat{\scaleobj{0.8}{\boxtimes}}}
\DeclareMathOperator{\disjun}{\sqcup}

\newcommand{\bdots}{\reflectbox{$\ddots$}}
\newcommand{\sltwo}{\mathfrak{sl}_2}
\newcommand{\GL}{\mathrm{GL}}
\newcommand{\Hom}{\mathrm{Hom}}
\newcommand{\bbHom}{\mathbb{H}\mathrm{om}}
\newcommand{\End}{\mathrm{End}}

\newcommand{\Vect}{\mathrm{Vect}}

\newcommand{\bbVect}{\mathbb{V}\mathrm{ect}}

\newcommand{\Rib}{\mathcal{R}}

\newcommand{\op}{\mathrm{op}}

\newcommand{\Man}{\mathrm{Man}}
\newcommand{\cobbE}{\hat{\mathbb{A}}^Z}
\newcommand{\cobfE}{\hat{\mathbf{A}}}
\newcommand{\WRT}{\mathrm{WRT}}
\newcommand{\CGP}{\mathrm{CGP}}
\newcommand{\gltqft}{\bbV^Z_{\calC}}

\newcommand\restr[3][0]{{\raisebox{-#1pt}{$\left. \raisebox{#1pt}{$#2$} \vphantom{\big|} \right|_{#3}$}}}
\newcommand{\two}{\langle 2 \rangle}

\DeclareRobustCommand{\one}{\mathbin{\text{\includegraphics[height=\heightof{$\mathbf{1}$}]{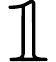}}}}
\DeclareRobustCommand{\bbnabla}{\mathbin{\text{\reflectbox{\raisebox{\depth}{\scalebox{1}[-1]{\includegraphics[height=\heightof{$\mathbf{\Delta}$}]{bbDelta.pdf}}}}}}}
\DeclareRobustCommand{\fourth}{\mathbin{\text{\includegraphics[height=\heightof{$\prime$}]{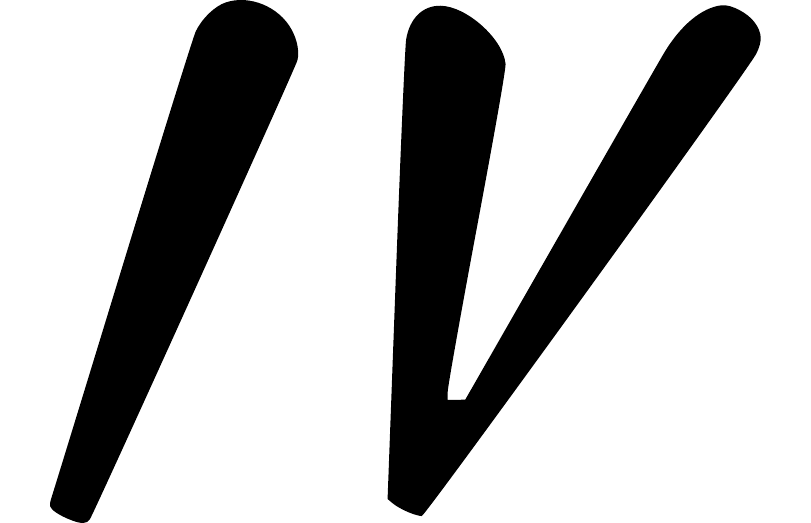}}}}
\DeclareRobustCommand{\fifth}{\mathbin{\text{\includegraphics[height=\heightof{$\prime$}]{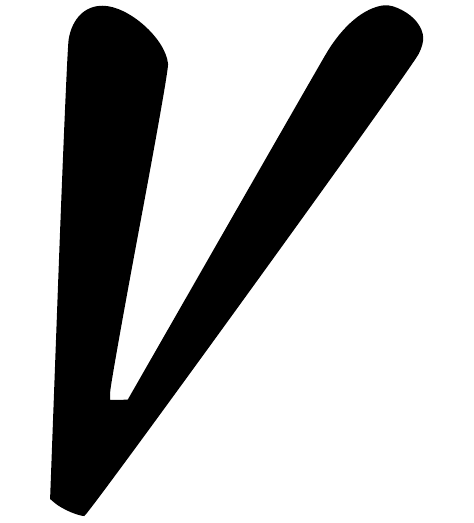}}}}

\newcommand{\PGr}{Z}

\newcommand{\bfCob}{\mathbf{Cob}}
\newcommand{\bfadCob}{\mathbf{Cob}^{\mathbf{ad}}}
\newcommand{\rmadCob}{\mathrm{Cob}^{\mathrm{ad}}}

\newcommand{\bfCat}{\mathbf{Cat}}
\newcommand{\coCat}{\hat{\mathbf{C}}\mathbf{at}}
\newcommand{\twoCat}{\mathbf{2Cat}}
\newcommand{\smtwoCat}{\mathbf{2Cat^{sm}}}
\newcommand{\cmpl}{\bfC}
\newcommand{\grcmpl}{\bfC}
\newcommand{\trbr}{\bftau}
\newcommand{\twobr}{\bfc}
\newcommand{\gmmbr}{\bfc^{\gamma}}
\newcommand{\cobbr}{\bftau}

\newcommand{\Kar}{\mathrm{Kar}}
\newcommand{\Mat}{\mathrm{Mat}}
\newcommand{\Deg}{\mathrm{Deg}}

\newcommand{\sqbinom}[2]{\left[ \begin{matrix} #1 \\ #2 \end{matrix} \right]}
\newcommand{\Proj}{\mathrm{Proj}}
\newcommand{\Inj}{\mathrm{Inj}}

\newcommand{\adSk}{\check{\mathcal{S}}}

\newcommand{\bbProj}{\mathbb{P}\mathrm{roj}}

\newcommand{\rmCol}{\mathrm{Col}}

\makeatletter
\newcommand{\subalign}[1]{
  \vcenter{
    \Let@ \restore@math@cr \default@tag
    \baselineskip\fontdimen10 \scriptfont\tw@
    \advance\baselineskip\fontdimen12 \scriptfont\tw@
    \lineskip\thr@@\fontdimen8 \scriptfont\thr@@
    \lineskiplimit\lineskip
    \ialign{\hfil$\m@th\scriptstyle##$&$\m@th\scriptstyle{}##$\crcr
      #1\crcr
    }
  }
}
\makeatother

\def\clap#1{\hbox to 0pt{\hss#1\hss}}
\def\mathllap{\mathpalette\mathllapinternal}
\def\mathrlap{\mathpalette\mathrlapinternal}
\def\mathclap{\mathpalette\mathclapinternal}
\def\mathllapinternal#1#2{%
\llap{$\mathsurround=0pt#1{#2}$}}
\def\mathrlapinternal#1#2{%
\rlap{$\mathsurround=0pt#1{#2}$}}
\def\mathclapinternal#1#2{%
\clap{$\mathsurround=0pt#1{#2}$}}

\begin{document}

\frontmatter

\title{Non-Semisimple Extended Topological Quantum Field Theories}


\author{Marco De Renzi}
\address{Institut de Mathématiques de Jussieu -- Paris Rive Gauche, UMR 7586, Université Paris Diderot -- Paris 7, Bâtiment Sophie Germain, Boite Courrier 7012, 75205 Paris Cedex 13, France}
\curraddr{Department of Mathematics, Faculty of Science and Engineering, Waseda University, 3-4-1 \={O}kubo, Shinjuku-ku, Tokyo, 169-8555, Japan}
\email{m.derenzi@kurenai.waseda.jp}
\thanks{The author would like to thank Christian Blanchet, Bertrand Patureau-Mirand, Nathan Geer, François Costantino, and Jun Murakami for many fruitful discussions and useful comments.}





\begin{abstract}
 We develop the general theory for the construction of \textit{Extended Topological Quantum Field Theories} (\textit{ETQFTs}) associated with the Costantino-Geer-Patureau quantum invariants of closed 3-manifolds. In order to do so, we introduce \textit{relative modular categories}, a class of ribbon categories which are modeled on representations of unrolled quantum groups, and which can be thought of as a non-semisimple analogue to modular categories. Our approach exploits a 2-categorical version of the universal construction introduced by Blanchet, Habegger, Masbaum, and Vogel. The 1+1+1-EQFTs thus obtained are realized by symmetric monoidal 2-functors which are defined over non-rigid 2-categories of \textit{admissible} cobordisms decorated with colored ribbon graphs and cohomology classes, and which take values in 2-categories of complete graded linear categories. In particular, our construction extends the family of graded 2+1-TQFTs defined for the unrolled version of quantum $\sltwo$ by Blanchet, Costantino, Geer, and Patureau to a new family of graded ETQFTs. The non-semisimplicity of the theory is witnessed by the presence of non-semisimple graded linear categories associated with \textit{critical} 1-manifolds.
\end{abstract}

\maketitle

\tableofcontents

\include{preface}

\mainmatter

\include{chapter_1}
\include{chapter_2}
\include{chapter_3}
\include{chapter_4}
\include{chapter_5}

\include{chapter_6}
\include{chapter_7}

\appendix
\include{appendix_a}
\include{appendix_b}

\include{appendix_c}
\include{appendix_d}
\include{appendix_e}

\backmatter
\bibliographystyle{amsalpha}
\include{bibliography}
\printindex

\end{document}

%% file: preface.tex
%

\chapter*{Preface}

In their 2014 paper \cite{CGP14} Costantino, Geer, and Patureau developed the general theory for the construction of a new class of non-semisimple quantum invariants of closed 3-manifolds of Witten-Reshetikhin-Turaev type. Their work is based on surgery presentations, and it exploits some rather complicated algebraic structures: if $G$ and $\PGr$ are abelian groups called the \textit{structure group} and the \textit{periodicity group} respectively, and if $X$ is a ``small'' subset of $G$ called the \textit{critical set}, then their machinery provides a quantum invariant $\CGP_{\calC}$ of decorated closed 3-manifolds for every \textit{pre-modular\footnote{In \cite{CGP14} the authors use the term \textit{modular} instead of \textit{pre-modular}. Our change of terminology is motivated by the semisimple theory, where quantum invariants are defined for any non-degerate pre-modular category, and modularity is only needed in order to prove symmetric monoidality of functorial extensions.} $G$-category $\calC$ relative to $(\PGr,X)$} satisfying a certain non-degeneracy condition. Such relative pre-modular categories, introduced here in Definition \ref{D:relative_pre-modular_category}, are ribbon categories with three important features: they carry a $G$-structure, they have finiteness properties only up to the action of $\PGr$ and, more importantly, they are not necessarily semisimple, with the deviation from semisimplicity being measured by the set $X \subset G$. The Costantino-Geer-Patureau invariants are defined for closed 3-manifolds $M$ equipped with $\calC$-colored ribbon graphs $T$ and cohomology classes $\omega$ with $G$-coefficients, but not for arbitrary ones: indeed, $T$ and $\omega$ should satisfy a certain crucial \textit{admissibility} condition, which is introduced here in Section \ref{S:admissible_cobordisms}. A small survey of the construction and of some of its applications can be found in \cite{D15}.

An explicit family of examples was obtained by Costantino, Geer, and Patureau using certain non-\-de\-gen\-er\-ate relative pre-modular categories of finite-dimensional complex-weight representations of the so called \textit{unrolled quantum group} $U_q^H \sltwo$ when $q$ is a primitive $r$-th root of unity for even $r \geqslant 4$ not a multiple of $8$. The structure group for these categories is $\C / 2 \Z$, with critical set $\Z / 2 \Z$, and their periodicity group is isomorphic to $\Z$. This family of examples extends the multivariable Alexander polynomial, the Kashaev invariants, and the Akutsu-Deguchi-Ohtsuki invariants to framed colored links in arbitrary 3-manifolds. A very important phenomenon in this theory is the appearance of the abelian Reidemeister torsion, which can be recovered by the $\sltwo$ Costantino-Geer-Patureau invariants at the level $r=4$. This distinguishes dramatically the non-semisimple theory from the semisimple one, as the relationship between the Reidemeister torsion and the standard Witten-Reshetikhin-Turaev invariants remains a great open question, part of the famous asymptotic conjecture of Witten. In particular, the Costantino-Geer-Patureau invariants can be used to recover the whole classification of lens spaces, of which there exist infinitely many pairs which are not homeomorphic, and yet cannot be distinguished using any semisimple quantum invariant.

In \cite{BCGP16} Blanchet, Costantino, Geer, and Patureau extended this family of $\sltwo$ quantum invariants to a family of non-semisimple $\Z$-graded \textit{Topological Quantum Field Theories}, or \textit{TQFTs} for short, consisting in symmetric monoidal functors with sources given by categories of admissible decorated cobordisms, and with targets given by categories of $\Z$-graded vector spaces. Remark that the periodicity group of $U_q^H \sltwo$-modules appears here as the grading group for vector spaces. This is not random. Furthermore, when the level $r$, which is the order of the root of unity $q$, is a multiple of 4, then the braiding on $\Z$-graded vector spaces is the super-symmetric one. This family of $\Z$-graded TQFTs has surprising new properties which produce unprecedented results: indeed, if we consider the induced representations of mapping class groups of surfaces, the actions of Dehn twists along non-separating simple closed curves always have infinite order, something that never happened in the semisimple theory.

\section*{Main results}

\textit{Extended TQFTs}, usually abbreviated as \textit{ETQFTs}, are symmetric monoidal 2-functors from 2-categories of cobordisms to linear 2-categories. They provide localizations of TQFTs, meaning that they allow for computations of state spaces with cut-and-paste methods, while further enriching their structure, thus providing tools for deeper applications. It is therefore quite natural to ask the following:
\begin{enumerate}
 \item Is it possible to upgrade the family of $\Z$-graded TQFTs associated with unrolled quantum $\sltwo$ to a new family of non-semisimple $\Z$-graded ETQFTs?
 \item Can we find conditions for a non-degenerate pre-modular $G$-category $\calC$ relative to $(\PGr,X)$ under which $\CGP_{\calC}$ extends to a $\PGr$-graded ETQFT?
\end{enumerate}
The goal of this memoir is to give a positive answer to both of these questions. Indeed, we prove that if $\calC$ is a non-degenerate relative pre-modular category, then we only need to make a small additional requirement in order to reach our goal. We call this condition \textit{relative modularity}, we introduce it in Definition \ref{D:relative_modular_category}, and we draw analogies with the standard modularity condition coming from the semisimple theory. Since in particular this property is met by all categories of finite-dimensional weight representations of $U_q^H \sltwo$, our construction applies to all of the Blanchet-Costantino-Geer-Patureau $\Z$-graded TQFTs. Our ETQFTs will then be given by symmetric monoidal 2-functors defined over 2-categories of decorated cobordisms satisfying a certain admissibility condition, and taking values in 2-categories of complete $\PGr$-graded linear categories. We state here our main result.

\begin{theorem*}\label{main}
 If $\calC$ is a modular $G$-category relative to $(\PGr,X)$, then $\CGP_{\calC}$ extends to a $\PGr$-graded ETQFT
 \[
  \cobbE_{\calC} : \bfadCob_{\calC} \rightarrow \coCat_{\Bbbk}^{\PGr}.
 \]
\end{theorem*}

Unlike the non-semisimple ETQFTs constructed by Kerler and Lyubashenko in \cite{KL01}, the 2-functor $\cobbE_{\calC}$ is defined also for non-connected surfaces. This is possible thanks to a key ingredient which was missing from earlier attempts at non-semisimple construction: modified traces. The theory, which was developed in \cite{GPT09}, \cite{GKP11}, \cite{GKP13}, and \cite{GPV13} by Geer and Patureau together with a number of collaborators, and which is further studied in various different contexts such as \cite{GR17}, \cite{BBG18}, and \cite{H18}, plays a crucial role in our setting, as well as in other recent non-semisimple constructions such as \cite{BBG17}, \cite{DGP17}, \cite{BGPR18}, and \cite{CGPT18}. In our case, the $2$-functor $\cobbE_{\calC}$ can be described in terms of the relative modular category $\calC$ which is used as a building block: connected objects of $\smash{\bfadCob_{\calC}}$ are mapped to complete $\PGr$-graded linear categories which, up to equivalence, are given by $\PGr$-graded extensions of homogeneous subcategories of projective objects of $\calC$. Furthermore, generating 1-morphisms of $\smash{\bfadCob_{\calC}}$ can be translated into $\PGr$-graded linear functors between these $\PGr$-graded linear categories, and they can be interpreted as algebraic structures on them. This description presents us immediately with a new phenomenon: the 2-functor $\cobbE_{\calC}$ is fully monoidal with respect to disjoint union, but the images of certain objects of the domain 2-category are in general non-semisimple. This result may seem to contradict Theorem 3 of \cite{BDSV15}, which can roughly be stated as follows: not only every modular category $\calC$ determines a 1+1+1-ETQFT featuring $\calC$ as circle category, but these are essentially all the possibilities. In particular, images of closed 1-manifolds under 1+1+1-ETQFTs are necessarily semisimple. This apparent incongruity can in fact be explained as follows: all the objects of the cobordism 2-category considered in \cite{BDSV15} are fully dualizable, and semisimplicity is a consequence of this property. On the other hand, in the definition of the admissible cobordism 2-category $\smash{\bfadCob_{\calC}}$ we forbid some non-admissible cobordisms by removing them from the categories of morphisms. This produces a non-rigid 2-category featuring some \textit{critical} non-dualizable objects, thus making room for non-semisimplicity in the theory.

\section*{Where to find a quantum invariant and a TQFT inside an ETQFT}

We claimed that the symmetric monoidal 2-functor $\cobbE_{\calC}$ extends the quantum invariant $\CGP_{\calC}$, but what does this mean? The idea comes from the analogy with the following well-known property: every 2+1-TQFT yields a quantum invariant of closed 3-manifolds. Here is a recipe to read it off directly. A closed 3-manifold can be interpreted in a unique way as a cobordism between two copies of the empty surface. Therefore, a TQFT associates with it a linear endomorphism of a certain vector space which, as a consequence of monoidality, is isomorphic to the base field $\Bbbk$, the unit for the tensor product on vector spaces. Such a map is just a product with some fixed number in $\Bbbk$. Since morphisms of cobordism categories are given by isomorphism classes of cobordisms, this number actually depends only on the diffeomorphism class of the closed 3-manifold, so we can intepret it as a quantum invariant. 
This picture directly generalizes to the level of 2-categories, so that every 1+1+1-ETQFT yields a 2+1-TQFT. The procedure that recovers it is completely analogous. A closed surface can be interpreted in a unique way as a cobordism between two copies of the empty 1-manifold. Therefore, an ETQFT associates with it a linear endofunctor of a certain linear category which, as a consequence of monoidality, is equivalent to the linear category of finite-dimensional vector spaces over $\Bbbk$, which is in turn  equivalent to the unit for the complete tensor product of complete linear categories. Such a functor is, up to equivalence of linear categories, just a tensor product with some fixed vector space, which we can interpret as a state space of a TQFT. 
Everything extends word by word to the $\PGr$-graded case: every $\PGr$-graded 2+1-TQFT yields a quantum invariant of closed 3-manifolds, and every $\PGr$-graded 1+1+1-ETQFT yields a $\PGr$-graded 2+1-TQFT. In particular, hidden behind the higher structure of the symmetric monoidal 2-functor $\cobbE_{\calC}$, we still have a recipe which associates a number with every admissible decorated closed 3-manifold and a $\PGr$-graded vector space with every admissible decorated closed surface. In this sense, our construction recovers the quantum invariants of \cite{CGP14}, and generalizes the TQFTs of \cite{BCGP16} to all relative modular categories.

\section*{Outline of the construction}

The definition of the Witten-Reshetikhin-Turaev quantum invariants of closed 3-manifolds contained in \cite{RT91} makes use of surgery presentations which are colored with certain finite-dimensional simple modules over the \textit{restricted} version of quantum $\sltwo$ at an even root of unity. A famous result by Turaev, which first appeared in \cite{T94}, generalizes this construction by characterizing its building blocks: non-degenerate pre-modular categories. This allows us to associate with every such category $\calC$ a quantum invariant $\WRT_{\calC}$ of decorated closed 3-manifolds, where this time surgery presentations are colored with objects of the indexing category, rather than representations of quantum $\sltwo$. It has been known for some time that these quantum invariants can be extended to ETQFTs every time the category $\calC$ is modular. One possible way to construct these semisimple ETQFTs is to plug $\WRT_{\calC}$ into a machinery which is derived from \cite{BHMV95}, and which is called the \textit{extended universal construction}. This operation produces a 2-functor denoted
\[
 \cobfE_{\calC} : \bfCob_{\calC} \rightarrow \coCat_{\Bbbk}.
\]
Here $\bfCob_{\calC}$ is a symmetric monoidal 2-category of decorated cobordisms of dimension 1+1+1, and $\coCat_{\Bbbk}$ is the symmetric monoidal 2-category of complete linear categories. It can then be shown that $\cobfE_{\calC}$ is symmetric monoidal if, and in fact only if, as follows from Theorem 3 of \cite{BDSV15}, $\calC$ is modular.

The idea of this memoir is to try and adapt the same procedure to the non-semisimple setting of Costantino, Geer, and Patureau. Their construction, which is again based on surgery presentations, provides a quantum invariant $\CGP_{\calC}$ of admissible decorated closed 3-manifolds for every non-degenerate pre-modular $G$-category $\calC$ relative to $(\PGr,X)$. By means of the extended universal construction we can extend $\CGP_{\calC}$ to a 2-functor
\[
 \cobfE_{\calC} : \bfadCob_{\calC} \rightarrow \coCat_{\Bbbk}.
\]
In order to do this, we have to construct a suitable symmetric monoidal 2-category $\smash{\bfadCob_{\calC}}$ of admissible cobordisms decorated with $\calC$-colored ribbon graphs and cohomology classes with $G$-coefficients. However, unlike in the semisimple case, $\cobfE_{\calC}$ is not an ETQFT on the nose, and its deviation from monoidality can be parametrized by means of the periodicity group $\PGr$. Nevertheless, we can define a $\PGr$-graded extension
\[
 \cobbE_{\calC} : \bfadCob_{\calC} \rightarrow \coCat_{\Bbbk}^{\PGr}
\]
of $\cobfE_{\calC}$ with target the symmetric monoidal 2-category of complete $\PGr$-graded linear categories. The idea is to replace the image of an object of $\smash{\bfadCob_{\calC}}$ under $\cobfE_{\calC}$, which is a complete linear category, with a suitably defined complete $\PGr$-graded linear category obtained by integrating the obstruction to monoidality into its structure. We can then show that the resulting 2-functor $\cobbE_{\calC}$ is indeed symmetric monoidal, and thus provides a $\PGr$-graded ETQFT. While functoriality comes for free from the extended universal construction, there is a price to pay when it comes to characterizing the result. Indeed, some work is required in order to figure out the explicit relationship between the input relative modular category $\calC$ and the image of the associated 2-functor $\cobbE_{\calC}$. Nevertheless, the description is for the most part extremely clear, with images of connected 1-dimensional manifolds corresponding to natural subcategories of $\calC$, and with images of generating 2-dimensional cobordisms corresponding to meaningful functors between them.

\section*{Structure of the exposition}

The memoir is organized as follows: we begin by introducing our main ingredient, relative modular categories, in Chapter \ref{Ch:relative_modular_categories}; We devote Chapter \ref{Ch:admissible_cobordisms} to the detailed construction of the symmetric monoidal 2-category $\bfadCob_{\calC}$ of admissible cobordisms, which we will then use as domain for our non-semisimple ETQFTs; In Chapter \ref{Ch:extenstion_of_CGP_invariants} we review the construction of the Costantino-Geer-Patureau invariant $\CGP_{\calC}$, and we apply the extended universal construction to extend it to $2$-functor ${\hat{\bfA}_{\calC} : \bfadCob_{\calC} \rightarrow \coCat_{\Bbbk}}$; In Chapter \ref{Ch:combinatorial_topological_properties} we study how the behaviour of $\CGP_{\calC}$ under combinatorial and topological operations affects the properties of the $2$-functor $\hat{\bfA}_{\calC}$; In Chapter \ref{Ch:graded_extensions} we show $\hat{\bfA}_{\calC}$ is not an ETQFT in general, we carefully analyze its deviation from monoidality, and we upgrade the construction by defining the $2$-functor ${\cobbE_{\calC} : \bfadCob_{\calC} \rightarrow \coCat_{\Bbbk}^{\PGr}}$; Our main result is contained in Chapter \ref{Ch:symmetric_monoidality}, where we show that $\cobbE_{\calC}$ is indeed a symmetric monoidal $2$-functor, and that it contains a generalized version of the Blanchet-Costantino-Geer-Patureau TQFT; Chapter \ref{Ch:characterization_of_image} is devoted to the explicit description, in terms of the relative modular category $\calC$, of the $\PGr$-graded linear categories and of the $\PGr$-graded linear functors associated with generating objects and 1-morphisms of $\smash{\bfadCob_{\calC}}$ respectively; Appendix \ref{A:unrolled_quantum_groups} builds an explicit family of examples of relative modular categories coming from the representation theory of unrolled quantum groups at roots of unity, which can be plugged into our machinery to produce concrete examples of graded ETQFTs; Appendix \ref{A:cobordisms_with_corners}  fixes the notation we use for manifolds and cobordisms with corners; Appendix \ref{A:Maslov} generalizes the theory of Lagrangian subspaces and Maslov indices to the case of surfaces with boundary; Appendix \ref{A:symmetric_monoidal} gathers definitions and results concerning $2$-categories and symmetric monoidal structures; Appendix \ref{A:co_lin_&_gr_lin_cat} contains definitions and properties of complete linear categories and of complete $\PGr$-graded linear categories.

\aufm{Marco De Renzi}


%% file: chapter_1.tex
%
%
%

\chapter{Relative modular categories}\label{Ch:relative_modular_categories}

This chapter is devoted to the definition of the algebraic structures which are going to play the lead role in our construction: relative modular categories. These are a non-semisimple analogue to modular categories and, since their definition involves a lot of different ingredients, we first set the ground by recalling concepts like ribbon linear categories, group structures, group actions, and projective traces.

\section{Pivotal and ribbon linear categories}

Before starting, a quick remark about conventions: by appealing to Theorem \ref{T:coherence_for_2-cat}, we tacitly assume every monoidal category ${\calC}$ we consider is strict. However, monoidal functors ${F : \calC \rightarrow \calC'}$ between monoidal categories will still be equipped with coherence data, given by isomorphisms ${\varepsilon : \one' \rightarrow F(\one)}$ and natural transformations ${\mu : \otimes' \circ F \times F \Rightarrow F \circ \otimes}$. Throughout this memoir, ${\Bbbk}$ denotes an algebraically closed field, although we actually only need elements of $\Bbbk$ to admit square roots.

A \textit{linear category} is a ${\Vect_{\Bbbk}}$-enriched category, where ${\Vect_{\Bbbk}}$ is the monoidal category of vector spaces over ${\Bbbk}$, and a \textit{linear functor} between linear categories is a ${\Vect_{\Bbbk}}$-enriched functor. A linear category is \textit{additive} if it admits a zero object ${0 \in \calC}$, and if every pair of objects ${V,V' \in \calC}$ admits a direct sum, which we specify by fixing the choice of a linear functor ${\oplus : \calC \times \calC \rightarrow \calC}$. A \textit{monoidal linear category} is a linear category equipped with a monoidal structure whose tensor unit is given by a simple object ${\one \in \calC}$, and whose tensor product is given by a bilinear functor ${\otimes : \calC \times \calC \rightarrow \calC}$. A good reference for additive and monoidal categories is provided by Sections 1.2 and 2.1 of \cite{EGNO15}.

A \textit{pivotal linear category} is a rigid monoidal linear category $\calC$ together with a strict monoidal linear functor $D : \calC^{\op} \rightarrow \calC$ specifying duals of objects and morphisms, and with a monoidal natural isomorphism $\varphi : D \circ D \Rightarrow \id_{\calC}$ called the \textit{pivotal structure}. For every object $V \in \calC$ and every morphism $f \in \Hom_{\calC}(V,V')$ we use the notation $V^* := D(V)$ and $f^* := D(f)$, and we specify left and right evaluation and coevaluation morphisms
\begin{gather*}
 \lev_V \in \Hom_{\calC}(V^* \otimes V,\one), \quad \lcoev_V \in \Hom_{\calC}(\one,V \otimes V^*), \\
 \rev_V \in \Hom_{\calC}(V \otimes V^*,\one), \quad \rcoev_V \in \Hom_{\calC}(\one,V^* \otimes V)\phantom{,}
\end{gather*}
satisfying
\begin{align*}
 f^* &= (\lev_{V'} \otimes \id_{V^*}) \circ (\id_{V'^*} \otimes f \otimes \id_{V^*}) \circ (\id_{V'^*} \otimes \lcoev_V) \\
 &= (\id_{V^*} \otimes \rev_{V'}) \circ (\id_{V^*} \otimes f \otimes \id_{V'^*}) \circ (\rcoev_V \otimes \id_{V'^*}).
\end{align*}
If $c$ is a braiding on a pivotal linear category $\calC$, then we define the \textit{twist of $\calC$} to be the natural isomorphism $\vartheta : \id_{\calC} \Rightarrow \id_{\calC}$ associating with every object $V \in \calC$ the morphism
\[
 \vartheta_V := ({\id_{V}} \otimes {\rev_V}) \circ (c_{V,V} \otimes {\id_{V^*}}) \circ ({\id_V} \otimes {\lcoev_{V}}).
\]
The twist satisfies
\[
 \vartheta_{V \otimes V'} = c_{V',V} \circ c_{V,V'} \circ (\vartheta_V \otimes \vartheta_{V'})
\]
for all objects ${V,V' \in \calC}$, and we say the braiding ${c}$ is \textit{compatible with the pivotal structure} if ${\vartheta}$ satifies ${(\vartheta_V)^* = \vartheta_{V^*}}$ for every object ${V \in \calC}$. A \textit{ribbon linear category} is a pivotal linear category together with a compatible braiding. See Sections 4.7 and 8.10 of \cite{EGNO15} for a reference about pivotal and ribbon categories.

\section{Group structures and ribbon graphs}\label{S:group_structures}

In order to introduce group structures, let us fix an abelian group ${G}$. A \textit{$G$-structure} on an additive monoidal linear category ${\calC}$ is an equivalence
\[
 \calC \cong \bigoplus_{g \in G} \calC_g
\]
for a family ${\{ \calC_g \mid g \in G \}}$ of full subcategories of ${\calC}$ satisfying:
\begin{enumerate}
 \item $V \in \calC_g$, $V' \in \calC_{g'}$ $\Rightarrow$ $V \otimes V' \in \calC_{g + g'}$;
 \item $V \in \calC_g$, $V' \in \calC_{g'}$, $g \neq g'$ $\Rightarrow$ $\Hom_{\calC}(V,V') = 0$.
\end{enumerate}
An additive monoidal linear category ${\calC}$ equipped with a ${G}$-structure is called a \textit{$G$-category}\footnote{A $G$-category is usually called a $G$-graded category, see for instance \cite{TV12}, \cite{GP13}, and \cite{CGP14}. We change terminology here because we will make extensive use of the term \textit{graded category} in the enriched sense, with grading appearing on morphisms rather than objects.}. For every ${g \in G}$ the subcategory ${\calC_g}$ is called the \textit{homogeneous subcategory of index ${g}$}. If $V$ is an object of $\calC_g$, then 
we say it is a \textit{homogeneous object of index $g$}. Analogously, if $f$ is a morphism of $\calC_g$, then 
we say it is a \textit{homogeneous morphism of index $g$}. We say a ${G}$-structure on an additive pivotal linear category ${\calC}$ is \textit{compatible with the pivotal structure} if ${V^* \in \calC_{- g}}$ for every ${V \in \calC_g}$. A \textit{pivotal ${G}$-category} is then an additive pivotal linear category equipped with a compatible ${G}$-structure, and a \textit{ribbon ${G}$-category} is a pivotal ${G}$-category which is ribbon.

When working with ribbon ${G}$-categories, we will use a version of the category ${\Rib_{\calC}}$ of ${\calC}$-colored ribbon graphs which is adapted to group structures: the \textit{category ${\Rib_{\calC}^G}$ of ${G}$-homogeneous ${\calC}$-colored ribbon graphs} is the subcategory of ${\Rib_{\calC}}$ whose objects ${(\underline{\varepsilon},\underline{V})}$ are finite sequences ${\left( (\varepsilon_1,V_1), \ldots, (\varepsilon_k,V_k) \right)}$ where ${\varepsilon_i \in \{ +, - \}}$ is a sign and ${V_i}$ is a homogeneous object of ${\calC}$ for every integer ${1 \leqslant i \leqslant k}$, and whose morphisms ${T : (\underline{\varepsilon},\underline{V}) \rightarrow (\underline{\varepsilon'},\underline{V'})}$ are isotopy classes of ${\calC}$-colored ribbon graphs inside ${D^2 \times I}$ from ${P_{(\underline{\varepsilon},\underline{V})}}$ to ${P_{(\underline{\varepsilon'},\underline{V'})}}$ whose ${\calC}$-coloring is ${G}$-homogeneous, meaning that colors of edges of ${T}$ are given by homogeneous objects of ${\calC}$. Here ${P_{(\underline{\varepsilon},\underline{V})}}$ and ${P_{(\underline{\varepsilon'},\underline{V'})}}$ are the standard ${\calC}$-colored ribbon sets inside ${D^2}$ associated with the objects ${(\underline{\varepsilon},\underline{V})}$ and ${(\underline{\varepsilon'},\underline{V'})}$ of ${\Rib_{\calC}^G}$, where the term ribbon set simply denotes a discrete set of oriented framed marked points, which is what you get when you intersect transversely a ribbon graph inside a 3-manifold with a surface. We denote with ${F_{\calC}}$ the restriction of the Reshetikhin-Turaev functor associated with ${\calC}$ to the category ${\Rib_{\calC}^G}$. 
Remark that $F_{\calC}$ induces an equivalence relation, called \textit{skein equivalence}, on formal linear combinations of morphisms of $\Rib_{\calC}^G$, which we still interpret as morphisms of $\Rib_{\calC}^G$. More precisely, if $T_1,\ldots,T_m,T'_1,\ldots,T'_{m'} : (\underline{\varepsilon},\underline{V}) \rightarrow (\underline{\varepsilon'},\underline{V'})$ are morphisms of ${\Rib_{\calC}^G}$ satisfying
\[
 \sum_{i=1}^m \alpha_i \cdot F_{\calC}(T_i) = \sum_{i'=1}^{m'} \alpha'_{i'} \cdot F_{\calC}(T'_{i'}),
\]
for some coefficients $\alpha_1,\ldots,\alpha_m,\alpha'_1,\ldots,\alpha'_{m'} \in \Bbbk$, and if we set
\[
 T := \sum_{i=1}^m \alpha_i \cdot T_i, \quad T' := \sum_{i'=1}^{m'} \alpha'_{i'} \cdot T'_{i'},
\]
then we say $T$ is \textit{skein equivalent} to $T'$, and we write $T \doteq T'$.

\section{Group actions and group realizations}\label{S:group_actions}

In order to define relative modular categories we will need a special kind of group action. Let us fix for this section an abelian group $\PGr$, and let us denote with $\PGr$ also the discrete category over $\PGr$. This category has a natural monoidal structure, with tensor unit given by the identity element of $\PGr$, and with tensor product given by the operation of $\PGr$. An \textit{action of $\PGr$ on a linear category $\calC$} is a monoidal functor $R : \PGr \rightarrow \End_{\Bbbk}(\calC)$, where $\End_{\Bbbk}(\calC)$ denotes the category of linear endofunctors of $\calC$, which is a monoidal category with tensor unit given by the identity functor of $\calC$, and with tensor product given by composition. If $R$ is an action of $\PGr$ on $\calC$, then we denote with $R^k$ the linear endofunctor $R(k) \in \End_{\Bbbk}(\calC)$ for every $k \in \PGr$. Remark that actions preserve properties of objects of $\calC$ like simplicity: indeed, for every element $k \in \PGr$ and for all objects $V,V' \in \calC$, the linear map $R^k_{V,V'}$ from $\Hom_{\calC}(V,V')$ to $\Hom_{\calC}(R^k(V),R^k(V'))$ is invertible. Its inverse is built using the action of $-k \in \PGr$ and the structural morphisms $\varepsilon : \id_{\calC} \Rightarrow R^0$ and $\mu_{-k,k} : R^{-k} \circ R^k \Rightarrow R^{-k+k}$ of $R$, and it maps every $f \in \Hom_{\calC}(R^k(V),R^k(V'))$ to
\[
 \varepsilon_{V'}^{-1} \circ ( \mu_{-k,k} )_{V'} \circ R^{-k}_{R^k(V),R^k(V')}(f) \circ ( \mu_{-k,k}^{-1} )_V \circ \varepsilon_V \in \Hom_{\calC}(V,V').
\]
An action of $\PGr$ on a linear category $\calC$ is \textit{free} if it induces a free action on the set of isomorphism classes of simple objects of $\calC$. Remark that if $\calC$ is a monoidal linear category, then every monoidal functor $F : \PGr \rightarrow \calC$ yields an action $R^F$ of $\PGr$ on $\calC$ determined by the right translation linear endofunctors $R^{F(k)} \in \End_{\Bbbk}(\calC)$ mapping every object $V$ of $\calC$ to $V \otimes F(k)$, and every morphism $f$ of $\calC$ to $f \otimes \id_{F(k)}$ for every $k \in \PGr$.

A \textit{realization of $\PGr$ in a ribbon linear category $\calC$} is a monoidal functor $\sigma : \PGr \rightarrow \calC$ satisfying $\vartheta_{\sigma(k)} = \id_{\sigma(k)}$ for every $k \in \PGr$. For every realization $\sigma$ of $\PGr$ in $\calC$ we denote with $\sigma(\PGr)$ the set of objects $\{ \sigma(k) \in \calC \mid k \in \PGr \}$. Let us collect some properties of group realizations: first of all, for every $k \in \PGr$ we have
\begin{align*}
 \dim_{\calC}(\sigma(k))^2 
 &= \dim_{\calC}(\sigma(-k)) \dim_{\calC}(\sigma(k)) \\
 &= \dim_{\calC}(\sigma(-k) \otimes \sigma(k)) \\
 &= \dim_{\calC}(\sigma(0)) \\
 &= 1.
\end{align*}
Furthermore, for every $k \in \PGr$ we have
\begin{align*}
 \dim_{\calC}(\sigma(k))
 &= \tr_{\calC}(\id_{\sigma(k)}) \\
 &= \tr_{\calC}(\vartheta_{\sigma(k)}) \\
 &= \tr_{\calC}(c_{\sigma(k),\sigma(k)}).
\end{align*}
A realization $\sigma$ of $\PGr$ in a ribbon linear category $\calC$ is \textit{free} if the associated action $R^{\sigma}$ of $\PGr$ on $\calC$ by right translation linear endofunctors is free.

\section{Projective traces and ambidextrous objects}

The last ingredient for the definition of relative modular categories is the most important one, and the one that was crucially missing from earlier attempts at non-semisimple constructions. In this section, we will recall the notion of a trace on the ideal of projective objects in a ribbon linear category, which is central in the theory of renormalized topological invariants developed by Geer and Patureau together with many collabrators. We say an object ${V}$ of a category ${\calC}$ is a \textit{retract} of another object ${V'}$ of ${\calC}$ if there exist morphisms ${f \in \Hom_{\calC}(V,V')}$ and ${f' \in \Hom_{\calC}(V',V)}$ satisfying ${f' \circ f = \id_V}$. A subcategory ${\calC'}$ of ${\calC}$ is said to be \textit{closed under retraction} if every retract of every object of ${\calC'}$ is an object of ${\calC'}$. Moreover a subcategory ${\calC'}$ of a monoidal category ${\calC}$ is said to be \textit{absorbent} if for all objects ${V}$ of ${\calC}$ and ${V'}$ of ${\calC'}$ the tensor products ${V \otimes V'}$ and ${V' \otimes V}$ are objects of ${\calC'}$. An \textit{ideal} ${\scrI}$ of a monoidal category ${\calC}$ is then a full subcategory of ${\calC}$ which is absorbent and closed under retraction. If ${\calC}$ is a pivotal category, then the full subcategory ${\Proj(\calC)}$ of projective objects of ${\calC}$ is an ideal of ${\calC}$, and it coincides with the full subcategory ${\Inj(\calC)}$ of injective objects. Moreover, thanks to Lemma 17 of \cite{GPV13}, if ${V}$ is a projective object of ${\calC}$ with epic evaluation then ${\Proj(\calC)}$ coincides with the ideal ${\scrI_{V}}$ generated by ${V}$, that is the smallest ideal of ${\calC}$ containing ${V}$.

The \textit{right partial trace} of an endomorphism ${f \in \End_{\calC}(V \otimes V')}$ is the endomorphism ${\tr_{\rmr}(f) \in \End_{\calC}(V)}$ given by
\[
 \tr_{\rmr}(f) := ({\id_V} \otimes {\rev_{V'}}) \circ (f \otimes {\id_{V'^*}}) \circ ({\id_V} \otimes {\lcoev_{V'}}).
\]
A \textit{trace ${\rmt}$ on the ideal of projective objects ${\Proj(\calC)}$} of a ribbon linear category ${\calC}$, sometimes also called a \textit{projective trace on ${\calC}$}, is a family 
\[
 \rmt := \{ \rmt_V : \End_{\calC}(V) \rightarrow \C \mid V \in \Proj(\calC) \}
\]
of linear maps satisfying:
\begin{enumerate}
 \item \textit{Ciclicity}: ${\rmt_{V}(f' \circ f) = \rmt_{V'}(f \circ f')}$ for all objects ${V,V'}$ of ${\Proj(\calC)}$ and for all morphisms ${f \in \Hom_{\calC}(V,V')}$ and ${f' \in \Hom_{\calC}(V',V)}$;
 \item \textit{Partial trace}: ${\rmt_{V \otimes V'} (f)= \rmt_V(\tr_{\rmr}(f))}$ for all objects ${V}$ of ${\Proj(\calC)}$ and ${V'}$ of ${\calC}$ and for every morphism ${f \in \End_{\calC}(V \otimes V')}$.
\end{enumerate}
Thanks to the braiding of ${\calC}$, a projective trace also satisfies ${\rmt_{V \otimes V'} (f) = \rmt_{V'}(\tr_{\rml}(f))}$ for all objects ${V}$ of ${\calC}$ and ${V'}$ of ${\Proj(\calC)}$ and for every morphism ${f \in \End_{\calC}(V \otimes V')}$, where the endomorphism ${\tr_{\rml}(f) \in \End_{\calC}(V')}$, called the \textit{left partial trace} of ${f}$, is given by
\[
 \tr_{\rml}(f) := ({\lev_V} \otimes {\id_{V'}}) \circ ({\id_{V^*}} \otimes f) \circ ({\rcoev_V} \otimes {\id_{V'}}).
\]
See Figure \ref{F:partial_traces} for a graphical representation of partial trace operations. Remark that our convention differs from the one used in \cite{T94} for what concerns orientations of edges. This is done in order to have a coherent notation for the orientation induced on the boundary throughout the exposition.

\begin{figure}[hbtp]\label{F:partial_traces}
 \centering
 \includegraphics{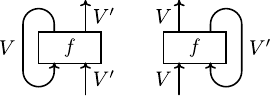}
 \caption{Partial traces $\tr_{\rml}(f)$ and $\tr_{\rmr}(f)$ of a morphism ${f \in \End_{\calC}(V \otimes V')}$.}
\end{figure}

A trace ${\rmt}$ on ${\Proj(\calC)}$ is \textit{non-zero} if there exists an endomorphism ${f}$ of a projective object ${V}$ satisfying ${\rmt_V(f) \neq 0}$, and it is \textit{non-degenerate} if
\[
 \rmt_V(\cdot \circ \cdot) : \Hom_{\calC}(V',V) \otimes \Hom_{\calC}(V,V') \rightarrow \Bbbk
\]
is a non-degenerate pairing for every object ${V}$ in ${\Proj(\calC)}$ and for every object ${V'}$ in ${\calC}$. Every trace ${\rmt}$ on ${\Proj(\calC)}$ determines a \textit{projective dimension} defined as ${\rmd(V) := \rmt_V(\id_V)}$ for every ${V \in \calC}$.

The existence of projective traces is ensured by the existence of a special kind of objects: if ${\calC}$ is a ribbon linear category then we say a simple object ${V}$ of ${\calC}$ is \textit{ambidextrous} if
\[
 \tr_{\rml}(f) = \tr_{\rmr}(f)
\]
for every ${f \in \End_{\calC}(V \otimes V)}$. If ${\calC}$ is a ribbon linear category and if ${V}$ is a projective ambidextrous object of ${\calC}$ then, thanks to Theorem 3.3.2 of \cite{GKP11}, there exists a unique trace ${\rmt}$ on ${\Proj(\calC)}$ which satisfies ${\rmd(V) = 1}$. Moreover, it follows from Section 5 of \cite{GPV13} that if ${\rmt}$ is a non-zero trace on ${\Proj(\calC)}$ then ${\rmd(V) \neq 0}$ for every simple projective object ${V}$ of ${\calC}$ whose left evaluation is an epimorphism. We call such a ${V}$ an \textit{object with epic evaluation}. Corollary 3.2.1 of \cite{GKP13} gives a condition ensuring the existence of ambidextrous objects. We also remark that sometimes traces on more general ideals exist, but we will not focus on them in this memoir, as we will crucially exploit properties of projective objects in our construction.

\section{Relative pre-modular categories}\label{S:relative_pre-modular_categories}

Relative pre-modular categories are a non-semisimple version of pre-modular categories carrying both a group structure and a group action. However, in order for them to be manageable, we need to be able to control their non-semisimplicity. In particular, we want their generic homogeneous subcategory with respect to the group structure to be semisimple, with the non-semisimple part of the category localized in a set of critical indices which has to be somewhat small, in a sense to be made precise. Furthermore, although we do not ask for generic homogeneous subcategories to admit a finite number of isomorphism classes of simple objects, as this does not happen in significant examples, we want them to admit a finite number of orbits of such isomorphism classes with respect to the group action. Before explaning the meaning of all this in further detail, let us start by quickly fixing our terminology. We say a set ${D = \{ V_i \in \calC \mid  i \in \rmI \}}$ of objects of a linear category ${\calC}$ is a \textit{dominating set} if for every ${V \in \calC}$ there exist ${V_{i_1}, \ldots, V_{i_m} \in D}$ and ${r_j \in \Hom_{\calC}(V_{i_j},V)}$ and ${s_j \in \Hom_{\calC}(V,V_{i_j})}$ for every integer ${1 \leqslant j \leqslant m}$ satisfying 
\[
 \id_V = \sum_{j=1}^m r_j \circ s_j,
\]
in which case we also say \textit{${D}$ dominates ${\calC}$}. Furthermore, we say ${D}$ is \textit{completely reduced} if $\dim_{\C} \Hom_{\calC}(V_i,V_j) = \delta_{ij}$
for all ${i,j \in \rmI}$. A \textit{semisimple category} is then a linear category ${\calC}$ together with a completely reduced dominating set. Next, we say a subset ${X}$ of an abelian group ${G}$ is \textit{small symmetric} if:
\begin{enumerate}
 \item ${X = -X}$;
 \item ${\displaystyle G \not\subset \bigcup_{i=1}^m (g_i + X)}$ for all ${m \in \N}$ and all ${g_1, \ldots, g_m \in G}$.
\end{enumerate}
If ${G}$ is an abelian group and ${X \subset G}$ is a small symmetric subset then all elements in ${G \smallsetminus X}$ are called \textit{generic}, and all elements in ${X}$ are called \textit{critical}. We are now ready to give our first definition, which is insipired by Definition 2 of \cite{GP13} and by Definition 4.2 of \cite{CGP14}.

\begin{definition}\label{D:relative_pre-modular_category}
 If ${G}$ and ${\PGr}$ are abelian groups, and if ${X \subset G}$ is a small symmetric subset, then a \textit{pre-modular ${G}$-category relative to ${(\PGr,X)}$} is a ribbon linear category ${\calC}$ together with:
 \begin{enumerate}
  \item A compatible ${G}$-structure on ${\calC}$;
  \item A free realization ${\sigma : \PGr \rightarrow \calC_0}$;
  \item A non-zero trace ${\rmt}$ on ${\Proj(\calC)}$.
 \end{enumerate}
 These data satisfy the following conditions:
 \begin{enumerate}
  \item \textit{Generic semisimplicity}: For every ${g \in G \smallsetminus X}$ there exists a finite ordered set ${\Theta(\calC_g) = \{ V_i \in \calC_g \mid i \in \rmI_g \}}$ of objects with epic evaluation such that
  \[
   \Theta(\calC_g) \otimes \sigma(\PGr) := \{ V_i \otimes \sigma(k) \mid i \in \rmI_g, k \in \PGr \}
  \]
  is a completely reduced dominating set for ${\calC_g}$;
  \item \textit{Compatibility}: There exists a bilinear map ${\psi : G \times \PGr \rightarrow \Bbbk^*}$ such that 
  \[
   c_{\sigma(k),V} \circ c_{V,\sigma(k)} = \psi(g,k) \cdot \id_{V \otimes \sigma(k)}
  \]
  for every ${g \in G}$, for every ${V \in \calC_g}$, and for every ${k \in \PGr}$.
 \end{enumerate}
\end{definition}

If $\calC$ is a pre-modular $G$-category relative to $(\PGr,X)$, then $G$ is called the \textit{structure group}, $\PGr$ is called the \textit{periodicity group}, and $X$ is called the \textit{critical set} of ${\calC}$. Bilinearity of $\psi$ means
\begin{gather*}
 \psi(0,k) = 1, \quad \psi(g+g',k) = \psi(g,k) \psi(g',k), \\
 \psi(g,0) = 1, \quad \psi(g,k+k') = \psi(g,k) \psi(g,k')\phantom{,}
\end{gather*}
for all $g,g' \in G$ and all $k,k' \in \PGr$. In the trivial case $G = \PGr = \{ 0 \}$ and $X = \varnothing$ a pre-modular ${G}$-category relative to ${(\PGr,X)}$ is just a pre-modular category, as explained in Section \ref{S:relation_with_semisimple_theory}.

Let us discuss Definition \ref{D:relative_pre-modular_category} by comparing it with Definition 4.2 of \cite{CGP14}. First of all, remark we are asking the non-zero trace $\rmt$ to be defined on the ideal $\Proj(\calC)$, instead of an arbitrary one, because, as we already mentioned earlier, we will rely on specific properties of projective objects in order to perform our construction. Here is the very first example, which we will use multiple times.

\begin{remark}\label{R:epic_evaluations}
 The requirement about epic evaluations in Definition \ref{D:relative_pre-modular_category} has the following direct consequence: if $U \in \Proj(\calC)$ and if $V_i \in \Theta(\calC_g)$, then the morphism $\id_U \otimes \rev_{V_i}$ is still epic, because $U$ is dualizable. Therefore, since $U$ is projective, there exists a section $s_{U,i}$ of $\id_U \otimes \rev_{V_i}$, which is a morphism in $\Hom_{\calC}(U,U \otimes V_i \otimes V_i^*)$ satisfying $(\id_U \otimes \rev_{V_i}) \circ s_{U,i} = \id_U$.
\end{remark}

This observation has a number of consequences.

\begin{proposition}\label{P:non-degeneracy_of_trace}
 If $\calC$ is a pre-modular $G$-category relative to $(\PGr,X)$ then the trace $\rmt$ on $\Proj(\calC)$ is non-degenerate.
\end{proposition}

\begin{proof}
 Let us consider a homogeneous projective object $U$, a homogeneous object $U'$ of the same index, and a non-zero morphism $f \in \Hom_{\calC}(U,U')$ between them. We claim that
 \[
  ((\id_{U^*} \otimes f) \circ \rcoev_U) \otimes \id_{V_i} \in \Hom_{\calC}(V_i,U^* \otimes U' \otimes V_i)
 \]
 is a non-zero morphism of $\calC_g$ for every $V_i \in \Theta(\calC_g)$. Indeed, it can be written as
 \[
  \left( \id_{U^*} \otimes \left( f \circ (\id_U \otimes \rev_{V_i}) \right) \otimes \id_{V_i} \right) \circ \left( \lcoev_U \otimes \id_{V_i} \otimes \rcoev_{V_i} \right),
 \]
 which is the image of $f \circ (\id_U \otimes \rev_{V_i})$ under the linear isomorphism between the morphism spaces $\Hom_{\calC}(U \otimes V_i \otimes V_i^*,U')$ and $\Hom_{\calC}(V_i,U^* \otimes U' \otimes V_i)$ induced by the pivotal structure. But $f \circ (\id_U \otimes \rev_{V_i})$ is a non-zero morphism because
 \[
  f \circ (\id_U \otimes \rev_{V_i}) \circ s_{U,i} = f
 \]
 for every section $s_{U,i}$ of $\id_U \otimes \rev_{V_i}$. This proves our claim. Now, since $\calC_g$ is semisimple and $V_i$ is simple, there exists a retraction $r_{f,i}$ of ${\left( (\id_{U^*} \otimes f) \circ \rcoev_U \right) \otimes \id_{V_i}}$, which is a morphism in $\Hom_{\calC}(U^* \otimes U' \otimes V_i,V_i)$ satisfying
 \[
  r_{f,i} \circ \left( \left( (\id_{U^*} \otimes f) \circ \rcoev_U \right) \otimes \id_{V_i} \right) = \id_{V_i}.
 \]
 Let $f' \in \Hom_{\calC}(U',U)$ denote the morphism
 \[
  ( \id_U \otimes \rev_{V_i} ) \circ ( \id_U \otimes r_{f,i} \otimes \id_{V_i^*} ) \circ ( \lcoev_U \otimes \id_{U'} \otimes \lcoev_{V_i} ).
 \]
 Using first ciclicity and then left and right partial trace properties, the three endomorphisms represented in Figure \ref{F:proof_non-degeneracy_of_trace} have the same trace.
 \begin{figure}[t]\label{F:proof_non-degeneracy_of_trace}
  \centering
  \includegraphics{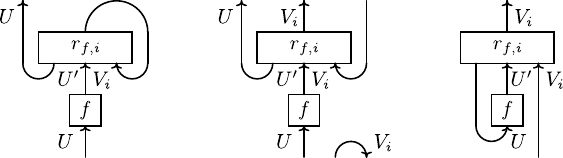}
  \caption{Endomorphisms of $\calC$ sharing the same trace.}
 \end{figure}
 This means precisely
 \begin{align*}
  &\rmt_U ( f' \circ f ) = \rmt_U \left( ( \id_U \otimes \rev_{V_i} ) \circ ( \id_U \otimes r_{f,i} \otimes \id_{V_i^*} ) \circ ( \lcoev_U \otimes f \otimes \lcoev_{V_i} ) \right) \\
  &\hspace{\parindent} = \rmt_{U \otimes V_i \otimes V_i^*} \left( ( \id_U \otimes r_{f,i} \otimes \id_{V_i^*} ) \circ ( \lcoev_U \otimes f \otimes \lcoev_{V_i} ) \circ ( \id_U \otimes \rev_{V_i} ) \right) \\
  &\hspace{\parindent} = \rmt_{V_i} \left( \tr_{\rml} \left( \tr_{\rmr} \left( ( \id_U \otimes r_{f,i} \otimes \id_{V_i^*} ) \circ ( \lcoev_U \otimes f \otimes \lcoev_{V_i} ) \circ ( \id_U \otimes \rev_{V_i} ) \right) \right) \right) \\
  &\hspace{\parindent} = \rmt_{V_i} \left( r_{f,i} \circ \left( \left( (\id_{U^*} \otimes f) \circ \rcoev_U \right) \otimes \id_{V_i} \right) \right) \\
  &\hspace{\parindent} = \rmt_{V_i} ( \id_{V_i} ) \\
  &\hspace{\parindent} \neq 0. \qedhere
 \end{align*}
\end{proof}

A second direct consequence of Remark \ref{R:epic_evaluations} is that we can do without condition (8) in Definition 4.2 of \cite{CGP14}, as explained in Remark 4.4 of \cite{CGP14}. Also, another observation: our notion of free realization is actually different from the one given in \cite{CGP14}, because we are asking nothing of the braiding. However, the compatibility condition in Definition \ref{D:relative_pre-modular_category} implies
\[
 c_{\sigma(k'),\sigma(k)} \circ c_{\sigma(k),\sigma(k')} = \id_{\sigma(k) \otimes \sigma(k')}
\]
for all $k,k' \in \PGr$. Thus, when we restrict the free realization $\sigma$ to the subcategory $\PGr_+$ of $\PGr$ whose set of objects is the kernel of the homomorphism ${\dim_{\calC}} \circ \sigma : \PGr \rightarrow \Z^*$, we indeed obtain a free realization in the sense of Section 4.3 of \cite{CGP14}. Of course, we loose nothing by having this free realization as part of the action of a larger group, which is what we are doing here. Remark this means that for every $V \in \Proj(\calC)$ and every $k \in \PGr$ we have
\[
 \rmd(V \otimes \sigma(k)) = \rmd(V) \dim_{\calC}(\sigma(k)).
\]

\begin{remark}\label{R:critical_fusion}
 We can fix, for every critical $x \in X$, a generic $g_x \in G \smallsetminus X$ satisfying $x + g_x \in G \smallsetminus X$, which always exists because $X$ is small. We can also fix some $i_x \in \rmI_{g_x}$, which specifies a projective simple object $V_x := V_{i_x} \in \Theta(\calC_{g_x})$. Then the linear category $\Proj(\calC_x)$ is dominated by the set 
 \[
  \Theta(\calC_{x + g_x}) \otimes \sigma(\PGr) \otimes \{ V^*_x \} 
  := \left\{ V_i \otimes \sigma(k) \otimes V^*_x \mid i \in \rmI_{x + g_x}, k \in \PGr \right\}.
 \]
 In other words, if $V \in \Proj(\calC_x)$, then there exist $V_{i_1}, \ldots, V_{i_m} \in \Theta(\calC_{x + g_x})$ and $s_j \in \Hom_{\calC}(V,V_{i_j} \otimes \sigma(k_j) \otimes V^*_x)$ and $r_j \in \Hom_{\calC}(V_{i_j} \otimes \sigma(k_j) \otimes V^*_x,V)$ for some $k_j \in \PGr$ and for every integer 
 $1 \leqslant j \leqslant m$ satisfying
 \[
  \id_V = \sum_{j=1}^m r_j \circ s_j.
 \]
 Indeed, since $V$ is projective, then, thanks to Remark \ref{R:epic_evaluations}, there exists a section $s_{V,x} \in \Hom_{\calC}(V,V \otimes V_x \otimes V_x^*)$ of the epimorphism $\id_V \otimes \rev_{V_x}$. Thus, since $V \otimes V_x$ has index $x + g_x$, and since $\calC_{x + g_x}$ is semisimple and dominated by $\Theta(\calC_{x + g_x}) \otimes \sigma(\PGr)$, the claim follows. Remark however that in this case the dominating set of objects $\Theta(\calC_{x + g_x}) \otimes \sigma(\PGr) \otimes \{ V^*_x \}$ is not completely reduced.
\end{remark}

\section{Main definition}\label{S:main_definition}

Relative modular categories are relative pre-modular categories satisfying a strong non-degeneracy condition. Their definition requires some preliminary work. If $\calC$ is a pre-modular $G$-category relative to $(\PGr,X)$ then the associated \textit{Kirby color of index $g \in G \smallsetminus X$} is the formal linear combination of objects
\[
 \Omega_g := \sum_{i \in \rmI_g} \rmd(V_i) \cdot V_i.
\]
If $T$ is a $G$-homogeneous $\calC$-colored ribbon graph, and if $K \subset D^2 \times I$ is a framed knot disjoint from $T$, the operation of labeling $K$ with $\Omega_g$ for some generic index $g \in G \smallsetminus X$ produces a formal linear combination of $G$-homogeneous $\calC$-colored ribbon graphs which we still interpret as a $G$-homogeneous $\calC$-colored ribbon graph, whose image under the Rehetikhin-Turaev functor $F_{\calC}$ is obtained by expanding linearly. Remark that here, as well as everywhere else in this memoir, we tacitly assume all knots and links to be equipped with a fixed orientation.

Let us pause for a moment to make an important remark. Because of our definition of free realization, which does not require categorical dimensions of objects in the image to be equal to $1$, what we call a Kirby color is not a Kirby color in the sense Definition 4.6 of \cite{CGP14}. Indeed, if we wanted to follow the original notion, we would need to consider representatives for orbits of isomorphism classes of simple objects of $\calC_g$ under the action of $\PGr_+ := \left\{ k \in \PGr \mid \dim_{\calC}(\sigma(k)) = 1 \right\}$, which is a subgroup of index at most $2$ in $\PGr$. In other words, we would need to consider, for every $g \in G \smallsetminus X$, the linear combination
\[
 \tilde{\Omega}_g := \sum_{[k] \in \PGr/\PGr_+} \dim_{\calC}(\sigma(k)) \sum_{i \in \rmI_g} \rmd(V_i) \cdot V_i \otimes \sigma(k).
\]
However, we claim that, for the purposes of our construction, the original Kirby color $\tilde{\Omega}_g$ is actually equivalent to $\left| \PGr/\PGr_+ \right| \cdot \Omega_g$. In order to make this claim precise, we need a preliminary definition which is inspired by Section 5.2 of \cite{CGP14}, so let $T$ be a $G$-homogeneous $\calC$-colored ribbon graph, let $K \subset D^2 \times I$ be a $G$-homogeneous $\calC$-colored framed knot disjoint from $T$, and let $\varSigma_K \subset D^2 \times I$ be a Seifert surface for a parallel copy of $K$ determined by the framing. If we suppose $\varSigma_K$ is transverse to $T \cup K$, and if for every point $p \in \varSigma_K \cap (T \cup K)$ we denote with $\varepsilon_p \in \{ +,- \}$ the sign of the intersection of $\varSigma_K$ and $T \cup K$ at $p$, and with $g_p \in G$ the index of the color of the edge of $T \cup K$ intersecting $\varSigma_K$ at $p$, then we can set
\[
 \lk^G (K,T \cup K) := \sum_{p \in \varSigma_K \cap (T \cup K)} \varepsilon_p g_p \in G.
\]

\begin{lemma}\label{L:equivalence_of_Kirby_colors}
 If $T$ is a $G$-homogeneous $\calC$-colored ribbon graph, if $K \subset D^2 \times I$ is an $\Omega_g$-colored framed knot disjoint from $T$, if $\tilde{K} \subset D^2 \times I$ is obtained from $K$ by replacing the color $\Omega_g$ with $\tilde{\Omega}_g$, and if $\lk^G(K,T \cup K) = 0$, then
 \[
  \left| \PGr/\PGr_+ \right| \cdot T \cup K \doteq T \cup \tilde{K}.
 \]
\end{lemma}

\begin{proof}
 The proof is analogous to the one of Lemma 5.8 of \cite{CGP14}. Indeed, suppose that $k \in \PGr$ determines a non-trivial generator $[k] \in \PGr/\PGr_+$, so that
 \[
  \tilde{\Omega}_g = \sum_{i \in \rmI_g} \rmd(V_i) \cdot (V_i - V_i \otimes \sigma(k)).
 \]
 Then, up to skein equivalence, we have $T \cup \tilde{K} \doteq T \cup K - T \cup K \cup K'$, where $K'$ is a $\sigma(k)$-colored parallel copy of $K$ determined by the framing. Now the compatibility condition between the $G$-structure and the $\PGr$-action of $\calC$ allows us to turn every undercrossing of the $\sigma(k)$-colored framed knot $K'$ with the rest of the tangle into an overcrossing, and the price we need to pay for this operation is a coefficient $\psi(\lk^G(K,T \cup K),k) = 1$. By applying the same strategy to self-crossings of the $\sigma(k)$-colored component we can turn it into an unknot with framing either 0 or 1. The result now follows from $\tr_{\calC}(\vartheta_{\sigma(k)}) = \dim_{\calC}(\sigma(k)) = -1$.
\end{proof}

\begin{remark}\label{R:equivalence_of_Kirby_colors}
 As a consequence of Lemma \ref{L:equivalence_of_Kirby_colors}, our Kirby colors have all the properties of the original ones. In particular, they satisfy Lemma 5.9 of \cite{CGP14}, which implies the handle-slide invariance of the corresponding graphical calculus. The reason why the technical hypothesis of Lemma \ref{L:equivalence_of_Kirby_colors} is not restrictive is explained in Remark 5.6 of \cite{CGP14}. See also Proposition 6.20 of \cite{BCGP16} for the analogous statement in the special case of unrolled quantum $\sltwo$. 
\end{remark}

 Next, for a pre-modular $G$-category $\calC$ relative to $(\PGr,X)$ there exist constants $\Delta_-,\Delta_+ \in \Bbbk$, called the \textit{negative} and the \textit{positive stabilization coefficient} respectively, realizing the skein equivalences of Figure \ref{F:stabilization_coefficients}. 
\begin{figure}[b]\label{F:stabilization_coefficients}
 \centering
 \includegraphics{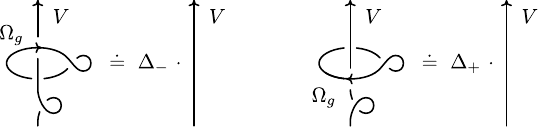}
 \caption{Skein equivalences defining $\Delta_-$ and $\Delta_+$.}
\end{figure}
Observe that $\Delta_-$ and $\Delta_+$ do not depend neither on $V \in \calC_g$, nor on $g \in G \smallsetminus X$, as proved in Lemma 5.10 of \cite{CGP14}. We call \textit{negative} and \textit{positive stabilization along a $V$-colored edge} the operations that replace the right-hand sides with the left-hand sides of the first and of the second skein equivalence respectively.

\begin{definition}\label{D:non-degeneracy}
 A pre-modular $G$-category relative to $(\PGr,X)$ is \textit{non-\-de\-gen\-er\-ate} if $\Delta_-\Delta_+ \neq 0$.
\end{definition}

Non-degenerate relative pre-modular categories give rise to quantum invariants of admissible closed $3$-manifolds, as explained in Section \ref{S:CGP_invariants}. What we need in order to extend these invariants to ETQFTs is a slightly stronger non-degeneracy condition, which provides the main definition of this memoir.

\begin{definition}\label{D:relative_modular_category}
 A \textit{modular $G$-category relative to $(\PGr,X)$} is a pre-modular $G$-category relative to $(\PGr,X)$ which admits a \textit{relative modularity parameter} $\zeta \in \Bbbk^*$ realizing the skein equivalence of Figure \ref{F:relative_modularity} for all $g,h \in G \smallsetminus X$ and for all $i,j \in \rmI_g$.
\end{definition}

\begin{figure}[hbtp]\label{F:relative_modularity}
 \centering
 \includegraphics{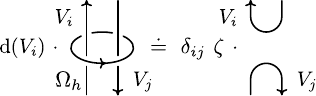}
 \caption{Relative modularity condition.}
\end{figure}

There are many examples of relative modular categories. The main family comes from representations of unrolled quantum groups at roots of unity. An explicit construction can be found in Appendix \ref{A:unrolled_quantum_groups}, as well as in references within. The next result shows the non-degeneracy condition of Definition \ref{D:relative_modular_category} is stronger than the one of \ref{D:non-degeneracy}.

\begin{proposition}\label{P:non-degeneracy_of_relative_modular_categories}
 If $\calC$ is a modular $G$-category relative to $(\PGr,X)$ then $\calC$ is non-\-de\-gen\-er\-ate, with $\Delta_- \Delta_+ = \zeta \neq 0$.
\end{proposition}

\begin{figure}[tbh]\label{F:non-degeneracy_of_relative_modular_categories}
 \centering
 \includegraphics{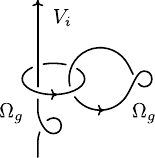}
 \caption{$\calC$-Colored ribbon tangle $T_i$ witnessing $\Delta_- \Delta_+ = \zeta$.}
\end{figure}

\begin{proof}
 We claim the non-degeneracy of $\calC$ can be proved by considering the $\calC$-colored ribbon tangle $T_i$ represented in Figure \ref{F:non-degeneracy_of_relative_modular_categories} for any $i \in \rmI_g$. Indeed, remark that $T_i$ can be obtained from the identity $\calC$-colored ribbon tangle $\id_{(+,V_i)}$ by first performing a negative stabilization on the $V_i$-colored edge, and then performing a positive stabilization on the resulting $\Omega_g$-colored framed knot. Thus we have
 \[
  \Delta_- \Delta_+ \cdot \id_{V_i} \doteq T_i \doteq \zeta \cdot \id_{V_i},
 \]
 where the second skein equivalence is a consequence of the relative modularity condition.
\end{proof}

\section{Relation with semisimple theory}\label{S:relation_with_semisimple_theory}

At first sight the relative modularity condition of Definition \ref{D:relative_modular_category} may seem strange, but it is actually a direct generalization of the standard modularity condition, which is usually stated in terms of the invertibility of the positive Hopf link matrix. In order to explain this, let us consider a pre-modular $G$-category $\calC$ relative to $(\PGr,X)$. In the trivial case $G = \PGr = \{ 0 \}$ we have $X = \varnothing$, which means $\Proj(\calC) = \calC$. Then, thanks to Lemma 16 and Corollary 17 of \cite{GPT09}, every projective dimension is just a scalar multiple of the categorical dimension, which means we may as well directly suppose $\rmd = \dim_{\calC}$. In particular, $\calC$ is semisimple, with $\Theta(\calC) = \{ V_i  \mid i \in \rmI \}$ providing a finite completely reduced dominating set, so it is pre-modular also in the usual sense. In this case, we drop the term \textit{relative} from the modularity condition of Definition \ref{D:relative_modular_category}. If $\calC$ is a pre-modular category then the \textit{positive Hopf link matrix of $\calC$} is the $\rmI^2$-matrix $(S^+_{ij})_{(i,j) \in \rmI^2}$ whose $(i,j)$-th entry ${S^+_{ij}}$ is given by the evaluation of the Reshetikhin-Turaev functor $F_{\calC}$ against the $\calC$-colored framed link depicted in the left-hand part of Figure \ref{F:S-matrix} for every $(i,j) \in \rmI^2$.

\begin{figure}[hbtp]\label{F:S-matrix}
 \centering
 \includegraphics{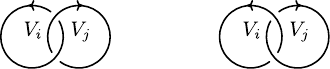}
 \caption{The $\calC$-colored framed links representing the $(i,j)$-th entries of the positive and of the negative Hopf link matrix of $\calC$.}
\end{figure}

\begin{proposition}\label{P:equivalence_of_modularity_conditions}
 A pre-modular category ${\calC}$ satisfies the modularity condition of Definition \ref{D:relative_pre-modular_category} if and only if its positive Hopf link matrix is invertible.
\end{proposition}

\begin{proof}
 If the positive Hopf link matrix of $\calC$ is invertible, then it is known that the modularity condition holds, as follows from Exercise 3.10.2 of \cite{T94}. Therefore, let us suppose $\calC$ satisfies the modularity condition, and let us consider the negative Hopf link matrix of $\calC$, that is the $\rmI^2$-matrix $(S^-_{ij})_{(i,j) \in \rmI^2}$ whose $(i,j)$-th entry $S^-_{ij}$ is defined as the evaluation of the Reshetikhin-Turaev functor $F_{\calC}$ against the $\calC$-colored framed link depicted in the right-hand part of Figure \ref{F:S-matrix} for every $(i,j) \in \rmI^2$. Now we have 
 \[
  S^+_{ij} = \tr_{\calC} \left( F_{\calC} \left( T^+_{ij} \right) \right), \quad S^-_{ij} = \tr_{\calC} \left( F_{\calC} \left( T^-_{ij} \right) \right),
 \]
 for the $\calC$-colored ribbon tangles depicted in Figure \ref{F:long_Hopf_links} for every $(i,j) \in \rmI^2$.
 
 \begin{figure}[htbp]\label{F:long_Hopf_links}
  \centering
  \includegraphics{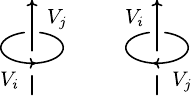}
  \caption{${\calC}$-Colored ribbon tangles $T^+_{ij}$ and $T^-_{ij}$.}
 \end{figure}
 
 Therefore we have 
 \begin{align*}
  \sum_{h \in \rmI} S^+_{ih} S^-_{hj} 
  &= \sum_{h \in \rmI} \tr_{\calC} \left( F_{\calC} \left( T^+_{ih} \right) \right) \tr_{\calC} \left( F_{\calC} \left( T^-_{hj} \right) \right) \\
  &= \sum_{h \in \rmI} \dim_{\calC}(V_h) \tr_{\calC} \left( F_{\calC} \left( T^+_{ih} \circ T^-_{hj} \right) \right) \\
  &= \zeta \delta_{ij},
 \end{align*}
 for every $(i,j) \in \rmI^2$, where the last equality follows from the skein equivalence of Figure \ref{F:equivalence_modularity}. \qedhere
 
 \begin{figure}[htbp]\label{F:equivalence_modularity}
  \centering
  \includegraphics{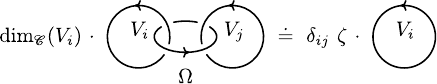}
  \caption{Skein equivalence witnessing $(S^+_{ij})_{(i,j) \in \rmI^2}^{-1} = \zeta^{-1} (S^-_{ij})_{(i,j) \in \rmI^2}$.}
 \end{figure}
 
\end{proof}


%% file: chapter_2.tex
%
%
%

\chapter{Admissible cobordisms}\label{Ch:admissible_cobordisms}

In this chapter we introduce our source $2$-cat\-e\-go\-ry for ETQFTs and graded ETQFTs: the symmetric monoidal $2$-cat\-e\-go\-ry ${\bfadCob_{\calC}}$ of admissible cobordisms of dimension 1+1+1. Since our goal is to extend the Cos\-tan\-tino-Geer-Patu\-reau quantum invariants, which require cohomology classes for their definition, decorations on cobordisms will be richer than usual. Every manifold considered in this memoir is assumed to be oriented, except when explicitly stated otherwise.

\section{Group colorings}\label{S:group_colorings}

In this section we fix an abelian group ${G}$, and we define \textit{${G}$-colorings}, which are special decorations we progressively introduce on closed manifolds, cobordisms, and cobordisms with corners. They consist of a relative cohomology class with coefficients in ${G}$, together with a finite set of base points which is needed in order to induce ${G}$-colorings on horizontal and vertical gluings. The role of $G$-colorings will be to determine indices for surgery presentations of closed 3-manifolds. Appendix \ref{A:cobordisms_with_corners} explains our notation for cobordisms and cobordisms with corners.

\begin{definition}\label{D:G-coloring_object}
 If ${\varGamma}$ is a closed $1$-dimensional smooth manifold, then a \textit{$G$-coloring of $\varGamma$} is a relative cohomology class 
 \[
  \xi \in H^1(\varGamma,A;G)
 \]
 together with a finite set ${A \subset \varGamma}$ containing exactly one base point in every connected component of ${\varGamma}$.
\end{definition}

\begin{definition}\label{D:G-coloring_1-morphism}
 If ${\xi}$ and ${\xi'}$ are ${G}$-colorings of ${\varGamma}$ and ${\varGamma'}$, if ${\varSigma}$ is a $2$-dimensional cobordism from ${\varGamma}$ to ${\varGamma'}$, and if ${P}$ is a ribbon set inside ${\varSigma}$, then a \textit{$G$-coloring of ${(\varSigma,P)}$ extending ${\xi}$ and ${\xi'}$} is a relative cohomology class 
 \[
  \vartheta \in H^1(\varSigma \smallsetminus P,A \cup A' \cup B;G)
 \]
 satisfying ${j_{\varGamma}^*(\vartheta) = \xi}$ and ${j_{\varGamma'}^*(\vartheta) = \xi'}$ for the embeddings 
 \begin{gather*}
  j_{\varGamma} : (\varGamma,A) \hookrightarrow (\varSigma \smallsetminus P,A \cup A' \cup B), \\
  j_{\varGamma'} : (\varGamma',A') \hookrightarrow (\varSigma \smallsetminus P,A \cup A' \cup B)\phantom{,}
 \end{gather*}
 induced by the boundary identifications ${f_{\varSigma_-}}$ and ${f_{\varSigma_+}}$, together with a finite set ${B \subset \varSigma \smallsetminus P}$ containing at least one base point in the interior of every connected component of ${\varSigma}$ disjoint from the incoming boundary ${\partial_- \varSigma}$. Remark that we still denote with ${A \subset \partial_- \varSigma}$ and ${A' \subset \partial_+ \varSigma}$ the images of ${A \subset \varGamma}$ and ${A' \subset \varGamma'}$ under the boundary identifications ${f_{\varSigma_-}}$ and ${f_{\varSigma_+}}$ respectively. Our notation for cobordisms is explained in Definition \ref{D:cobordism}.
\end{definition}

\begin{remark}\label{R:cylinder_G-coloring_object}
 Every ${G}$-coloring ${\xi}$ of ${\varGamma}$ induces a ${G}$-coloring ${I \times \xi}$ of ${I \times \varGamma}$ extending ${\xi}$ given by ${p_{\varGamma}^*(\xi)}$ for the projection
 \[
  p_{\varGamma} : (I \times \varGamma,\{ 0,1 \} \times A) \rightarrow (\varGamma,A).
 \]
 The trivial cobordism ${I \times \varGamma}$ is introduced in Remark \ref{R:trivial_cobordism}. 
\end{remark}

\begin{definition}\label{D:G-coloring_2-morphism}
 If ${\vartheta}$ and ${\vartheta'}$ are ${G}$-colorings of ${(\varSigma,P)}$ and ${(\varSigma',P')}$ which extend ${G}$-colorings ${\xi}$ and ${\xi'}$ of ${\varGamma}$ and ${\varGamma'}$, if ${M}$ is a 3-dimensional cobordism with corners from ${\varSigma}$ to ${\varSigma'}$, and if ${T \subset M}$ is a ribbon graph from ${P}$ to ${P'}$, then a \textit{$G$-coloring of ${(M,T)}$ extending ${\vartheta}$ and ${\vartheta'}$} is a relative cohomology class 
 \[
  \omega \in H^1(M \smallsetminus T,((A \cup A') \times \{ 0,1 \}) \cup B \cup B';G),
 \]
 satisfying ${j_{\varGamma}^*(\omega) = p_{\varGamma}^*(\xi)}$, ${j_{\varGamma'}^*(\omega) = p_{\varGamma'}^*(\xi')}$, ${j_{\varSigma}^*(\omega) = \vartheta}$, and ${j_{\varSigma'}^*(\omega) = \vartheta'}$ for the embeddings
 \begin{gather*}
  j_{\varGamma} : (\varGamma \times I,A \times \{ 0,1 \}) \hookrightarrow (M \smallsetminus T,((A \cup A') \times \{ 0,1 \}) \cup B \cup B'), \\
  j_{\varGamma'} : (\varGamma' \times I,A' \times \{ 0,1 \}) \hookrightarrow (M \smallsetminus T,((A \cup A') \times \{ 0,1 \}) \cup B \cup B'), \\
  j_{\varSigma} : (\varSigma \smallsetminus P,A \cup A' \cup B) \hookrightarrow (M \smallsetminus T,((A \cup A') \times \{ 0,1 \}) \cup B \cup B'), \\
  j_{\varSigma'} : (\varSigma' \smallsetminus P',A \cup A' \cup B') \hookrightarrow (M \smallsetminus T,((A \cup A') \times \{ 0,1 \}) \cup B \cup B')\phantom{,}
 \end{gather*}
 induced by the boundary identifications ${f_{M^{\rmv}_-}}$, ${f_{M^{\rmv}_+}}$, ${f_{M^{\rmh}_-}}$, and ${f_{M^{\rmh}_+}}$, and for the projections
 \begin{gather*}
  p_{\varGamma} : (\varGamma \times I,A \times \{ 0,1 \}) \rightarrow (\varGamma,A), \\
  p_{\varGamma'} : (\varGamma' \times I,A' \times \{ 0,1 \}) \rightarrow (\varGamma',A').
 \end{gather*}
 Remark that we still denote with ${A \times \{ 0,1 \}}$, ${A' \times \{ 0,1 \}}$, ${B}$, and ${B'}$ the images of ${A \times \{ 0,1 \} \subset \varGamma \times I}$, ${A' \times \{ 0,1 \} \subset \varGamma'  \times I}$, ${B \subset \varSigma}$, and ${B' \subset \varSigma'}$ under the boundary identifications ${f_{M^{\rmv}_-}}$, ${f_{M^{\rmv}_+}}$, ${f_{M^{\rmh}_-}}$, and ${f_{M^{\rmh}_+}}$ respectively. Our notation for cobordisms with corners is explained in Definition \ref{D:cobordism_with_corners}.
\end{definition}

\begin{remark}\label{R:cylinder_G-coloring_1-morphism}
 Every ${G}$-coloring ${\vartheta}$ of ${(\varSigma,P)}$ extending ${G}$-colorings ${\xi}$ and ${\xi'}$ of ${\varGamma}$ and ${\varGamma'}$ induces a ${G}$-coloring ${\vartheta \times I}$ of ${(\varSigma \times I,P \times I)}$ extending ${\vartheta}$ given by ${p_{\varSigma}^*(\vartheta)}$ for the projection
 \[
  p_{\varSigma} : \left( (\varSigma \smallsetminus P) \times I,(A \cup A' \cup B) \times \{ 0,1 \} \right) \rightarrow
  (\varSigma \smallsetminus P,A \cup A' \cup B).
 \]
 The trivial cobordism with corners ${\varSigma \times I}$ is introduced in Remark \ref{R:trivial_cobordism_with_corners}.
\end{remark}

Let us fix a ribbon $G$-category $\calC$. If $\varSigma$ is a $2$-di\-men\-sion\-al cobordism from $\varGamma$ to $\varGamma'$, then we say a $G$-ho\-mo\-ge\-ne\-ous $\calC$-col\-ored ribbon set $P \subset \varSigma$ and a $G$-col\-or\-ing $\vartheta$ of $(\varSigma,P)$ extending $\xi$ and $\xi'$ are \textit{compatible} if, for every oriented vertex $p \in P$, the index of the color of $p$ is given by $\langle \vartheta, m_p \rangle$, where $m_p$ is the homology class of a positive meridian of $p$. A \textit{$(\calC,G)$-col\-or\-ing $(P,\vartheta)$ of $\varSigma$ relative to $\xi$ and $\xi'$} is then a $G$-ho\-mo\-ge\-ne\-ous $\calC$-col\-ored ribbon set $P \subset \varSigma$ together with a compatible $G$-col\-or\-ing $\vartheta$ of $(\varSigma,P)$ extending $\xi$ and $\xi'$. Analogously, if $M$ is a $3$-di\-men\-sion\-al cobordism with corners from $\varSigma$ to $\varSigma'$, then we say a $G$-homogeneous $\calC$-col\-ored ribbon graph $T \subset M$ from $P$ to $P'$ and a $G$-col\-or\-ing $\omega$ of $(M,T)$ extending $\vartheta$ and $\vartheta'$ are \textit{compatible} if, for every oriented edge $e \subset T$, the index of the color of $e$ is given by $\langle \omega, m_e \rangle$, where $m_e$ is the homology class of a positive meridian of $e$. A \textit{$(\calC,G)$-col\-or\-ing $(T,\omega)$ of $M$ relative to $(P,\vartheta)$ and $(P',\vartheta')$} is then a $G$-ho\-mo\-ge\-ne\-ous $\calC$-col\-ored ribbon graph $T \subset M$ from $P$ to $P'$ together with a compatible $G$-col\-or\-ing $\omega$ of $(M,T)$ extending $\vartheta$ and $\vartheta'$.

\section{2-Category of decorated cobordisms}

In this section we fix a ribbon ${G}$-category ${\calC}$ and we introduce the symmetric monoidal $2$-cat\-e\-go\-ry ${\bfCob_{\calC}}$ of decorated cobordisms of dimension 1+1+1. Alongside standard decorations like ${\calC}$-colored ribbon sets, ${\calC}$-colored ribbon graphs, Lagrangian subspaces, and signature defects, we also equip objects and morphisms of ${\bfCob_{\calC}}$ with ${G}$-colorings as introduced in the previous section. Remark that Maslov indices need to be discussed in the case of surfaces with boundary, so we do this in Appendix \ref{A:Maslov}. We also point out that ${\bfCob_{\calC}}$ is a rigid $2$-cat\-e\-go\-ry, meaning all of its objects are fully dualizable, while the domain $2$-category for our ETQFTs will be a non-rigid subcategory of this one. Here, we are using the term $2$-category in the fully weak sense, as a synonym for \textit{bicategory}. See Appendix \ref{A:symmetric_monoidal} for the definition of symmetric monoidal $2$-cat\-e\-go\-ry.

\begin{definition}
 An \textit{object ${\bbGamma}$ of ${\bfCob_{\calC}}$} is a pair ${(\varGamma,\xi)}$ where:
 \begin{enumerate}
  \item ${\varGamma}$ is a smooth closed $1$-dimensional manifold; 
  \item ${\xi}$ is a \textit{$G$-coloring} of ${\varGamma}$.
 \end{enumerate}
\end{definition}

See Definition \ref{D:G-coloring_object} for the notion of ${G}$-coloring ${\xi}$ of ${\varGamma}$.

\begin{definition}
 A \textit{$1$-morphism ${\bbSigma : \bbGamma \rightarrow \bbGamma'}$ of ${\bfCob_{\calC}}$} is a 4-tuple ${(\varSigma,P,\vartheta,\calL)}$ where: 
 \begin{enumerate}
  \item ${\varSigma}$ is a $2$-dimensional cobordism from ${\varGamma}$ to ${\varGamma'}$;
  \item ${P \subset \varSigma}$ is a ${G}$-homogeneous ${\calC}$-colored ribbon set;
  \item ${\vartheta}$ is a compatible ${G}$-coloring of ${(\varSigma,P)}$ extending ${\xi}$ and ${\xi'}$;
  \item ${\calL \subset H_1(\varSigma;\R)}$ is a Lagrangian subspace.
 \end{enumerate}
\end{definition}

See Definition \ref{D:G-coloring_1-morphism} for the notion of ${G}$-coloring ${\vartheta}$ of ${(\varSigma,P)}$, and see Definition \ref{D:Lagrangian} for the notion of Lagrangian subspace.

\begin{definition}
 A \textit{$2$-morphism ${\bbM : \bbSigma \Rightarrow \bbSigma'}$ of ${\bfCob_{\calC}}$} between $1$-morphisms ${\bbSigma,\bbSigma' : \bbGamma \rightarrow \bbGamma'}$ is an equivalence class of 4-tuples ${(M,T,\omega,n)}$ where:
 \begin{enumerate}
  \item ${M}$ is a 3-dimensional cobordism with corners from ${\varSigma}$ to ${\varSigma'}$;
  \item ${T \subset M}$ is a ${G}$-homogeneous ${\calC}$-colored ribbon graph from ${P}$ to ${P'}$;
  \item ${\omega}$ is a compatible ${G}$-coloring of ${(M,T)}$ extending ${\vartheta}$ and ${\vartheta'}$;
  \item ${n \in \Z}$ is an integer called the \textit{signature defect}.
 \end{enumerate}
 Two such 4-tuples ${(M,T,\omega,n)}$ and ${(M',T',\omega',n')}$ are equivalent if ${n = n'}$ and if there exists an isomorphism of cobordisms with corners ${f : M \rightarrow M'}$ satisfying ${f(T) = T'}$ and ${f^*(\omega') = \omega}$.
\end{definition}

See Definition \ref{D:cobordism_with_corners} for the notion of cobordism with corners and related isomorphisms, and see Definition \ref{D:G-coloring_2-morphism} for the notion of ${G}$-coloring ${\omega}$ of ${(M,T)}$.

\begin{definition}
 The \textit{identity $1$-morphism ${\id_{\bbGamma} : \bbGamma \rightarrow \bbGamma}$ associated with an object ${\bbGamma}$ of ${\bfCob_{\calC}}$} is the 4-tuple
 \[
  \left( I \times \varGamma,\varnothing,I \times \xi,H_1(I \times \varGamma;\R) \right),
 \]
 where ${I \times \xi}$ is the ${G}$-coloring introduced in Remark \ref{R:cylinder_G-coloring_object}.
\end{definition}

\begin{definition}
 The \textit{identity $2$-morphism ${\id_{\bbSigma} : \bbSigma \Rightarrow \bbSigma}$ associated with a $1$-morphism ${\bbSigma : \bbGamma \rightarrow \bbGamma'}$ of ${\bfCob_{\calC}}$} is the equivalence class of the 4-tuple
 \[
  \left( \varSigma \times I,P \times I,\vartheta \times I,0 \right),
 \]
 where ${\vartheta \times I}$ is the ${G}$-coloring introduced in Remark \ref{R:cylinder_G-coloring_1-morphism}.
\end{definition}

In order to define horizontal and vertical compositions of $1$-morphisms and $2$-morphisms, we need to be able to induce uniquely determined ${G}$-colorings on horizontal and vertical gluings of decorated cobordisms. The following remark will be repeatedly used for this purpose.

\begin{remark}\label{R:gluing_cohomology}
 Let ${X_-}$, ${X_+}$, and ${Y}$ be topological spaces, let ${f_- : Y \hookrightarrow X_-}$ and ${f_+ : Y \rightarrow X_+}$ be embeddings, and let ${E_- \subset X_-}$, ${E_+ \subset X_+}$, and ${F \subset Y}$ be subspaces satisfying ${f_-(F) = E_- \cap f_-(Y_-)}$ and ${f_+(F) = E_+ \cap f_+(Y_+)}$. If
 \[
  \alpha_- \in H^1(X_-,E_-;G), \quad \alpha_+ \in H^1(X_+,E_+;G) 
 \]
 are relative cohomology classes satisfying ${j_{Y_-}^*(\alpha_-) = j_{Y_+}^*(\alpha_+)}$ for the embeddings 
 \[
  j_{Y_-} : (Y,F) \hookrightarrow (X_-,E_-), \quad
  j_{Y_+} : (Y,F) \hookrightarrow (X_+,E_+)
 \]
 induced by ${f_-}$ and ${f_+}$ respectively, then there exists a relative cohomology class 
 \[
  {\alpha_- \cup_Y \alpha_+ \in H^1(X_- \cup_Y X_+,E_- \cup_F E_+;G)}
 \]
 satisfying ${i_{X_-}^*(\alpha_- \cup_Y \alpha_+) = \alpha_-}$ and ${i_{X_+}^*(\alpha_- \cup_Y \alpha_+) = \alpha_+}$ for the inclusions 
 \begin{gather*}
  i_{X_-} : (X_-,E_-) \hookrightarrow (X_- \cup_Y X_+,E_- \cup_F E_+), \\
  i_{X_+} : (X_+,E_+) \hookrightarrow (X_- \cup_Y X_+,E_- \cup_F E_+).
 \end{gather*}
 This follows from the Mayer-Vietoris sequence 
 \begin{center}
  \begin{tikzpicture}
   \node[right] (P0) at (0,0) {$\cdots$};
   \node[right] (P1) at (1.725,0) {$H^0(Y,F;G)$};
   \node[left] (P2) at (10.725,0) {$H^1(X_- \cup_Y X_+,E_- \cup_F E_+;G)$};
   \node (P3) at (12.45,0) {};
   \node (P4) at (0,-1) {};
   \node[right] (P5) at (1.725,-1) {$H^1(X_-,E_-;G) \oplus H^1(X_+,E_+;G)$};
   \node[left] (P6) at (10.725,-1) {$H^1(Y,F;G)$};
   \node[left] (P7) at (12.45,-1) {$\cdots$};
   \draw
   (P0) edge[->] (P1)
   (P1) edge[->] node[above] {$\partial^*$} (P2)
   (P2) edge[-] (P3)
   (P4) edge[->] node[above] {$(i_{X_-}^*,i_{X_+}^*)$} (P5)
   (P5) edge[->] node[above] {$j^*_{Y_-} - j^*_{Y_+}$} (P6)
   (P6) edge[->] (P7);
  \end{tikzpicture} 
 \end{center}
 Such a relative cohomology class is furthermore unique if the map 
 \[
  \partial^* : H^0(Y,F;G) \rightarrow H^1(X_- \cup_Y X_+,E_- \cup_F E_+;G)
 \]
 is zero. This happens, for instance, when ${H^0(Y,E;G) = 0}$.
\end{remark}

Now if ${\bbSigma : \bbGamma \rightarrow \bbGamma'}$ and ${\bbSigma' : \bbGamma' \rightarrow \bbGamma''}$ are $1$-mor\-phisms of ${\bfCob_{\calC}}$, then the gluing of ${\varSigma}$ to ${\varSigma'}$ along ${\varGamma'}$, as given in Definition \ref{D:gluing}, determines a ${G}$-coloring ${\vartheta \cup_{\xi'} \vartheta'}$ of ${(\varSigma \cup_{\varGamma'} \varSigma',P \cup P')}$ extending ${\xi}$ and ${\xi''}$: its relative cohomology class is given by the pull-back ${q_{A'}^* (\vartheta \cup_{\varGamma'} \vartheta')}$ for the relative cohomology class ${\vartheta \cup_{\varGamma'} \vartheta'}$ obtained from Remark \ref{R:gluing_cohomology} by taking
\begin{gather*}
 (X_-,E_-) = (\varSigma \smallsetminus P,A \cup A' \cup B), \\
 (X_+,E_+) = (\varSigma' \smallsetminus P',A' \cup A'' \cup B'), \\
 (Y,F) = (\varGamma',A'),
\end{gather*}
and for the inclusion $q_{A'}$ of the pair
\[
 \left( (\varSigma \cup_{\varGamma'} \varSigma') \smallsetminus (P \cup P'), A \cup A'' \cup B \cup B' \right)
\]
into the pair 
\[
 \left( (\varSigma \cup_{\varGamma'} \varSigma') \smallsetminus (P \cup P'),
 A \cup A' \cup A'' \cup B \cup B' \right),
\]
while its base set is given by ${B \cup B'}$.

\begin{definition}
 The \textit{horizontal composition ${\bbSigma' \circ \bbSigma : \bbGamma \rightarrow \bbGamma''}$ of $1$-mor\-phisms ${\bbSigma : \bbGamma \rightarrow \bbGamma'}$ and ${\bbSigma' : \bbGamma' \rightarrow \bbGamma''}$ of ${\bfCob_{\calC}}$} is the 4-tuple
 \[
  \left( \varSigma \cup_{\varGamma'} \varSigma',P \cup P',\vartheta \cup_{\xi'} \vartheta',i_{\varSigma *}(\calL) + i_{\varSigma'*}(\calL') \right),
 \]
 where ${i_{\varSigma*} : H_1(\varSigma;\R) \rightarrow H_1(\varSigma \cup_{\varGamma'} \varSigma';\R)}$ and ${i_{\varSigma'*} : H_1(\varSigma';\R) \rightarrow H_1(\varSigma \cup_{\varGamma'} \varSigma';\R)}$ are induced by inclusion, as in Proposition \ref{P:sum_of_Lagrangians}.
\end{definition}

Analogously, if ${\bbSigma,\bbSigma'' : \bbGamma \rightarrow \bbGamma'}$ and ${\bbSigma',\bbSigma''' : \bbGamma' \rightarrow \bbGamma''}$ are $1$-mor\-phisms, and if ${\bbM : \bbSigma \Rightarrow \bbSigma''}$ and ${\bbM' : \bbSigma' \Rightarrow \bbSigma'''}$ are $2$-mor\-phisms of ${\bfCob_{\calC}}$, then the horizontal gluing of ${M}$ to ${M'}$ along ${\varGamma' \times I}$, as given in Definition \ref{D:horizontal_gluing}, determines a ${G}$-coloring ${\omega \cup_{\xi' \times I} \omega'}$ of ${(M \cup_{\varGamma' \times I} M',T \cup T')}$ extending ${\vartheta \cup_{\varGamma'} \vartheta'}$ and ${\vartheta'' \cup_{\varGamma'} \vartheta'''}$: it is given by the pull-back ${q_{A' \times \{ 0,1 \}}^* (\omega \cup_{\varGamma' \times I} \omega')}$
for the relative cohomology class ${\omega \cup_{\varGamma' \times I} \omega'}$ obtained from Remark \ref{R:gluing_cohomology} by taking
\begin{gather*}
 (X_-,E_-) = (M \smallsetminus T,((A \cup A') \times \{ 0,1 \}) \cup B \cup B'), \\
 (X_+,E_+) = (M' \smallsetminus T',((A' \cup A'') \times \{ 0,1 \}) \cup B'' \cup B'''), \\
 (Y,F) = (\varGamma' \times I,A' \times \{ 0,1 \}),
\end{gather*}
and for the inclusion ${q_{A' \times \{ 0,1 \}}}$ of the pair
\[
 \left( (M \cup_{\varGamma' \times I} M') \smallsetminus (T \cup T'),((A \cup A'') \times \{ 0,1 \}) \cup B \cup B' \cup B'' \cup B''' \right)
\]
into the pair
\[
 \left( (M \cup_{\varGamma' \times I} M') \smallsetminus (T \cup T'),((A \cup A' \cup A'') \times \{ 0,1 \}) \cup B \cup B' \cup B'' \cup B''' \right).
\]

\begin{definition}
 The \textit{horizontal composition $\bbM' \circ \bbM : \bbSigma' \circ \bbSigma \Rightarrow \bbSigma''' \circ \bbSigma''$ of $2$-mor\-phisms $\bbM : \bbSigma \Rightarrow \bbSigma''$ and $\bbM' : \bbSigma' \Rightarrow \bbSigma'''$ of ${\bfCob_{\calC}}$} between $1$-mor\-phisms $\bbSigma,\bbSigma'' : \bbGamma \rightarrow \bbGamma'$ and $\bbSigma',\bbSigma''' : \bbGamma' \rightarrow \bbGamma''$ is the equivalence class of the 4-tuple
 \[
  \left( M \cup_{\varGamma' \times I} M',T \cup T',\omega \cup_{\xi' \times I} \omega', n + n' \right).
 \]
\end{definition}

If ${\bbSigma,\bbSigma',\bbSigma'' : \bbGamma \rightarrow \bbGamma'}$ are $1$-mor\-phisms, and if ${\bbM : \bbSigma \Rightarrow \bbSigma'}$ and ${\bbM' : \bbSigma' \Rightarrow \bbSigma''}$ are $2$-mor\-phisms of ${\bfCob_{\calC}}$, then the vertical gluing of ${M}$ to ${M'}$ along ${\varSigma'}$, as given in Definition \ref{D:vertical_gluing}, determines a ${G}$-coloring ${\omega \cup_{\vartheta'} \omega'}$ of ${(M \cup_{\varSigma'} M',T \cup_{P'} T')}$ extending ${\vartheta}$ and ${\vartheta''}$: it is given by the pull-back $q_{(A \cup A') \times \{ 1/2 \}}^* (\omega \cup_{\varSigma'} \omega')$ 
for the relative cohomology class ${\omega \cup_{\varSigma'} \omega'}$ obtained from Remark \ref{R:gluing_cohomology} by taking
\begin{gather*}
 (X_-,E_-) = (M \smallsetminus T,((A \cup A') \times \{ 0,1 \}) \cup B \cup B'), \\
 (X_+,E_+) = (M' \smallsetminus T',((A \cup A') \times \{ 0,1 \}) \cup B' \cup B''), \\
 (Y,F) = (\varSigma',A \cup A' \cup B'),
\end{gather*}
and for the inclusion ${q_{(A \cup A') \times \{ 1/2 \}}}$ of the pair
\[
 \left( (M \cup_{\varSigma'} M') \smallsetminus (T \cup_P T'),((A \cup A') \times \{ 0,1 \}) \cup B \cup B'' \right)
\]
into the pair
\[
 \left( (M \cup_{\varSigma'} M') \smallsetminus (T \cup_P T'),\left( (A \cup A') \times \left\{ 0,\frac{1}{2},1 \right\} \right) \cup B \cup B' \cup B'' \right).
\]
Remark that we still denote with ${A \times \left\{ 0,\frac{1}{2},1 \right\}}$ and with ${A' \times \left\{ 0,\frac{1}{2},1 \right\}}$ the images of ${A \times \left\{ 0,\frac{1}{2},1 \right\} \subset \varGamma \times I}$ and ${A' \times \left\{ 0,\frac{1}{2},1 \right\} \subset \varGamma \times I}$ under the boundary identifications ${(f_{M^{\rmv}_-} \cup_{\varGamma} f_{M'^{\rmv}_-}) \circ u_{\varGamma}}$ and ${(f_{M^{\rmv}_+} \cup_{\varGamma'} f_{M'^{\rmv}_+}) \circ u_{\varGamma'}}$, which, we recall, factor through ${(\varGamma \times I) \cup_{\varGamma} (\varGamma \times I)}$ and through ${(\varGamma' \times I) \cup_{\varGamma'} (\varGamma' \times I)}$ respectively, as explained in Definition \ref{D:vertical_gluing}.

\begin{definition}\label{D:vertical_composition}
 The \textit{vertical composition ${\bbM' \ast \bbM : \bbSigma \Rightarrow \bbSigma''}$ of $2$-mor\-phisms ${\bbM : \bbSigma \Rightarrow \bbSigma'}$ and ${\bbM' : \bbSigma' \Rightarrow \bbSigma''}$ of ${\bfCob_{\calC}}$} between $1$-mor\-phisms ${\bbSigma,\bbSigma',\bbSigma'' : \bbGamma \rightarrow \bbGamma'}$ is the equivalence class of the 4-tuple
 \[
  \left( M \cup_{\varSigma'} M',T \cup_{P'} T',\omega \cup_{\vartheta'} \omega',
  n + n' - \mu(M_*(\calL),\calL',M'^*(\calL'')) \right),
 \]
 where the push-forward ${M_*(\calL)}$ and the pull-back ${M'^*(\calL'')}$ are defined by
 \begin{gather*}
  M_*(\calL) := \{ x' \in H_1(\varSigma';\R) \mid j_{\varSigma' *}(x') \in j_{\varSigma *} (\calL) \}, \\
  M'^*(\calL'') := \{ x' \in H_1(\varSigma';\R) \mid j'_{\varSigma' *}(x') \in j'_{\varSigma'' *}(\calL'') \}
 \end{gather*}
 for the embeddings
 \[
  j_{\varSigma} : \varSigma \hookrightarrow M, \quad
  j_{\varSigma'} : \varSigma' \hookrightarrow M, \quad
  j'_{\varSigma'} : \varSigma' \hookrightarrow M', \quad
  j'_{\varSigma''} : \varSigma'' \hookrightarrow M'\phantom{,}
 \]
 induced by the boundary identifications ${f_{M^{\rmh}_-}}$, ${f_{M^{\rmh}_+}}$, ${f_{M'^{\rmh}_-}}$, and ${f_{M'^{\rmh}_+}}$, and where ${\mu}$ is the Maslov index introduced in Definition \ref{D:Maslov}. 
\end{definition}

See Proposition \ref{P:Lagrangian_relations} for a proof that ${M_*(\calL)}$ and ${M'^*(\calL'')}$ are indeed Lagrangian subspaces of ${H_1(\varSigma';\R)}$. If the reader wonders why ${\mu}$ did not appear in horizontal compositions of $2$-mor\-phisms of ${\bfCob_{\calC}}$, the answer is simple: horizontal gluings are performed along surfaces of the form ${\varGamma' \times I}$, whose first homology with real coefficients contains a unique Lagrangian, the whole space. Therefore, every Maslov index in ${H_1(\varGamma' \times I;\R)}$ is zero. Now, remark that the definitions we introduced up to here almost produce the structure of a $2$-cat\-e\-go\-ry as given in Definition \ref{D:2-category}. What we still miss are left and right unitors and associators. These can be easily defined, and it is important that these definitions can indeed be given. However, we will not, because if we did, then we would invoke Theorem \ref{T:coherence_for_2-cat} anyway. Thus, from now on, we pretend ${\bfCob_{\calC}}$ is a strict $2$-cat\-e\-go\-ry, and we move on to discuss its symmetric monoidal structure, which is induced by disjoint union.

\begin{definition}
 The \textit{tensor unit of ${\bfCob_{\calC}}$} is the unique object associated with the empty manifold, and it is denoted ${\varnothing}$. We refer to 1-endomorphisms of $\varnothing$ as \textit{closed 1-morphisms}, and to 2-endomorphisms of $\id_{\varnothing}$ as \textit{closed 2-morphisms}.
\end{definition}

If ${X}$ and ${X'}$ are topological spaces, if ${E \subset X}$ and ${E' \subset X'}$ are subspaces, and if ${\alpha \in H^k(X,E;G)}$ and ${\alpha' \in H^k(X',E';G)}$ are cohomology classes, then we denote with ${\alpha \sqcup \alpha'}$ the cohomology class in ${H^k(X \sqcup X',E \sqcup E';G)}$ given by
\[
 q_{X'}^*(e_{X'}^{*-1}(\alpha)) + q_X^*(e_X^{*-1}(\alpha'))
\]
for the inclusions
\begin{gather*}
 q_{X'} : (X \sqcup X',E \sqcup E') \hookrightarrow (X \sqcup X',E \sqcup X'), \\
 q_X : (X \sqcup X',E \sqcup E') \hookrightarrow (X \sqcup X',X \sqcup E'), \\
 e_{X'} : (X,E) \hookrightarrow (X \sqcup X',E \sqcup X'), \\
 e_X : (X',E') \hookrightarrow (X \sqcup X',X \sqcup E').
\end{gather*}
This induces a disjoint union operation on ${G}$-colorings, which is given by standard disjoint union on base sets.

\begin{definition}
 The \textit{tensor product ${\bbGamma \disjun \bbGamma'}$ of objects ${\bbGamma}$ and ${\bbGamma'}$ of ${\bfCob_{\calC}}$} is the pair ${(\varGamma \sqcup \varGamma',\xi \sqcup \xi')}$.
\end{definition}

\begin{definition}
 The \textit{tensor product ${\bbSigma \disjun \bbSigma' : \bbGamma \disjun \bbGamma' \rightarrow \bbGamma'' \disjun \bbGamma'''}$ of $1$-mor\-phisms ${\bbSigma : \bbGamma \rightarrow \bbGamma''}$ and ${\bbSigma' : \bbGamma' \rightarrow \bbGamma'''}$ of ${\bfCob_{\calC}}$} is the 4-tuple 
 \[
  (\varSigma \sqcup \varSigma',P \sqcup P',\vartheta \sqcup \vartheta',\calL \oplus \calL').
 \]
\end{definition}

\begin{definition}
 The \textit{tensor product $\bbM \disjun \bbM' : \bbSigma \disjun \bbSigma' \Rightarrow \bbSigma'' \disjun \bbSigma'''$ of $2$-morphisms $\bbM : \bbSigma \Rightarrow \bbSigma''$ and $\bbM' : \bbSigma' \Rightarrow \bbSigma'''$ of $\bfCob_{\calC}$} between $1$-mor\-phisms $\bbSigma,\bbSigma'' : \bbGamma \rightarrow \bbGamma''$ and $\bbSigma',\bbSigma''' : \bbGamma' \rightarrow \bbGamma'''$ is the equivalence class of the 4-tuple
 \[
  (M \sqcup M',T \sqcup T',\omega \sqcup \omega',n+n').
 \]
\end{definition}


In order to make all this into a tensor product $2$-func\-tor, we need to define coherence 2-morphisms of the form
\[
 \disjun_{(\bbSigma'',\bbSigma'''),(\bbSigma,\bbSigma')} : (\bbSigma'' \disjun \bbSigma''') \circ (\bbSigma \disjun \bbSigma') \Rightarrow {(\bbSigma'' \circ \bbSigma)} \disjun {(\bbSigma''' \circ \bbSigma')} 
\]
for all 2-morphisms $\bbSigma : \bbGamma \rightarrow \bbGamma''$, $\bbSigma' : \bbGamma' \rightarrow \bbGamma'''$, $\bbSigma'' : \bbGamma'' \rightarrow \bbGamma^{\fourth}$, $\bbSigma''' : \bbGamma''' \rightarrow \bbGamma^{\fifth}$ of ${\bfCob_{\calC}}$. These are naturally induced by the canonical isomorphisms
\[
 (\varSigma \sqcup \varSigma') \cup_{\varGamma'' \sqcup \varGamma'''} (\varSigma'' \sqcup \varSigma''') \cong(\varSigma \cup_{\varGamma''} \varSigma'') \sqcup (\varSigma' \cup_{\varGamma'''} \varSigma''').
\]
Now, remark that most of the remaining symmetric monoidal structure of ${\bfCob_{\calC}}$ can simply be skipped by appealing to Theorem \ref{T:coherence_for_sym_mon_2-cat}. Again, it is important to know that the rest of the coherence data can indeed be defined, which is the case, but we will avoid doing this because then we would forget about it anyway. Thus, the last piece of structure we need to discuss is the braiding, and we do this directly in the quasi-strict version of ${\bfCob_{\calC}}$.

\begin{definition}\label{D:braiding_1-morphism}
 The \textit{braiding $1$-mor\-phism ${\cobbr_{\bbGamma,\bbGamma'} : \bbGamma \disjun \bbGamma' \rightarrow \bbGamma' \disjun \bbGamma}$ associated with objects ${\bbGamma}$ and ${\bbGamma'}$ of ${\bfCob_{\calC}}$} is the 4-tuple
 \[
  (I \ttimes (\varGamma \sqcup \varGamma'),\varnothing,I \times (\xi \sqcup \xi'), H_1(I \times (\varGamma \sqcup \varGamma');\R)),
 \]
 where ${I \ttimes (\varGamma \sqcup \varGamma')}$ is the cobordism commuting ${\varGamma}$ past ${\varGamma'}$, as given in Definition \ref{D:flip_cobordism}.
\end{definition}

\begin{definition}\label{D:braiding_2-morphism}
 The \textit{braiding $2$-mor\-phism} 
 \[
  \cobbr_{\bbSigma,\bbSigma'} :
  \cobbr_{\bbGamma'',\bbGamma'''} \circ (\bbSigma \disjun \bbSigma') \Rightarrow 
  (\bbSigma' \disjun \bbSigma) \circ \cobbr_{\bbGamma,\bbGamma'}
 \]
 \textit{associated with $1$-mor\-phisms ${\bbSigma : \bbGamma \rightarrow \bbGamma''}$ and 
 ${\bbSigma' : \bbGamma' \rightarrow \bbGamma'''}$ of ${\bfCob_{\calC}}$} is the equivalence class of the 4-tuple
 \[
  ((\varSigma \sqcup \varSigma') \ttimes I,(P \sqcup P') \times I,(\vartheta \sqcup \vartheta') \times I, 0),
 \]
 where ${(\varSigma \sqcup \varSigma') \ttimes I}$ is the cobordism with corners commuting ${\varSigma}$ past ${\varSigma'}$, as given in Definition \ref{D:flip_cobordism_with_corners}.
\end{definition}

\section{2-Category of admissible cobordisms}\label{S:admissible_cobordisms}

In this section we fix a pre-modular $G$-category $\calC$ relative to $(\PGr,X)$ and we define the symmetric monoidal $2$-category $\bfadCob_{\calC}$ of admissible cobordisms of dimension 1+1+1. This is a sub-$2$-category of $\bfCob_{\calC}$ that will provide the source for our $\PGr$-graded ETQFTs. Its definition rests on the key notion of \textit{admissibility}, a condition for morphisms of $\bfCob_{\calC}$ which will allow us to define the Cos\-tan\-tino-Geer-Patu\-reau quantum invariants, and which will be responsible for the non-semisimplicity of our $\PGr$-graded ETQFTs. In order to introduce it, we first need to fix some terminology.

If $\varSigma$ is a $2$-di\-men\-sion\-al cobordism, then we say a $G$-homogeneous $\calC$-colored ribbon set $P \subset \varSigma$ is \textit{projective} if it admits a \textit{projective vertex}, that is a vertex $p \in P$ whose color is a projective object of $\calC$, and we say a $G$-coloring $\vartheta$ of $(\varSigma,P)$ is \textit{generic} if it admits a \textit{generic curve}, that is an embedded closed oriented curve $\gamma \subset \varSigma \smallsetminus P$ whose homology class, still denoted $\gamma$, satisfies $\langle \vartheta, \gamma \rangle \in G \smallsetminus X$. Analogously, if $M$ is a $3$-di\-men\-sion\-al cobordism with corners, then we say a $G$-homogeneous $\calC$-col\-ored ribbon graph $T \subset M$ is \textit{projective} if it admits a \textit{projective edge}, that is an edge $e \subset T$ whose color is a projective object of $\calC$, and we say a $G$-coloring $\omega$ of $(M,T)$ is \textit{generic} if it admits a \textit{generic curve}, that is an embedded closed oriented curve $\kappa \subset M \smallsetminus T$ whose homology class, still denoted $\kappa$, satisfies $\langle \omega, \kappa \rangle \in G \smallsetminus X$.

Let $\varSigma$ be a $2$-di\-men\-sion\-al cobordism with connected components $\varSigma_1,\ldots,\varSigma_m$, and let $(P,\vartheta)$ be a $(\calC,G)$-col\-or\-ing of $\varSigma$ inducing $(\calC,G)$-col\-or\-ings $(P_i,\vartheta_i)$ of $\varSigma_i$ for every integer $1 \leqslant i \leqslant m$. We say $(P,\vartheta)$ is \textit{admissible} if either $P_i$ is projective or $\vartheta_i$ is generic for every connected component $\varSigma_i$ of $\varSigma$ which is disjoint from the incoming boundary $\partial_- \varSigma$, and we say $(P,\vartheta)$ is \textit{strongly admissible} if the same condition holds for every connected component $\varSigma_i$ of $\varSigma$, regardless of its intersection with $\partial_- \varSigma$. Analogously, let $M$ be a $3$-di\-men\-sion\-al cobordism with corners with connected components $M_1,\ldots,M_m$, and let $(T,\omega)$ be a $(\calC,G)$-coloring of $M$ inducig $(\calC,G)$-colorings $(T_i,\omega_i)$ of $M_i$ for every integer $1 \leqslant i \leqslant m$. We say $(T,\omega)$ is \textit{admissible} if either $T_i$ is projective or $\omega_i$ is generic for every connected component $M_i$ of $M$ which is disjoint from the incoming horizontal boundary $\partial^{\rmh}_- M$, and we say $(T,\omega)$ is \textit{strongly admissible} if the same condition holds for every connected component $M_i$ of $M$, regardless of its intersection with $\partial^{\rmh}_- M$.

\begin{definition}\label{D:admissible_1-morphisms}
 A $1$-mor\-phism $\bbSigma = (\varSigma,P,\vartheta,\calL)$ of $\bfCob_{\calC}$ is \textit{admissible} if $(T,\vartheta)$ is an admissible $(\calC,G)$-col\-or\-ing of $\varSigma$, and it is \textit{strongly admissible} if $(P,\vartheta)$ is strongly admissible.
\end{definition}

Remark that the previous definition has the following direct consequences:
\begin{enumerate}
 \item If $\bbGamma$ is an object of $\bfCob_{\calC}$, then its identity $\id_{\bbGamma}$ is an admissible $1$-mor\-phism, as no component is disjoint from the incoming boundary;
 \item If $\bbSigma : \bbGamma \rightarrow \bbGamma'$ and $\bbSigma' : \bbGamma' \rightarrow \bbGamma''$ are admissible $1$-mor\-phisms of $\bfCob_{\calC}$, then their composition $\bbSigma' \circ \bbSigma$ is an admissible $1$-mor\-phism;
 \item If $\bbSigma$ and $\bbSigma'$ are admissible $1$-mor\-phisms of $\bfCob_{\calC}$, then their tensor product $\bbSigma \disjun \bbSigma'$ is an admissible $1$-mor\-phism.
\end{enumerate}

\begin{definition}\label{D:admissible_2-morphisms}
 A $2$-mor\-phism $\bbM = (M,T,\omega,n)$ of $\bfCob_{\calC}$ is \textit{admissible} if $(T,\omega)$ is an admissible $(\calC,G)$-coloring of $M$, and it is \textit{strongly admissible} if $(T,\omega)$ is strongly admissible.
\end{definition}

Remark that the previous definition has the following direct consequences:
\begin{enumerate}
 \item If $\bbSigma$ is a $1$-mor\-phism of $\bfCob_{\calC}$, then its identity $\id_{\bbSigma}$ is an admissible $2$-mor\-phism, as no component is disjoint from the incoming horizontal boundary;
 \item If $\bbM : \bbSigma \Rightarrow \bbSigma''$ and $\bbM' : \bbSigma' \Rightarrow \bbSigma'''$ are admissible $2$-mor\-phisms of $\bfCob_{\calC}$ between $1$-mor\-phisms $\bbSigma,\bbSigma'' : \bbGamma \rightarrow \bbGamma'$ and $\bbSigma',\bbSigma''' : \bbGamma' \rightarrow \bbGamma''$, then their horizontal composition $\bbM' \circ \bbM$ is an admissible $2$-mor\-phism;
 \item If $\bbM : \bbSigma \Rightarrow \bbSigma'$ and $\bbM' : \bbSigma' \Rightarrow \bbSigma''$ are admissible $2$-mor\-phisms of $\bfCob_{\calC}$ between $1$-mor\-phisms $\bbSigma,\bbSigma',\bbSigma'' : \bbGamma \rightarrow \bbGamma'$, then their vertical composition $\bbM' \ast \bbM$ is an admissible $2$-mor\-phism;
 \item If $\bbM$ and $\bbM'$ are admissible $2$-mor\-phisms of $\bfCob_{\calC}$, then their tensor product $\bbM \disjun \bbM'$ is an admissible $2$-mor\-phism.
\end{enumerate}

\begin{definition}\label{D:sym_mon_2-cat_of_adm_cob}
 The \textit{symmetric monoidal $2$-cat\-e\-go\-ry of admissible cobordisms of dimension} 1+1+1 is the symmetric monoidal $2$-cat\-e\-go\-ry $\bfadCob_{\calC}$ whose objects are objects of $\bfCob_{\calC}$, whose morphisms are admissible morphisms of $\bfCob_{\calC}$, and whose symmetric monoidal structure is inherited from the one of $\bfCob_{\calC}$.
\end{definition}


%% file: chapter_3.tex
%
%
%

\chapter{Extension of Costantino-Geer-Patureau invariants}\label{Ch:extenstion_of_CGP_invariants}

In this chapter we fix a non-degenerate pre-modular $G$-category $\calC$ relative to $(\PGr,X)$ and we review the definition of the associated Costantino-Geer-Patureau quantum invariant $\CGP_{\calC}$ of admissible closed 3-manifolds constructed in \cite{CGP14}. To this invariant we apply a $2$-categorical version of the universal construction of \cite{BHMV95}. This will produce a strict $2$-functor $\bfA_{\calC}$ from the 2-category $\bfadCob_{\calC}$ of admissible cobordisms, introduced in Definition \ref{D:sym_mon_2-cat_of_adm_cob}, to the 2-category $\bfCat_{\Bbbk}$ of linear categories, introduced in Definition \ref{D:sym_mon_2-cat_of_lin_cat}. This $2$-functor will turn out not to be symmetric monoidal in general, but it will provide the basis for our definition of a $\PGr$-graded ETQFT extending $\CGP_{\calC}$. Starting from this chapter, we will begin to draw graphical representations of some morphisms of $\bfadCob_{\calC}$ inside $\R^3$. Here is our convention for reading them: edges of $\calC$-colored ribbon graphs will be represented by oriented solid blue lines with blackboard framing, and coupons will be equipped with orientations of horizontal and vertical boundaries indicating the order of tensor products and the direction of morphisms; $G$-colorings will appear as labels, denoting the result of evaluation, attached to oriented red dashed-dotted lines, representing relative homology classes.

\section{Projective and generic stabilizations}\label{S:projective_generic_stabilizations}

We start by introducing two operations which will be used later in order to define the Costantino-Geer-Patureau invariants for all closed 2-morphisms of $\bfadCob_{\calC}$. These operations transform an admissible $(\calC,G)$-coloring of a 3-dimensional cobordism with corners into a formal linear combination of admissible $(\calC,G)$-colorings, which we still interpret as an admissible $(\calC,G)$-coloring. Their purpose is to turn any admissible closed 3-manifold into a decorated closed 3-manifold against which we can compute the Costantino-Geer-Patureau invariant. Both operations consist in replacing small portions of decorations in order to ensure the presence of an edge colored by a simple projective object of generic index, as represented in Figure \ref{F:projective_generic_stabilization}. In order to explain all this in more detail, let us fix a generic $g \in G \smallsetminus X$, a simple projective $V_i \in \Theta(\calC_g)$, an arbitrary $h \in G$, and a projective $U \in \Proj(\calC_h)$. First of all, we denote with $\smash{\bbS^2_{(-,U),(+,U)}}$ the closed 1-morphism of $\bfadCob_{\calC}$ given by 
\[
 \left(S^2,P_{(-,U),(+,U)},\vartheta_{(-,U),(+,U)},\{ 0 \}\right),
\]
where $P_{(-,U),(+,U)} \subset S^2$ is the $\calC$-colored ribbon set given by the south pole with negative orientation and by the north pole with positive orientation, both with color $U$, and where $\vartheta_{(-,U),(+,U)}$ is the unique compatible $G$-coloring of $(S^2,P_{(-,U),(+,U)})$. Next, we denote with $\smash{\bbD^3_{\id_{(+,U)}} : \id_{\varnothing} \Rightarrow \bbS^2_{(-,U),(+,U)}}$ the $2$-morphism of $\bfadCob_{\calC}$ given by 
\[
 (D^3,\id_{(+,U)},\omega_{\id_{(+,U)}}, 0),
\]
where $\id_{(+,U)} \subset D^3$ is the $\calC$-colored ribbon tangle given by a single vertical edge joining the south and the north pole, with framing zero and color $U$, and where $\omega_{\id_{(+,U)}}$ is the unique compatible $G$-coloring of $(D^3,\id_{(+,U)})$. Similarly, we denote with $\smash{\bbD^3_{T_{U,i}} : \id_{\varnothing} \Rightarrow \bbS^2_{(-,U),(+,U)}}$ the $2$-morphism of $\bfadCob_{\calC}$ given by 
\[
 (D^3,T_{U,i},\omega_{T_{U,i}}, 0),
\]
where $T_{U,i} \subset D^3$ is the $\calC$-colored ribbon tangle represented in the bottom left part of Figure \ref{F:projective_generic_stabilization}, where $s_{U,i}$ is a section of the epimorphism $\smash{\id_U \otimes \rev_{V_i}}$ given by Remark \ref{R:epic_evaluations}, and where $\omega_{T_{U,i}}$ is the unique compatible $G$-coloring of $(D^3,T_{U,i})$. Then, if $(M,T,\omega,n)$ is a 2-morphism of $\bfadCob_{\calC}$ between 1-morphisms $\bbSigma, \bbSigma' : \bbGamma \rightarrow \bbGamma'$, a projective edge $e \subset T$ of color $U \in \Proj(\calC)$ determines a decomposition 
\[
 (M,T,\omega,n) = \bbM_e \ast \left( \bbD^3_{\id_{(+,U)}} \disjun \id_{\bbSigma} \right)
\]
for some 2-morphism $\bbM_e : \bbS^2_{(-,U),(+,U)} \disjun \bbSigma \Rightarrow \bbSigma'$ of $\bfadCob_{\calC}$, and we say a $(\calC,G)$-coloring $(T_e,\omega_e)$ of $M$ satisfying
\[
 (M,T_e,\omega_e,n) = \bbM_e \ast  \left( \bbD^3_{T_{U,i}} \disjun \id_{\bbSigma} \right)\phantom{,}
\]
is obtained from $(T,\omega)$ by \textit{projective stabilization of index $g$ along the edge $e$}. This defines the first operation, so let us move on to the second one. We begin by denoting with $\smash{(\bbS^1 \times \bbS^1)_g}$ the closed 1-morphism of $\bfadCob_{\calC}$ given by 
\[
 (S^1 \times S^1,\varnothing,\vartheta_g,\calL_m),
\]
where $\vartheta_g$ is the $G$-coloring of $(S^1 \times S^1,\varnothing)$ determined by $\langle \vartheta_g,\ell \rangle = g$ and by $\langle \vartheta_g,m \rangle = 0$ for the homology classes $\ell$ and $m$ of the longitude $S^1 \times \{ (0,1) \}$ and of the meridian $\{ (1,0) \} \times S^1$ respectively, and where the Lagrangian $\calL_m$ is generated by $m$. Next, we denote with $(\bbS^1 \times\overline{\bbD^2})_g : \id_{\varnothing} \Rightarrow (\bbS^1 \times \bbS^1)_g $ the $2$-morphism of $\bfadCob_{\calC}$ given by 
\[
 (S^1 \times \overline{D^2},\varnothing,\omega_g,0)
\]
where $\omega_g$ is the unique $G$-coloring of $(S^1 \times \overline{D^2},\varnothing)$ which extends $\vartheta_g$. Lastly, we denote with $(\bbS^1 \times\overline{\bbD^2})_{L_- \cup L_+} : \id_{\varnothing} \Rightarrow (\bbS^1 \times \bbS^1)_g$ the $2$-morphism of $\bfadCob_{\calC}$ given by 
\[
 (S^1 \times \overline{D^2},L_- \cup L_+,\omega_{L_- \cup L_+},0)
\]
where $L_- \cup L_+ \subset S^1 \times \overline{D^2}$ is the $\calC$-colored framed link represented in the bottom right part of Figure \ref{F:projective_generic_stabilization}, and where $\omega_{L_- \cup L_+}$ is the unique compatible $G$-col\-or\-ing of ${(S^1 \times \overline{D^2},L_- \cup L_+)}$ which extends $\vartheta_g$. Then, just like before, if $(M,T,\omega,n)$ is a 2-morphism of $\bfadCob_{\calC}$ between 1-morphisms $\bbSigma, \bbSigma' : \bbGamma \rightarrow \bbGamma'$, a generic curve $\kappa \subset M \smallsetminus T$ of index $g \in G \smallsetminus X$ determines a decomposition
\[
 (M,T,\omega,n) = \bbM_{\kappa} \ast \left( (\bbS^1 \times \overline{\bbD^2})_g \disjun \id_{\bbSigma} \right)
\]
for some 2-morphism $\bbM_{\kappa} : (\bbS^1 \times \bbS^1)_g \disjun \bbSigma \Rightarrow \bbSigma'$ of $\bfadCob_{\calC}$, and we say a $(\calC,G)$-coloring $(T_{\kappa},\omega_{\kappa})$ of $M$ satisfying
\[
 (M,T_{\kappa},\omega_{\kappa},n) = \Delta_-^{-1} \Delta_+^{-1} \cdot \bbM_{\kappa} \ast \left( (\bbS^1 \times \overline{\bbD^2})_{L_- \cup L_+} \disjun \id_{\bbSigma} \right)\phantom{,}
\]
is obtained from $(T,\omega)$ by \textit{generic stabilization of index $g$ along the curve $\kappa$}. Remark that, in order to define this operation, we are using the non-degeneracy condition of Definition \ref{D:non-degeneracy}.

\begin{figure}[htb]\label{F:projective_generic_stabilization}
 \centering
 \includegraphics{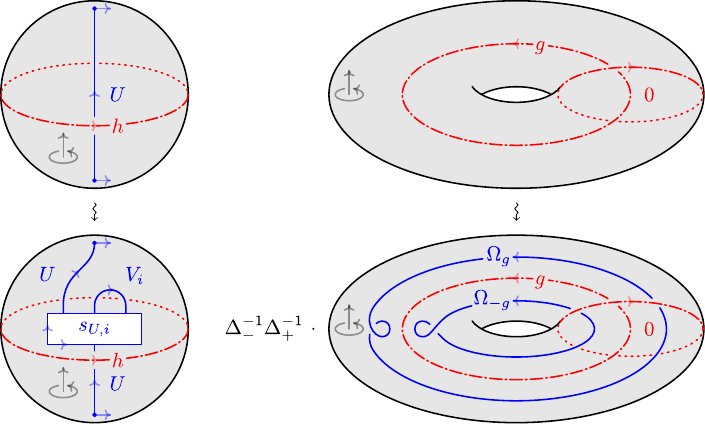}
 \caption{In the left-hand part, a projective stabilization of index $g$ replacing $(\id_{(+,U)},\omega_{\id_{(+,U)}})$ with $(T_{U,i},\omega_{T_{U,i}})$; In the right-hand part, a generic stabilization of index $g$ replacing $(\varnothing,\omega_g)$ with $\Delta_-^{-1} \Delta_+^{-1} \cdot (L_- \cup L_+,\omega_{L_- \cup L_+})$.}
\end{figure}

%
%

\section{Costantino-Geer-Patureau invariants}\label{S:CGP_invariants}

We start this section by recalling the definition of the renormalized invariant of admissible closed ribbon graphs, which was first introduced by Geer, Patureau, and Turaev in \cite{GPT09}. This construction relies crucially on a key ingredient: the non-degenerate trace $\rmt$ on $\Proj(\calC)$. Let us begin by fixing some terminology: if $T \subset S^3$ is a closed $G$-homogeneous $\calC$-colored ribbon graph, then we say an endomorphism $T_{(\underline{\varepsilon},\underline{V})} : (\underline{\varepsilon},\underline{V}) \rightarrow (\underline{\varepsilon},\underline{V})$ of $\Rib_{\calC}^G$ is a \textit{cutting presentation of $T$} if
\[
 T = \tr_{\Rib_{\calC}^G} (T_{(\underline{\varepsilon},\underline{V})}) = \rev_{(\underline{\varepsilon},\underline{V})} \circ (T_{(\underline{\varepsilon},\underline{V})} \otimes \id_{(\underline{\varepsilon},\underline{V})^*}) \circ \lcoev_{(\underline{\varepsilon},\underline{V})},
\]
where we are realizing $D^2 \times I$ as a subspace of $S^3$. Remark that if $T$ is projective, then it admits a cutting presentation of the form $T_{(+,V)}$ for some projective object $V \in \Proj(\calC)$. Such a presentation is of course far from being unique. Nonetheless, as follows from Theorem 3 of \cite{GPT09}, if $T_{(+,V)}$ is a cutting presentation of $T$ for some projective object $V \in \Proj(\calC)$, then
\[
 F'_{\calC}(T) := \rmt_V(F_{\calC}(T_{(+,V)}))
\]
is an invariant of the isotopy class of $T$. We call $F'_{\calC}$ the \textit{renormalized invariant of admissible closed ribbon graphs}. Remark that for every $\calC$-colored ribbon graph ${T \subset S^3}$ there exists a unique compatible $G$-coloring $\omega$ of $(S^3,T)$, and if $T$ is projective then $(S^3,T,\omega,0)$ is an admissible closed $2$-morphism of $\bfadCob_{\calC}$. Remark also that the definition of $F'_{\calC}$ does not actually require non-degeneracy of $\calC$. We will however use this condition now, because we need to choose a square root $\calD$ of $\Delta_- \Delta_+$, and to set 
\[
 \delta := \frac{\calD}{\Delta_-} = \frac{\Delta_+}{\calD}.
\] 
These constants will appear in the formula defining the Costantino-Geer-Patureau invariant. In order to review the construction, we first need to fix our terminology a little further, so let us consider $1$-mor\-phisms $\bbSigma,\bbSigma' : \bbGamma \rightarrow \bbGamma'$ of $\bfCob_{\calC}$, let $M$ be a connected $3$-dimensional cobordism with corners from $\varSigma$ to $\varSigma'$, and let $(M',T',\omega',n')$ be a connected $2$-morphism of $\bfCob_{\calC}$ from $\bbSigma$ to $\bbSigma'$. A surgery presentation of $M'$ in $M$ is a framed link $L = L_1 \cup \ldots \cup L_{\ell} \subset M$ together with a positive diffeomorphism $f_L$ between the surgered cobordism with corners $M(L)$ and $M'$. Remark that, since $T'$ can be isotoped inside the image of the exterior $M_L$ of $L$, we can pull it back to a $\calC$-colored ribbon graph in $M$ disjoint from $L$, which we still denote $T'$. Analogously, the restriction of $\omega'$ to $f_L(M_L) \smallsetminus T'$ can be pulled back to a cohomology class in $H^1(M \smallsetminus (L \cup T');G)$, which we still denote $\omega'$. Then, we say a surgery presentation $L = L_1 \cup \ldots \cup L_{\ell} \subset M$ of $M'$ is \textit{computable} with respect to $(T',\omega')$ if $\langle \omega',m_i \rangle \in G \smallsetminus X$ for every integer $1 \leqslant i \leqslant \ell$, where $m_i$ is the homology class of a positive meridian of the component $L_i$. In this case, we denote with $L_{\omega'}$ the $\calC$-colored framed link in $M$ obtained by labeling every $L_i$ with the Kirby color $\smash{\Omega_{\langle \omega',m_i \rangle}}$. Remark that, if $(T',\omega')$ is strongly admissible, then $M'$ admits a computable surgery presentation in $M$ with respect to some $\smash{(\tilde{T}',\tilde{\omega}')}$ obtained by performing projective or generic stabilization. Indeed, if $L \subset M$ is a surgery presentation of $M'$, we can distinguish two cases: either $T'$ is projective, or $\omega'$ is generic. If $T'$ is projective, then, as follows from Remark 4.4 and from the proof of Theorem 4.10 of \cite{CGP14}, a projective stabilization of sufficiently generic index $g$ allows us to turn $L$ into a computable surgery presentation with respect to the resulting $(\calC,G)$-coloring by sliding the edge of color $V_i$ over all surgery components. On the other hand, if $\omega'$ is generic, then we can perform a generic stabilization, which leads us back to the first case.  Once again, although computable surgery presentations are far from being unique, we have the following result.

\begin{proposition}\label{P:CGP_for_admissible_manifolds} 
 If $\bbM = (M,T,\omega,n)$ is a closed connected 2-morphism of $\bfadCob_{\calC}$, and if $L = L_1 \cup \ldots \cup L_{\ell} \subset S^3$ is a computable surgery presentation of $M$ with respect to some admissible $(\calC,G)$-coloring $(\tilde{T},\tilde{\omega})$ obtained from $(T,\omega)$ by performing projective or generic stabilizations, then 
\[
 \CGP_{\calC}(\bbM) := \calD^{-1 - \ell} \delta^{n - \sigma(L)} F'_{\calC} \left( L_{\tilde{\omega}} \cup \tilde{T} \right)
\]
is an invariant of the isomorphism class of $(M,T,\omega,n)$, where $\sigma(L)$ is the signature of the linking matrix of $L$.
\end{proposition}

\begin{proof}
 The formula for $\CGP_{\calC}(\bbM)$ can be rewritten as
 \[
  \CGP_{\calC}(\bbM) = \calD^{-1-b_1(M)} \delta^{n} \frac{F'_{\calC} \left( L_{\tilde{\omega}} \cup \tilde{T} \right)}{\Delta_-^{\sigma_-(L)} \Delta_+^{\sigma_+(L)}}
 \]
 where $b_1(M)$ is the first Betti number of $M$, and where $\sigma_-(L)$ and $\sigma_+(L)$ are the positive and the negative signature of $L$ respectively. Therefore, the statement follows directly from Remark 4.4 and from Theorem 4.10 of \cite{CGP14} when $T$ is projective. When $\omega$ is generic, this essentially follows from Lemma 4.10 of \cite{BCGP16}, after suitably translating the proof from the language of unrolled quantum $\sltwo$ to the language of non-degenerate relative pre-modular categories. For the convenience of the reader, let us give the explicit argument here: if $\kappa, \kappa' \subset M$ are two distinct generic curves for $\omega$, then let $(T_{\kappa},\omega_{\kappa})$ be obtained from $(T,\omega)$ by generic stabilization along $\kappa$, let $(T_{\kappa'},\omega_{\kappa'})$ be obtained from $(T,\omega)$ by generic stabilization along $\kappa'$, and let $(T_{\kappa \cup \kappa'},\omega_{\kappa \cup \kappa'})$ be obtained from $(T,\omega)$ by generic stabilization along both $\kappa$ and $\kappa'$. If we set $L_{\kappa} := f_{\kappa}^{-1}(L_- \cup L_+) \subset M$ and $L_{\kappa'} := f_{\kappa'}^{-1}(L_- \cup L_+) \subset M$,
 then remark that surgery along both $L_{\kappa}$ and $L_{\kappa'}$ is trivial. Now let us consider a surgery presentation $L \subset S^3$ of $M$, and let $I_L$ be the finite set of indices determined by $L$ and $\omega$. Since $X \subset G$ is small, we can choose $g,g' \in G$ such that $(g + I_L) \cup (g' + I_L) \cup (g + g' + I_L) \subset G \smallsetminus X$.
Then, let $(T_{\kappa,L_{\kappa}},\omega_{\kappa,L_{\kappa}})$ be obtained from $(T_{\kappa},\omega_{\kappa})$ by projective stabilization of index $g$ along $L_{\kappa}$, let $(T_{\kappa',L_{\kappa'}},\omega_{\kappa',L_{\kappa'}})$ be obtained from $(T_{\kappa'},\omega_{\kappa'})$ by projective stabilization of index $g'$ along $L_{\kappa'}$, and let $(T_{\kappa \cup \kappa',L_{\kappa} \cup L_{\kappa'}},\omega_{\kappa \cup \kappa',L_{\kappa} \cup L_{\kappa'}})$ be obtained from $(T_{\kappa \cup \kappa'},\omega_{\kappa \cup \kappa'})$ by projective stabilization of index $g$ and $g'$ along $L_{\kappa}$ and $L_{\kappa'}$ respectively. If we slide the edges of color $V_i$ and $V_{i'}$ over all components of $L$, we turn $L$ into a computable surgery presentation of $M$ with respect to $(T_{\kappa \cup \kappa',L_{\kappa} \cup L_{\kappa'}},\omega_{\kappa \cup \kappa',L_{\kappa} \cup L_{\kappa'}})$. Now, if we slide back the edge of color $V_{i'}$ to its original position, and if we undo the second projective stabilization, we turn $L_{\kappa'} \cup L$ into a computable surgery presentation of $M$ with respect to $(T_{\kappa,L_{\kappa}},\omega_{\kappa,L_{\kappa}})$. On the other hand, if we slide back the edge of color $V_i$ to its original position, and if we undo the first projective stabilization, we turn $L_{\kappa} \cup L$ into a computable surgery presentation of $M$ with respect to $(T_{\kappa',L_{\kappa'}},\omega_{\kappa',L_{\kappa'}})$.
\end{proof}

We call $\CGP_{\calC}$ the \textit{Costantino-Geer-Patureau invariant of admissible closed 3-manifolds}.

\section{Universal construction}

In this section we consider the category $\rmadCob_{\calC}$ of closed $1$-mor\-phisms of $\bfCob_{\calC}$ and of admissible $2$-mor\-phisms between them, and we extend the invariant $\CGP_{\calC}$ to a functor $\rmV_{\calC} : \rmadCob_{\calC} \rightarrow \Vect_{\Bbbk}$. This is done by means of the universal construction of Blanchet, Habegger, Masbaum, and Vogel, which is a very natural recipe for doing so. The idea behind it is very straight-forward: first, we should linearize sets of morphisms of $\rmadCob_{\calC}$, and then, we should quotient these vector spaces by means of the invariant $\CGP_{\calC}$. More precisely, let us consider an object $\bbSigma$ of $\rmadCob_{\calC}$, let $\calV(\bbSigma)$ be the free vector space over the set $\rmadCob_{\calC}(\id_{\varnothing},\bbSigma)$ of morphisms $\bbM_{\bbSigma} : \id_{\varnothing} \Rightarrow \bbSigma$ of $\rmadCob_{\calC}$, and let $\calV'(\bbSigma)$ be the free vector space over the set $\rmadCob(\bbSigma,\id_{\varnothing})$ of morphisms $\bbM'_{\bbSigma} : \bbSigma \Rightarrow \id_{\varnothing}$ of $\rmadCob_{\calC}$. We define the bilinear pairing
\[
 \langle \cdot,\cdot \rangle_{\bbSigma} : \calV'(\bbSigma) \times \calV(\bbSigma) \rightarrow \Bbbk,
\]
sending every vector $(\bbM'_{\bbSigma},\bbM_{\bbSigma}) \in \calV'(\bbSigma) \times \calV(\bbSigma)$ to the number
\[
 \CGP_{\calC}(\bbM'_{\bbSigma} \ast \bbM_{\bbSigma}) \in \Bbbk.
\]

\begin{definition}\label{D:univ_vector_space}
 The \textit{universal vector space of $\bbSigma$} is the vector space
 \[
  \rmV_{\calC}(\bbSigma) := \calV(\bbSigma) / \calV'(\bbSigma)^{\perp}.
 \]
 The \textit{dual universal vector space of $\bbSigma$} is the vector space
 \[
  \rmV'_{\calC}(\bbSigma) := \calV'(\bbSigma) / \calV(\bbSigma)^{\perp}.
 \]
\end{definition}

Remark that this terminology is coherent, because for every object $\bbSigma$ of $\rmadCob_{\calC}$ the induced pairing
\[
 \langle \cdot,\cdot \rangle_{\bbSigma} : \rmV'_{\calC}(\bbSigma) \otimes \rmV_{\calC}(\bbSigma) \rightarrow \Bbbk
\]
is non-degenerate, so that $\rmV'_{\calC}(\bbSigma)$ is indeed dual to $\rmV_{\calC}(\bbSigma)$. Now, let us consider a morphism $\bbM : \bbSigma \Rightarrow \bbSigma'$ of $\rmadCob_{\calC}$.

\begin{definition}\label{D:univ_linear_map}
 The \textit{universal linear map of $\bbM$} is the linear map 
 \[
  \rmV_{\calC}(\bbM) : \rmV_{\calC}(\bbSigma) \rightarrow \rmV_{\calC}(\bbSigma')
 \]
 sending every vector $[\bbM_{\bbSigma}] \in \rmV_{\calC}(\bbSigma)$ to the vector
 \[
  [\bbM \ast \bbM_{\bbSigma}] \in \rmV_{\calC}(\bbSigma').
 \]
 The \textit{dual universal linear map of $\bbM$} is the linear map
 \[
  \rmV'_{\calC}(\bbM) : \rmV'_{\calC}(\bbSigma') \rightarrow \rmV'_{\calC}(\bbSigma)
 \]
 sending every vector $[\bbM'_{\bbSigma}] \in \rmV'_{\calC}(\bbSigma')$ to the vector 
 \[
  [\bbM'_{\bbSigma} \ast \bbM] \in \rmV'_{\calC}(\bbSigma).
 \]
\end{definition}

The \textit{universal construction} extends $\CGP_{\calC}$ to a pair of functors
\[
 \rmV_{\calC} : \rmadCob_{\calC} \rightarrow \Vect_{\Bbbk}, \quad
 \rmV'_{\calC} : (\rmadCob_{\calC})^{\op} \rightarrow \Vect_{\Bbbk}
\]
called the \textit{universal quantization functor} and the \textit{dual universal quantization functor} respectively. Some properties of these functors can be established right away. For instance, we can define a natural transformation of the form
\[
 \mu : \otimes \circ \left( \rmV_{\calC} \times \rmV_{\calC} \right) \Rightarrow \rmV_{\calC} \circ \disjun
\]
as follows: for all objects $\bbSigma$ and $\bbSigma'$ of $\rmadCob_{\calC}$, we consider the linear map
\[
 \mu_{\bbSigma,\bbSigma'} : \rmV_{\calC}(\bbSigma) \otimes \rmV_{\calC}(\bbSigma') \rightarrow \rmV_{\calC}(\bbSigma \disjun \bbSigma')
\]
sending every vector $[\bbM_{\bbSigma}] \otimes [\bbM_{\bbSigma'}] \in \rmV_{\calC}(\bbSigma) \otimes \rmV_{\calC}(\bbSigma')$ to the vector
\[
 \left[ \bbM_{\bbSigma} \disjun \bbM_{\bbSigma'} \right] \in \rmV_{\calC}(\bbSigma \disjun \bbSigma').
\]
Similarly, we can define a natural transformation of the form
\[
 \mu' : \otimes \circ \left( \rmV'_{\calC} \times \rmV'_{\calC} \right) \Rightarrow \rmV'_{\calC} \circ \disjun
\]
as follows: for all objects $\bbSigma$ and $\bbSigma'$ of $\rmadCob_{\calC}$, we consider the linear map
\[
 \mu'_{\bbSigma,\bbSigma'} : \rmV'_{\calC}(\bbSigma) \otimes \rmV'_{\calC}(\bbSigma') \rightarrow \rmV'_{\calC}(\bbSigma \disjun \bbSigma')
\]
sending every vector $[\bbM'_{\bbSigma}] \otimes [\bbM'_{\bbSigma'}] \in \rmV'_{\calC}(\bbSigma) \otimes \rmV'_{\calC}(\bbSigma')$ to the vector
\[
 \left[ \bbM_{\bbSigma} \disjun \bbM_{\bbSigma'} \right] \in \rmV'_{\calC}(\bbSigma \disjun \bbSigma').
\]

\begin{proposition}\label{P:lax_monoidality_TQFT}
 The linear maps $\mu_{\bbSigma,\bbSigma'} : \rmV_{\calC}(\bbSigma) \otimes \rmV_{\calC}(\bbSigma') \rightarrow \rmV_{\calC}(\bbSigma \disjun \bbSigma')$ and $\mu'_{\bbSigma,\bbSigma'} : \rmV'_{\calC}(\bbSigma) \otimes \rmV'_{\calC}(\bbSigma') \rightarrow \rmV'_{\calC}(\bbSigma \disjun \bbSigma')$ are injective for all $\bbSigma, \bbSigma' \in \rmadCob_{\calC}$.
\end{proposition}

\begin{proof}
 Let us consider a trivial vector of the form
 \[
  \sum_{i=1}^m \alpha_i \cdot \left[ \bbM_{\bbSigma,i} \disjun \bbM_{\bbSigma',i} \right] \in \rmV_{\calC}(\bbSigma \disjun \bbSigma).
 \]
 By definition, this means the pairing
 \[
  \sum_{i=1}^m \alpha_i \CGP_{\calC} \left( \bbM'_{\bbSigma \disjun \bbSigma'} \ast (\bbM_{\bbSigma,i} \disjun \bbM_{\bbSigma',i}) \right)
 \]
 is zero for every vector $\bbM'_{\bbSigma \disjun \bbSigma'} \in \rmV'_{\calC}(\bbSigma \disjun \bbSigma')$. In particular, every vector of the form $\bbM'_{\bbSigma} \disjun \bbM'_{\bbSigma'} \in \rmV'_{\calC}(\bbSigma \disjun \bbSigma')$ has trivial pairing. This means the vector
 \[
  \sum_{i=1}^m \alpha_i \cdot \left[ \bbM_{\bbSigma,i} \right] \otimes \left[ \bbM_{\bbSigma',i} \right] \in \rmV_{\calC}(\bbSigma) \otimes \rmV_{\calC}(\bbSigma')
 \]
 is trivial too. The proof in for the dual case is completely analogous.
\end{proof}

\section{Extended universal construction}\label{S:extended_universal_construction}

In this section we extend the functor $\rmV_{\calC} : \rmadCob_{\calC} \rightarrow \Vect_{\Bbbk}$ to a strict 2-functor $\bfA_{\calC} : \bfadCob_{\calC} \rightarrow \bfCat_{\Bbbk}$ by repeating the universal construction on a higher categorical level, and we refer to Appendix \ref{A:symmetric_monoidal} for the terminology and notation we use. The idea remains the same: first, we should linearize categories of morphisms of $\bfadCob_{\calC}$, and then, we should quotient these linear categories by means of the functor $\rmV_{\calC}$. This requires fixing some terminology. First of all, we need to say what it means to linearize a category. This is easy: every category determines a \textit{free linear category} sharing its set of objects, which is obtained by specifying the free vector spaces over its sets of morphisms as vector spaces of morphisms. Next, we need to say what it means to quotient a linear category. This requires a preliminary definition: a \textit{linear congruence} $R$ on a linear category $A$ is given by the choice of a linear subspace $R(x,y) \subset \Hom_A(x,y)$ for every pair of objects $x,y \in A$ satisfying:
\begin{enumerate}
 \item $f \in R(x,y)$ $\Rightarrow$ $g \circ f \in R(x,z)$ $\Forall g \in \Hom_A(y,z)$;
 \item $g \in R(y,z)$ $\Rightarrow$ $g \circ f \in R(x,z)$ $\Forall f \in \Hom_A(x,y)$.
\end{enumerate}
If $A$ is a linear category, and if $R$ is a linear congruence on $A$, then the \textit{quotient linear category} $A/R$ is the linear category with the same set of objects of $A$, and with vector spaces of morphisms defined by $\Hom_{A/R}(x,y) := \Hom_A(x,y)/R(x,y)$ for all $x,y \in A$. Then, we need to say how to obtain linear congruences from the universal quantization functor. This is done as before: a \textit{bilinear pairing} between linear categories $A$ and $A'$ is a functor $\langle \cdot,\cdot \rangle : A' \times A \rightarrow \Vect_{\Bbbk}$ which is linear in every factor, the \textit{annihilator of $A'$} with respect to $\langle \cdot,\cdot \rangle$ is the linear congruence $A'^{\perp}$ on $A$ defined by
\[
 A'^{\perp}(x,y) := \{ f \in \Hom_A(x,y) \mid \langle \id_{x'},f \rangle = 0 \Forall x' \in A' \},
\]
and analogously the \textit{annihilator of $A$} with respect to $\langle \cdot,\cdot \rangle$ is the linear congruence $A^{\perp}$ on $A'$ defined by
\[
 A^{\perp}(x',y') := \{ f' \in \Hom_{A'}(x',y') \mid \langle f',\id_x \rangle = 0 \Forall x \in A \}.
\]
We are now ready to implement our construction, so let us consider an object $\bbGamma$ of $\bfadCob_{\calC}$, let $\calA(\bbGamma)$ be the free linear category over the category $\bfadCob_{\calC}(\varnothing,\bbGamma)$ of $1$-mor\-phisms $\bbSigma_{\bbGamma} : \varnothing \rightarrow \bbGamma$ and $2$-mor\-phisms $\bbM_{\bbGamma} : \bbSigma_{\bbGamma} \Rightarrow \bbSigma''_{\bbGamma}$ of $\bfadCob_{\calC}$, and let $\calA'(\bbGamma)$ be the free linear category over the category $\bfadCob_{\calC}(\bbGamma,\varnothing)$ of $1$-mor\-phisms $\bbSigma'_{\bbGamma} : \bbGamma \rightarrow \varnothing$ and $2$-mor\-phisms $\bbM'_{\bbGamma} : \bbSigma'_{\bbGamma} \Rightarrow \bbSigma'''_{\bbGamma}$ of $\bfadCob_{\calC}$. We define the bilinear pairing
\[
 \langle \cdot,\cdot \rangle_{\bbGamma} : \calA'(\bbGamma) \times \calA(\bbGamma) \rightarrow \Vect_{\Bbbk},
\]
sending every object $(\bbSigma'_{\bbGamma}, \bbSigma_{\bbGamma}) \in \calA'(\bbGamma) \times \calA(\bbGamma)$ to the vector space
\[
 \rmV_{\calC}(\bbSigma'_{\bbGamma} \circ \bbSigma_{\bbGamma}) \in \Vect_{\Bbbk},
\]
and sending every morphism 
$(\bbM'_{\bbGamma},\bbM_{\bbGamma}) \in \Hom_{\calA'(\bbGamma) \times \calA(\bbGamma)}((\bbSigma'_{\bbGamma},\bbSigma_{\bbGamma}),(\bbSigma'''_{\bbGamma},\bbSigma''_{\bbGamma}))$
to the linear map
\[
 \rmV_{\calC}(\bbM'_{\bbGamma} \circ \bbM_{\bbGamma}) \in \Hom_{\Bbbk} \left( \rmV_{\calC}(\bbSigma'_{\bbGamma} \circ \bbSigma_{\bbGamma}),\rmV_{\calC}(\bbSigma'''_{\bbGamma} \circ \bbSigma''_{\bbGamma}) \right).
\]

\begin{definition}\label{D:univ_linear_category}
 The \textit{universal linear category of $\bbGamma$} is the linear category
 \[
  \bfA_{\calC}(\bbGamma) := \calA(\bbGamma)/\calA'(\bbGamma)^{\perp}.
 \]
 The \textit{dual universal linear category of $\bbGamma$} is the linear category
 \[
  \bfA'_{\calC}(\bbGamma) := \calA'(\bbGamma)/\calA(\bbGamma)^{\perp}.
 \]
\end{definition}

Let us consider a 1-morphism $\bbSigma : \bbGamma \rightarrow \bbGamma'$ of $\bfadCob_{\calC}$.

\begin{definition}\label{D:univ_linear_functor}
 The \textit{universal linear functor of $\bbSigma$} is the linear functor
 \[
  \bfA_{\calC}(\bbSigma) : \bfA_{\calC}(\bbGamma) \rightarrow \bfA_{\calC}(\bbGamma')
 \]
 sending every object $\bbSigma_{\bbGamma} \in \bfA_{\calC}(\bbGamma)$ to the object
 \[
  \bbSigma \circ \bbSigma_{\bbGamma} \in \bfA_{\calC}(\bbGamma'),
 \]
 and every morphism $\left[ \bbM_{\bbGamma} \right] \in \Hom_{\bfA_{\calC}(\bbGamma)}(\bbSigma_{\bbGamma},\bbSigma''_{\bbGamma})$ to the morphism
 \[
  \left[ \id_{\bbSigma} \circ \bbM_{\bbGamma} \right] \in \Hom_{\bfA_{\calC}(\bbGamma')}(\bbSigma \circ \bbSigma_{\bbGamma},\bbSigma \circ \bbSigma''_{\bbGamma}).
 \]
 The \textit{dual universal linear functor of $\bbSigma$} is the linear functor
 \[
  \bfA'_{\calC}(\bbSigma) : \bfA'_{\calC}(\bbGamma') \rightarrow \bfA'_{\calC}(\bbGamma)
 \]
 sending every object $\bbSigma'_{\bbGamma'} \in \bfA'_{\calC}(\bbGamma')$ to the object 
 \[
  \bbSigma'_{\bbGamma'} \circ \bbSigma \in \bfA'_{\calC}(\bbGamma),
 \]
 and every morphism $\left[ \bbM'_{\bbGamma'} \right] \in \Hom_{\bfA'_{\calC}(\bbGamma')}(\bbSigma'_{\bbGamma'},\bbSigma'''_{\bbGamma'})$ to the morphism
 \[
  \left[ \bbM'_{\bbGamma'} \circ \id_{\bbSigma} \right] \in \Hom_{\bfA'_{\calC}(\bbGamma)}(\bbSigma'_{\bbGamma'} \circ \bbSigma, \bbSigma'''_{\bbGamma'} \circ \bbSigma).
 \]
\end{definition}

Let us consider a 2-morphism $\bbM : \bbSigma \Rightarrow \bbSigma'$ of $\bfadCob_{\calC}$ between 1-morphisms $\bbSigma, \bbSigma' : \bbGamma \rightarrow \bbGamma'$.

\begin{definition}\label{D:univ_natural_transformation}
 The \textit{universal natural transformation of $\bbM$} is the natural transformation
 \[
  \bfA_{\calC}(\bbM) : \bfA_{\calC}(\bbSigma) \Rightarrow \bfA_{\calC}(\bbSigma')
 \]
 associating with every object $\bbSigma_{\bbGamma} \in \bfA_{\calC}(\bbSigma)$ the morphism 
 \[
  \left[ \bbM \circ \id_{\bbSigma_{\bbGamma}} \right] \in \Hom_{\bfA_{\calC}(\bbGamma')}(\bbSigma \circ \bbSigma_{\bbGamma},\bbSigma' \circ \bbSigma_{\bbGamma}).
 \]
 The \textit{dual universal natural transformation of $\bbM$} is the natural transformation
 \[
  \bfA'_{\calC}(\bbM) : \bfA'_{\calC}(\bbSigma) \Rightarrow \bfA'_{\calC}(\bbSigma')
 \]
 associating with every object $\bbSigma'_{\bbGamma'} \in \bfA'_{\calC}(\bbSigma)$ the morphism
 \[
  \left[ \id_{\bbSigma'_{\bbGamma'}} \circ \bbM \right] \in \Hom_{\bfA'_{\calC}(\bbGamma)}(\bbSigma'_{\bbGamma'} \circ \bbSigma,\bbSigma'_{\bbGamma'} \circ \bbSigma').
 \]
\end{definition}

The \textit{extended universal construction} extends $\CGP_{\calC}$ to a pair of $2$-func\-tors
\[
 \bfA_{\calC} : \bfadCob_{\calC} \rightarrow \bfCat_{\Bbbk}, \quad
 \bfA'_{\calC} : (\bfadCob_{\calC})^{\op} \rightarrow \bfCat_{\Bbbk}
\]
called the \textit{universal quantization $2$-func\-tor} and the \textit{dual universal quantization $2$-func\-tor} respectively, where $(\bfadCob_{\calC})^{\op}$ is the $2$-cat\-e\-go\-ry with same class of objects as $\bfadCob_{\calC}$, and with categories of morphisms $(\bfadCob_{\calC})^{\op}(\bbGamma,\bbGamma') := \bfadCob_{\calC}(\bbGamma',\bbGamma)$ for all objects $\bbGamma, \bbGamma' \in \bfadCob_{\calC}$. We denote with
\[
 \hat{\bfA}_{\calC} : \bfadCob_{\calC} \rightarrow \coCat_{\Bbbk}, \quad
 \hat{\bfA}'_{\calC} : (\bfadCob_{\calC})^{\op} \rightarrow \coCat_{\Bbbk}
\]
the completions of the universal quantization $2$-func\-tor and of the dual universal quantization $2$-func\-tor in the sense of Proposition \ref{P:co_strict_2-funct} and Remark \ref{R:2-completion}, where $\coCat_{\Bbbk}$ is the symmetric monoidal $2$-cat\-e\-go\-ry of complete linear categories, see Definition \ref{D:sym_mon_2-cat_of_co_lin_cat}. These $2$-func\-tors will provide our first candidates for ETQFTs extending $\CGP_{\calC}$, although, as we will see later, they fail at being symmetric monoidal, and we will need to refine our construction in order to obtain the desired feature. However, some properties of these $2$-func\-tors can be established right away. First of all, they are strict, as a consequence of our confusing $\bfadCob_{\calC}$, and thus also $(\bfadCob_{\calC})^{\op}$, with strict $2$-cat\-e\-go\-ries. Moreover, if $\sqtimes$ denotes the enriched tensor product of linear categories introduced in Section \ref{S:lin_&_gr_lin_cat}, we can define a $2$-trans\-for\-ma\-tion of the form
\[
 \bfmu : \sqtimes \circ \left( \bfA_{\calC} \times \bfA_{\calC} \right) \Rightarrow \bfA_{\calC} \circ \disjun
\]
as follows: for all objects $\bbGamma$ and $\bbGamma'$ of $\bfadCob_{\calC}$, we consider the linear functor
\[
 \bfmu_{\bbGamma,\bbGamma'} : \bfA_{\calC}(\bbGamma) \sqtimes \bfA_{\calC}(\bbGamma') \rightarrow \bfA_{\calC}(\bbGamma \disjun \bbGamma')
\]
sending every object $(\bbSigma_{\bbGamma},\bbSigma_{\bbGamma'})$ of $\bfA_{\calC}(\bbGamma) \sqtimes \bfA_{\calC}(\bbGamma')$ to the object $\bbSigma_{\bbGamma} \disjun \bbSigma_{\bbGamma'}$ of $\bfA_{\calC}(\bbGamma \disjun \bbGamma')$,
and every morphism 
\[
 [\bbM_{\bbGamma}] \otimes [\bbM_{\bbGamma'}]
\]
of 
$\Hom_{\bfA_{\calC}(\bbGamma) \sqtimes \bfA_{\calC}(\bbGamma')}((\bbSigma_{\bbGamma},\bbSigma_{\bbGamma'}),(\bbSigma''_{\bbGamma},\bbSigma''_{\bbGamma'}))$ 
to the morphism
\[
 \left[ \bbM_{\bbGamma} \disjun \bbM_{\bbGamma'} \right]
\]
of $\Hom_{\bfA_{\calC}(\bbGamma \disjun \bbGamma')}(\bbSigma_{\bbGamma} \disjun \bbSigma_{\bbGamma'},\bbSigma''_{\bbGamma} \disjun \bbSigma''_{\bbGamma'})$. For all $1$-mor\-phisms $\bbSigma : \bbGamma \rightarrow \bbGamma''$ and $\bbSigma' : \bbGamma' \rightarrow \bbGamma'''$ of $\bfadCob_{\calC}$, we consider the natural transformation 
\[
 \bfmu_{\bbSigma,\bbSigma'} : \bfA_{\calC}(\bbSigma \disjun \bbSigma') \circ \bfmu_{\bbGamma,\bbGamma'} \Rightarrow \bfmu_{\bbGamma'',\bbGamma'''} \circ \left( \bfA_{\calC}(\bbSigma) \sqtimes \bfA_{\calC}(\bbSigma') \right)
\]
associating with every object $(\bbSigma_{\bbGamma},\bbSigma_{\bbGamma'})$ of $\bfA_{\calC}(\bbGamma) \sqtimes \bfA_{\calC}(\bbGamma')$
 the coherence morphism
\[
 \left[ \disjun_{(\bbSigma,\bbSigma'),(\bbSigma_{\bbGamma},\bbSigma_{\bbGamma'})} \right]
\]
of $\Hom_{\bfA_{\calC}(\bbGamma'' \disjun \bbGamma''')}((\bbSigma \disjun \bbSigma') \circ (\bbSigma_{\bbGamma} \disjun \bbSigma_{\bbGamma'}),{(\bbSigma \circ \bbSigma_{\bbGamma})} \disjun {(\bbSigma' \circ \bbSigma_{\bbGamma'})})$. 
The same can be done in the dual case. Indeed, we can define a $2$-trans\-for\-ma\-tion of the form
\[
 \bfmu' : \sqtimes \circ \left( \bfA'_{\calC} \times \bfA'_{\calC} \right) \Rightarrow \bfA'_{\calC} \circ \disjun
\]
as follows: for all objects $\bbGamma$ and $\bbGamma'$ of $\bfadCob_{\calC}$, we consider the linear functor
\[
 \bfmu'_{\bbGamma,\bbGamma'} : \bfA'_{\calC}(\bbGamma) \sqtimes \bfA'_{\calC}(\bbGamma') \rightarrow \bfA'_{\calC}(\bbGamma \disjun \bbGamma')
\]
sending every object $(\bbSigma'_{\bbGamma},\bbSigma'_{\bbGamma'})$ of $\bfA'_{\calC}(\bbGamma) \sqtimes \bfA'_{\calC}(\bbGamma')$ to the object $\bbSigma'_{\bbGamma} \disjun \bbSigma'_{\bbGamma'}$ of $\bfA'_{\calC}(\bbGamma \disjun \bbGamma')$,
and every morphism 
\[
 [\bbM'_{\bbGamma}] \otimes [\bbM'_{\bbGamma'}]
\]
of $\Hom_{\bfA'_{\calC}(\bbGamma) \sqtimes \bfA'_{\calC}(\bbGamma')}((\bbSigma'_{\bbGamma},\bbSigma'_{\bbGamma'}),(\bbSigma'''_{\bbGamma},\bbSigma'''_{\bbGamma'}))$ to the morphism
\[
 \left[ \bbM'_{\bbGamma} \disjun \bbM'_{\bbGamma'} \right]
\]
of $\Hom_{\bfA'_{\calC}(\bbGamma \disjun \bbGamma')}(\bbSigma'_{\bbGamma} \disjun \bbSigma'_{\bbGamma'},\bbSigma'''_{\bbGamma} \disjun \bbSigma'''_{\bbGamma'})$. For all $1$-mor\-phisms $\bbSigma : \bbGamma \rightarrow \bbGamma''$ and $\bbSigma' : \bbGamma' \rightarrow \bbGamma'''$ of $\bfadCob_{\calC}$, we consider the natural transformation
\[
 \bfmu'_{\bbSigma,\bbSigma'} : \bfA'_{\calC}(\bbSigma \disjun \bbSigma') \circ \bfmu'_{\bbGamma'',\bbGamma'''} \Rightarrow \bfmu'_{\bbGamma,\bbGamma'} \circ \left( \bfA'_{\calC}(\bbSigma) \sqtimes \bfA'_{\calC}(\bbSigma') \right)
\]
associating with every object $(\bbSigma'_{\bbGamma''},\bbSigma'_{\bbGamma'''})$ of $\bfA'_{\calC}(\bbGamma'') \sqtimes \bfA'_{\calC}(\bbGamma''')$
the coherence morphism
\[
 \left[ \disjun_{(\bbSigma'_{\bbGamma''},\bbSigma'_{\bbGamma'''}),(\bbSigma,\bbSigma')} \right]
\]
of $\Hom_{\bfA'_{\calC}(\bbGamma \disjun \bbGamma')}((\bbSigma'_{\bbGamma''} \disjun \bbSigma'_{\bbGamma'''}) \circ (\bbSigma \disjun \bbSigma'),{(\bbSigma'_{\bbGamma''} \circ \bbSigma)} \disjun {(\bbSigma'_{\bbGamma'''} \circ \bbSigma')})$. 

\begin{proposition}\label{P:lax_monoidality_ETQFT}
 The linear functors $\bfmu_{\bbGamma,\bbGamma'} : \bfA_{\calC}(\bbGamma) \sqtimes \bfA_{\calC}(\bbGamma') \rightarrow \bfA_{\calC}(\bbGamma \disjun \bbGamma')$ and $\bfmu'_{\bbGamma,\bbGamma'} : \bfA'_{\calC}(\bbGamma) \sqtimes \bfA'_{\calC}(\bbGamma') \rightarrow \bfA'_{\calC}(\bbGamma \disjun \bbGamma')$ are faithful for all $\bbGamma, \bbGamma' \in \bfadCob_{\calC}$.
\end{proposition}

\begin{proof}
 Let us consider a trivial morphism of the form
 \[
  \sum_{i=1}^m \alpha_i \cdot \left[ \bbM_{\bbGamma,i} \disjun \bbM_{\bbGamma',i} \right] \in \Hom_{\bfA_{\calC}(\bbGamma \disjun \bbGamma')}(\bbSigma_{\bbGamma} \disjun \bbSigma_{\bbGamma'},\bbSigma''_{\bbGamma} \disjun \bbSigma''_{\bbGamma'}).
 \]
 By definition, this means the linear map 
 \[
  \sum_{i=1}^m \alpha_i \cdot \rmV_{\calC} \left( \id_{\bbSigma'_{\bbGamma \disjun \bbGamma'}} \circ (\bbM_{\bbGamma,i} \disjun \bbM_{\bbGamma',i}) \right)  
 \]
 is zero for every object $\bbSigma'_{\bbGamma \disjun \bbGamma'} \in \bfA'_{\calC}(\bbGamma \disjun \bbGamma')$. In particular, every object of the form $\bbSigma'_{\bbGamma} \disjun \bbSigma'_{\bbGamma'} \in \bfA'_{\calC}(\bbGamma \disjun \bbGamma')$ determines a zero linear map. This means the morphism
 \[
  \sum_{i=1}^m \alpha_i \cdot \left[ \bbM_{\bbGamma,i} \right] \otimes \left[ \bbM_{\bbGamma',i} \right] \in \Hom_{\bfA_{\calC}(\bbGamma)}(\bbSigma_{\bbGamma},\bbSigma''_{\bbGamma}) \otimes \Hom_{\bfA_{\calC}(\bbGamma')}(\bbSigma_{\bbGamma'},\bbSigma''_{\bbGamma'})
 \]
 is trivial too. The proof in for the dual case is completely analogous.
\end{proof}


We end this chapter with a final remark: for every admissible object $\bbSigma$ of $\rmadCob_{\calC}$, we have 
 \[
  \rmV_{\calC}(\bbSigma) = \Hom_{\bfA_{\calC}(\varnothing)}(\id_{\varnothing}, \bbSigma).
 \]
 Indeed, both $\rmV_{\calC}(\bbSigma)$ and $\Hom_{\bfA_{\calC}(\varnothing)}(\id_{\varnothing}, \bbSigma)$ are quotients of the free vector space generated by the set of $2$-morphisms of $\bfadCob_{\calC}$ from $\id_{\varnothing}$ to $\bbSigma$. In order to prove our claim, we have to show that they are the exact same quotient. It is clear that if the morphism
 \[
  \sum_{i=1}^m \alpha_i \cdot \left[ \bbM_{\varnothing,i} \right] \in 
  \Hom_{\bfA_{\calC}(\varnothing)}(\id_{\varnothing}, \bbSigma)
 \]
 is trivial, then the vector
 \[
  \sum_{i=1}^m \alpha_i \cdot \left[ \bbM_{\varnothing,i} \right] \in \rmV_{\calC}(\bbSigma)
 \]
 is trivial too. Indeed, if the linear map
 \[
  \sum_{i=1}^m \alpha_i \cdot \rmV_{\calC} \left( \id_{\bbSigma'_{\varnothing}} \circ \bbM_{\varnothing,i} \right) 
 \]
 is trivial for every object $\bbSigma'_{\varnothing} \in \bfA_{\calC}(\varnothing)$, then in particular it is trivial for $\bbSigma'_{\varnothing} = \id_{\varnothing}$. To show that the converse is also true, let us suppose the vector
 \[
  \sum_{i=1}^m \alpha_i \cdot \left[ \bbM_{\bbSigma,i} \right] \in 
  \rmV_{\calC}(\bbSigma)
 \]
 is trivial. This means 
 \[
  \sum_{i=1}^m \alpha_i \CGP_{\calC} \left( \bbM'_{\bbSigma} \ast \bbM_{\bbSigma,i} \right)
 \]
 is zero for every vector $[ \bbM'_{\bbSigma} ] \in \rmV'_{\calC}(\bbSigma)$. Now the morphism
 \[
  \sum_{i=1}^m \alpha_i \cdot \left[ \bbM_{\bbSigma,i} \right] \in 
  \Hom_{\bfA_{\calC}(\varnothing)}(\id_{\varnothing}, \bbSigma)
 \]
 is trivial if
 \[
  \sum_{i=1}^m \alpha_i \CGP_{\calC} \left( \bbM'_{\bbSigma'_{\varnothing} \circ \bbSigma} \ast \left( \id_{\bbSigma'_{\varnothing}} \circ \bbM_{\bbSigma,i} \right) \ast \bbM_{\bbSigma'_{\varnothing}} \right)
 \]
 is zero for every object $\bbSigma'_{\varnothing} \in \bfA_{\calC}'(\varnothing)$, and for all vectors $[ \bbM_{\bbSigma'_{\varnothing}} ] \in \rmV_{\calC}(\bbSigma'_{\varnothing})$ and $[ \bbM'_{\bbSigma'_{\varnothing} \circ \bbSigma} ] \in \rmV'_{\calC}(\bbSigma'_{\varnothing} \circ \bbSigma)$. But since
 \[
  \bbM'_{\bbSigma'_{\varnothing} \circ \bbSigma} \ast \left( \id_{\bbSigma'_{\varnothing}} \circ \bbM_{\bbSigma,i} \right) \ast \bbM_{\bbSigma'_{\varnothing}}
  = \left( \bbM'_{\bbSigma'_{\varnothing} \circ \bbSigma} \ast \left( \left( \id_{\bbSigma'_{\varnothing}} \ast \bbM_{\bbSigma'_{\varnothing}} \right) \circ \id_{\bbSigma} \right) \right) \ast \bbM_{\bbSigma,i}
 \]
 for every integer $1 \leqslant i \leqslant m$, we can conclude.


%% file: chapter_4.tex
%
%
%

\chapter{Combinatorial and topological properties}\label{Ch:combinatorial_topological_properties}

In this chapter we begin to study the $2$-func\-tor $\bfA_{\calC} : \bfadCob_{\calC} \rightarrow \bfCat_{\Bbbk}$ of Section \ref{S:extended_universal_construction}. In particular, we discuss properties related the behaviour of the Costantino-Geer-Patureau invariant $\CGP_{\calC}$ under combinatorial and topological operations. From now on, $\calC$ will denote a fixed modular $G$-category relative to $(\PGr,X)$.

\section{Skein equivalence}\label{S:skein_equivalence}

In this section we extend the notion of skein equivalence, which was introduced in Section \ref{S:group_structures} for morphisms of $\Rib_{\calC}^G$, to an equivalence relation on admissible $(\calC,G)$-colorings of a 3-dimensional cobordism with corners. The idea is to declare two admissible $(\calC,G)$-colorings skein equivalent when they are related by a local skein equivalence inside a $3$-disc. To better explain this, let us consider the map
\[
 \begin{array}{rccc}
  f_{D^3} : & D^2 \times I & \rightarrow & D^3 \\
  & \left( \left( x,y \right), t \right) & \mapsto & \left( x,y,(2t-1) \sqrt{1 - \left( x^2 + y^2 \right)} \right).
 \end{array}
\] 
If $(\underline{\varepsilon},\underline{V})$ and $(\underline{\varepsilon'},\underline{V'})$ are objects of $\Rib_{\calC}^G$, we can use $f_{D^3}$ to define by pull-back a standard $\calC$-colored ribbon set $P_{(\underline{\varepsilon},\underline{V})^*,(\underline{\varepsilon'},\underline{V'})}$ inside $S^2$, and we denote with $\bbS^2_{(\underline{\varepsilon},\underline{V})^*,(\underline{\varepsilon'},\underline{V'})}$ the closed $1$-mor\-phism of $\bfCob_{\calC}$ given by
\[
 (S^2,P_{(\underline{\varepsilon},\underline{V})^*,(\underline{\varepsilon'},\underline{V'})},\vartheta_{(\underline{\varepsilon},\underline{V})^*,(\underline{\varepsilon'},\underline{V'})},\{ 0 \}),
\]
where $\vartheta_{(\underline{\varepsilon},\underline{V})^*,(\underline{\varepsilon'},\underline{V'})}$ is the unique compatible $G$-coloring of $(S^2,P_{(\underline{\varepsilon},\underline{V})^*,(\underline{\varepsilon'},\underline{V'})})$. If $T \subset D^3$ is a $G$-homogeneous $\calC$-colored ribbon graph from $\varnothing$ to $P_{(\underline{\varepsilon},\underline{V})^*,(\underline{\varepsilon'},\underline{V'})}$, we denote with $\bbD^3_T : \id_{\varnothing} \Rightarrow \bbS^2_{(\underline{\varepsilon},\underline{V})^*,(\underline{\varepsilon'},\underline{V'})}$ the $2$-mor\-phism of $\bfCob_{\calC}$ given by
\[
 (D^3,T,\omega_T,0),
\]
where $\omega_T$ is the unique compatible $G$-coloring of $(D^3,T)$. Then, if $(M,T,\omega,n)$ is a $2$-mor\-phism of $\bfadCob_{\calC}$ between $1$-mor\-phisms $\bbSigma, \bbSigma' : \bbGamma \rightarrow \bbGamma'$ which decomposes as
\[
 (M,T,\omega,n) = \sum_{i=1}^m \alpha_i \cdot \bbM_{(\underline{\varepsilon},\underline{V})^*,(\underline{\varepsilon'},\underline{V'})} \ast \left( \bbD^3_{T_i} \disjun \id_{\bbSigma} \right)
\]
for some $2$-mor\-phism $\bbM_{(\underline{\varepsilon},\underline{V})^*,(\underline{\varepsilon'},\underline{V'})} : \bbS^2_{(\underline{\varepsilon},\underline{V})^*,(\underline{\varepsilon'},\underline{V'})} \disjun \bbSigma \Rightarrow \bbSigma'$ of $\bfadCob_{\calC}$, we say a $(\calC,G)$-coloring $(T',\omega')$ of $M$ satisfying
\[
 (M,T',\omega',n) = \sum_{i'=1}^{m'} \alpha'_{i'} \cdot \bbM_{(\underline{\varepsilon},\underline{V})^*,(\underline{\varepsilon'},\underline{V'})} \ast \left( \bbD^3_{T'_{i'}} \disjun \id_{\bbSigma} \right)
\]
is \textit{skein equivalent} to $(T,\omega)$ if
\[
 \sum_{i=1}^m \alpha_i \cdot f^{-1}_{D^3} \left( T_i \right) \doteq \sum_{i' = 1}^{m'} \alpha'_{i'} \cdot f^{-1}_{D^3} \left( T'_{i'} \right).
\]
The projective stabilization of Section \ref{S:projective_generic_stabilizations} provides an example of skein equivalence.

\begin{remark}\label{R:skein_equivalence}
 If $(M,T,\omega,n)$ is a closed $2$-mor\-phism of $\bfadCob_{\calC}$, and if the $(\calC,G)$-coloring $(T',\omega')$ of $M$ is skein equivalent to $(T,\omega)$, then
 \[
  \CGP_{\calC}(M,T,\omega,n) = \CGP_{\calC}(M,T',\omega',n).
 \]
\end{remark}

\begin{lemma}\label{L:skein_equivalence}
 If $[ M_{\varGamma},T,\omega,n ]$ is a morphism of $\Hom_{\bfA_{\calC}(\bbGamma)}(\bbSigma_{\bbGamma},\bbSigma''_{\bbGamma})$,
 and if the $(\calC,G)$-coloring $(T',\omega')$ of $M_{\varGamma}$ is skein equivalent to $(T,\omega)$, then
 \[
  [M_{\varGamma},T,\omega,n] = [M_{\varGamma},T',\omega',n]
 \]
 as morphisms of $\Hom_{\bfA_{\calC}(\bbGamma)} ( \bbSigma_{\bbGamma},\bbSigma''_{\bbGamma} )$.
\end{lemma}


\begin{proof}
 We have to show that for every object $\bbSigma'_{\bbGamma} \in \bfA'_{\calC}(\bbGamma)$ the linear map
 \[
  \rmV_{\calC} \left( \id_{\bbSigma'_{\bbGamma}} \circ \left( M_{\varGamma},T,\omega,n \right) \right) 
 \]
 is equal to the linear map
 \[
 \rmV_{\calC} \left( \id_{\bbSigma'_{\bbGamma}} \circ \left( M_{\varGamma},T',\omega',n \right) \right).
 \]
 This means we have to show that for all vectors $[\bbM_{\bbSigma'_{\bbGamma} \circ \bbSigma_{\bbGamma}}] \in \rmV_{\calC}(\bbSigma'_{\bbGamma} \circ \bbSigma_{\bbGamma})$ and $[\bbM'_{\bbSigma'_{\bbGamma} \circ \bbSigma''_{\bbGamma}}] \in \rmV'_{\calC}(\bbSigma'_{\bbGamma} \circ \bbSigma''_{\bbGamma})$ the invariant
 \[
  \CGP_{\calC} \left( \bbM'_{\bbSigma'_{\bbGamma} \circ \bbSigma''_{\bbGamma}} \ast \left( \id_{\bbSigma'_{\bbGamma}} \circ \left( M_{\varGamma},T,\omega,n \right) \right) \ast \bbM_{\bbSigma'_{\bbGamma} \circ \bbSigma_{\bbGamma}} \right)
 \]
 is equal to the invariant
 \[
  \CGP_{\calC} \left( \bbM'_{\bbSigma'_{\bbGamma} \circ \bbSigma''_{\bbGamma}} \ast \left( \id_{\bbSigma'_{\bbGamma}} \circ \left( M_{\varGamma},T',\omega',n \right) \right) \ast \bbM_{\bbSigma'_{\bbGamma} \circ \bbSigma_{\bbGamma}} \right).
 \]
 Now this follows directly from Remark \ref{R:skein_equivalence}.
\end{proof}

\section{Surgery axioms}\label{S:surgery_axioms}

In this section we introduce decorated surgeries of index $0$, $1$, and $2$. These operations transform a $2$-mor\-phism of $\bfadCob_{\calC}$ into a formal linear combination of $2$-mor\-phisms of $\bfadCob_{\calC}$, which we still interpret as a $2$-mor\-phism of $\bfadCob_{\calC}$, and they modify both topological properties of supports and combinatorial properties of decorations. On the level of topology, they consist in $3$-di\-men\-sion\-al surgeries, replacing a copy of $S^{i-1} \times (-1)^{i-1} D^{4-i}$ with a copy of $D^i \times S^{3-i}$ for some integer $0 \leqslant i \leqslant 2$, with the convention that $S^{-1} := \varnothing$ and that $D^0 = \{ 0 \}$. To explain how they affect decorations, we need to introduce some notation, so let us consider a generic $g \in G \smallsetminus X$, together with some $i \in \rmI_g$. We begin by defining index $0$ surgery morphisms, which are very simple: The \textit{index $0$ surgery $1$-mor\-phism $\bbSigma_0 : \varnothing \rightarrow \varnothing$ of $\bfadCob_{\calC}$} is just
\[
 \left( S^{-1} \times S^3,\varnothing,0,\{ 0 \} \right) = \id_{\varnothing};
\]
The \textit{index $0$ attaching $2$-mor\-phism $\bbA_0 : \varnothing \rightarrow \bbSigma_0$ of $\bfadCob_{\calC}$} is just
\[
 \left( S^{-1} \times \overline{D^4},\varnothing,0,0 \right) = \id_{\id_{\varnothing}};
\]
The \textit{$i$-colored index $0$ belt $2$-mor\-phism $\bbB_{i,0} : \varnothing \rightarrow \bbSigma_0$ of $\bfadCob_{\calC}$} is
\[
 \left( D^0 \times S^3,K_{B_{i,0}},\omega_{B_{0,i}},0 \right),
\]
where $K_{B_{i,0}} \subset D^0 \times S^3$ is the $\calC$-colored framed link given by an unknot with framing zero and color $V_i$, and where $\omega_{B_{i,0}}$ is the unique compatible $G$-coloring of $(D^0 \times S^3,K_{B_{i,0}})$.

Next, we move on to define index $1$ surgery morphisms, which are represented in Figure \ref{F:index_1_surgery}: The \textit{$i$-colored index $1$ surgery $1$-mor\-phism $\bbSigma_{i,1} : \varnothing \rightarrow \varnothing$ of $\bfadCob_{\calC}$} is
\[
 \left( S^0 \times S^2,P_{i,1},\vartheta_{i,1},\{ 0 \} \right), 
\]
where $P_{i,1} \subset S^0 \times S^2$ is the $\calC$-colored ribbon set given by the first south and the second north pole with positive orientation, and by the first north and the second south pole with negative orientation, all with color $V_i$, and where $\vartheta_{i,1}$ is the unique compatible $G$-coloring of $(S^0 \times S^2,P_{i,1})$, with a single base point on the equator of each copy of $S^2$ ; The \textit{$i$-colored index $1$ attaching $2$-mor\-phism $\bbA_{i,1} : \id_{\varnothing} \Rightarrow \bbSigma_{i,1}$ of $\bfadCob_{\calC}$} is
\[
 \left( S^0 \times D^3,T_{A_{i,1}},\omega_{A_{i,1}},0 \right), 
\]
where $T_{A_{i,1}} \subset S^0 \times D^3$ is the $\calC$-colored ribbon tangle given by two edges, both joining the south and the north pole of a copy of $S^2$, both with framing zero and color $V_i$, and where $\omega_{A_{i,1}}$ is the unique compatible $G$-coloring of ${(S^0 \times D^3,T_{A_{i,1}})}$; The \textit{$i$-colored index $1$ belt $2$-mor\-phism $\bbB_{i,1} : \id_{\varnothing} \Rightarrow \bbSigma_{i,1}$ of $\bfadCob_{\calC}$} is
\[
 \left( D^1 \times S^2,T_{B_{i,1}},\omega_{B_{i,1}},0 \right),
\]
where $T_{B_{i,1}} \subset D^1 \times S^2$ is the $\calC$-colored ribbon tangle given by two edges, one joining the south poles, one joining the north poles of the two copies of $S^2$, both with framing zero and color $V_i$, and where $\omega_{B_{i,1}}$ is the only compatible $G$-coloring of ${(D^1 \times S^2,T_{B_{i,1}})}$ which is a pull-back for the projection ${p_{S^2} : D^1 \times S^2 \rightarrow S^2}$.

\begin{figure}[htb]\label{F:index_1_surgery}
 \centering
 \includegraphics{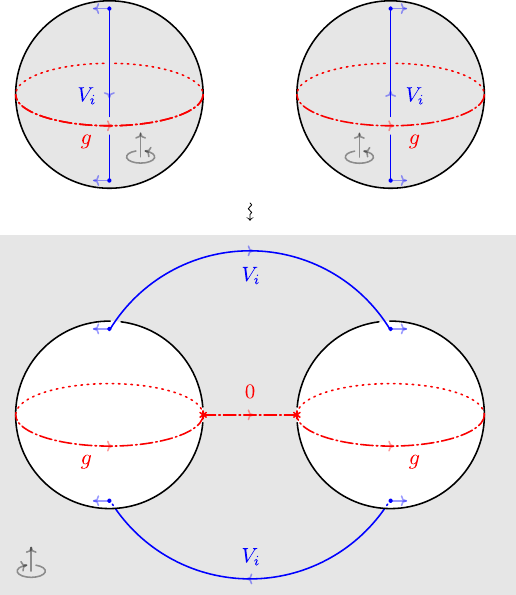}
 \caption{An $i$-colored index $1$ surgery replacing $\bbA_{i,1}$ with $\bbB_{i,1}$.}
\end{figure}

We finish by defining index $2$ surgery morphisms, which are represented in Figure \ref{F:index_2_surgery}: The \textit{$g$-colored index $2$ surgery $1$-mor\-phism $\bbSigma_{g,2} : \varnothing \rightarrow \varnothing$ of $\bfadCob_{\calC}$} is
\[
 \left( S^1 \times S^1,\varnothing,\vartheta_{g,2}, \calL_{g,2} \right),
\]
where $\vartheta_{g,2}$ is the $G$-coloring of $(S^1 \times S^1,\varnothing)$ determined by $\langle \vartheta_{g,2},\ell \rangle = 0$ and by $\langle \vartheta_{g,2},m \rangle = g$ for the homology classes $\ell$ and $m$ of the longitude $S^1 \times \{ (0,1) \}$ and of the meridian $\{ (1,0) \} \times S^1$ respectively, and where the Lagrangian $\calL_{g,2}$ is generated by $m$; The \textit{$g$-colored index $2$ attaching $2$-mor\-phism $\bbA_{g,2} : \id_{\varnothing} \Rightarrow \bbSigma_{g,2}$ of $\bfadCob_{\calC}$} is
\[
 \left( S^1 \times \overline{D^2},K_{A_{g,2}},\omega_{A_{g,2}},0 \right),
\]
where $K_{A_{g,2}} \subset S^1 \times \overline{D^2}$ is the $\calC$-colored framed knot given by the core ${S^1 \times \{ (0,0) \}}$, with framing zero and color $\Omega_g$, and where $\omega_{A_{g,2}}$ is the only compatible $G$-coloring of ${(S^1 \times \overline{D^2},K_{A_{g,2}})}$ which extends $\vartheta_{g,2}$; The \textit{$g$-colored index $2$ belt $2$-mor\-phism ${\bbB_{g,2} : \id_{\varnothing} \Rightarrow \bbSigma_{g,2}}$ of $\bfadCob_{\calC}$} is
\[
 \left( D^2 \times S^1,\varnothing,\omega_{B_{g,2}},0 \right)
\]
where $\omega_{B_{g,2}}$ is the only $G$-coloring of $(D^2 \times S^1,\varnothing)$ which extends $\vartheta_{g,2}$.

\begin{figure}[htb]\label{F:index_2_surgery}
 \centering
 \includegraphics{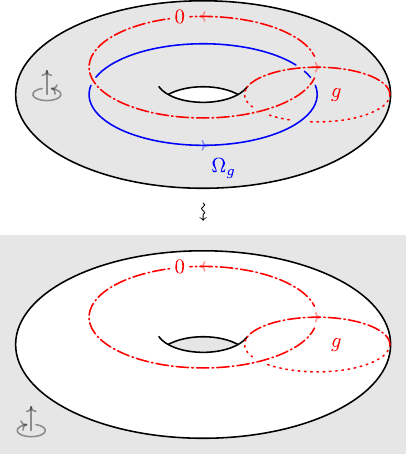}
 \caption{A $g$-colored index $2$ surgery replacing $\bbA_{g,2}$ with $\bbB_{g,2}$.}
\end{figure}

If $\bbM : \bbSigma \Rightarrow \bbSigma'$ is a $2$-mor\-phism of $\bfadCob_{\calC}$ between $1$-mor\-phisms $\bbSigma, \bbSigma' : \bbGamma \rightarrow \bbGamma'$ which decomposes as $\bbM_{i,j} \ast \left( \bbA_{i,j} \disjun \id_{\bbSigma} \right)$
for some generic $g \in G \smallsetminus X$, for some $i \in \rmI_g$, for some integer $0 \leqslant j \leqslant 2$, and for some $2$-mor\-phism $\bbM_{i,j} : \bbSigma_{i,j} \disjun \bbSigma \Rightarrow \bbSigma'$ of $\bfadCob_{\calC}$,  where we set $\bbSigma_{i,0} := \bbSigma_0$, $\bbA_{i,0} := \bbA_0$, $\bbSigma_{i,2} := \bbSigma_{g,2}$, $\bbA_{i,2} := \bbA_{g,2}$, and $\bbB_{i,2} := \bbB_{g,2}$, then we say the $2$-mor\-phism $\bbM_{i,j} \ast \left( \bbB_{i,j} \disjun \id_{\bbSigma} \right)$
is obtained from $\bbM$ by \textit{$i$-colored index $j$ surgery}. The generic stabilization of Section \ref{S:projective_generic_stabilizations} provides an example of a double index $2$ surgery.

\begin{proposition}\label{P:surgery_axioms}  
 If $\bbM_{i,j} : \bbSigma_{i,j} \Rightarrow \id_{\varnothing}$ is a closed $2$-mor\-phism of $\bfadCob_{\calC}$ for some generic $g \in G \smallsetminus X$, for some $i \in \rmI_g$, and for some integer $0 \leqslant j \leqslant 2$, then
 \[
  \CGP_{\calC}(\bbM_{i,j} \ast \bbB_{i,j}) = \lambda_{i,j} \CGP_{\calC}(\bbM_{i,j} \ast \bbA_{i,j})
 \]
 with $\lambda_{i,0} = \lambda_{i,1}^{-1} = \calD^{-1} \rmd(V_i)$ and $\lambda_{i,2} = \calD^{-1}$. 
\end{proposition}

\begin{proof}
 If $j = 0$, then the property reduces to the computation
 \begin{align*}
  \CGP_{\calC}(\bbB_{i,0}) &= \calD^{-1-0} \delta^{0 - 0} F'_{\calC}(K_{B_{i,0}}) \\
  &= \calD^{-1} \rmd(V_i) \\
  &= \calD^{-1} \rmd(V_i) \CGP_{\calC}(\bbA_0).
 \end{align*}

 If $j = 1$, then we can adapt the proof of Proposition 4.8 of \cite{BCGP16}. Indeed, we can distinguish two cases, according to whether or not the surgery operation involves two different connected components. We start from the first case, so let us suppose it does. Let us fix decompositions
 \[
  \bbSigma_{i,1} = \overline{\bbS^2_{(+,V_i),(-,V_i)}} \disjun \bbS^2_{(-,V_i),(+,V_i)}, \quad
  \bbA_{i,1} = \overline{\bbD^3_{\id_{(-,V_i)}}} \disjun \bbD^3_{\id_{(+,V_i)}}
 \]
 with respect to the $2$-mor\-phisms 
 \[
  \overline{\bbD^3_{\id_{(-,V_i)}}} : \id_{\varnothing} \Rightarrow \overline{\bbS^2_{(+,V_i),(-,V_i)}}, \quad \bbD^3_{\id_{(+,V_i)}} : \id_{\varnothing} \Rightarrow \bbS^2_{(-,V_i),(+,V_i)}
 \]
 of $\bfadCob_{\calC}$ represented in the top-left part and in the top-right part of Figure \ref{F:index_1_surgery} respectively. Then, let us consider connected $2$-morphisms 
 \[
  \bbM_{i,1} : \overline{\bbS^2_{(+,V_i),(-,V_i)}} \Rightarrow \id_{\varnothing}, \quad
  \bbM'_{i,1} : \bbS^2_{(-,V_i),(+,V_i)} \Rightarrow \id_{\varnothing}
 \] 
 of $\bfadCob_{\calC}$, and let us set
 \begin{gather*}
  (M,T,\omega,n) = \bbM_{i,1} \ast \overline{\bbD^3_{\id_{(-,V_i)}}}, \quad
  (M',T',\omega',n') = \bbM'_{i,1} \ast \bbD^3_{\id_{(+,V_i)}}, \\
  (M \sqcup M',T \sqcup T',\omega \sqcup \omega',n + n') = (\bbM_{i,1} \disjun \bbM'_{i,1}) \ast \bbA_{i,1}, \vphantom{\overline{\bbD^3_{\id_{(-,V_i)}}}} \\
  (M \# M',T \# T',\omega \# \omega',n + n') = (\bbM_{i,1} \disjun \bbM'_{i,1}) \ast \bbB_{i,1}. \vphantom{\overline{\bbD^3_{\id_{(-,V_i)}}}}
 \end{gather*}
 Up to performing projective stabilizations of sufficiently generic indices, we can suppose that $M$ admits a computable surgery presentation $L = L_1 \cup \ldots \cup L_{\ell} \subset S^3$ with respect to $(T,\omega)$, and that $M'$ admits a computable surgery presentation $L' = L'_1 \cup \ldots \cup L'_{\ell'} \subset S^3$ with respect to $(T',\omega')$. Then, $L \cup L'$ is a computable surgery presentation of $M \# M'$ with respect to $(T \# T',\omega \# \omega')$. Moreover, if $(L_{\omega} \cup T)_{(+,V_i)}$ and $(L'_{\omega'} \cup T')_{(+,V_i)}$ are cutting presentations of $L_{\omega} \cup T$ and of $L'_{\omega'} \cup T'$ respectively, then their composition $(L_{\omega} \cup T)_{(+,V_i)} \circ (L'_{\omega'} \cup T')_{(+,V_i)}$ is a cutting presentation of $(L \cup L')_{\omega \# \omega'} \cup (T \# T')$. This means
 \begin{align*}
  &F'_{\calC} \left( (L \cup L')_{\omega \# \omega'} \cup (T \# T') \right) = \rmt_{V_i} \left( F_{\calC} \left( (L_{\omega} \cup T)_{(+,V_i)} \circ (L'_{\omega'} \cup T')_{(+,V_i)} \right) \right) \\
  &\hspace{\parindent} = \rmd(V_i)^{-1} \rmt_{V_i} \left( F_{\calC} \left( (L_{\omega} \cup T)_{(+,V_i)} \right) \right) \rmt_{V_i} \left( F_{\calC} \left( (L'_{\omega'} \cup T')_{(+,V_i)} \right) \right) \\
  &\hspace{\parindent} = \rmd(V_i)^{-1} F'_{\calC} \left( L_{\omega} \cup T \right) F'_{\calC} \left( L'_{\omega'} \cup T' \right).
 \end{align*}
 Thus we get
 \begin{align*}
  &\CGP_{\calC} \left( (\bbM_{i,1} \disjun \bbM'_{i,1}) \ast \bbB_{i,1} \right) \\
  &\hspace{\parindent} = \calD^{-1 - (\ell + \ell')} \delta^{(n + n') - (\sigma(L) + \sigma(L'))} F'_{\calC} \left( (L \cup L')_{\omega \# \omega'} \cup (T \# T') \right) \\
  &\hspace{\parindent} = \calD^{-1 - \ell - \ell'} \delta^{n - \sigma(L) + n' - \sigma(L')} \rmd(V_i)^{-1} F'_{\calC} \left( L_{\omega} \cup T \right) F'_{\calC} \left( L'_{\omega'} \cup T' \right) \\
  &\hspace{\parindent} = \calD \rmd(V_i)^{-1} \CGP_{\calC} \left( \bbM_{i,1} \ast \overline{\bbD^3_{\id_{(-,V_i)}}} \right)
  \CGP_{\calC} \left( \bbM'_{i,1} \ast \bbD^3_{\id_{(+,V_i)}} \right) \\
  &\hspace{\parindent} = \calD \rmd(V_i)^{-1} \CGP_{\calC} \left( (\bbM_{i,1} \disjun \bbM'_{i,1}) \ast \bbA_{i,1} \right).
 \end{align*}

 Now, let us move on to the second case, let us consider a connected $2$-mor\-phism $\bbM_{i,1} : \bbSigma_{i,1} \Rightarrow \id_{\varnothing}$ of $\bfadCob_{\calC}$, and let us set $(M,T,\omega,n) = \bbM_{i,1} \ast \bbA_{i,1}$ and $(M',T',\omega',n') = \bbM_{i,1} \ast \bbB_{i,1}$. Up to performing a projective stabilization of sufficiently generic index, we can suppose that $M$ admits a computable surgery presentation $L = L_1 \cup \ldots \cup L_{\ell} \subset S^3$ with respect to $(T,\omega)$, and that there exists a curve $\gamma \subset M$ joining the base points of $\bbSigma_{i,1}$ whose relative homology class is evaluated to a generic index $h \in G \smallsetminus X$ by the $G$-coloring of $\bbM_{i,1}$. In order to check the second claim, remark that we can turn any relative homology class into a generic one by adding to it the homology class of a meridian of an edge whose color has sufficiently generic index. Furthermore, a computable surgery presentation $L'$ of $M'$ in $S^3$ with respect to $(T',\omega')$ can be obtained from $L$ by adding a single surgery component $L_{\ell + 1}$. In order to explain how, let us choose a cutting presentation $(L_{\omega} \cup T)_{((-,V_i),(+,V_i))}$ of $L_{\omega} \cup T$ in which $\gamma$ appears as an arc contained in the outgoing horizontal boundary $D^2 \times \{ 1 \}$ of $D^2 \times I$. Then, a cutting presentation $(L'_{\omega'} \cup T')_{((-,V_i),(+,V_i))}$ of $L'_{\omega'} \cup T'$ is represented in Figure \ref{F:proof_1-surgery}. 
 \begin{figure}[b]
  \centering
  \includegraphics{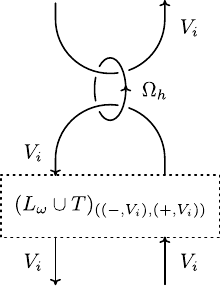}
  \caption{Cutting presentation of $L'_{\omega'} \cup T'$.}
  \label{F:proof_1-surgery}
 \end{figure}  
 Remark that
 \begin{align*}
  F'_{\calC}(L'_{\omega'} \cup T') &= \zeta \rmd(V_i)^{-1} F'_{\calC}(L_{\omega} \cup T) \\
  &= \calD^2 \rmd(V_i)^{-1} F'_{\calC}(L_{\omega} \cup T)
 \end{align*}
 thanks to the relative modularity condition of Definition \ref{D:relative_modular_category}, and thanks to Proposition \ref{P:non-degeneracy_of_relative_modular_categories}. Furthermore, remark that $n' = n$, and that $\sigma(L') = \sigma(L)$. This means
 \begin{align*}
  &\CGP_{\calC} \left( \bbM_{i,1} \ast \bbB_{i,1} \right)
  = \calD^{-1-\ell'} \delta^{n' - \sigma(L')} F'_{\calC}(L'_{\omega'} \cup T') \\
  &\hspace{\parindent} = \calD^{-\ell} \delta^{n - \sigma(L)} \rmd(V_i)^{-1} F'_{\calC} \left( L_{\omega} \cup T \right) \\
  &\hspace{\parindent} = \calD \rmd(V_i)^{-1} \CGP_{\calC} \left( (\bbM_{i,1} \disjun \bbM'_{i,1}) \ast \bbA_{i,1} \right).
 \end{align*}

 If $j = 2$, then we can adapt the proof of Proposition 4.9 of \cite{BCGP16}. Indeed, let us consider a connected $2$-mor\-phism $\bbM_{g,2} : \bbSigma_{g,2} \Rightarrow \id_{\varnothing}$ of $\bfadCob_{\calC}$, and let us set $(M,T,\omega,n) = \bbM_{g,2} \ast \bbA_{g,2}$ and $(M',T',\omega',n') = \bbM_{g,2} \ast \bbB_{g,2}$. Up to performing a projective stabilization of sufficiently generic index, we can suppose that $M$ admits a computable surgery presentation $L = L_1 \cup \ldots \cup L_{\ell} \subset S^3$ with respect to $(T,\omega)$. Then a computable surgery presentation $L'$ of $M'$ in $S^3$ with respect to $(T',\omega')$ is obtained from $L$ by adding a single surgery component $L_{\ell + 1}$ determined by the $\calC$-colored framed knot $K_{A_{g,2}}$ of $\bbA_{g,2}$. Furthermore, remark that $n' = n + \sigma(L \cup K_{A_{g,2}}) - \sigma(L)$, and that $K_{A_{g,2}} \cup T' = T$. Therefore, we get
 \begin{align*}
  \CGP_{\calC} \left( \bbM_{g,2} \ast \bbB_{g,2} \right)
  &= \calD^{-1-\ell'} \delta^{n' - \sigma(L')} F'_{\calC}\left( L'_{\omega'} \cup T' \right) \\
  &= \calD^{-2-\ell} \delta^{n - \sigma(L)} F'_{\calC} \left( L_{\omega} \cup K_{A_{g,2}} \cup T' \right) \\
  &= \calD^{-1} \CGP_{\calC} \left( \bbM_{g,2} \ast \bbA_{g,2} \right). \qedhere
 \end{align*}
\end{proof}

\begin{lemma}\label{L:surgery_axioms}
 If $[\bbM_{\bbGamma,i,j}]$ is a morphism of 
 $\Hom_{\bfA_{\calC}(\bbGamma)}(\bbSigma_{i,j} \disjun \bbSigma_{\bbGamma},\bbSigma''_{\bbGamma})$ for some generic $g \in G \smallsetminus X$, for some $i \in \rmI_g$, and for some integer $0 \leqslant j \leqslant 2$, then
 \[
  \left[ \bbM_{\bbGamma,i,j} \ast \left( \bbB_{i,j} \disjun \id_{\bbSigma_{\bbGamma}} \right) \right] 
  = \lambda_{i,j} \cdot \left[ \bbM_{\bbGamma,i,j} \ast \left( \bbA_{i,j} \disjun \id_{\bbSigma_{\bbGamma}} \right) \right]
 \]
 as morphisms of $\Hom_{\bfA_{\calC}(\bbGamma)}(\bbSigma_{i,j} \disjun \bbSigma_{\bbGamma},\bbSigma''_{\bbGamma})$.
\end{lemma}

\begin{proof}
 The proof is completely analogous to the proof of \ref{L:skein_equivalence}. Indeed, we have to show that for every object $\bbSigma'_{\bbGamma} \in \bfA'_{\calC}(\bbGamma)$ the linear map
 \[
  \rmV_{\calC} \left( \id_{\bbSigma'_{\bbGamma}} \circ \left( \bbM_{\bbGamma,i,j} \ast \left( \bbB_{i,j} \disjun \id_{\bbSigma_{\bbGamma}} \right) \right) \right) 
 \]
 is equal to the linear map
 \[
  \lambda_{i,j} \cdot \rmV_{\calC} \left( \id_{\bbSigma'_{\bbGamma}} \circ \left( \bbM_{\bbGamma,i,j} \ast \left( \bbA_{i,j} \disjun \id_{\bbSigma_{\bbGamma}} \right) \right) \right). 
 \]
 This means we have to show that for all vectors $[\bbM_{\bbSigma'_{\bbGamma} \circ \bbSigma_{\bbGamma}}] \in \rmV_{\calC}(\bbSigma'_{\bbGamma} \circ \bbSigma_{\bbGamma})$ and $[\bbM'_{\bbSigma'_{\bbGamma} \circ \bbSigma''_{\bbGamma}}] \in \rmV'_{\calC}(\bbSigma'_{\bbGamma} \circ \bbSigma''_{\bbGamma})$ the invariant
 \[
  \CGP_{\calC} \left( \bbM'_{\bbSigma'_{\bbGamma} \circ \bbSigma''_{\bbGamma}} \ast \left( \id_{\bbSigma'_{\bbGamma}} \circ \left( \bbM_{\bbGamma,i,j} \ast \left( \bbB_{i,j} \disjun \id_{\bbSigma_{\bbGamma}} \right) \right) \right) \ast \bbM_{\bbSigma'_{\bbGamma} \circ \bbSigma_{\bbGamma}} \right)
 \]
 is equal to the invariant
 \[
  \lambda_{i,j} \CGP_{\calC} \left( \bbM'_{\bbSigma'_{\bbGamma} \circ \bbSigma''_{\bbGamma}} \ast \left( \id_{\bbSigma'_{\bbGamma}} \circ \left( \bbM_{\bbGamma,i,j} \ast \left( \bbA_{i,j} \disjun \id_{\bbSigma_{\bbGamma}} \right) \right) \right) \ast \bbM_{\bbSigma'_{\bbGamma} \circ \bbSigma_{\bbGamma}} \right).
 \]
 Now this follows directly from Proposition \ref{P:surgery_axioms}.
\end{proof}

\section{Connection, domination, triviality}\label{S:connection_lemma}

In this section we derive some important consequences of skein equivalence and surgery axioms. Roughly speaking, these results can be summarized as follows: 
\begin{enumerate}
 \item The vector space of morphisms between a pair of objects of a universal linear category is generated by all admissible decorations of a fixed non-empty connected $3$-di\-men\-sion\-al cobordism with corners;
 \item A dominating set of objects of a universal linear category is obtained by considering all admissible decorations of a fixed non-empty $2$-di\-men\-sion\-al cobordism;
 \item To check whether a morphism of a universal linear category is trivial or not, it is enough to evaluate all admissible decorations of a fixed closed $3$-manifold obtained by first gluing horizontally a non-empty trivial $3$-di\-men\-sion\-al cobordism with corners to it, and then gluing vertically a pair of non-empty connected $3$-di\-men\-sion\-al cobordisms to the result.
\end{enumerate} 
In order to explain all this, we need some preliminary definition. First of all, if $\bbSigma = (\varSigma,P,\vartheta,\calL)$ and $\bbSigma' = (\varSigma',P',\vartheta',\calL')$ are $1$-mor\-phisms of $\bfadCob_{\calC}$ between objects $\bbGamma$ and $\bbGamma'$, and if $M$ is a $3$-dimensional cobordism with corners from $\varSigma$ to $\varSigma'$, then we denote with $\adSk(M;\bbSigma,\bbSigma')$ the free vector space over the set of all admissible $(\calC,G)$-colorings of $M$ relative to $(P,\vartheta)$ and $(P',\vartheta')$.

\begin{lemma}\label{L:connection_lemma}
 If $\bbSigma_{\bbGamma} = (\varSigma_{\varGamma},P,\vartheta,\calL)$ and $\bbSigma''_{\bbGamma} = (\varSigma''_{\varGamma},P'',\vartheta'',\calL'')$ are ob\-jects of $\bfA_{\calC}(\bbGamma)$, and if $M$ is a non-emp\-ty connected $3$-di\-men\-sion\-al cobordism with corners from $\varSigma_{\varGamma}$ to $\varSigma'_{\varGamma}$, then the linear map 
 \[
  \begin{array}{rccc}
   \rho_{M} : & \adSk(M;\bbSigma_{\bbGamma}, \bbSigma''_{\bbGamma}) & 
   \rightarrow & \Hom_{\bfA_{\calC}(\bbGamma)}(\bbSigma_{\bbGamma},\bbSigma''_{\bbGamma}) \\
   & (T,\omega) & \mapsto & [M,T,\omega,0]
  \end{array}
 \]
 is surjective.
\end{lemma}

\begin{proof}
 In order to prove surjectivity of $\rho_{M}$, we have to show that for every mor\-phism $[M_{\varGamma},T,\omega,n]$ of $\Hom_{\bfA_{\calC}(\bbGamma)}(\bbSigma_{\bbGamma},\bbSigma''_{\bbGamma})$ there exist ad\-mis\-si\-ble $(\calC,G)$-col\-or\-ings $(T_1,\omega_1), \ldots, (T_m,\omega_m)$ of $M$ relative to $(P,\vartheta)$ and $(P'',\vartheta'')$, and coefficients $\alpha_1, \ldots, \alpha_m \in \Bbbk$ such that
 \[
  \sum_{i = 1}^m \alpha_i \cdot [M,T_i,\omega_i,0] = [M_{\varGamma},T,\omega,n].
 \]
 First of all, we can suppose $M_{\varGamma}$ is connected. Indeed, if it is not, then we can choose some some $i \in \rmI_g$ with $g \in G \smallsetminus X$, and we can suppose every connected component of $M_{\varGamma}$ contains an edge colored with $V_i$. This is achieved by performing projective or generic stabilizations, which, thanks to Lemmas \ref{L:skein_equivalence} and \ref{L:surgery_axioms}, do not alter $[M_{\varGamma},T,\omega,n]$. Thus, thanks to Lemma \ref{L:surgery_axioms}, a finite sequence of $i$-colored index $1$ surgeries on $(M_{\varGamma},T,\omega,n)$ connecting its components determines a vector of $\Hom_{\bfA_{\calC}(\bbGamma)}(\bbSigma_{\bbGamma},\bbSigma''_{\bbGamma})$ which is a non-zero scalar multiple of $[M_{\varGamma},T,\omega,n]$. Then, since we are assuming $M_{\varGamma}$ is connected, we know that, up to performing projective or generic stabilizations, there exists a computable surgery presentation $L = L_1 \cup \ldots \cup L_{\ell} \subset M$ of $M_{\varGamma}$ with respect to $(T,\omega)$. This means that, thanks to Lemma \ref{L:surgery_axioms}, there exists an integer $n_L \in \Z$ such that 
 \begin{align*}
  [M_{\varGamma},T,\omega,n] &= \calD^{-\ell} \cdot 
  [M,L_{\omega} \cup T,\omega,n_L] \\
  &= \calD^{-\ell} \delta^{n_L} \cdot 
  [M,L_{\omega} \cup T,\omega,0]. \qedhere
 \end{align*}
\end{proof}

Here is an important remark: at this point in the construction, universal linear categories divert fom dual ones. Indeed, using the terminology introduced in Section \ref{S:admissible_cobordisms}, admissibility is equivalent to strong admissibility for all objects and morphisms of universal linear categories, while this is not true in the dual case, where the two conditions differ. This discrepancy is not visible in Lemmas \ref{L:skein_equivalence} and \ref{L:surgery_axioms}, which can be directly translated, together with their proofs, into analogous results for dual universal linear categories. On the other hand, in order to translate Lemma \ref{L:connection_lemma}, we need stronger hypotheses.

\begin{lemma}\label{L:connection_lemma_prime}
 If $\bbSigma'_{\bbGamma} = (\varSigma'_{\varGamma},P',\vartheta',\calL')$ and $\bbSigma'''_{\bbGamma} = (\varSigma'''_{\varGamma},P''',\vartheta''',\calL''')$ are strong\-ly ad\-mis\-si\-ble ob\-jects of $\bfA'_{\calC}(\bbGamma)$, and if $M'$ is a non-emp\-ty connected $3$-di\-men\-sion\-al cobordism with corners from $\varSigma'_{\varGamma}$ to $\varSigma'''_{\varGamma}$, then the linear map 
 \[
  \begin{array}{rccc}
   \rho'_{M'} : & \adSk(M';\bbSigma'_{\bbGamma}, \bbSigma'''_{\bbGamma}) & 
   \rightarrow & \Hom_{\bfA'_{\calC}(\bbGamma)}(\bbSigma'_{\bbGamma},\bbSigma'''_{\bbGamma}) \\
   & (T',\omega') & \mapsto & [M',T',\omega',0]
  \end{array}
 \]
 is surjective.
\end{lemma}

Remark that the stronger hypotheses of Lemma \ref{L:connection_lemma_prime} allow for the exact same proof of Lemma \ref{L:connection_lemma}.
Now, if $\bbGamma = (\varGamma,\xi)$ and $\bbGamma' = (\varGamma',\xi')$ are objects of $\bfadCob_{\calC}$, and if $\varSigma$ is a $2$-dimensional cobordism from $\varGamma$ to $\varGamma'$, then we denote with $\adSk(\varSigma;\bbGamma,\bbGamma')$ the set of all admissible $(\calC,G)$-colorings of $\varSigma$ relative to $\xi$ and $\xi'$. For the next statement, we need the notion of dominating set as given in Section \ref{S:relative_pre-modular_categories}.

\begin{lemma}\label{L:Morita_reduction}
 If $\bbGamma = (\varGamma,\xi)$ is an object of $\bfadCob_{\calC}$, if $\varSigma$ is a non-empty $2$-di\-men\-sion\-al cobordism from $\varnothing$ to $\varGamma$, and if $\calL \subset H_1(\varSigma;\R)$ is a Lagrangian, then the set of objects
 \[
  D(\varSigma;\bbGamma) := \left\{ \bbSigma_{(P,\vartheta)} := (\varSigma,P,\vartheta,\calL) \in \bfA_{\calC}(\bbGamma) \bigm| (P,\vartheta) \in \adSk(\varSigma;\varnothing,\bbGamma) \right\}
 \]
 dominates $\bfA_{\calC}(\bbGamma)$.
\end{lemma}

\begin{proof}
 If $\bbSigma_{\bbGamma} = (\varSigma_{\varGamma},P,\vartheta,\calL)$ and $\bbSigma''_{\bbGamma} = (\varSigma''_{\varGamma},P'',\vartheta'',\calL'')$ are objects of $\bfA_{\calC}(\bbGamma)$, if $M$ is a non-empty connected $3$-di\-men\-sion\-al cobordism with corners from $\varSigma_{\varGamma}$ to $\varSigma$, and if $M'$ is a non-empty connected $3$-di\-men\-sion\-al cobordism with corners from $\varSigma$ to $\varSigma''_{\varGamma}$, then, thanks to Lemma \ref{L:connection_lemma}, every morphism $[\bbM_{\bbGamma}]$ of $\Hom_{\bfA_{\calC}(\bbGamma)}(\bbSigma_{\bbGamma},\bbSigma''_{\bbGamma})$ is the image of some vector of ${\adSk(M \cup_{\varSigma} M';\bbSigma_{\bbGamma},\bbSigma''_{\bbGamma})}$. Up to isotopy, every $(\calC,G)$-coloring of $M \cup_{\varSigma} M'$ determines a $(\calC,G)$-coloring of $\varSigma$, which is a vector of $\adSk(\varSigma;\varnothing,\bbGamma)$. This means we have
 \[
  \left[ \bbM_{\bbGamma} \right] = \sum_{i=1}^m \alpha_i \left[ (M',T'_i,\omega'_i,0) \ast (M,T_i,\omega_i,0) \right],
 \]
 where $(T_i,\omega_i) \in \adSk(M;\bbSigma_{\bbGamma},\bbSigma_{(P_i,\vartheta_i)})$ and ${(T'_i,\omega'_i) \in \adSk(M';\bbSigma_{(P_i,\vartheta_i)},\bbSigma''_{\bbGamma})}$ for every integer $1 \leqslant i \leqslant m$ and for some $(P_1,\vartheta_1),\ldots,(P_m,\vartheta_m) \in \adSk(\varSigma;\varnothing,\bbGamma)$, 
\end{proof}

Remark that Lemma \ref{L:Morita_reduction} does not extend to the dual case, because morphisms of dual universal linear categories are not in general strongly admissible.
%
%
%
%

\begin{lemma}\label{L:triviality_lemma}
 If $\bbSigma_{\bbGamma} = (\varSigma_{\varGamma},P,\vartheta,\calL)$ and $\bbSigma''_{\bbGamma} = (\varSigma''_{\varGamma},P'',\vartheta'',\calL'')$ are objects of $\bfA_{\calC}(\bbGamma)$ for some object $\bbGamma = (\varGamma,\xi)$ of $\bfadCob_{\calC}$, if $\varSigma'$ is a non-empty $2$-di\-men\-sion\-al cobordism from $\varGamma$ to $\varnothing$, if $\calL' \subset H_1(\varSigma';\R)$ is a Lagrangian, if $M$ is a connected $3$-di\-men\-sion\-al cobordism from $\varnothing$ to $\Sigma_{\varGamma} \cup_{\varGamma} \Sigma'$, and if $M'$ is a connected $3$-di\-men\-sion\-al cobordism from $\Sigma''_{\varGamma} \cup_{\varGamma} \Sigma'$ to $\varnothing$, then a linear combination
 \[
  \sum_{i=1}^m \alpha_i \cdot [ \bbM_{\bbGamma,i} ] \in \Hom_{\bfA_{\calC}(\bbGamma)}(\bbSigma_{\bbGamma},\bbSigma''_{\bbGamma})
 \]
 is trivial if and only if
 \[
  \sum_{i=1}^m \alpha_i \CGP_{\calC} \left( (M',T',\omega',0) \ast \left( \id_{\bbSigma'_{(P',\vartheta')}}  \circ \bbM_{\bbGamma,i} \right) \ast (M,T,\omega,0) \right) = 0
 \]
 for all $(P',\vartheta') \in \adSk(\varSigma';\bbGamma,\varnothing)$, for all $(T,\omega) \in \adSk(M;\id_{\varnothing},\bbSigma'_{(P',\vartheta')} \circ \bbSigma_{\bbGamma})$, and for all $(T',\omega') \in \adSk(M';\bbSigma'_{(P',\vartheta')} \circ \bbSigma''_{\bbGamma},\id_{\varnothing})$.
\end{lemma}

\begin{proof}
 The morphism
 \[
  \sum_{i=1}^m \alpha_i \cdot [ \bbM_{\bbGamma,i} ] \in \Hom_{\bfA_{\calC}(\bbGamma)}(\bbSigma_{\bbGamma},\bbSigma''_{\bbGamma})
 \]
 is zero if and only if the linear map
 \[
  \sum_{i=1}^m \alpha_i \cdot \rmV_{\calC} \left( \id_{\bbSigma'_{\bbGamma}} \circ \bbM_{\bbGamma,i} \right) 
 \]
 is zero for every object $\bbSigma'_{\bbGamma} \in \bfA'_{\calC}(\bbGamma)$. This happens if and only if the invariant
 \[
  \sum_{i=1}^m \alpha_i \CGP_{\calC} \left( \bbM'_{\bbSigma'_{\bbGamma} \circ \bbSigma''_{\bbGamma}} \ast \left( \id_{\bbSigma'_{\bbGamma}} \circ \bbM_{\bbGamma,i} \right) \ast \bbM_{\bbSigma'_{\bbGamma} \circ \bbSigma_{\bbGamma}} \right)
 \]
 is zero for all vectors $[\bbM_{\bbSigma'_{\bbGamma} \circ \bbSigma_{\bbGamma}}] \in \rmV_{\calC}(\bbSigma'_{\bbGamma} \circ \bbSigma_{\bbGamma})$ and $[\bbM'_{\bbSigma'_{\bbGamma} \circ \bbSigma''_{\bbGamma}}] \in \rmV'_{\calC}(\bbSigma'_{\bbGamma} \circ \bbSigma''_{\bbGamma})$. Up to performing projective or generic stabilizations, we can suppose the $\calC$-colored ribbon graph of $\bbM_{\bbSigma'_{\bbGamma} \circ \bbSigma_{\bbGamma}}$ contains a projective edge in every connected component of the support of $\bbM_{\bbSigma'_{\bbGamma} \circ \bbSigma_{\bbGamma}}$. Therefore, up to isotoping these projective edges through all connected components of the support of $\id_{\bbSigma'_{\bbGamma}}$, we can suppose $\id_{\bbSigma'_{\bbGamma}}$ is strongly admissible. Then, thanks to Lemma \ref{L:connection_lemma_prime}, we know $[\id_{\bbSigma'_{\bbGamma}}]$ is the image of a vector of $\adSk(M'' \cup_{\varSigma'} M''';\bbSigma'_{\bbGamma},\bbSigma'_{\bbGamma})$ for some non-empty connected $3$-di\-men\-sion\-al cobordisms with corners $M''$ from $\varSigma'_{\varGamma}$ to $\varSigma'$ and $M'''$ from $\varSigma'$ to $\varSigma'''_{\varGamma}$. Up to isotopy, every $(\calC,G)$-coloring of $M'' \cup_{\varSigma'} M'''$ determines a $(\calC,G)$-coloring of $\varSigma'$, which is a vector of $\adSk(\varSigma';\bbGamma,\varnothing)$. This means we have
 \[
  \left[ \id_{\bbSigma'_{\bbGamma}} \right] = \sum_{i'=1}^{m'} \alpha'_{i'} \left[ (M''',T'''_{i'},\omega'''_{i'},0) \ast (M'',T''_{i'},\omega''_{i'},0) \right],
 \]
 where $(T''_{i'},\omega''_{i'}) \in \adSk(M'';\bbSigma'_{\bbGamma},\bbSigma'_{(P'_{i'},\vartheta'_{i'})})$ and $(T'''_{i'},\omega'''_{i'}) \in \adSk(M''';\bbSigma'_{(P'_{i'},\vartheta'_{i'})},\bbSigma'_{\bbGamma})$ for every integer $1 \leqslant i' \leqslant m'$ and for some $(P'_1,\vartheta'_1),\ldots,(P'_{m'},\vartheta'_{m'}) \in \adSk(\varSigma';\bbGamma,\varnothing)$. Now we can apply Lemma \ref{L:connection_lemma} to
 \[
  \left[ \left( (M'',T''_{i'},\omega''_{i'},0) \circ \id_{\bbSigma_{\bbGamma}} \right) \ast \bbM_{\bbSigma'_{\bbGamma} \circ \bbSigma_{\bbGamma}} \right]
 \]
 and to
 \[
  \left[ \bbM'_{\bbSigma'_{\bbGamma} \circ \bbSigma''_{\bbGamma}} \ast \left( (M''',T'''_{i'},\omega'''_{i'},0) \circ \id_{\bbSigma''_{\bbGamma}} \right) \right]
 \]
 for every integer $1 \leqslant i' \leqslant m'$ in order to conclude.
\end{proof}

We end this chapter with a final remark: many properties of morphism spaces of universal linear categories can be proved, with identical or easier arguments, for universal vector spaces. Let us spell out the analogues, in the context of the universal construction, of the main results we obtained so far.

\begin{lemma}\label{L:connection_lemma_TQFT}
 If $\bbSigma = (\varSigma,P,\vartheta,\calL)$ is an ob\-ject of $\rmadCob_{\calC}$, and if $M$ is a non-emp\-ty connected $3$-di\-men\-sion\-al cobordism with corners from $\varnothing$ to $\varSigma$, then the linear map 
 \[
  \begin{array}{rccc}
   \rho_{M} : & \adSk(M;\id_{\varnothing}, \bbSigma) & 
   \rightarrow & \rmV_{\calC}(\bbSigma) \\
   & (T,\omega) & \mapsto & [M,T,\omega,0]
  \end{array}
 \]
 is surjective.
\end{lemma}

\begin{proof}
 This statement is almost a special case of Lemma \ref{L:connection_lemma}, apart from the fact that objects of $\rmadCob_{\calC}$ need not be admissible. However, the exact same proof works in this case: indeed, every vector of $\rmV_{\calC}(\bbSigma)$ is represented by a linear combination of strongly admissible morphisms of $\rmadCob_{\calC}$, and the admissibility of sources and targets plays no role at all in the proofs of Lemmas \ref{L:skein_equivalence} and \ref{L:surgery_axioms}.
\end{proof}

\begin{lemma}\label{L:triviality_lemma_TQFT}
 If $\bbSigma = (\varSigma,P,\vartheta,\calL)$ is an admissible ob\-ject of $\rmadCob_{\calC}$, if $M'$ is a connected $3$-di\-men\-sion\-al cobordism from $\Sigma$ to $\varnothing$, then a linear combination
 \[
  \sum_{i=1}^m \alpha_i \cdot [ \bbM_{\bbSigma,i} ] \in \rmV_{\calC}(\bbSigma)
 \]
 is trivial if and only if
 \[
  \sum_{i=1}^m \alpha_i \CGP_{\calC} \left( (M',T',\omega',0) \ast \bbM_{\bbSigma,i} \right) = 0
 \]
 for all $(T',\omega') \in \adSk(M';\bbSigma,\id_{\varnothing})$.
\end{lemma}

\begin{proof}
 The vector
 \[
  \sum_{i=1}^m \alpha_i \cdot [ \bbM_{\bbSigma,i} ] \in \rmV_{\calC}(\bbSigma)
 \]
 is zero if and only if the invariant
 \[
  \sum_{i=1}^m \alpha_i \CGP_{\calC} \left( \bbM'_{\bbSigma} \ast \bbM_{\bbSigma,i} \right)
 \]
 is zero for all vectors $[\bbM'_{\bbSigma}] \in \rmV'_{\calC}(\bbSigma')$. Thanks to Lemma \ref{L:connection_lemma}, we know $[\bbM'_{\bbSigma}]$ is the image of a vector of $\adSk(M';\bbSigma,\id_{\varnothing})$.
\end{proof}


%% file: chapter_5.tex
%
%
%

\chapter{Graded extensions}\label{Ch:graded_extensions}

In this chapter we prove the $2$-func\-tor $\hat{\bfA}_{\calC} : \bfadCob_{\calC} \rightarrow \coCat_{\Bbbk}$ of Section \ref{S:extended_universal_construction} is not in general an ETQFT. The obstruction to monoidality can be characterized by means of the periodicity group $\PGr$. In order to do this, we first introduce the various ingredients involved in the definition of a free realization of $\PGr$ in $\bfA_{\calC}(\varnothing)$. Next, we assemble everything into an action of $\PGr$ on all universal linear categories. This allows for the construction of $\PGr$-graded extensions of universal linear categories, which lead to the definition of a strict $2$-functor $\bbA^Z_{\calC}$ whose target, the $2$-category $\bfCat_{\Bbbk}^Z$ of $Z$-graded linear categories, is introduced in Definition \ref{D:sym_mon_2-cat_of_gr_lin_cat}. The completion of this $2$-functor will provide a $\PGr$-graded ETQFT extending $\CGP_{\calC}$.

\section{2-Spheres}\label{S:2-spheres}

We suppose from now on that the element $0 \in G$ is critical. Indeed, should $0$ be generic, we would have $\Proj(\calC) = \calC$, because every object of $\calC$ can be realized as a tensor product with $\one \in \calC_0$. Then, thanks to Lemma 16 and Corollary 17 of \cite{GPT09}, this would force the projective trace $\rmt$ to be just a scalar multiple of the standard categorical trace $\tr_{\calC}$, and similarly the renormalized invariant $F'_{\calC}$ would just be a scalar multiple of the standard Reshetikhin-Turaev functor $F_{\calC}$. Then, $\CGP_{\calC}$ would boil down to a semisimple Witten-Reshetikhin-Turaev invariant with $G$-structure. In particular, the rest of the construction could very well be carried on also without this assumption, but this would force us to use two different notations for the rest of the memoir. Since the semisimple case produces a milder generalization of what was already known, we prefer to focus directly on the non-semisimple one. Then, let us consider some generic $g \in G \smallsetminus X$, some $i,j \in \rmI_g$, and some $k \in \PGr$.

\begin{definition}\label{D:2-sphere}
 The \textit{$(i,j,k)$-colored $2$-sphere $\smash{\bbS^2_{i,j,k}}$} is the closed $1$-mor\-phism of $\bfadCob_{\calC}$ given by 
 \[ 
  (S^2,P_{\varnothing,((+,V_i),(+,\sigma(k)),(-,V_j))},\vartheta_{\varnothing,((+,V_i),(+,\sigma(k)),(-,V_j))},\{ 0 \})
 \]
 in the notation of Section \ref{S:skein_equivalence}, with a single base point at the south pole.
\end{definition}

This family of closed $1$-mor\-phisms plays a major role in our construction, and we devote the next few sections to the study of their properties.

\begin{lemma}\label{L:2-sphere}
 For all $i,j \in \rmI_g$ and all $k,k' \in \PGr$ we have
 \[
  \dim_{\Bbbk} \left( \Hom_{\bfA_{\calC}(\varnothing)} \left( \bbS^2_{i,j,k},\bbS^2_{i,j,k'} \right) \right) = \delta_{ij} \delta_{k k'}.
 \]
\end{lemma}

\begin{proof}
 Every morphism of $\Hom_{\bfA_{\calC}(\varnothing)} ( \bbS^2_{i,j,k},\bbS^2_{i,j,k'} )$ is the image of a vector of $\adSk(S^2 \times I;\bbS^2_{i,j,k},\bbS^2_{i,j,k'})$ thanks to Lemma \ref{L:connection_lemma}. Furthermore, up to isotopy and skein equivalence, we can restrict to admissible $(\calC,G)$-colorings of the form $(T_f,\omega_h)$ like the one represented in Figure \ref{F:2-sphere_lemma_1} with 
 \[
  f \in \Hom_{\calC}(V_i \otimes \sigma(k) \otimes V_j^*,V_i \otimes \sigma(k') \otimes V_j^*), \quad
  h \in G.
 \]
 \begin{figure}[t]\label{F:2-sphere_lemma_1}
  \centering
  \includegraphics{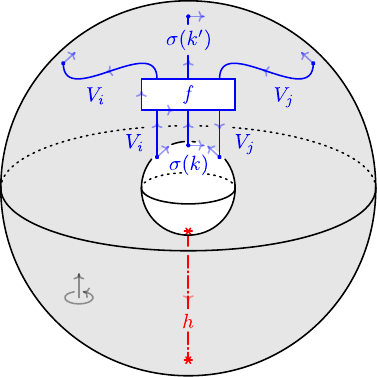}
  \caption{$(\calC,G)$-Coloring of $S^2 \times I$ of the form $(T_f,\omega_h)$.}
 \end{figure}
 To determine which $(\calC,G)$-colorings correspond to trivial morphisms, we can appeal to Lemma \ref{L:triviality_lemma}, so let us specify $\smash{\overline{S^2}}$ as a $2$-di\-men\-sion\-al cobordism from $\varnothing$ to $\varnothing$, let us specify $S^2 \times \overline{I}$ as a $3$-di\-men\-sion\-al cobordism from $\varnothing$ to $S^2 \sqcup \overline{S^2}$, and let us specify $S^2 \times I$ as a $3$-di\-men\-sion\-al cobordism from $S^2 \sqcup \overline{S^2}$ to $\varnothing$. Now the triviality of the morphism $[S^2 \times I,T_f,\omega_h,0]$ can be tested by considering all $(P',\vartheta') \in \adSk(\overline{S^2};\varnothing,\varnothing)$ and all
 \begin{gather*}
  (T,\omega) \in \adSk \left( S^2 \times \overline{I};\id_{\varnothing},\bbS^2_{i,j,k} \disjun \overline{\bbS^2_{(P',\vartheta')}} \right), \\
  (T',\omega') \in \adSk \left( S^2 \times I;\bbS^2_{i,j,k'} \disjun \overline{\bbS^2_{(P',\vartheta')}},\id_{\varnothing} \right),
 \end{gather*}
 and by computing the Costantino-Geer-Patureau invariant $\CGP_{\calC}$ of the resulting closed $2$-mor\-phism of $\bfadCob_{\calC}$. We can choose a surgery presentation composed of a single unknot with framing zero, whose computability can always be forced by performing generic or projective stabilization outside of $(T_f,\omega_h)$. Therefore, up to isotopy, skein equivalence, and multiplication by invertible scalars, we only need to compute the projective trace of the morphism $F_{\calC}(T_{f,f',h+h'})$ for all
 \[
  f' \in \Hom_{\calC} \left( V_i \otimes \sigma(k') \otimes V_j^*,V_i \otimes \sigma(k) \otimes V_j^* \right), \quad
  h' \in \{ -h \} + (G \smallsetminus X),
 \]
 where the $\calC$-colored ribbon graph $T_{f,f',h + h'}$ is represented in the left-hand part of Figure \ref{F:2-sphere_lemma_2}. 
 \begin{figure}[ht]\label{F:2-sphere_lemma_2}
  \centering
  \includegraphics{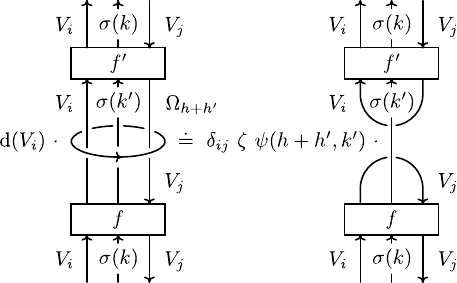}
  \caption{Skein equivalence following from the compatibility between the $G$-structure and the $\PGr$-action, and from the relative modularity of $\calC$.}
 \end{figure}
 Now the ribbon structure of $\calC$ and the semisimplicity of $\calC_g$ yield
 \[
  \dim_{\Bbbk} \left( \Hom_{\calC}(V_i \otimes \sigma(k) \otimes V_j^*,\sigma(k')) \right) = \delta_{ij} \delta_{kk'}.
 \]
 This means
 \[
  \dim_{\Bbbk} \left( \Hom_{\bfA_{\calC}(\varnothing)} \left( \bbS^2_{i,j,k},\bbS^2_{i,j,k'} \right) \right) \leqslant \delta_{ij} \delta_{kk'}.
 \]
 The dimension of $\End_{\bfA_{\calC}(\varnothing)}(\bbS^2_{i,i,k})$ is furthermore exactly equal to $1$, because the non-triviality of $[ \id_{\bbS^2_{i,i,k}} ]$ follows by considering
 \begin{gather*}
  f = \id_{V_i \otimes \sigma(k) \otimes V_i^*}, \quad h = 0, \\
  f' = \id_{V_i \otimes \sigma(k) \otimes V_i^*}, \quad h' \in G \smallsetminus X. \qedhere
 \end{gather*}
\end{proof}

Recall Remark \ref{R:critical_fusion}, in which we fixed, for every critical $x \in X$, a generic $g_x \in G \smallsetminus X$ satisfying $x + g_x \in G \smallsetminus X$, together with some $i_x \in \rmI_{g_x}$ specifying a simple projective object $V_x := V_{i_x} \in \Theta(\calC_{g_x})$. Now, as a direct consequence of Remark \ref{R:critical_fusion} and of Lemmas \ref{L:skein_equivalence} and \ref{L:Morita_reduction}, the universal linear category $\bfA_{\calC}(\varnothing)$ is dominated by the set of obects $\{ \bbS^2_{i,i_0,k} \mid i \in \rmI_{g_0}, k \in \PGr \}$. However, this dominating set is redundant, as we just established that, for all $i \in \rmI_{g_0} \smallsetminus \{ i_0 \}$, and for all $k \in \PGr$, the $(i,i_0,k)$-colored $2$-sphere $\bbS^2_{i,i_0,k}$ is a zero object of $\bfA_{\calC}(\varnothing)$. Therefore, we set 
\[
 \bbS^2_k := \bbS^2_{i_0,i_0,k},
\]
and we refer to $\bbS^2_k$ simply as the \textit{$k$-colored 2-sphere}. Now the linear category $\bfA_{\calC}(\varnothing)$ is Morita equivalent to its full subcategory with set of objects $\{ \bbS^2_k \mid k \in \PGr \}$, which is not Morita equivalent to the tensor unit $\Bbbk$ of $\bfCat_{\Bbbk}$ unless $\PGr = \{ 0 \}$. This means the periodicity group $\PGr$ measures the deviation from monoidality of $\hat{\bfA}_{\calC}$.

\section{3-Discs}\label{S:3-discs}

The family of $2$-spheres we just specified provides the first brick in the definition of a free realization of $\PGr$ in $\bfA_{\calC}(\varnothing)$. In the next two sections we introduce the remaining ingredients.

\begin{definition}\label{D:3-discs}
 The \textit{$0$-col\-ored $3$-disc} $\bbD^3_0 : \id_{\varnothing} \Rightarrow \bbS^2_0$ is the $2$-mor\-phism of $\bfadCob_{\calC}$ given by 
 \[
  \left( D^3,T_0,\omega_0,0 \right)
 \]
 where $T_0 \subset D^3$ is the $\calC$-colored ribbon graph represented in Figure \ref{F:3-disc}, where $\varepsilon \in \Hom_{\calC}(\one,\sigma(0))$ is a coherence morphism of $\sigma : \PGr \rightarrow \calC_0$, and where $\omega_0$ is the unique compatible $G$-coloring of $(D^3,T_0)$. Similarly, the \textit{inverse $0$-col\-ored $3$-disc} $\overline{\bbD^3_0} : \bbS^2_0 \Rightarrow \id_{\varnothing}$ is the $2$-mor\-phism of $\bfadCob_{\calC}$ given by 
 \[
  \left( \overline{D^3},\overline{T_0},\omega_0,0 \right)
 \]
 where $\overline{T_0} \subset \overline{D^3}$ is the $\calC$-colored ribbon graph obtained from $T_0$ by reversing the orientation of all edges and vertical boundaries of coupons, and by replacing the color $\varepsilon$ with its inverse $\varepsilon^{-1}$.
\end{definition}

\begin{figure}[htb]\label{F:3-disc}
 \centering
 \includegraphics{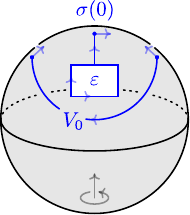}
 \caption{The $0$-col\-ored $3$-disc $\bbD^3_0$.}
\end{figure}

Let us discuss properties of compositions of these morphisms.

\begin{lemma}\label{L:3-discs_composition}
 We have the equality
 \[
  \left[ \bbD^3_0 \ast \overline{\bbD^3_0} \right] = \calD^{-1} \rmd(V_0) \cdot \left[ \id_{\bbS^2_0} \right]
 \]
 between morphisms of $\End_{\bfA_{\calC}(\varnothing)}(\bbS^2_0)$.
\end{lemma}

\begin{proof}
 This result is an immediate consequence of Lemma \ref{L:surgery_axioms} with color $i = i_{g_0}$ and index $j = 1$.
\end{proof}

Remark that, since the equality 
\[
 \left[ \overline{\bbD^3_0} \ast \bbD^3_0 \right] = \calD^{-1} \rmd(V_0) \cdot \left[ \id_{\id_{\varnothing}} \right]
\]
between morphisms of $\End_{\bfA_{\calC}(\varnothing)}(\id_{\varnothing})$ follows immediately from Lemma \ref{L:surgery_axioms} with color $i = i_{g_0}$ and index $j = 0$, the objects $\id_{\varnothing}$ and $\bbS^2_0$ of $\bfA_{\calC}(\varnothing)$ are isomorphic.

\section{3-Pants}\label{S:3-pants}

We define the \textit{$3$-pant cobordism $P^3$} as the $3$-dimensional cobordism from $S^2 \sqcup S^2$ to $S^2$ whose support is given by $D^3$ minus two open balls of radius $\frac{1}{4}$ and center $(-\frac{1}{2},0,0)$ and $(+\frac{1}{2},0,0)$ respectively, and whose incoming and outgoing horizontal boundary identifications are induced by $\id_{S^2}$ through rescaling and translation. Then, let us consider $k,k' \in \PGr$.

\begin{definition}\label{D:3-pants}
 The \textit{$(k,k')$-colored $3$-pant} $\bbP^3_{k,k'} : \bbS^2_k \disjun \bbS^2_{k'} \Rightarrow \bbS^2_{k + k'}$ is the $2$-mor\-phism of $\bfadCob_{\calC}$ given by 
 \[
  \left( P^3,T_{k,k'},\omega_{k,k'},0 \right)
 \]
 where $T_{k,k'} \subset P^3$ is the $\calC$-colored ribbon graph represented in Figure \ref{F:3-pant}, where $\mu_{k,k'} \in \Hom_{\calC}(\sigma(k) \otimes \sigma(k'),\sigma(k+k'))$ is a coherence morphism of $\sigma : \PGr \rightarrow \calC_0$, and where $\omega_{k,k'}$ is the only compatible $G$-coloring of $(P^3,T_{k,k'})$ which vanishes on all relative homology classes contained in the southern hemisphere. Similarly, the \textit{inverse $(k,k')$-colored $3$-pant} $\overline{\bbP^3_{k,k'}} : \bbS^2_{k + k'} \Rightarrow \bbS^2_k \disjun \bbS^2_{k'}$ is the $2$-mor\-phism of $\bfadCob_{\calC}$ given by 
 \[
  \left( \overline{P^3},\overline{T_{k,k'}},\omega_{k,k'},0 \right)
 \]
 where $\overline{T_{k,k'}} \subset \overline{P^3}$ is the $\calC$-colored ribbon graph obtained from $T_{k,k'}$ by reversing the orientation of all edges and vertical boundaries of coupons, and by replacing the color $\mu_{k,k'}$ with its inverse $\mu_{k,k'}^{-1}$.
\end{definition}

\begin{figure}[htb]\label{F:3-pant}
 \centering
 \includegraphics{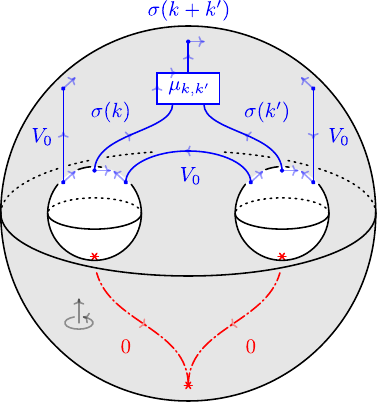}
 \caption{The $(k,k')$-colored $3$-pant $\bbP^3_{k,k'}$.}
\end{figure}

\begin{lemma}\label{L:3-pant}
 For all $k,k',k'' \in \PGr$ we have
 \begin{gather*}
  \dim_{\Bbbk} \Hom_{\bfA_{\calC}(\varnothing)} \left( \bbS^2_k \disjun \bbS^2_{k'},\bbS^2_{k''} \right) = \delta_{(k + k')k''}\phantom{,} \\
  \dim_{\Bbbk} \Hom_{\bfA_{\calC}(\varnothing)} \left( \bbS^2_k,\bbS^2_{k'} \disjun \bbS^2_{k''} \right) = \delta_{k(k' + k'')}.
 \end{gather*}
\end{lemma}

\begin{proof}
 The argument is the exact same one we used for the proof of Lemma \ref{L:2-sphere}. Indeed, every morphism of $\Hom_{\bfA_{\calC}(\varnothing)} ( \bbS^2_k \disjun \bbS^2_{k'},\bbS^2_{k''} )$ is the image of a vector of $\adSk(P^3;\bbS^2_k \disjun \bbS^2_{k'},\bbS^2_{k''})$ thanks to Lemma \ref{L:connection_lemma}. Furthermore, up to isotopy and skein equivalence, we can restrict to admissible $(\calC,G)$-colorings of the form $(T_f,\omega_{h,h'})$ like the one represented in Figure \ref{F:3-pant_lemma_1} with 
 \[
  f \in \Hom_{\calC}(V_0 \otimes \sigma(k) \otimes V_0^* \otimes V_0 \otimes \sigma(k') \otimes V_0^*,V_0 \otimes \sigma(k'') \otimes V_0^*), \quad
  h,h' \in G.
 \]
 \begin{figure}[tb]\label{F:3-pant_lemma_1}
  \centering
  \includegraphics{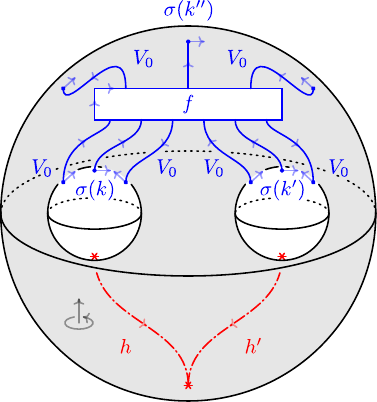}
  \caption{$(\calC,G)$-Coloring of $P^3$ of the form $(T_f,\omega_{h,h'})$.}
 \end{figure}
 To determine which $(\calC,G)$-colorings correspond to trivial morphisms, we can appeal to Lemma \ref{L:triviality_lemma}, so let us specify $\smash{\overline{S^2}}$ as a $2$-di\-men\-sion\-al cobordism from $\varnothing$ to $\varnothing$, let us specify $\overline{P^3}$ as a $3$-di\-men\-sion\-al cobordism from $\varnothing$ to $S^2 \sqcup S^2 \sqcup \overline{S^2}$, and let us specify $S^2 \times I$ as a $3$-di\-men\-sion\-al cobordism from $S^2 \sqcup \overline{S^2}$ to $\varnothing$. Now the triviality of the morphism $[P^3,T_f,\omega_{h,h'},0]$ can be tested by considering all $(P',\vartheta') \in \adSk(\overline{S^2};\varnothing,\varnothing)$ and all
 \begin{gather*}
  (T,\omega) \in \adSk \left( \overline{P^3};\id_{\varnothing},\bbS^2_k \disjun \bbS^2_{k'} \disjun \overline{\bbS^2_{(P',\vartheta')}} \right), \\
  (T',\omega') \in \adSk \left( S^2 \times I;\bbS^2_{k''} \disjun \overline{\bbS^2_{(P',\vartheta')}},\id_{\varnothing} \right),
 \end{gather*}
 and by computing the Costantino-Geer-Patureau invariant $\CGP_{\calC}$ of the resulting closed $2$-mor\-phism of $\bfadCob_{\calC}$. We can choose a surgery presentation composed of a two-component unlink with framing zero, whose computability can always be forced by performing generic or projective stabilization outside of $(T_f,\omega_{h,h'})$. Therefore, up to isotopy, skein equivalence, and multiplication by invertible scalars, we only need to compute the projective trace of the morphism $F_{\calC}(T_{f,f',h + h'',h' + h'''})$ for all
 \begin{gather*}
  f' \in \Hom_{\calC} \left( V_0 \otimes \sigma(k'') \otimes V_0^*,V_0 \otimes \sigma(k) \otimes V_0^* \otimes V_0 \otimes \sigma(k') \otimes V_0^* \right), \\
 h'' \in \{ -h \} + (G \smallsetminus X), \quad h''' \in \{ -h' \} + (G \smallsetminus X),
 \end{gather*}
 where the $\calC$-colored ribbon graph $T_{f,f',h + h'',h' + h'''}$ is represented in the left-hand part of Figure \ref{F:3-pant_lemma_2}. 
 \begin{figure}[b]\label{F:3-pant_lemma_2}
  \centering
  \includegraphics{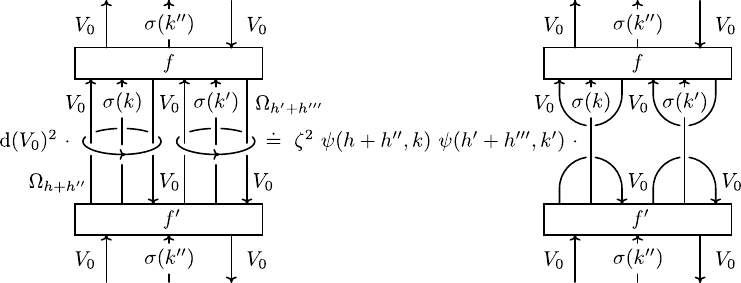}
  \caption{Skein equivalence following from the compatibility between the $G$-struc\-ture and the $\PGr$-ac\-tion, and from the relative modularity of $\calC$.}
 \end{figure}
 Now the ribbon structure of $\calC$ and the semisimplicity of $\calC_{g_0}$ yield
 \[
  \dim_{\Bbbk} \left( \Hom_{\calC}(\sigma(k) \otimes \sigma(k'),V_0 \otimes \sigma(k'') \otimes V_0^*) \right) = \delta_{(k + k')k''}.
 \]
 This means
 \[
  \dim_{\Bbbk} \left( \Hom_{\bfA_{\calC}(\varnothing)} \left( \bbS^2_k \disjun \bbS^2_{k'},\bbS^2_{k''} \right) \right) \leqslant \delta_{(k + k')k''}.
 \]
 The dimension of $\Hom_{\bfA_{\calC}(\varnothing)} ( \bbS^2_k \disjun \bbS^2_{k'},\bbS^2_{k + k'} )$ is furthermore exactly equal to $1$, because the non-triviality of $[ \bbP^3_{k,k'} ]$ follows by considering 
 \begin{gather*}
  f = \id_{V_0} \otimes \left( \mu_{k,k'} \circ (\id_{\sigma(k)} \otimes \lev_{V_0} \otimes \id_{\sigma(k')}) \right) \otimes \id_{V_0^*}, \quad 
  h = 0, \\
  f' = \id_{V_0} \otimes \left( (\id_{\sigma(k)} \otimes \rcoev_{V_0} \otimes \id_{\sigma(k')}) \circ  \mu_{k,k'}^{-1} \right) \otimes \id_{V_0^*}, \quad 
  h' = g_0.
 \end{gather*}
 The proof of the second statement is completely analogous. \qedhere
\end{proof}

Let us discuss properties of compositions of these morphisms.

\begin{lemma}\label{L:unitarity_of_3-discs}
 For every $k \in \PGr$ we have the equalities
 \begin{gather*}
  \left[ \bbP^3_{0,k} \ast \left( \bbD^3_0 \disjun \id_{\bbS^2_k} \right) \right]
  = \left[ \bbP^3_{k,0} \ast \left( \id_{\bbS^2_k} \disjun \bbD^3_0 \right) \right]
  = \left[ \id_{\bbS^2_k} \right], \\
  \left[ \left( \overline{\bbD^3_0} \disjun \id_{\bbS^2_k} \right) \ast \overline{\bbP^3_{0,k}} \right]
  = \left[ \left( \id_{\bbS^2_k} \disjun \overline{\bbD^3_0} \right) \ast \overline{\bbP^3_{k,0}} \right]
  = \left[ \id_{\bbS^2_k} \right]\phantom{,}
 \end{gather*}
 between vectors of $\End_{\bfA_{\calC}(\varnothing)}(\bbS^2_k)$.
\end{lemma}

\begin{proof}
 The equalities 
 \[
  \mu_{0,k} \circ (\varepsilon \otimes \id_{\sigma(k)}) = \mu_{k,0} \circ (\id_{\sigma(k)} \otimes \varepsilon) = \id_{\sigma(k)}
 \]
 follow from the definition of the coherence data for the monoidal functor $\sigma$. Therefore, the result follows from Lemma \ref{L:skein_equivalence} by 
 remarking that the cobordisms
 \[
  \left( D^3 \sqcup (S^2 \times I) \right) \cup_{S^2 \sqcup S^2} P^3, \quad
  \left( (S^2 \times I) \sqcup D^3 \right) \cup_{S^2 \sqcup S^2} P^3, \quad
  S^2 \times I
 \]
 are isomorphic.
\end{proof}

\begin{lemma}\label{L:associativity_of_3-pants}
 For all $k,k',k'' \in \PGr$ we have the equalities
 \[
  \left[ \bbP^3_{k + k',k''} \ast \left( \bbP^3_{k,k'} \disjun \id_{\bbS^2_{k''}} \right) \right] 
  = \left[ \bbP^3_{k,k' + k''} \ast \left( \id_{\bbS^2_k} \disjun \bbP^3_{k',k''} \right) \right]
 \]
 between vectors of $\Hom_{\bfA_{\calC}(\varnothing)}(\bbS^2_k \disjun \bbS^2_{k'} \disjun \bbS^2_{k''},\bbS^2_{k + k' + k''})$ and
 \[
  \left[ \left( \overline{\bbP^3_{k,k'}} \disjun \id_{\bbS^2_{k''}} \right) \ast \overline{\bbP^3_{k + k',k''}} \right]
  = \left[ \left( \id_{\bbS^2_k} \disjun \overline{\bbP^3_{k',k''}} \right) \ast \overline{\bbP^3_{k,k' + k''}} \right]
 \]
 between vectors of $\Hom_{\bfA_{\calC}(\varnothing)}(\bbS^2_{k + k' + k''},\bbS^2_k \disjun \bbS^2_{k'} \disjun 
 \bbS^2_{k''})$.
\end{lemma}

\begin{proof}
 The equality
 \[
  \mu_{k + k',k''} \circ (\mu_{k,k'} \otimes \id_{\sigma(k'')}) = \mu_{k,k' +  k''} \circ (\id_{\sigma(k)} \otimes \mu_{k',k''})
 \]
 follows from the definition of the coherence data for the monoidal functor $\sigma$. Therefore, the result follows from Lemma \ref{L:skein_equivalence} by 
 remarking that the cobordisms
 \[
  \left( P^3 \sqcup (S^2 \times I) \right) \cup_{S^2 \sqcup S^2} P^3, \quad
  \left( (S^2 \times I) \sqcup P^3 \right) \cup_{S^2 \sqcup S^2} P^3
 \]
 are isomorphic.
\end{proof}

\begin{lemma}\label{L:3-pants_composition_1}
 For all $k,k' \in \PGr$ we have the equality
 \[
  \left[ \overline{\bbP^3_{k,k'}} \ast \bbP^3_{k,k'} \right] = \calD \rmd(V_0)^{-1} \cdot \left[ \id_{\bbS^2_k} \disjun \id_{\bbS^2_{k'}} \right]
 \]
 between vectors of $\End_{\bfA_{\calC}(\varnothing)}(\bbS^2_k \disjun \bbS^2_{k'})$.
\end{lemma}

\begin{proof}
 The equality
 \[
  \mu_{k,k'}^{-1} \circ \mu_{k,k'} = \left( (\id_{\sigma(k)} \otimes \varepsilon^{-1}) \circ \mu_{k,0}^{-1} \right) \otimes \left( \mu_{0,k'} \circ (\varepsilon \otimes \id_{\sigma(k')}) \right)
 \]
 follows from the definition of the coherence data for the monoidal functor $\sigma$. Therefore, the equality 
 \[
  \left[ \overline{\bbP^3_{k,k'}} \ast \bbP^3_{k,k'} \right] =
  \left[ \left( \id_{\bbS^2_k} \disjun \bbP^3_{0,k'} \right) \ast \left( \overline{\bbP}^3_{k,0} \disjun \id_{\bbS^2_{k'}} \right) \right]
 \]
 between vectors of $\End_{\bfA_{\calC}(\varnothing)}(\bbS^2_k \disjun \bbS^2_{k'})$ follows from Lemma \ref{L:skein_equivalence} by remarking that the cobordisms
 \[
  P^3 \cup_{S^2} \overline{P^3}, \quad
  \left( (S^2 \times I) \sqcup \overline{P^3} \right) \cup_{S^2 \sqcup S^2 \sqcup S^2} 
  \left( P^3 \sqcup (S^2 \times I) \right)
 \]
 are isomorphic. The result now follows from Lemmas \ref{L:3-discs_composition} and \ref{L:unitarity_of_3-discs}.
\end{proof}

\begin{lemma}\label{L:3-pants_composition_2}
 For all $k,k' \in \PGr$ we have the equality 
 \[
  \left[ \bbP^3_{k,k'} \ast \overline{\bbP^3_{k,k'}} \right] = \calD \rmd(V_0)^{-1} \cdot \left[ \id_{\bbS^2_{k + k'}} \right]
 \]
 between vectors of $\End_{\bfA_{\calC}(\varnothing)}(\bbS^2_{k + k'})$.
\end{lemma}

\begin{proof}
 Since $\dim_{\Bbbk} \End_{\bfA_{\calC}(\varnothing)}(\bbS^2_{k + k'}) = 1$, we must have 
 \[
  \left[ \bbP^3_{k,k'} \ast \overline{\bbP^3_{k,k'}} \right] = \alpha \cdot \left[ \id_{\bbS^2_{k + k'}} \right]
 \]
 for some $\alpha \in \Bbbk$. Thanks to Lemma \ref{L:3-pants_composition_1}, we have the chain of equalities
 \begin{align*}
  \left[ \id_{\bbS^2_k} \disjun \id_{\bbS^2_{k'}} \right] & = 
  \calD^{-2} \rmd(V_0)^2 \cdot \left[ \overline{\bbP^3_{k,k'}} \ast \bbP^3_{k,k'} 
  \ast \overline{\bbP^3_{k,k'}} \ast \bbP^3_{k,k'} \right] \\
  & = \calD^{-2} \rmd(V_0)^2 \alpha \cdot \left[ \overline{\bbP^3_{k,k'}} \ast \bbP^3_{k,k'} \right] \\
  & = \calD^{-1} \rmd(V_0) \alpha \cdot \left[ \id_{\bbS^2_k} \disjun \id_{\bbS^2_{k'}} \right],
 \end{align*}
 which implies $\alpha = \calD \rmd(V_0)^{-1}$.
\end{proof}

Remark that, thanks to Lemmas \ref{L:3-pants_composition_1} and \ref{L:3-pants_composition_2}, the objects $\bbS^2_k \disjun \bbS^2_{k'}$ and $\bbS^2_{k + k'}$ of $\bfA_{\calC}(\varnothing)$ are isomorphic. Then, let us consider $k,k' \in \PGr$.

\begin{definition}\label{D:twisted_3-pant}
 The \textit{twisted $(k,k')$-colored $3$-pant} $\tilde{\bbP}^3_{k,k'} : \bbS^2_k \disjun \bbS^2_{k'} \Rightarrow \bbS^2_{k + k'}$ is the $2$-mor\-phism of $\bfadCob_{\calC}$ given by 
 \[
  \bbP^3_{k',k} \ast \cobbr_{\bbS^2_k,\bbS^2_{k'}},
 \]
 where the braiding $2$-morphism $\cobbr_{\bbS^2_k,\bbS^2_{k'}}$ is given in Definition \ref{D:braiding_2-morphism}.
\end{definition}

Let us consider the unique bilinear map $\gamma : \PGr \times \PGr \rightarrow \Z^*$ satisfying
\[
 \gamma(k,k') \cdot \mu_{k,k'} = \mu_{k',k} \circ c_{\sigma(k),\sigma(k')}. 
\]

\begin{lemma}\label{L:twisted_3-pant_lemma}
 For all $k,k' \in \PGr$ we have the equality
 \[
  \left[ \tilde{\bbP}^3_{k,k'} \right] = \gamma(k,k') \cdot \left[ \bbP^3_{k,k'} \right]
 \]
 between vectors of $\Hom_{\bfA_{\calC}(\varnothing)}(\bbS^2_k \disjun \bbS^2_{k'},\bbS^2_{k + k'})$.
\end{lemma}

\begin{proof}
 We know $\left[ \bbP^3_{k,k'} \right]$ generates $\Hom_{\bfA_{\calC}(\varnothing)}(\bbS^2_k \disjun \bbS^2_{k'},\bbS^2_{k + k'})$, so there have to exist coefficients $\alpha_{k,k'} \in \Bbbk$ for all $k,k' \in \PGr$ satisfying
 \[
  \left[ \tilde{\bbP}^3_{k,k'} \right] = \alpha_{k,k'} \cdot \left[ \bbP^3_{k,k'} \right].
 \]
 In order to compute $\alpha_{k,k'}$, we can appeal to Lemma \ref{L:triviality_lemma}, so let us specify $\smash{\overline{S^2}}$ as a $2$-di\-men\-sion\-al cobordism from $\varnothing$ to $\varnothing$, let us specify $\overline{P^3}$ as a $3$-di\-men\-sion\-al cobordism from $\varnothing$ to $S^2 \sqcup S^2 \sqcup \overline{S^2}$, and let us specify $S^2 \times I$ as a $3$-di\-men\-sion\-al cobordism from $S^2 \sqcup \overline{S^2}$ to $\varnothing$. 
 Up to isotopy, skein equivalence, and multiplication by invertible scalars, we only need to compare the projective traces of the morphisms $F_{\calC}(T_{f,f',g + g'',g' + g'''})$ and $F_{\calC}(T_{\tilde{f},f',g + g'',g' + g'''})$ for all
 \begin{gather*}
  f' \in \Hom_{\calC} \left( V_0 \otimes \sigma(k'') \otimes V_0^*,V_0 \otimes \sigma(k) \otimes V_0^* \otimes V_0 \otimes \sigma(k') \otimes V_0^* \right), \\
 g'' \in \{ -g \} + (G \smallsetminus X), \quad g''' \in \{ -g' \} + (G \smallsetminus X),
 \end{gather*}
 where the $\calC$-colored ribbon graph $T_{f,f',g + g'',g' + g'''}$ is represented in the left-hand part of Figure \ref{F:3-pant_lemma_2}, where $k'' = k + k'$, where $g = g' = 0$, and where $f$ and $\tilde{f}$ are given by the evaluation of the Reshetikhin-Turaev functor $F_{\calC}$ against the $\calC$-colored ribbon graphs represented in the left-hand part and in the right-hand part of Figure \ref{F:twisted_3-pant_lemma} respectively.
 \begin{figure}[tb]\label{F:twisted_3-pant_lemma}
  \centering
  \includegraphics{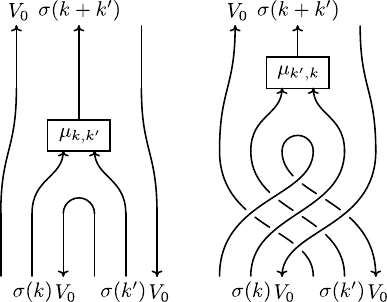}
  \caption{$\calC$-Colored ribbon graphs representing $f$ and $\tilde{f}$ respectively.}
 \end{figure} 
 By using twice the compatibility condition between the $G$-structure and the $\PGr$-action, we can pull the $\sigma(k)$-colored edge in the graphical representation of $\tilde{f}$ on top of all $V_0$-colored edges. This means
 \[
  T_{\tilde{f},f',g + g'',g' + g'''} \doteq \gamma(k,k') \cdot T_{f,f',g + g'',g' + g'''}. \qedhere
 \]
\end{proof}

\section{Graded extended universal construction}\label{S:graded_extended_universal_construction}

In this section we define a $\PGr$-graded refinement of the $2$-functors $\bfA_{\calC}$ and $\bfA'_{\calC}$ based on the results of the previous sections. What we proved so far can be summarized as follows: we have a realization $\bbS^2$ of $\PGr$ in $\bfA_{\calC}(\varnothing) = \bfA'_{\calC}(\varnothing)$ sending every $k \in \PGr$ to $\bbS^2_k \in \bfA_{\calC}(\varnothing) = \bfA'_{\calC}(\varnothing)$, with coherence data provided by the $0$-colored $3$-discs $\bbD^3_0$, and by the $(k,k')$-colored $3$-pants $\bbP^3_{k,k'}$. Now, recall the definition of the $2$-transformations $\bfmu : \sqtimes \circ \left( \bfA_{\calC} \times \bfA_{\calC} \right) \Rightarrow \bfA_{\calC} \circ \disjun$ and $\bfmu' : \sqtimes \circ \left( \bfA'_{\calC} \times \bfA'_{\calC} \right) \Rightarrow \bfA'_{\calC} \circ \disjun$ given in Proposition \ref{P:lax_monoidality_ETQFT}.

\begin{proposition}\label{P:action_by_spheres}
 For every object $\bbGamma$ of $\bfadCob_{\calC}$, we have an action $R^{\bbS^2}$ of $\PGr$ on $\bfA_{\calC}(\bbGamma)$ determined by the right translation linear endofunctors 
 \[
  R^{\bbS^2_k} := \bfmu_{\bbGamma,\varnothing} \circ (\id_{\bfA_{\calC}(\bbGamma)} \times \bbS^2_k) \in \End_{\Bbbk}(\bfA_{\calC}(\bbGamma))
 \]
 for every $k \in \PGr$, and similarly we have an action ${R'}^{\bbS^2}$ of $\PGr$ on $\bfA'_{\calC}(\bbGamma)$ determined by the right translation linear endofunctors 
 \[ R'^{\bbS^2_k} := \bfmu'_{\bbGamma,\varnothing} \circ (\id_{\bfA'_{\calC}(\bbGamma)} \times \bbS^2_k) \in \End_{\Bbbk}(\bfA'_{\calC}(\bbGamma))
 \]
 for every $k \in \PGr$.
\end{proposition}

The proof of Proposition \ref{P:action_by_spheres} is given by Lemmas \ref{L:unitarity_of_3-discs}, \ref{L:associativity_of_3-pants}, \ref{L:3-pants_composition_1}, and \ref{L:3-pants_composition_2}. We are now ready to upgrade the extended universal construction by means of an operation called \textit{$\PGr$-graded extension}, which is introduced here in Definition \ref{D:graded_extension}. Let us consider an object $\bbGamma$ of $\bfadCob_{\calC}$.

\begin{definition}\label{D:univ_gr_linear_category}
 The \textit{universal $\PGr$-graded linear category of $\bbGamma$} is the $\PGr$-graded linear category 
 \[
  \bbA^Z_{\calC}(\bbGamma) := R^{\bbS^2}(\bfA_{\calC}(\bbGamma))
 \]
 obtained from $\bfA_{\calC}(\bbGamma)$ by $\PGr$-graded extension with respect to the action $R^{\bbS^2}$ of $\PGr$ given by Proposition \ref{P:action_by_spheres}. The \textit{dual universal $\PGr$-graded linear category of $\bbGamma$} is the $\PGr$-graded linear category 
 \[
  \bbA'^Z_{\calC}(\bbGamma) := {R'}^{\bbS^2}(\bfA'_{\calC}(\bbGamma))
 \]
 obtained from $\bfA'_{\calC}(\bbGamma)$ by $\PGr$-graded extension with respect to the action ${R'}^{\bbS^2}$ of $\PGr$ given by Proposition \ref{P:action_by_spheres}.
\end{definition}

Let us spell this definition out: objects of $\bbA^Z_{\calC}(\bbGamma)$ are the same as objects $\bbSigma_{\bbGamma}$ of $\bfA_{\calC}(\bbGamma)$, degree $k$ morphisms of $\bbA^Z_{\calC}(\bbGamma)$ from $\bbSigma_{\bbGamma}$ to $\bbSigma''_{\bbGamma}$ are morphisms $[ \bbM^k_{\bbGamma} ]$ of $\Hom_{\bfA_{\calC}(\bbGamma)}(\bbSigma_{\bbGamma},\bbSigma''_{\bbGamma} \disjun \bbS^2_{-k})$, identities of $\bbA^Z_{\calC}(\bbGamma)$ are
\[
 \left[ \id_{\bbSigma_{\bbGamma}}^0 \right] := \left[ \id_{\bbSigma_{\bbGamma}} \disjun \bbD^3_0 \right]
\]
in $\Hom_{\bfA_{\calC}(\bbGamma)}(\bbSigma_{\bbGamma},\bbSigma_{\bbGamma} \disjun \bbS^2_0)$ for every $\bbSigma_{\bbGamma} \in \bfA_{\calC}(\bbGamma)$, and compositions of $\bbA^Z_{\calC}(\bbGamma)$ are
\[
 \left[ \bbM''^{\ell}_{\bbGamma} \right] \diamond \left[ \bbM^k_{\bbGamma} \right] := \left[ ( \id_{\bbSigma^{\fourth}_{\bbGamma}} \disjun \bbP^3_{-\ell,-k} ) \ast ( \bbM''^{\ell}_{\bbGamma} \disjun \id_{\bbS^2_{-k}} ) \ast \bbM^k_{\bbGamma} \right]
\]
in $\Hom_{\bfA_{\calC}(\bbGamma)}(\bbSigma_{\bbGamma},\bbSigma^{\fourth}_{\bbGamma} \disjun \bbS^2_{-\ell-k})$ for every $[ \bbM^k_{\bbGamma} ] \in \Hom_{\bfA_{\calC}(\bbGamma)}(\bbSigma_{\bbGamma},\bbSigma''_{\bbGamma} \disjun \bbS^2_{-k})$ and every $[ \bbM''^{\ell}_{\bbGamma} ] \in \Hom_{\bfA_{\calC}(\bbGamma)}(\bbSigma''_{\bbGamma},\bbSigma^{\fourth}_{\bbGamma} \disjun \bbS^2_{-\ell})$. The same description holds for the dual case: objects of $\bbA'^Z_{\calC}(\bbGamma)$ are the same as objects $\bbSigma'_{\bbGamma}$ of $\bfA'_{\calC}(\bbGamma)$, degree $k$ morphisms of $\bbA'^Z_{\calC}(\bbGamma)$ from $\bbSigma'_{\bbGamma}$ to $\bbSigma'''_{\bbGamma}$ are morphisms $[ \bbM'^k_{\bbGamma} ]$ of $\Hom_{\bfA'_{\calC}(\bbGamma)}(\bbSigma'_{\bbGamma},\bbSigma'''_{\bbGamma} \disjun \bbS^2_{-k})$, identities of $\bbA'^Z_{\calC}(\bbGamma)$ are
\[
 \left[ \id_{\bbSigma'_{\bbGamma}}^0 \right] := \left[ \id_{\bbSigma'_{\bbGamma}} \disjun \bbD^3_0 \right]
\]
in $\Hom_{\bfA'_{\calC}(\bbGamma)}(\bbSigma'_{\bbGamma},\bbSigma'_{\bbGamma} \disjun \bbS^2_0)$ for every $\bbSigma'_{\bbGamma} \in \bfA'_{\calC}(\bbGamma)$, and compositions of $\bbA'^Z_{\calC}(\bbGamma)$ are
\[
 \left[ \bbM'''^{\ell}_{\bbGamma} \right] \diamond \left[ \bbM'^k_{\bbGamma} \right] := \left[ ( \id_{\bbSigma^{\fifth}_{\bbGamma}} \disjun \bbP^3_{-\ell,-k} ) \ast ( \bbM'''^{\ell}_{\bbGamma} \disjun \id_{\bbS^2_{-k}} ) \ast \bbM'^k_{\bbGamma} \right]
\]
in $\Hom_{\bfA'_{\calC}(\bbGamma)}(\bbSigma'_{\bbGamma},\bbSigma^{\fifth}_{\bbGamma} \disjun \bbS^2_{-\ell-k})$ for every $[ \bbM'^k_{\bbGamma} ] \in \Hom_{\bfA'_{\calC}(\bbGamma)}(\bbSigma'_{\bbGamma},\bbSigma'''_{\bbGamma} \disjun \bbS^2_{-k})$ and every $[ \bbM'''^{\ell}_{\bbGamma} ] \in \Hom_{\bfA'_{\calC}(\bbGamma)}(\bbSigma'''_{\bbGamma},\bbSigma^{\fifth}_{\bbGamma} \disjun \bbS^2_{-\ell})$. Moving on, let us consider a 1-morphism $\bbSigma : \bbGamma \rightarrow \bbGamma'$ of $\bfadCob_{\calC}$.

\begin{definition}\label{D:univ_gr_linear_functor}
 The \textit{universal $\PGr$-graded linear functor of $\bbSigma$} is the $\PGr$-graded linear functor 
 \[
  \bbA^Z_{\calC}(\bbSigma) : \bbA^Z_{\calC}(\bbGamma) \rightarrow \bbA^Z_{\calC}(\bbGamma')
 \]
 sending every object $\bbSigma_{\bbGamma} \in \bbA^Z_{\calC}(\bbGamma)$
 to the object 
 \[
  \bbSigma \circ \bbSigma_{\bbGamma} \in \bbA^Z_{\calC}(\bbGamma'),
 \]
 and every degree $k$ morphism $[ \bbM^k_{\bbGamma} ] \in \Hom_{\bbA^Z_{\calC}(\bbGamma)}(\bbSigma_{\bbGamma},\bbSigma''_{\bbGamma})$ 
 to the degree $k$ morphism 
 \[
  \left[ \id_{\bbSigma} \circ \bbM^k_{\bbGamma} \right] \in \Hom_{\bbA^Z_{\calC}(\bbGamma')} \left( \bbSigma \circ \bbSigma_{\bbGamma},\bbSigma \circ \bbSigma''_{\bbGamma} \right).
 \] 
 The \textit{dual universal $\PGr$-graded linear functor of $\bbSigma$} is the $\PGr$-graded linear functor 
 \[
  \bbA'^Z_{\calC}(\bbSigma) : \bbA'^Z_{\calC}(\bbGamma') \rightarrow \bbA'^Z_{\calC}(\bbGamma)
 \]
 sending every object $\bbSigma'_{\bbGamma'} \in \bbA'^Z_{\calC}(\bbGamma')$
 to the object 
 \[
  \bbSigma'_{\bbGamma'} \circ \bbSigma \in \bbA'^Z_{\calC}(\bbGamma),
 \]
 and every degree $k$ morphism $[ \bbM'^k_{\bbGamma'} ] \in \Hom_{\bbA'^Z_{\calC}(\bbGamma')}(\bbSigma'_{\bbGamma'},\bbSigma'''_{\bbGamma'})$ 
 to the degree $k$ morphism 
 \[
  \left[ \bbM'^k_{\bbGamma'} \circ \id_{\bbSigma} \right] \in \Hom_{\bbA'^Z_{\calC}(\bbGamma)} \left( \bbSigma'_{\bbGamma'} \circ \bbSigma,\bbSigma'''_{\bbGamma'} \circ \bbSigma \right).
 \]
\end{definition}

Next, let us consider a $2$-morphism $\bbM : \bbSigma \Rightarrow \bbSigma'$ of $\bfadCob_{\calC}$ between $1$-morphisms $\bbSigma, \bbSigma' : \bbGamma \rightarrow \bbGamma'$.

\begin{definition}\label{D:univ_gr_natural_transformation}
 The \textit{universal $\PGr$-graded natural transformation of $\bbM$} is the $\PGr$-graded natural transformation $\bbA^Z_{\calC}(\bbM) : \bbA^Z_{\calC}(\bbSigma) \Rightarrow \bbA^Z_{\calC}(\bbSigma')$
 associating with every object $\bbSigma_{\bbGamma} \in \bbA^Z_{\calC}(\bbGamma)$
 the degree $0$ morphism 
 \[
  \left[ \left( \bbM \circ \id_{\bbSigma_{\bbGamma}} \right) \disjun \bbD^3_0 \right] \in \Hom_{\bbA^Z_{\calC}(\bbGamma')}(\bbSigma \circ \bbSigma_{\bbGamma},\bbSigma' \circ \bbSigma_{\bbGamma}).
 \] 
 The \textit{dual universal $\PGr$-graded natural transformation of $\bbM$} is the $\PGr$-graded natural transformation $\bbA'^Z_{\calC}(\bbM) : \bbA'^Z_{\calC}(\bbSigma) \Rightarrow \bbA'^Z_{\calC}(\bbSigma')$
 associating with every object ${\bbSigma'_{\bbGamma} \in  \bbA'^Z_{\calC}(\bbGamma')}$ the degree $0$ morphism 
 \[
  \left[ \left( \id_{\bbSigma'_{\bbGamma'}} \circ \bbM \right) \disjun \bbD^3_0 \right] \in \Hom_{\bbA'^Z_{\calC}(\bbGamma)}(\bbSigma'_{\bbGamma'} \circ \bbSigma,\bbSigma'_{\bbGamma'} \circ \bbSigma').
 \]
\end{definition}

The \textit{$\PGr$-graded extended universal construction} extends $\CGP_{\calC}$ to a pair of $2$-func\-tors
\[
 \bbA^Z_{\calC} : \bfadCob_{\calC} \rightarrow \bfCat_{\Bbbk}^{\PGr}, \quad
 \bbA'^Z_{\calC} : (\bfadCob_{\calC})^{\op} \rightarrow \bfCat_{\Bbbk}^{\PGr}
\]
called the \textit{universal $\PGr$-graded quantization $2$-func\-tor} and the \textit{dual universal $\PGr$-graded quantization $2$-func\-tor} respectively. Remark that these $2$-func\-tors are strict, as a consequence of our confusing $\bfadCob_{\calC}$, and thus also $(\bfadCob_{\calC})^{\op}$, with strict $2$-cat\-e\-go\-ries. We denote with
\[
 \hat{\bbA}^Z_{\calC} : \bfadCob_{\calC} \rightarrow \coCat_{\Bbbk}^{\PGr}, \quad
 \hat{\bbA}'^Z_{\calC} : (\bfadCob_{\calC})^{\op} \rightarrow \coCat_{\Bbbk}^{\PGr}
\]
the completions of the universal $\PGr$-graded quantization $2$-func\-tor and of the dual universal $\PGr$-graded quantization $2$-func\-tor in the sense of Proposition \ref{P:gr_co_strict_2-funct} and Remark \ref{R:2-gr-completion}, where $\coCat_{\Bbbk}^{\PGr}$ is the symmetric monoidal $2$-cat\-e\-go\-ry of complete $\PGr$-graded linear categories, see Definition \ref{D:sym_mon_2-cat_of_co_gr_lin_cat}.

%% file: chapter_6.tex
\chapter{Symmetric monoidality}\label{Ch:symmetric_monoidality}

The first part of this chapter contains the proof of our main result: the $2$-func\-tor $\hat{\bbA}^Z_{\calC} : \bfadCob_{\calC} \rightarrow \coCat_{\Bbbk}^{\PGr}$ of Section \ref{S:graded_extended_universal_construction} 
is symmetric monoidal. The second part is devoted to a description of the $\PGr$-graded TQFT contained in $\hat{\bbA}^Z_{\calC}$, which will be studied in detail in the next chapter.

\section{Graded ETQFT}\label{S:graded_ETQFT}

In order to establish the symmetric monoidality of the $2$-func\-tor $\hat{\bbA}^Z_{\calC}$, we first need to construct the appropriate coherence data. To this end, we start by defining a $\PGr$-graded linear functor of the form
\[
 \bfepsilon : \Bbbk \rightarrow \bbA^Z_{\calC}(\varnothing).
\]
We do this by sending the unique object of $\Bbbk$ to the empty object $\id_{\varnothing}$ of $\bbA^Z_{\calC}(\varnothing)$.

\begin{proposition}\label{P:epsilon}
 The $\PGr$-graded linear functor 
 \[ 
  \bfepsilon : \Bbbk \rightarrow \bbA^Z_{\calC}(\varnothing)
 \]
 is a $\PGr$-Morita equivalence.
\end{proposition}

\begin{proof}
 The $\PGr$-graded linear category $\bbA^Z_{\calC}(\varnothing)$ is dominated by $\{ \id_{\varnothing} \}$, and $\id_{\varnothing}$ is a simple object, so the result follows immediately from Proposition \ref{P:gr_Morita_equivalence}.
\end{proof}

Next, we need to define a $2$-trans\-for\-ma\-tion of the form
\[
 \bfmu : \sqtimes \circ \left( \bbA^Z_{\calC} \times \bbA^Z_{\calC} \right) \Rightarrow \bbA^Z_{\calC} \circ \disjun.
\]
In order to do this, we need to start by defining, for all objects $\bbGamma$ and $\bbGamma'$ of $\bfadCob_{\calC}$, a $\PGr$-graded linear functor of the form
\[
 \bfmu_{\bbGamma,\bbGamma'} : \bbA^Z_{\calC}(\bbGamma) \sqtimes \bbA^Z_{\calC}(\bbGamma') \rightarrow \bbA^Z_{\calC}(\bbGamma \disjun \bbGamma').
\]
We do this by sending every object $(\bbSigma_{\bbGamma},\bbSigma_{\bbGamma'})$ of $\bbA^Z_{\calC}(\bbGamma) \sqtimes \bbA^Z_{\calC}(\bbGamma')$ to the object $ \bbSigma_{\bbGamma} \disjun \bbSigma_{\bbGamma'}$ of $\bbA^Z_{\calC}(\bbGamma \disjun \bbGamma')$,
and every degree $k + k'$ morphism 
\[ 
 [\bbM^k_{\bbGamma}] \otimes [\bbM^{k'}_{\bbGamma'}]
\]
of $\Hom_{\bbA^Z_{\calC}(\bbGamma) \sqtimes \bbA^Z_{\calC}(\bbGamma')}((\bbSigma_{\bbGamma},\bbSigma_{\bbGamma'}),(\bbSigma''_{\bbGamma},\bbSigma''_{\bbGamma'}))$ to the degree $k + k'$ morphism
\[
 \left[ \left( \id_{\bbSigma''_{\bbGamma}} \disjun \left( ( \id_{\bbSigma''_{\bbGamma'}} \disjun \bbP^3_{-k,-k'} ) \ast ( \cobbr_{\bbS^2_{-k},\bbSigma''_{\bbGamma'}} \disjun \id_{\bbS^2_{-k'}} ) \right) \right) \ast ( \bbM^k_{\bbGamma} \disjun \bbM^{k'}_{\bbGamma'} ) \right]
\]
of $\Hom_{\bbA^Z_{\calC}(\bbGamma \disjun \bbGamma')}(\bbSigma_{\bbGamma} \disjun \bbSigma_{\bbGamma'},\bbSigma''_{\bbGamma} \disjun \bbSigma''_{\bbGamma'})$, where $\cobbr_{\bbS^2_{-k},\bbSigma''_{\bbGamma'}}$ denotes the braiding $2$-mor\-phism of Definition \ref{D:braiding_2-morphism}.

\begin{proposition}\label{P:mu}
 For all objects $\bbGamma$ and $\bbGamma'$ of $\bfadCob_{\calC}$
 \[
  \bfmu_{\bbGamma,\bbGamma'} : \bbA^Z_{\calC}(\bbGamma) \sqtimes \bbA^Z_{\calC}(\bbGamma') \rightarrow \bbA^Z_{\calC}(\bbGamma \disjun \bbGamma')
 \]
 is a $\PGr$-Morita equivalence.
\end{proposition}

\begin{proof}
 We start by proving $\bfmu_{\bbGamma,\bbGamma'}$ is actually a $\PGr$-graded linear functor. First of all, we need to show that
 \[
  \bfmu_{\bbGamma,\bbGamma'} \left( [\id^0_{\bbSigma_{\bbGamma}}] \otimes [\id^0_{\bbSigma_{\bbGamma'}}] \right) = \left[ \id^0_{\bbSigma_{\bbGamma} \disjun \bbSigma_{\bbGamma'}} \right].
 \]
 For all objects $\bbSigma_{\bbGamma} \in \bbA^Z_{\calC}(\bbGamma)$ and $\bbSigma_{\bbGamma'} \in \bbA^Z_{\calC}(\bbGamma')$. This follows directly from Lemma \ref{L:unitarity_of_3-discs}. Next, we need to show that
 \begin{align*}
  &\bfmu_{\bbGamma,\bbGamma'} \left( [\bbM''^{\ell'}_{\bbGamma}] \otimes [\bbM''^{\ell'}_{\bbGamma'}] \right) \diamond \bfmu_{\bbGamma,\bbGamma'} \left( [\bbM^k_{\bbGamma}] \otimes [\bbM^{k'}_{\bbGamma'}] \right) \\
  &\hspace{\parindent} = \gamma(k,\ell') \cdot \bfmu_{\bbGamma,\bbGamma'} \left( \left( [\bbM''^{\ell'}_{\bbGamma}] \diamond [\bbM^k_{\bbGamma}] \right) \otimes \left( [\bbM''^{\ell'}_{\bbGamma'}] \diamond [\bbM^{k'}_{\bbGamma'}] \right) \right)
 \end{align*}
 for every degree $k + k'$ morphism
 \[
  [\bbM^k_{\bbGamma}] \otimes [\bbM^{k'}_{\bbGamma'}] \in \Hom_{\bbA^Z_{\calC}(\bbGamma)}(\bbSigma_{\bbGamma},\bbSigma''_{\bbGamma}) \otimes \Hom_{\bbA^Z_{\calC}(\bbGamma')}(\bbSigma_{\bbGamma'},\bbSigma''_{\bbGamma'})
 \]
 and every degree $\ell + \ell'$ morphism
 \[
  [\bbM''^{\ell}_{\bbGamma}] \otimes [\bbM''^{\ell'}_{\bbGamma'}] \in \Hom_{\bbA^Z_{\calC}(\bbGamma)}(\bbSigma''_{\bbGamma},\bbSigma^{\fourth}_{\bbGamma}) \otimes \Hom_{\bbA^Z_{\calC}(\bbGamma')}(\bbSigma''_{\bbGamma'},\bbSigma^{\fourth}_{\bbGamma'}).
 \]
 This follows from
 \begin{align*}
  &\left[ {\bbP^3_{\phantom{3}}}_{\mathllap{-} \ell - \ell', - k - k'} \ast ( {\bbP^3_{\phantom{3}}}_{\mathllap{-}\ell,-\ell'} \disjun {\bbP^3_{\phantom{3}}}_{\mathllap{-}k,-k'} ) \ast ( \id_{{\bbS^2_{\phantom{2}}}_{\mathllap{-}\ell}} \disjun \cobbr_{{\bbS^2_{\phantom{2}}}_{\mathllap{-}k},{\bbS^2_{\phantom{2}}}_{\mathllap{-}\ell'}} \disjun \id_{{\bbS^2_{\phantom{2}}}_{\mathllap{-}k'}} ) \right] \\
  &\hspace{\parindent} = \left[ {\bbP^3_{\phantom{3}}}_{\mathllap{-} \ell - \ell' - k, - k'} \ast \left( \big( {\bbP^3_{\phantom{3}}}_{\mathllap{-} \ell - \ell', - k} \ast ( {\bbP^3_{\phantom{3}}}_{\mathllap{-}\ell,-\ell'} \disjun \id_{{\bbS^2_{\phantom{2}}}_{\mathllap{-}k}} ) \ast ( \id_{{\bbS^2_{\phantom{2}}}_{\mathllap{-}\ell}} \disjun \cobbr_{{\bbS^2_{\phantom{2}}}_{\mathllap{-}k},{\bbS^2_{\phantom{2}}}_{\mathllap{-}\ell'}} ) \big) \disjun \id_{{\bbS^2_{\phantom{2}}}_{\mathllap{-}k'}} \right) \right] \\
  &\hspace{\parindent} = \left[ {\bbP^3_{\phantom{3}}}_{\mathllap{-} \ell - \ell' - k, - k'} \ast \left( \Big( {\bbP^3_{\phantom{3}}}_{\mathllap{-} \ell, - \ell' - k} \ast \big( \id_{{\bbS^2_{\phantom{2}}}_{\mathllap{-} \ell}} \disjun {( {\bbP^3_{\phantom{3}}}_{\mathllap{-}\ell',-k}  \ast \cobbr_{{\bbS^2_{\phantom{2}}}_{\mathllap{-}k},{\bbS^2_{\phantom{2}}}_{\mathllap{-}\ell'}} )} \big) \Big) \disjun \id_{{\bbS^2_{\phantom{2}}}_{\mathllap{-}k'}} \right) \right] \\
  &\hspace{\parindent} = \gamma(-k,-\ell') \cdot \left[ {\bbP^3_{\phantom{3}}}_{\mathllap{-} \ell - k - \ell', - k'} \ast \left( \big( {\bbP^3_{\phantom{3}}}_{\mathllap{-} \ell,- k - \ell'} \ast ( \id_{{\bbS^2_{\phantom{2}}}_{\mathllap{-}\ell}} \disjun {\bbP^3_{\phantom{3}}}_{\mathllap{-}k,-\ell'} ) \big) \disjun \id_{{\bbS^2_{\phantom{2}}}_{\mathllap{-}k'}} \right) \right] \\
  &\hspace{\parindent} = \gamma(k,\ell') \cdot \left[ {\bbP^3_{\phantom{3}}}_{\mathllap{-} \ell - k - \ell', - k'} \ast \left( \big( {\bbP^3_{\phantom{3}}}_{\mathllap{-} \ell - k, - \ell'} \ast ( {\bbP^3_{\phantom{3}}}_{\mathllap{-}\ell,-k} \disjun \id_{{\bbS^2_{\phantom{2}}}_{\mathllap{-}\ell'}} ) \big) \disjun \id_{{\bbS^2_{\phantom{2}}}_{\mathllap{-}k'}} \right) \right] \\
  &\hspace{\parindent} = \gamma(k,\ell') \cdot \left[ {\bbP^3_{\phantom{3}}}_{\mathllap{-} \ell - k, - \ell' - k'} \ast ( {\bbP^3_{\phantom{3}}}_{\mathllap{-}\ell,-k} \disjun {\bbP^3_{\phantom{3}}}_{\mathllap{-}\ell',-k'} ) \right],
 \end{align*}
 where the first, the second, the fourth, and the fifth equality follow from Lemma \ref{L:associativity_of_3-pants}, and where the third equality follows from Lemma \ref{L:twisted_3-pant_lemma}.
 Now, we can move on to prove $\bfmu_{\bbGamma,\bbGamma'}$ is a $\PGr$-Morita equivalence. Once again, the strategy is to use Proposition \ref{P:gr_Morita_equivalence}. To begin with, if $\bbA^Z_{\calC}(\bbGamma)$ is dominated by 
 \[
  D(\varSigma;\bbGamma) = \left\{ \bbSigma_{(P,\vartheta)} = (\varSigma,P,\vartheta,\calL) \in \bbA^Z_{\calC}(\bbGamma) \bigm| (P,\vartheta) \in \adSk(\varSigma;\varnothing,\bbGamma) \right\},
 \]
 and if $\bbA^Z_{\calC}(\bbGamma')$ is dominated by 
 \[
  D(\varSigma';\bbGamma') = \left\{ \bbSigma'_{(P',\vartheta')} = (\varSigma',P',\vartheta',\calL') \in \bbA^Z_{\calC}(\bbGamma') \Bigm| (P',\vartheta') \in \adSk(\varSigma';\varnothing,\bbGamma') \right\},
 \]
 then $\bbA^Z_{\calC}(\bbGamma \disjun \bbGamma')$ is dominated by 
 \[
  D(\varSigma \sqcup \varSigma';\bbGamma \disjun \bbGamma') = \left\{ \bbSigma_{(P,\vartheta)} \disjun \bbSigma'_{(P',\vartheta')} \in \bbA^Z_{\calC}(\bbGamma \disjun \bbGamma') \biggm| 
  \begin{array}{l} 
   (P,\vartheta) \in \adSk(\varSigma;\varnothing,\bbGamma) \\
   (P',\vartheta') \in \adSk(\varSigma';\varnothing,\bbGamma')
  \end{array} \right\},
 \]
 as follows from Lemma \ref{L:Morita_reduction}. Therefore, $\bfmu_{\bbGamma,\bbGamma'}$ defines a bijection between generators. Next, to see that it is faithful, let us consider a trivial degree $k$ morphism
 \[
  \sum_{i=1}^m \alpha_i \cdot \bfmu_{\bbGamma,\bbGamma'} 
  \left( [ \bbM^{k - k_i}_{\bbGamma,i} ] \otimes [ \bbM^{k_i}_{\bbGamma',i} ] \right) \in \Hom_{\bbA^Z_{\calC}(\bbGamma \disjun \bbGamma')}(\bbSigma_{\bbGamma} \disjun \bbSigma_{\bbGamma'},\bbSigma''_{\bbGamma} \disjun \bbSigma''_{\bbGamma'}).
 \] 
 Now remark that, for all $1$-mor\-phisms 
 \[
  \bbSigma : \bbGamma \rightarrow \varnothing, \quad
  \bbSigma' : \bbGamma' \rightarrow \varnothing,
 \]
 and for all $2$-mor\-phisms
 \begin{gather*}
  \bbM : \id_{\varnothing} \Rightarrow \left( \bbSigma \circ \bbSigma_{\bbGamma} \right), \quad
  \bbM' : \id_{\varnothing} \Rightarrow \left( \bbSigma' \circ \bbSigma_{\bbGamma'} \right), \\
  \bbM'' : \left( (\bbSigma \circ \bbSigma''_{\bbGamma}) \disjun \bbS^2_{- k + k_i} \right) \Rightarrow \id_{\varnothing}, \quad
  \bbM''' : \left( (\bbSigma' \circ \bbSigma''_{\bbGamma'}) \disjun \bbS^2_{-k_i} \right) \Rightarrow \id_{\varnothing}
 \end{gather*}
 of $\bfadCob_{\calC}$, the Costantino-Geer-Patureau invariant of the admissible closed $3$-man\-i\-fold obtained by vertically composing
 \[
  \bbM \disjun \bbM'
 \]
 to
 \[
  \left( \id_{\bbSigma} \circ \bbM^{k - k_i}_{\bbGamma,i} \right) \disjun \left( \id_{\bbSigma'} \circ \bbM^{k_i}_{\bbGamma',i} \right),
 \]
 then to
 \[
  \id_{\bbSigma \circ \bbSigma''_{\bbGamma}} \disjun \left( \left( \id_{\bbSigma' \circ \bbSigma''_{\bbGamma'}} \disjun \bbP^3_{- k + k_i,- k_i} \right) \ast \left( \cobbr_{\bbS^2_{- k + k_i},\bbSigma' \circ \bbSigma''_{\bbGamma'}} \disjun \id_{\bbS^2_{-k_i}} \right) \right),
 \]
 then to
 \[
  \id_{\bbSigma \circ \bbSigma''_{\bbGamma}} \disjun \left( \left( \cobbr_{\bbSigma' \circ \bbSigma''_{\bbGamma'},\bbS^2_{- k + k_i}} \disjun \id_{\bbS^2_{-k_i}} \right) \right) \ast \left( \id_{\bbSigma' \circ \bbSigma''_{\bbGamma'}} \disjun \overline{\bbP^3_{- k + k_i, - k_i}} \right),
 \]
 and then to
 \[
  \bbM'' \disjun \bbM'''
 \]
 equals
 \[
  \calD \rmd(V_0)^{-1}
  \CGP_{\calC} \left( \bbM'' \ast \left( \id_{\bbSigma} \circ \bbM^{k-k_i}_{\bbGamma,i} \right) \ast \bbM \right) 
  \CGP_{\calC} \left( \bbM''' \ast \left( \id_{\bbSigma'} \circ \bbM^{k_i}_{\bbGamma',i} \right) \ast \bbM' \right)
 \]
 for every integer $1 \leqslant i \leqslant m$. This implies the degree $k$ morphism
 \[
  \sum_{i=1}^m \alpha_i \cdot [ \bbM^{k-k_i}_{\bbGamma,i} ] \otimes [ \bbM^{k_i}_{\bbGamma',i} ] \in \Hom_{\bbA^Z_{\calC}(\bbGamma)}(\bbSigma_{\bbGamma},\bbSigma''_{\bbGamma}) \otimes \Hom_{\bbA^Z_{\calC}(\bbGamma')}(\bbSigma_{\bbGamma'},\bbSigma''_{\bbGamma'})
 \]
 is trivial, which proves $\bfmu_{\bbGamma,\bbGamma'}$ is faithful. To see it is also full, we need to show that, for all objects 
 \[
  \bbSigma_{\bbGamma} = (\varSigma_{\varGamma},P,\vartheta,\calL), \quad
  \bbSigma''_{\bbGamma} = (\varSigma''_{\varGamma},P'',\vartheta'',\calL'')
 \]
 of $\bbA^Z_{\calC}(\bbGamma)$, and for all objects 
 \[
  \bbSigma_{\bbGamma'} = (\varSigma_{\varGamma'},P',\vartheta',\calL'), \quad
  \bbSigma''_{\bbGamma'} = (\varSigma''_{\varGamma'},P''',\vartheta''',\calL''')
 \]
 of $\bbA^Z_{\calC}(\bbGamma')$, every degree $k$ morphism
 \[
  [ \bbM^k_{\bbGamma \disjun \bbGamma'} ] \in \Hom_{\bbA^Z_{\calC}(\bbGamma \disjun \bbGamma')} \left( \bbSigma_{\bbGamma} \disjun \bbSigma_{\bbGamma'}, \bbSigma''_{\bbGamma} \disjun \bbSigma''_{\bbGamma'} \right)
 \]
 can be written as
 \[
  \sum_{i=1}^m \alpha_i \cdot \bfmu_{\bbGamma,\bbGamma'} \left( [\bbM^{k-k_i}_{\bbGamma,i}] \otimes [\bbM^{k_i}_{\bbGamma',i}] \right)
 \]
 for some coefficients $\alpha_i \in \Bbbk$, for some degree $k-k_i$ morphisms 
 \[
  [ \bbM^{k-k_i}_{\bbGamma,i} ] \in \Hom_{\bbA^Z_{\calC}(\bbGamma)}(\bbSigma_{\bbGamma},\bbSigma''_{\bbGamma}),
 \]
 and for some degree $k_i$ morphisms 
 \[
  [ \bbM^{k_i}_{\bbGamma',i} ] \in \Hom_{\bbA^Z_{\calC}(\bbGamma')}(\bbSigma_{\bbGamma'},\bbSigma''_{\bbGamma'}),
 \]
 with $1 \leqslant i \leqslant m$. In orded to do so, let us consider non-empty connected $3$-dimensional cobordisms with corners $M$ from $\varSigma_{\varGamma}$ to $\varSigma''_{\varGamma} \sqcup S^2$ and $M'$ from $\varSigma_{\varGamma'}$ to $\varSigma''_{\varGamma'} \sqcup S^2$. Then, thanks to Lemma \ref{L:connection_lemma}, $[ \bbM^k_{\bbGamma \disjun \bbGamma'} ]$ is the image of some vector of $\adSk(P^3(M,M');\bbSigma_{\bbGamma} \disjun \bbSigma_{\bbGamma'},\bbSigma''_{\bbGamma} \disjun \bbSigma''_{\bbGamma'} \disjun \bbS^3_{-k})$, where the cobordism with corners $P^3(M,M')$ is obtained by vertically gluing
 \[
  M \sqcup M'
 \]
 to
 \[
  ( \varSigma''_{\varGamma} \times I ) \sqcup \left( ( S^2 \sqcup \varSigma''_{\varGamma'} ) \ttimes I \right) \sqcup ( S^2 \times I )
 \]
 along $\varSigma''_{\varGamma} \sqcup S^2 \sqcup \varSigma''_{\varGamma'} \sqcup S^2$, and then to
 \[
  ( \varSigma''_{\varGamma} \times I ) \sqcup ( \varSigma''_{\varGamma'} \times I ) \sqcup P^3
 \]
 along $\varSigma''_{\varGamma} \sqcup \varSigma''_{\varGamma'} \sqcup S^2 \sqcup S^2$. Up to isotopy and skein equivalence, Lemmas \ref{L:2-sphere} and \ref{L:3-pant} allow us to conclude.
\end{proof}


Now, we need to define, for all $1$-mor\-phisms $\bbSigma : \bbGamma \rightarrow \bbGamma''$ and $\bbSigma' : \bbGamma' \rightarrow \bbGamma'''$ of $\bfadCob_{\calC}$, a $\PGr$-graded natural transformation of the form
\[
 \bfmu_{\bbSigma,\bbSigma'} : \bbA^Z_{\calC}(\bbSigma \disjun \bbSigma') \circ \bfmu_{\bbGamma,\bbGamma'} \Rightarrow \bfmu_{\bbGamma'',\bbGamma'''} \circ \left( \bbA^Z_{\calC}(\bbSigma) \sqtimes \bbA^Z_{\calC}(\bbSigma') \right).
\]
We do this by associating with every object $(\bbSigma_{\bbGamma},\bbSigma_{\bbGamma'})$ of $\bbA^Z_{\calC}(\bbGamma) \sqtimes \bbA^Z_{\calC}(\bbGamma')$
the degree $0$ morphism
\[
 \left[ \disjun_{(\bbSigma,\bbSigma'),(\bbSigma_{\bbGamma},\bbSigma_{\bbGamma'})} \disjun \bbD^3_0 \right]
\]
of $\Hom_{\bbA^Z_{\calC}(\bbGamma'' \disjun \bbGamma''')}((\bbSigma \disjun \bbSigma') \circ (\bbSigma_{\bbGamma} \disjun \bbSigma_{\bbGamma'}),{(\bbSigma \circ \bbSigma_{\bbGamma})} \disjun {(\bbSigma' \circ \bbSigma_{\bbGamma'})})$,
which we denote $\bfmu_{(\bbSigma,\bbSigma'),(\bbSigma_{\bbGamma},\bbSigma_{\bbGamma'})}^0$.

\begin{lemma}
 For all $1$-mor\-phisms $\bbSigma : \bbGamma \rightarrow \bbGamma''$ and $\bbSigma' : \bbGamma' \rightarrow \bbGamma'''$ of $\bfadCob_{\calC}$
 \[
  \bfmu_{\bbSigma,\bbSigma'} : \bbA^Z_{\calC}(\bbSigma \disjun \bbSigma') \circ \bfmu_{\bbGamma,\bbGamma'} \Rightarrow \bfmu_{\bbGamma'',\bbGamma'''} \circ \left( \bbA^Z_{\calC}(\bbSigma) \sqtimes \bbA^Z_{\calC}(\bbSigma') \right).
 \]
 is a $\PGr$-graded natural transformation
\end{lemma}

\begin{proof}
 We need to show that
 \begin{align*}
  &\bfmu_{\bbGamma'',\bbGamma'''} \left( \bbA^Z_{\calC}(\bbSigma) \left( [\bbM^k_{\bbGamma}] \right) \otimes \bbA^Z_{\calC}(\bbSigma') \left( [\bbM^{k'}_{\bbGamma'}] \right) \right) \diamond \bfmu_{(\bbSigma,\bbSigma'),(\bbSigma_{\bbGamma},\bbSigma_{\bbGamma'})}^0 \\
  &\hspace{\parindent} = \bfmu_{(\bbSigma,\bbSigma'),(\bbSigma''_{\bbGamma},\bbSigma''_{\bbGamma'})}^0 \diamond \bbA^Z_{\calC}(\bbSigma \disjun \bbSigma') \left( \bfmu_{\bbGamma,\bbGamma'} \left( [\bbM^k_{\bbGamma}] \otimes [\bbM^{k'}_{\bbGamma'}] \right) \right)
 \end{align*}
 for every degree $k + k'$ morphism
 \[
  [\bbM^k_{\bbGamma}] \otimes [\bbM^{k'}_{\bbGamma'}] \in \Hom_{\bbA^Z_{\calC}(\bbGamma)}(\bbSigma_{\bbGamma},\bbSigma''_{\bbGamma}) \otimes \Hom_{\bbA^Z_{\calC}(\bbGamma')}(\bbSigma_{\bbGamma'},\bbSigma''_{\bbGamma'}).
 \]
 This follows directly from the properties of $\disjun$, and from Lemma \ref{L:unitarity_of_3-discs}.
\end{proof}

It is now easy to check that $\bfmu$ is a $2$-transformation. Next, we need to define a $2$-modification of the form
\[
 \bfH : ( \bbA^Z_{\calC} \triangleright \cobbr ) \circ \bfmu \Rrightarrow ( \bfmu \triangleleft \bfT ) \circ \left( \gmmbr \triangleleft ( \bbA^Z_{\calC} \times \bbA^Z_{\calC} ) \right),
\]
where the $2$-func\-tor $\bfT : \bfadCob_{\calC} \times \bfadCob_{\calC} \rightarrow \bfadCob_{\calC} \times \bfadCob_{\calC}$ transposes the two copies of $\bfadCob_{\calC}$. In order to do this, we need to define, for all objects $\bbGamma$ and $\bbGamma'$ of $\bfadCob_{\calC}$, a $\PGr$-graded natural transformation of the form
\[
 \bfH_{\bbGamma,\bbGamma'} : \bbA^Z_{\calC}(\cobbr_{\bbGamma,\bbGamma'}) \circ \bfmu_{\bbGamma,\bbGamma'} \Rightarrow \bfmu_{\bbGamma',\bbGamma} \circ \gmmbr_{\bbA^Z_{\calC}(\bbGamma),\bbA^Z_{\calC}(\bbGamma')}.
\]
We do this by associating with every object $(\bbSigma_{\bbGamma},\bbSigma_{\bbGamma'})$ of $\bbA^Z_{\calC}(\bbGamma) \sqtimes \bbA^Z_{\calC}(\bbGamma')$
the degree $0$ morphism 
\[
 \left[ \cobbr_{\bbSigma_{\bbGamma},\bbSigma_{\bbGamma'}} \disjun \bbD^3_0 \right]
\]
of $\Hom_{\bbA^Z_{\calC}(\bbGamma' \disjun \bbGamma)}(\cobbr_{\bbGamma,\bbGamma'} \circ ( \bbSigma_{\bbGamma} \disjun \bbSigma_{\bbGamma'} ),\bbSigma_{\bbGamma'} \disjun \bbSigma_{\bbGamma})$, which we denote $\bfH_{\bbSigma_{\bbGamma},\bbSigma_{\bbGamma'}}^0$.

\begin{lemma}
 For all objects $\bbGamma$ and $\bbGamma'$ of $\bfadCob_{\calC}$
 \[
  \bfH_{\bbGamma,\bbGamma'} : \bbA^Z_{\calC}(\cobbr_{\bbGamma,\bbGamma'}) \circ \bfmu_{\bbGamma,\bbGamma'} \Rightarrow \bfmu_{\bbGamma',\bbGamma} \circ \gmmbr_{\bbA^Z_{\calC}(\bbGamma),\bbA^Z_{\calC}(\bbGamma')}
 \]
 is a $\PGr$-graded natural transformation.
\end{lemma}

\begin{proof}
 We have to show that
 \begin{align*}
  &\gamma(k,k') \cdot \bfmu_{\bbGamma',\bbGamma} \left( [ \bbM^{k'}_{\bbGamma'} ] \otimes [ \bbM^k_{\bbGamma} ] \right) \diamond \bfH_{\bbSigma_{\bbGamma},\bbSigma_{\bbGamma'}}^0 \\
  &\hspace{\parindent} = \bfH_{\bbSigma''_{\bbGamma},\bbSigma''_{\bbGamma'}}^0 \diamond \left( \bbA^Z_{\calC}(\cobbr_{\bbGamma,\bbGamma'}) \left( \bfmu_{\bbGamma,\bbGamma'} \left( [ \bbM^k_{\bbGamma} ] \otimes [ \bbM^{k'}_{\bbGamma'} ] \right) \right) \right)
 \end{align*}
 for every degree $k + k'$ morphism
 \[
  [\bbM^k_{\bbGamma}] \otimes [\bbM^{k'}_{\bbGamma'}] \in \Hom_{\bbA^Z_{\calC}(\bbGamma)}(\bbSigma_{\bbGamma},\bbSigma''_{\bbGamma}) \otimes \Hom_{\bbA^Z_{\calC}(\bbGamma')}(\bbSigma_{\bbGamma'},\bbSigma''_{\bbGamma'})
 \]
 This follows directly from the properties of $\cobbr$, and from Lemmas \ref{L:unitarity_of_3-discs} and \ref{L:twisted_3-pant_lemma}.
\end{proof}

We are now ready to prove our main result.

\begin{theorem}\label{T:symmetric_monoidality}
 The $2$-functor $\hat{\bbA}^Z_{\calC} : \bfadCob_{\calC} \rightarrow \coCat_{\Bbbk}^{\PGr}$ is symmetric monoidal.
\end{theorem}

\begin{proof}
 First of all, the completion $\cmpl \circ \bfepsilon$ of the $\PGr$-graded linear functor $\bfepsilon$ is an equivalence, as a direct consequence of Proposition \ref{P:epsilon}. Next, the $2$-transformation $( \cmpl \triangleright \bfmu ) \circ ( \bfchi \triangleleft (\bfF \times \bfF))$ is a composition of equivalences thanks to Proposition \ref{P:mu}, where $\bfchi : \cmpl \circ {\sqtimes} \circ (\cmpl \times \cmpl) \Rightarrow \cmpl \circ \sqtimes$ is part of the coherence data of the symmetric monoidal $2$-functor $\cmpl$ given by Proposition \ref{P:gr_co_strict_2-funct}. Furthermore, we claim
 \begin{gather*}
  \bfmu_{\varnothing,\bbGamma} \circ (\bfepsilon \sqtimes \id_{\bbA^Z_{\calC}(\bbGamma)}) = \id_{\bbA^Z_{\calC}(\bbGamma)}, \quad \bfmu_{\bbGamma,\varnothing} \circ (\id_{\bbA^Z_{\calC}(\bbGamma)} \sqtimes \bfepsilon) = \id_{\bbA^Z_{\calC}(\bbGamma)}, \\
  \bfmu_{\bbGamma,\bbGamma' \disjun \bbGamma''} \circ \left( \id_{\bbA^Z_{\calC}(\bbGamma)} \sqtimes \bfmu_{\bbGamma',\bbGamma''} \right) = \bfmu_{\bbGamma \disjun \bbGamma',\bbGamma''} \circ \left( \bfmu_{\bbGamma,\bbGamma'} \sqtimes \id_{\bbA^Z_{\calC}(\bbGamma'')} \right)
 \end{gather*}
 for all $\bbGamma, \bbGamma', \bbGamma'' \in \bfadCob_{\calC}$. On the level of objects, these equalities are an immediate consequence of the quasi-strictness of $\bfadCob_{\calC}$. On the level of morphisms, the first and the second equality follow from Lemma \ref{L:unitarity_of_3-discs}, while the third equality follows from Lemma \ref{L:associativity_of_3-pants}. This means that, apart from the invertible $2$-modification $\cmpl \triangleright \bfH$, we can assemble the coherence data for $\bbA^Z_{\calC}$ simply by using identity $2$-modifications. Therefore, all the conditions we need to check are greatly simplified. For instance, the first two axioms, namely Equations (HTA1) and (HTA2) of \cite{GPS95}, reduce to equalities between different compositions of identity $2$-modifications, thanks again to the quasi-strictness of $\bfadCob_{\calC}$. On the other hand, Equations (BHA1), (BHA2), and (SHA1) of \cite{M00} follow directly from the properties of the braiding $\cobbr$. 
\end{proof}

\section{Graded TQFT}\label{S:graded_TQFT}

In this section, we consider the category $\rmadCob_{\calC}$ of closed $1$-mor\-phisms of $\bfCob_{\calC}$ and of admissible $2$-mor\-phisms between them, and we describe the $\PGr$-graded TQFT $\bbV^Z_{\calC} : \rmadCob_{\calC} \rightarrow \Vect_{\Bbbk}^{\PGr}$ contained in $\hat{\bbA}^Z_{\calC}$. Indeed, as it was explained in the introduction, every $\PGr$-graded ETQFT yields a $\PGr$-graded TQFT, and $\PGr$-graded linear functors associated with closed $1$-morphisms of $\bfadCob_{\calC}$ determine $\PGr$-graded vector spaces which we explicitly characterize. When $\calC$ is taken to be the relative modular category of finite-dimensional weight representations of the unrolled version of quantum $\sltwo$ at even roots of unity, our construction recovers the $\Z$-graded TQFT defined in \cite{BCGP16}. To start, let us consider an object $\bbSigma$ of $\rmadCob_{\calC}$.

\begin{definition}\label{D:univ_gr_vector_space}
 The \textit{universal $\PGr$-graded vector space of $\bbSigma$} is the $\PGr$-graded vector space $\bbV^Z_{\calC}(\bbSigma)$ whose space of degree $k$ vectors is given by
 \[
  \bbV^k_{\calC}(\bbSigma) := \rmV_{\calC}(\bbSigma \disjun \bbS^2_{-k})
 \]
 for every $k \in \PGr$. The \textit{dual universal $\PGr$-graded vector space of $\bbSigma$} is the $\PGr$-graded vector space $\bbV'^Z_{\calC}(\bbSigma)$ whose space of degree $k$ vectors is given by 
 \[
  \bbV'^k_{\calC}(\bbSigma) := \rmV'_{\calC}(\bbSigma \disjun \bbS^2_k)
 \]
 for every $k \in \PGr$.
\end{definition}

Remark that this terminology is coherent, because for every object $\bbSigma$ of $\rmadCob_{\calC}$ the induced pairing $\langle \cdot,\cdot \rangle_{\bbSigma} : \bbV'^Z_{\calC}(\bbSigma) \otimes \bbV^Z_{\calC}(\bbSigma) \rightarrow \Bbbk$ whose degree $0$ component
\[
 \langle \cdot , \cdot \rangle^0_{\bbSigma} : \bigoplus_{k \in \PGr} \bbV'^{-k}_{\calC}(\bbSigma) \otimes \bbV^k_{\calC}(\bbSigma) \rightarrow \Bbbk
\]
sends every vector $[ \bbM'^{-k}_{\bbSigma} ] \otimes [ \bbM^k_{\bbSigma} ] \in \bbV'^{-k}_{\calC}(\bbSigma) \otimes \bbV^k_{\calC}(\bbSigma)$ to the number
\[
 \CGP_{\calC}(\bbM'^{-k}_{\bbSigma} \ast \bbM^k_{\bbSigma}) \in \Bbbk
\]
is non-degenerate, so that $\bbV'^Z_{\calC}(\bbSigma)$ is indeed dual to $\bbV^Z_{\calC}(\bbSigma)$. Now, let us consider a morphism $\bbM : \bbSigma \Rightarrow \bbSigma'$ of $\rmadCob_{\calC}$.

\begin{definition}
 The \textit{universal $\PGr$-graded linear map of $\bbM$} is the $\PGr$-graded linear map $\bbV^Z_{\calC}(\bbM) : \bbV^Z_{\calC}(\bbSigma) \rightarrow \bbV^Z_{\calC}(\bbSigma')$ whose degree $k$ component is given by 
 \[
  \bbV^k_{\calC}(\bbM) := \smash{\rmV_{\calC}(\bbM \disjun \id_{\bbS^2_{-k}})}
 \]
 for every $k \in \PGr$.  The \textit{dual universal $\PGr$-graded linear map of $\bbM$} is the $\PGr$-graded linear map $\bbV'^Z_{\calC}(\bbM) : \bbV'^Z_{\calC}(\bbSigma') \rightarrow \bbV'^Z_{\calC}(\bbSigma)$ whose degree $k$ component is given by 
 \[
  \bbV'^k_{\calC}(\bbM) := \rmV_{\calC}(\bbM \disjun \id_{\bbS^2_k})
 \]
 for every $k \in \PGr$.
\end{definition}

The \textit{$\PGr$-graded universal construction} extends $\CGP_{\calC}$ to a pair of functors
\[
 \bbV^Z_{\calC} : \rmadCob_{\calC} \rightarrow \Vect_{\Bbbk}^{\PGr}, \quad
 \bbV'^Z_{\calC} : (\rmadCob_{\calC})^{\op} \rightarrow \Vect_{\Bbbk}^{\PGr}
\]
called the \textit{universal $\PGr$-graded quantization functor} and the \textit{dual universal $\PGr$-graded quantization functor} respectively. In order to establish the symmetric monoidality of the functor $\bbV^Z_{\calC}$, we first need to construct the appropriate coherence data. To this end, we start by considering the $\PGr$-graded linear
\[
 \varepsilon : \Bbbk \rightarrow \bbV^Z_{\calC}(\id_{\varnothing}).
\]
which sends the generator $1$ of $\Bbbk$ to the empty vector $[\id_{\id_{\varnothing}}]$ of $\bbV^Z_{\calC}(\id_{\varnothing})$.

\begin{proposition}\label{P:epsilon_TQFT}
 The $\PGr$-graded linear map 
 \[
  \varepsilon : \Bbbk \rightarrow \bbV^Z_{\calC}(\id_{\varnothing})
 \]
 is a $\PGr$-graded isomorphism.
\end{proposition}

\begin{proof}
 The claim follows immediately from the multiplicativity and from the non-triviality of the invariant $\CGP_{\calC}$.
\end{proof}

Next, we consider the natural transformation
\[
 \mu : \otimes \circ \left( \bbV^Z_{\calC} \times \bbV^Z_{\calC} \right) \Rightarrow \bbV^Z_{\calC} \circ \disjun
\]
whose components are given, for all objects $\bbSigma$ and $\bbSigma'$ of $\rmadCob_{\calC}$, by the $\PGr$-graded linear map
\[
 \mu_{\bbSigma,\bbSigma'} : \bbV^Z_{\calC}(\bbSigma) \otimes \bbV^Z_{\calC}(\bbSigma') \rightarrow \bbV^Z_{\calC}(\bbSigma \disjun \bbSigma')
\]
sending every degree $k + k'$ vector 
\[ 
 [\bbM^k_{\bbSigma}] \otimes [\bbM^{k'}_{\bbSigma'}]
\]
of $\bbV^Z_{\calC}(\bbSigma) \otimes \bbV^Z_{\calC}(\bbSigma')$ to the degree $k + k'$ vector
\[
 \left[ \left( \id_{\bbSigma} \disjun \left( ( \id_{\bbSigma'} \disjun \bbP^3_{-k,-k'} ) \ast ( \cobbr_{\bbS^2_{-k},\bbSigma'} \disjun \id_{\bbS^2_{-k'}} ) \right) \right) \ast ( \bbM^k_{\bbSigma} \disjun \bbM^{k'}_{\bbSigma'} ) \right]
\]
of $\bbV^Z_{\calC}(\bbSigma \disjun \bbSigma')$.

\begin{proposition}\label{P:mu_TQFT}
 For all objects $\bbSigma$ and $\bbSigma'$ of $\rmadCob_{\calC}$
 \[
  \mu_{\bbSigma,\bbSigma'} : \bbV^Z_{\calC}(\bbSigma) \otimes \bbV^Z_{\calC}(\bbSigma') \rightarrow \bbV^Z_{\calC}(\bbSigma \disjun \bbSigma')
 \]
 is a $\PGr$-graded isomorphism.
\end{proposition}

\begin{proof}
 The argument is the exact same one we used in order to establish faithfulness and fullness of the $\PGr$-graded functors 
 \[ 
  \bfmu_{\bbGamma,\bbGamma'} : \bbA^Z_{\calC}(\bbGamma) \sqtimes \bbA^Z_{\calC}(\bbGamma') \rightarrow \bbA^Z_{\calC}(\bbGamma \disjun \bbGamma')
 \]
 for objects $\bbGamma$ and $\bbGamma'$ of $\bfadCob_{\calC}$ in the proof of Proposition \ref{P:mu}, with the only difference that we have to invoke Lemma \ref{L:connection_lemma_TQFT} instead of Lemma \ref{L:connection_lemma} for establishing the surjectivity of $\mu_{\bbSigma,\bbSigma'}$.
\end{proof}

We are now ready to prove the following result.

\begin{theorem}\label{T:graded_TQFT}
 The functor 
 \[
  \bbV^Z_{\calC} : \rmadCob_{\calC} \rightarrow \Vect_{\Bbbk}^{\PGr}
 \]
 is  symmetric monoidal.
\end{theorem}

\begin{proof}
 In order to prove our claim we only need to check that
 \begin{gather*}
  \mu_{\id_{\varnothing},\bbSigma} \circ (\varepsilon \otimes \id_{\bbV^Z_{\calC}(\bbSigma)}) = \id_{\bbV^Z_{\calC}(\bbSigma)}, \quad \mu_{\bbSigma,\id_{\varnothing}} \circ (\id_{\bbV^Z_{\calC}(\bbSigma)} \otimes \varepsilon) = \id_{\bbV^Z_{\calC}(\bbSigma)}, \\
  \mu_{\bbSigma,\bbSigma' \disjun \bbSigma''} \circ \left( \id_{\bbV^Z_{\calC}(\bbSigma)} \otimes \mu_{\bbSigma',\bbSigma''} \right) = \mu_{\bbSigma \disjun \bbSigma',\bbSigma''} \circ \left( \mu_{\bbSigma,\bbSigma'} \otimes \id_{\bbV^Z_{\calC}(\bbSigma'')} \right), \\
  \bbV^Z_{\calC}(\cobbr_{\bbSigma,\bbSigma'}) \circ \mu_{\bbSigma,\bbSigma'} = \mu_{\bbSigma',\bbSigma} \circ c^{\gamma}_{\bbV^Z_{\calC}(\bbSigma),\bbV^Z_{\calC}(\bbSigma')}
 \end{gather*}
 for all $\bbSigma, \bbSigma', \bbSigma'' \in \rmadCob_{\calC}$. The first and the second equality follow from Lemma \ref{L:unitarity_of_3-discs}, the third equality follows from Lemma \ref{L:associativity_of_3-pants}, while the fourth equality follows from Lemma \ref{L:twisted_3-pant_lemma}.
\end{proof}

Our goal for the remainder of the memoir will be characterize universal $\PGr$-graded linear categories and universal $\PGr$-graded linear functors associated with generating objects and $1$-morphisms of $\bfadCob_{\calC}$. This will allow us to produce computations of universal $\PGr$-graded vector spaces associated with closed $1$-morphisms of $\bfadCob_{\calC}$ in terms of the relative modular category $\calC$. In order to do this, we need to introduce a $\PGr$-graded linear version of the functor $\bbV^Z_{\calC}$ providing our $\PGr$-graded TQFT. First, we need to recall the definition of the $\PGr$-graded linear category $\bbVect_{\Bbbk}^{\PGr}$ of $\PGr$-graded vector spaces given by Example \ref{E:gr_cat_of_gr_vector_spaces}. Next, we need to consider the $\PGr$-graded linear functor 
\[
 \gltqft : \bbA^Z_{\calC}(\varnothing) \rightarrow \bbVect_{\Bbbk}^{\PGr}
\]
sending every object $\bbSigma_{\varnothing} \in \bbA^Z_{\calC}(\varnothing)$ to the $\PGr$-graded vector space $\bbV^Z_{\calC}(\bbSigma_{\varnothing})$, and every degree $k$ morphism $[ \bbM^k_{\varnothing} ] \in \Hom_{\bbA^Z_{\calC}(\varnothing)}(\bbSigma_{\varnothing},\bbSigma''_{\varnothing})$ to the $\PGr$-graded linear map $\gltqft ( [ \bbM^k_{\varnothing} ] ) : \bbV^Z_{\calC}(\bbSigma_{\varnothing}) \rightarrow R^{-k}(\bbV^Z_{\calC}(\bbSigma''_{\varnothing}))$
sending every degree $\ell$ vector $[ \bbM^{\ell}_{\bbSigma_{\varnothing}} ]$ of $\bbV^Z_{\calC}(\bbSigma_{\varnothing})$ to the degree $k + \ell$ vector
\[
 \left[ \left( \id_{\bbSigma''_{\varnothing}} \disjun \bbP^3_{-k,-\ell} \right) \ast \left( \bbM^k_{\varnothing} \disjun \id_{\bbS^2_{-\ell}} \right) \ast \bbM^{\ell}_{\bbSigma_{\varnothing}} \right]
\]
of $\bbV^Z_{\calC}(\bbSigma''_{\varnothing})$.
We will think about this problem as follows: if $\bbSigma$ is a closed $1$-morphism of $\bfadCob_{\calC}$, then we look for a $\PGr$-graded linear functor $\bbF_{\bbSigma} : \Bbbk \rightarrow \bbVect_{\Bbbk}^{\PGr}$, together with a $\PGr$-graded natural isomorphism of the form
\begin{center}
 \begin{tikzpicture}[descr/.style={fill=white}]
  \node (P0) at (45:{3.25/sqrt(2)}) {$\bbVect_{\Bbbk}^{\PGr}$};
  \node (P1) at (45+90:{3.25/sqrt(2)}) {$\Bbbk$};
  \node (P2) at (45+2*90:{3.25/sqrt(2)}) {$\bbA^Z_{\calC}(\varnothing)$};
  \node (P3) at (45+3*90:{3.25/sqrt(2)}) {$\bbA^Z_{\calC}(\varnothing)$};
  \node[above] at (0,0) {$\Downarrow$};
  \node[below] at (0,0) {$\bbeta_{\bbSigma}$};
  \draw
  (P1) edge[->] node[above,yshift=5pt] {$\bbF_{\bbSigma}$} (P0)
  (P1) edge[->] node[left,xshift=-5pt] {$\bfepsilon$} (P2)
  (P2) edge[->] node[below,yshift=-5pt] {$\bbA^Z_{\calC}(\bbSigma)$} (P3)
  (P3) edge[->] node[right,xshift=5pt] {$\gltqft$} (P0);
 \end{tikzpicture}
\end{center}
Since the $\PGr$-graded linear category $\Bbbk$ features a single object, we are actually looking for a single $\PGr$-graded vector space, together with a $\PGr$-graded linear isomorphism from $\gltqft(\bbSigma)$ to it. The idea is then to try and decompose the closed $1$-morphism $\bbSigma$ as a compostion of tensor products of generating $1$-morphisms, whose images are easier to describe and compute.


%% file: chapter_7.tex
%
%
%

\chapter{Characterization of the image}\label{Ch:characterization_of_image}

This chapter is devoted to a characterization of the $\PGr$-graded ETQFT $\hat{\bbA}^Z_{\calC}$
by means of the relative modular category $\calC$. 
First, we identify $\PGr$-graded linear categories associated with a set of generating objects of $\bfadCob_{\calC}$ composed of $1$-spheres. Next, we study $\PGr$-graded linear functors associated with a set of generating $1$-morphisms of $\bfadCob_{\calC}$ composed of $2$-discs, $2$-pants, and $2$-cylinders. We end this chapter with some computations of $\PGr$-graded vector spaces associated with closed $1$-morphisms of $\bfadCob_{\calC}$, as an application of the characterization we provide.

\section{1-Spheres}\label{S:1-spheres}

We start by describing universal $\PGr$-graded linear categories associated with generating objects of $\bfadCob_{\calC}$, given by $1$-spheres, in terms of homogeneous subcategories of projective objects of $\calC$. In order to do this, we first need to fix some notation. Recalling Definition \ref{D:graded_extension}, we denote with $R^{\sigma}(\calC)$ the $\PGr$-graded linear category obtained from $\calC$ by $\PGr$-graded extension with respect to the free action $R^{\sigma}$ of $\PGr$ determined by the right translation linear endofunctors $R^{\sigma(k)} \in \End_{\Bbbk}(\calC)$ mapping every object $V$ of $\calC$ to $V \otimes \sigma(k)$, and every morphism $f$ of $\calC$ to $f \otimes \id_{\sigma(k)}$ for every $k \in \PGr$. 
For all objects $V,V' \in \calC$, we denote with $\bbHom_{\calC}(V,V')$ the $\PGr$-graded vector space of morphisms of $R^{\sigma}(\calC)$ from $V$ to $V'$. Furthermore, since $R^{\sigma}$ restricts to a free action of $\PGr$ on $\Proj(\calC_g)$ for every $g \in G$, we set $\bbProj(\calC_g) := R^{\sigma}(\Proj(\calC_g))$.
Next, we need to introduce some objects and $1$-morphisms of $\bfCob_{\calC}$, so let us fix a $g \in G$ and an object $V \in \calC_g$.

\begin{definition}\label{D:1-sphere}
 The \textit{$g$-colored $1$-sphere} $\bbS^1_g$ is the object of $\bfCob_{\calC}$ given by 
 \[
  \left( S^1,\xi_g \right)
 \]
 where $\xi_g$ is the $G$-coloring of $S^1$ determined by $\langle \xi_g,[S^1] \rangle = g$, with a single base point at the south pole.
\end{definition}

%
%

\begin{definition}\label{D:2-disc}
 The \textit{$V$-col\-ored $2$-disc} $\bbD^2_V : \varnothing \rightarrow \bbS^1_g$ is the $1$-morphism of $\bfCob_{\calC}$ given by 
 \[
  \left( D^2,P_V,\vartheta_V,\{ 0 \} \right)
 \]
 where $P_V \subset D^2$ is the standard $\calC$-colored ribbon set associated with $(+,V) \in \Rib_{\calC}^G$, and where $\vartheta_V$ is the only compatible $G$-coloring of $(D^2,P_V)$ which vanishes on all relative homology classes contained in the southern hemisphere, as represented in Figure \ref{F:2-disc}. Similarly, the \textit{dual $V$-col\-ored $2$-disc} $\overline{\bbD^2_V} : \bbS^1_g \rightarrow \varnothing$ is the $1$-mor\-phism of $\bfadCob_{\calC}$ given by 
 \[
  \left( \overline{D^2},\overline{P_V},\vartheta_V,\{ 0 \} \right)
 \]
 where $\overline{P_V} \subset \overline{D^2}$ is the $\calC$-colored ribbon set obtained from $P_V$ by a change of sign.
\end{definition}

\begin{figure}[hbt]\label{F:2-disc}
 \centering
 \includegraphics{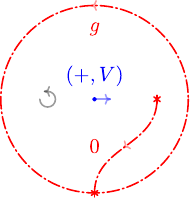}
 \caption{The $1$-morphism $\bbD^2_V$ of $\bfadCob_{\calC}$.}
\end{figure}

Let us consider a $3$-dimensional cobordism with corners $Y$ from $D^2$ to $D^2 \sqcup S^2$ whose support is given by $D^2 \times I \subset \R^3$ minus an open ball of radius $\frac{1}{4}$ and center $(0,0,\frac 23)$, 
and let us fix a section $s_{V,0} \in \Hom_{\calC}(V,V \otimes V_0 \otimes V_0^*)$
of the epimorphism $\id_V \otimes \rev_{V_0}$ for every object $V \in \bbProj(\calC_g)$, whose existence is ensured by Remark \ref{R:epic_evaluations}. Let
\[
 \bbF_g : \bbProj(\calC_g) \rightarrow \bbA^Z_{\calC}(\bbS^1_g)
\]
be the $\PGr$-graded linear functor sending every object $V$ of $\bbProj(\calC_g)$ to the object $\bbD^2_V$ of $\bbA^Z_{\calC}(\bbS^1_g)$, and every degree $k$ morphism $f^k$ of $\bbHom_{\calC}(V,V')$ to the degree $k$ morphism $\calD^{-1} \rmd(V_0) \cdot [ \bbY_{f^k} ]$ of $\smash{\Hom_{\bbA^Z_{\calC}(\bbS^1_g)}(\bbD^2_V,\bbD^2_{V'})}$, where the $2$-morphism ${\bbY_{f^k } : \bbD^2_V \Rightarrow \bbD^2_{V'} \disjun \bbS^2_{-k}}$ of $\bfadCob_{\calC}$ is given by
\[
 (Y,T_{f^k},\omega_{f^k},0 )
\]
for the $(\calC,G)$-coloring $(T_{f^k},\omega_{f^k})$ of $Y$ represented in Figure \ref{F:punctured_3-cylinder}.

\begin{figure}[hbt]\label{F:punctured_3-cylinder}
 \centering
 \includegraphics{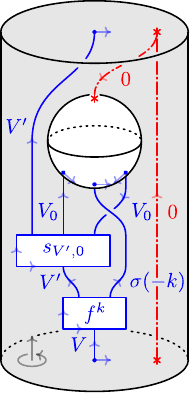}
 \caption{The $2$-morphism $\bbY_{f^k}$ of $\bfadCob_{\calC}$.}
\end{figure}

\begin{proposition}\label{P:univ_lin_cat}
 For every $g \in G$ the $\PGr$-graded linear functor
 \[ 
  \bbF_g : \bbProj(\calC_g) \rightarrow \bbA^Z_{\calC}(\bbS^1_g)
 \]
 is a $\PGr$-Morita equivalence.
\end{proposition}

\begin{proof} 
 We start by proving $\bbF_g$ is actually a $\PGr$-graded linear functor. First of all, we need to show that
 \[
  \left[ \id_{\bbD^2_V}^0 \right] = \calD^{-1} \rmd(V_0) \cdot \left[ \bbY_{\id_V^0} \right]
 \]
 for every object $V \in \bbProj(\calC_g)$. This follows directly from Lemma \ref{L:3-discs_composition}. Next, we need to show that
 \[
  \calD^{-2} \rmd(V_0)^2 \cdot \left[ \bbY_{f'^{k'}} \right] \diamond \left[ \bbY_{f^k} \right] = 
  \calD^{-1} \rmd(V_0) \cdot \left[ \bbY_{f'^{k'} \diamond f^k} \right]
 \]
 for every degree $k$ morphism $f^k \in \bbHom_{\calC}(V,V')$ and every degree $k'$ morphism $\smash{f'^{k'} \in \bbHom_{\calC}(V',V'')}$.
 This follows directly from the skein equivalence of Figure \ref{F:punctured_3-cylinder_lemma_1}, which gives a graphical representation of the equality we need to check.
 \begin{figure}[b]\label{F:punctured_3-cylinder_lemma_1}
  \centering
  \includegraphics{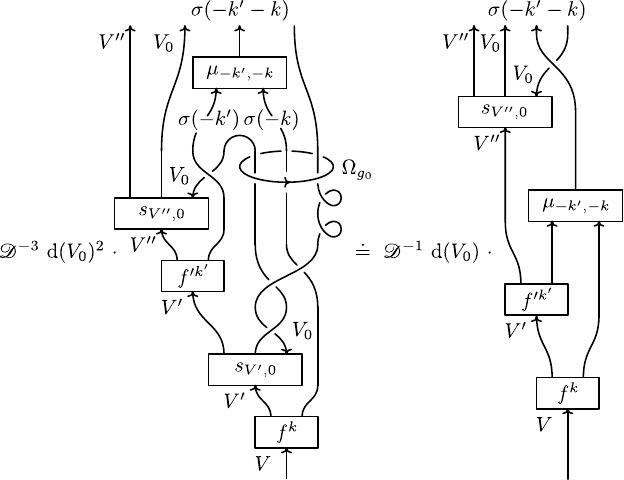}
  \caption{Skein equivalence witnessing the functoriality of $\bbF_g$.} 
 \end{figure}
 Remark that the $\calC$-colored ribbon graph represented in the left-hand part contains a surgey component accounting for the additional $3$-pant cobordism intervening in the composition in $\bbA^Z_{\calC}(\bbS^1_g)$, and that we have to slide a $V_0$-colored edge over it in order to make the surgery presentation computable. Remark also that the extra coefficient $\calD^{-1}$ on the left comes from Lemma \ref{L:surgery_axioms} with index $j = 2$. To prove the skein equivalence, we use twice the compatibility condition between the $G$-structure and the $\PGr$-action of $\calC$ in order to turn every undercrossing of the $\sigma(-k)$-colored edge with the rest of the graph into an overcrossing. Next, the relative monoidality condition allows us to remove the surgery component, and the price to pay is a coefficient $\calD^{-2} \rmd(V_0)$. Then, thanks to isotopy and to the definition of the section $s_{V',0}$, we can conclude. 
 Now, we can move on to prove $\bbF_g$ is a $\PGr$-Morita equivalence. Once again, the strategy is to use Proposition \ref{P:gr_Morita_equivalence}. First, we need to show the family of objects $\{ \bbD^2_V \mid V \in \bbProj(\calC_g) \}$
 dominates $\bbA^Z_{\calC}(\bbS^1_g)$. To do so, let us consider, for every $V \in \bbProj(\calC_g)$ and every $h \in G$, the $1$-morphism $\bbD^2_{V,h} : \varnothing \rightarrow \bbS^1_g$ of $\bfCob_{\calC}$ given by $\left( D^2,P_V,\vartheta_{V,h},\{ 0 \} \right)$,
 where $\vartheta_{V,h}$ is the only compatible $G$-coloring of $(D^2,P_V)$ satisfying $\langle \vartheta_{V,h},\gamma \rangle = h$ for the relative homology class $\gamma$ represented in Figure \ref{F:2-disc}. Then, thanks to Lemmas \ref{L:skein_equivalence} and \ref{L:Morita_reduction}, we know $\{ \bbD^2_{V,h} \mid V \in \bbProj(\calC_g), h \in G \}$ is a dominating set of objects for $\bbA^Z_{\calC}(\bbS^2_g)$. However, this set is redundant, as the objects $\bbD^2_{V,h}$ and $\bbD^2_{V,0}$ are isomorphic in $\bbA^Z_{\calC}(\bbS^2_g)$ for every $V \in \bbProj(\calC_g)$ and every $h \in G$. This can be seen by considering the invertible degree $0$ morphism of $[(\bbD^2 \times \bbI)_{V,h} \disjun \bbD^3_0] \in \Hom_{\bbA^Z_{\calC}(\calC_g)}(\bbD^2_{V,h},\bbD^2_{V,0})$, where the $2$-morphism $(\bbD^2 \times \bbI)_{V,h} : \bbD^2_{V,h} \Rightarrow \bbD^2_{V,0}$ of $\bfadCob_{\calC}$ is given by $(D^2 \times I,P_V \times I,\omega_{V,h},0)$
 for the $(\calC,G)$-coloring $(P_V \times I,\omega_{V,h})$ of $D^2 \times I$ represented in Figure \ref{F:punctured_3-cylinder_lemma_2}.
 \begin{figure}[t]\label{F:punctured_3-cylinder_lemma_2}
 	\centering
 	\includegraphics{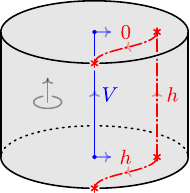}
 	\caption{The $2$-morphism $(\bbD^2 \times \bbI)_{V,h}$ of $\bfadCob_{\calC}$.}
 \end{figure}
 Therefore, we just need to show that $\bbF_g$ is faithful and full. Let us start with fullness, so let us fix some $k \in \PGr$ and some objects $V ,V' \in \bbProj(\calC_g)$. Every degree $k$ morphism of $\smash{\Hom_{\bbA^Z_{\calC}(\bbS^1_g)} ( \bbD^2_V,\bbD^2_{V'} )}$ is the image of a vector of $\smash{\adSk(Y;\bbD^2_V,\bbD^2_{V'} \disjun \bbS^2_{-k})}$ thanks to Lemma \ref{L:connection_lemma}. Furthermore, up to isotopy and skein equivalence, we can restrict to admissible $(\calC,G)$-colorings of the form $(T_f,\omega_h)$ like the one represented in Figure \ref{F:punctured_3-cylinder_lemma_3} 
 \begin{figure}[b]\label{F:punctured_3-cylinder_lemma_3}
  \centering
  \includegraphics{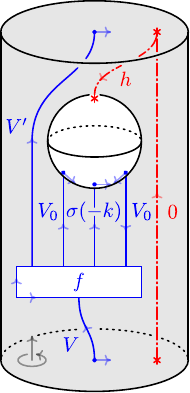}
  \caption{$(\calC,G)$-Coloring of $Y$ of the form $(T_f,\omega_h)$.}
 \end{figure}
  with 
 \[
  f \in \Hom_{\calC}(V,V' \otimes V_0 \otimes \sigma(-k) \otimes V_0^*), \quad
  h \in G.
 \]
 Indeed, remark that every $G$-coloring of $(Y,T_f)$ has to vanish on the relative homology class $B \times I$ joining the base points of the two copies of $D^2$, which is a sum of relative homology classes contained in the boundary whose evaluations are prescribed to be zero. Now, to determine which $(\calC,G)$-colorings correspond to trivial morphisms, we can appeal to Lemma \ref{L:triviality_lemma}, so let us specify $\smash{\overline{D^2}}$ as a $2$-di\-men\-sion\-al cobordism from $S^1$ to $\varnothing$, let us specify $D^3$ as a $3$-di\-men\-sion\-al cobordism from $\varnothing$ to $D^2 \cup_{S^1} \overline{D^2}$, and let us specify $S^2 \times I$ as a $3$-di\-men\-sion\-al cobordism from $(D^2 \cup_{S^1} \overline{D^2}) \sqcup \overline{S^2}$ to $\varnothing$. Now the triviality of the morphism $[Y,T_f,\omega_h,0]$ can be tested by considering all $V'' \in \calC_g$ and all
 \begin{gather*}
  (T,\omega) \in \adSk \left( D^3;\id_{\varnothing},\overline{\bbD^2_{V''}} \circ \bbD^2_V \right), \\
  (T',\omega') \in \adSk \left( S^2 \times I;\left( \overline{\bbD^2_{V''}} \circ \bbD^2_{V'} \right) \disjun \bbS^2_{-k},\id_{\varnothing} \right),
 \end{gather*}
 and by computing the Costantino-Geer-Patureau invariant $\CGP_{\calC}$ of the resulting closed $2$-mor\-phism of $\bfadCob_{\calC}$. We can choose a surgery presentation composed of a single unknot with framing zero, whose computability can always be forced by performing generic or projective stabilization outside of $(T_f,\omega_h)$. Therefore, up to isotopy, skein equivalence, and multiplication by invertible scalars, we only need to compute the projective trace of the morphism $F_{\calC}(T_{f,f',h+h'})$ for all
 \[
  f' \in \Hom_{\calC} \left( V' \otimes V_0 \otimes \sigma(-k) \otimes V_0^*,V \right), \quad
  h' \in \{ -h \} + (G \smallsetminus X),
 \]
 where the $\calC$-colored ribbon graph $T_{f,f',h + h'}$ is represented in the left-hand part of Figure \ref{F:punctured_3-cylinder_lemma_4}. 
 \begin{figure}[b]\label{F:punctured_3-cylinder_lemma_4}
  \centering
  \includegraphics{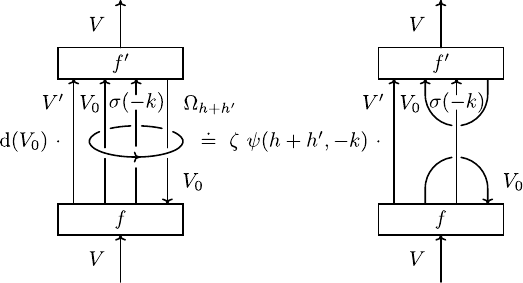}
  \caption{Skein equivalence following from the compatibility between the $G$-structure and the $\PGr$-action, and from the relative modularity of $\calC$.}
 \end{figure}
 This means that, if we consider
 \begin{gather*}
  e_{V',-k} := \id_{V'} \otimes \left( ( \rev_{V_0} \otimes \id_{\sigma(-k)} ) \circ ( \id_{V_0} \otimes c_{\sigma(-k),V_0^*} ) \right), \\ 
  s_{V',-k} := ( \id_{V'} \otimes \id_{V_0} \otimes c_{\sigma(-k),V_0^*}^{-1} ) \circ ( s_{V',0} \otimes \id_{\sigma(-k)} ), 
 \end{gather*}
 which satisfy $e_{V',-k} \circ s_{V',-k} = \id_{V' \otimes \sigma(-k)}$,
 we get
 \begin{align*}
  \left[ Y,T_f,\omega_h,0 \right] 
  &= \psi(h,-k) \cdot \left[ Y,T_{s_{V',-k} \circ e_{V',-k} \circ f},\omega_0,0 \right] \\
  &= \psi(h,-k) \cdot \bbY_{e_{V',-k} \circ f}.
 \end{align*}
 Therefore, we only need to show that $\bbF_g$ is faithful. This follows from the non-degeneracy of the projective trace $\rmt$: indeed, if we consider some non-trivial degree $k$ morphism $f^k \in \bbHom_{\calC}(V,V')$, then, thanks to Proposition \ref{P:non-degeneracy_of_trace}, there exists a morphism $f' \in \Hom_{\calC}(V' \otimes \sigma(-k),V)$ satisfying $\rmt_V(f' \circ f^k) \neq 0$.
 This means that, if we choose a retraction $r_{V',0} \in \Hom_{\calC}(V' \otimes V_0 \otimes V_0^*,V')$ of the monomorphism $\id_{V'} \otimes \lcoev_{V_0}$, whose existence is ensured by the injectivity of $V'$, and if we consider
 \begin{gather*}
  r_{V',-k} := ( r_{V',0} \otimes \id_{\sigma(-k)} ) \circ ( \id_{V'} \otimes \id_{V_0} \otimes c_{\sigma(-k),V_0^*} ), \\
  c_{V',-k} := \id_{V'} \otimes \left( ( \id_{V_0} \otimes c_{\sigma(-k),V_0^*}^{-1} ) \circ ( \lcoev_{V_0} \otimes \id_{\sigma(-k)} ) \right),
 \end{gather*}
 which satisfy $r_{V',-k} \circ c_{V',-k} = \id_{V' \otimes \sigma(-k)}$,
 we get
 \[
  \rmt_{V} \left( F_{\calC} \left( T_{s_{V',-k} \circ f^k,f' \circ r_{V',-k},g_0} \right) \right) = \zeta \psi(g_0,-k) \rmt_V \left(f' \circ f^k \right) \neq 0. \qedhere
 \]
\end{proof}

\begin{remark}\label{R:independence_of_section}
 The proof of the fullness of $\bbF_g$, and in particular the argument involving Figure \ref{F:punctured_3-cylinder_lemma_3}, shows that in fact, for every $f^k \in \bbHom_{\calC}(V,V')$, the definition of $\bbY_{f^k}$ is independent of the choice of the section $s_{V',0}$ of $\id_{V'} \otimes \rev_{V_0}$. This property will be used in the following. 
\end{remark}

\section{2-Discs}\label{S:2-discs}

We move on to describe universal $\PGr$-graded linear functors 
induced by generating $1$-mor\-phisms of $\bfadCob_{\calC}$ in terms of various structures over $\PGr$-graded subcategories of projective objects of $\calC$. In this section we focus on $2$-discs, which come in two flavors. The first family of universal $\PGr$-graded linear functors we are going to consider is quite trivial, and can be described in terms of constant functors.
Indeed, for every $g \in G$ and every object $V \in \Proj(\calC_g)$, we have a constant $\PGr$-graded linear functor from $\Bbbk$ to $\bbProj(\calC_g)$, which we still denote $V$, which fits into a commutative diagram of $\PGr$-graded linear functors of the form
\begin{center}
 \begin{tikzpicture}[descr/.style={fill=white}]
  \node (P0) at ({atan(-0.5)}:{3.25/(2*sin(atan(0.5)))}) {$\bbA^Z_{\calC}(\bbS^1_g)$};
  \node (P1) at (90:{3.25/2}) {$\bbProj(\calC_g)$};
  \node (P2) at ({atan(-0.5) + 2*90}:{3.25/(2*sin(atan(0.5)))}) {$\Bbbk$};
  \node (P3) at ({atan(0.5) + 2*90}:{3.25/(2*sin(atan(0.5)))}) {$\bbA^Z_{\calC}(\varnothing)$};
  \draw
  (P1) edge[->] node[right,xshift=10pt] {$\bbF_g$} (P0)
  (P2) edge[->] node[above,yshift=5pt] {$V$} (P1)
  (P2) edge[->] node[left,xshift=-5pt] {$\bfepsilon$} (P3)
  (P3) edge[->] node[below,yshift=-5pt] {$\bbA^Z_{\calC} \left( \bbD^2_V \right)$} (P0);
 \end{tikzpicture}
\end{center}

Next, we consider a family of universal $\PGr$-graded linear functors which can be described in terms of Hom functors. 
First of all, for every $g \in G$ and every object $V \in \calC_g$, not necessarily projective, let
\[
 \bbHom_{\calC}(V, \cdot) : \bbProj(\calC_g) \rightarrow \bbVect_{\Bbbk}^{\PGr}
\]
be the $\PGr$-graded linear functor sending every object $V'$ of $\bbProj(\calC_g)$ to the $\PGr$-graded vector space $\bbHom_{\calC}(V,V')$, and every degree $k'$ morphism $f'^{k'}$ of $\bbHom_{\calC}(V',V'')$ to the $\PGr$-graded linear map $\bbHom_{\calC}(V,f'^{k'}) : \bbHom_{\calC}(V,V') \rightarrow R^{-k'}\left(\bbHom_{\calC}(V,V'')\right)$ sending every degree $k$ vector $f^k$ of $\bbHom_{\calC}(V,V')$ to the degree $k + k'$ vector $f'^{k'} \diamond f^k$ of $\bbHom_{\calC}(V,V'')$. Next, let us consider a $3$-dimensional cobordism $X$ from $\varnothing$ to $(D^2 \cup_{S^1} \overline{D^2}) \sqcup S^2$ whose support is given by $D^3 \subset \R^3$ minus an open ball of radius $\frac{1}{4}$ and center $(0,0,\frac 13)$.
Then, for every object $V' \in \bbProj(\calC_g)$, we consider the $\PGr$-graded linear map
\[
 \bbeta_{V,V'} : \bbHom_{\calC}(V,V') \rightarrow \bbV^Z_{\calC} \left( \overline{\bbD^2_V} \circ \bbD^2_{V'} \right)
\]
sending every degree $k$ vector $f^k$ of $\bbHom_{\calC}(V,V')$ to the degree $k$ vector $[\bbX_{f^k}]$ of $\bbV^Z_{\calC}(\overline{\bbD^2_V} \circ \bbD^2_{V'})$, where the $2$-morphism
\[ 
 \bbX_{f^k} :\id_{\varnothing} \Rightarrow \left( \overline{\bbD^2_V} \circ \bbD^2_{V'}  \right) \disjun \bbS^2_{-k}
\]
of $\bfadCob_{\calC}$ is given by $( X,T_{f^k},\omega_{f^k},0 )$ for the $(\calC,G)$-coloring $(T_{f^k},\omega_{f^k})$ of $X$ represented in Figure \ref{F:punctured_3-disc}.

\begin{figure}[hbt]\label{F:punctured_3-disc}
 \centering
 \includegraphics{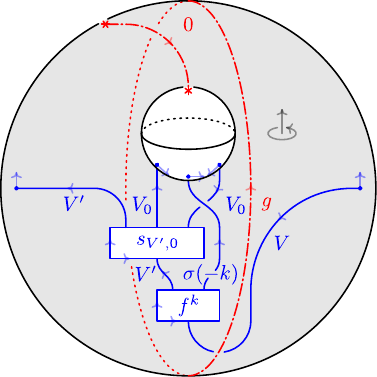}
 \caption{The $2$-morphism $\bbX_{f^k}$ of $\bfadCob_{\calC}$.}
\end{figure}

\begin{proposition}\label{P:counit}
 For every $g \in G$, and for every object $V \in \calC_g$,
 the $\PGr$-graded linear maps $\bbeta_{V,V'}$ define a $\PGr$-graded natural 
 isomorphism
 \begin{center}
  \begin{tikzpicture}[descr/.style={fill=white}]
   \node (P0) at ({atan(-0.5)}:{3.25/(2*sin(atan(0.5)))}) {$\bbA^Z_{\calC}(\varnothing)$};
   \node (P1) at ({atan(0.5)}:{3.25/(2*sin(atan(0.5)))}) {$\bbVect_{\Bbbk}^{\PGr}$};
   \node (P2) at ({atan(-0.5) + 2*90}:{3.25/(2*sin(atan(0.5)))}) {$\bbProj(\calC_g)$};
   \node (P3) at (3*90:{3.25/2}) {$\bbA^Z_{\calC}(\bbS^1_g)$};
   \node[above] at (0,0) {$\Downarrow$};
   \node[below] at (0,0) {$\bbeta_{V}$};
   \draw
   (P0) edge[->] node[right,xshift=5pt] {$\gltqft$} (P1)
   (P2) edge[->] node[above,yshift=5pt] {$\bbHom_{\calC}(V, \cdot )$} (P1)
   (P2) edge[->] node[left,xshift=-10pt] {$\bbF_{g}$} (P3)
   (P3) edge[->] node[below,yshift=-5pt] {$\bbA^Z_{\calC} \left( \overline{\bbD^2_V} \right)$} (P0);
  \end{tikzpicture}
 \end{center}
\end{proposition}

\begin{proof}
 The result is established by showing that the $\PGr$-graded linear maps
 \[
  \bbeta_{V,V'} : \bbHom_{\calC}(V,V') \rightarrow \bbV^Z_{\calC} \left( \overline{\bbD^2_V} \circ \bbD^2_{V'} \right)
 \]
 are natural with respect to degree $k'$ morphism $f'^{k'} \in \bbHom_{\calC}(V',V'')$, and that they are invertible for all objects $V' \in \bbProj(\calC_g)$. All these properties are checked by adapting the proof of Proposition \ref{P:univ_lin_cat}. First of all, naturality of $\bbeta_V$ is proved just like functoriality of $\bbF_g$. Indeed, we need to show that
 \[
  \calD^{-1} \rmd(V_0) \cdot \bbV^Z_{\calC} \left( \left[ \id_{\overline{\bbD^2_V}} \circ \bbY_{f'^{k'}} \right] \right)
  \left( \left[ \bbX_{f^k} \right] \right) \\
  = \left[ \bbX_{f'^{k'} \diamond f^k} \right]
 \]
 for every degree $k$ morphism $f^k \in \bbHom_{\calC}(V,V')$ and every degree $k'$ morphism $f'^{k'} \in \bbHom_{\calC}(V',V'')$. This follows directly from the skein equivalence of Figure \ref{F:punctured_3-cylinder_lemma_1}. Next, surjectivity and injectivity of $\bbeta_{V,V'}$ are shown just like fullness and faitfhulness of $\bbF_g$. Indeed, for what concerns surjectivity, every degree $k$ vector of $\smash{\bbV^Z_{\calC}(\overline{\bbD^2_V} \circ \bbD^2_{V'})}$ is the image of a vector of $\smash{\adSk(X;\id_{\varnothing},(\overline{\bbD^2_V} \circ \bbD^2_{V'}) \disjun \bbS^2_{-k})}$ thanks to Lemma \ref{L:connection_lemma_TQFT}. Furthermore, up to isotopy and skein equivalence, we can restrict to admissible $(\calC,G)$-colorings of the form $(T_f,\omega_h)$ like the one represented in Figure \ref{F:punctured_3-disc_lemma_1} 
 \begin{figure}[b]\label{F:punctured_3-disc_lemma_1}
  \centering
  \includegraphics{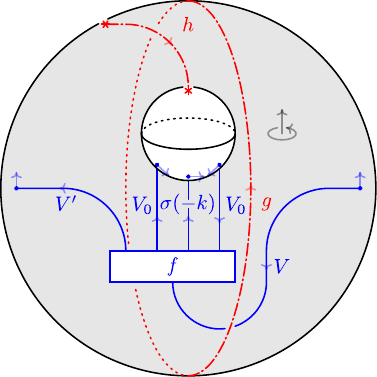}
  \caption{$(\calC,G)$-Coloring of $X$ of the form $(T_f,\omega_h)$.}
 \end{figure}
  with 
 \[
  f \in \Hom_{\calC}(V,V' \otimes V_0 \otimes \sigma(-k) \otimes V_0^*), \quad
  h \in G.
 \]
 To determine which $(\calC,G)$-colorings correspond to trivial morphisms, we can appeal to Lemma \ref{L:triviality_lemma_TQFT}, so let us specify $S^2 \times I$ as a $3$-di\-men\-sion\-al cobordism from $(D^2 \cup_{S^1} \overline{D^2}) \sqcup S^2$ to $\varnothing$. Now the triviality of the degree $k$ vector $[X,T_f,\omega_h,0]$ can be tested by considering all
 \[
  (T',\omega') \in \adSk \left( S^2 \times I; \left( \overline{\bbD^2_V} \circ \bbD^2_{V'} \right) \disjun \bbS^2_{-k},\id_{\varnothing} \right),
 \]
 and by computing the Costantino-Geer-Patureau invariant $\CGP_{\calC}$ of the resulting closed $2$-mor\-phism of $\bfadCob_{\calC}$. We can choose a surgery presentation composed of a single unknot with framing zero, whose computability can always be forced by performing generic or projective stabilization outside of $(T_f,\omega_h)$. Therefore, up to isotopy, skein equivalence, and multiplication by invertible scalars, we only need to compute the projective trace of the morphism $F_{\calC}(T_{f,f',h+h'})$ for all
 \[
  f' \in \Hom_{\calC} \left( V' \otimes V_0 \otimes \sigma(-k) \otimes V_0^*,V \right), \quad
  h' \in \{ -h \} + (G \smallsetminus X),
 \]
 where the $\calC$-colored ribbon graph $T_{f,f',h + h'}$ is represented in the left-hand part of Figure \ref{F:punctured_3-cylinder_lemma_3}. This means that, using the notation introduced during the proof of Proposition \ref{P:univ_lin_cat}, we get 
 \begin{align*}
  \left[ X,T_f,\omega_h,0 \right] 
  &= \psi(h,-k) \cdot \left[ X,T_{s_{V',-k} \circ e_{V',-k} \circ f},\omega_0,0 \right] \\
  &= \psi(h,-k) \cdot \bbX_{e_{V',-k} \circ f}.
 \end{align*}
 Therefore, we only need to show injectivity, which follows from the non-de\-gen\-er\-a\-cy of the projective trace $\rmt$: indeed, if we consider some non-trivial degree $k$ morphism $f^k \in \bbHom_{\calC}(V,V')$, then, thanks to Proposition \ref{P:non-degeneracy_of_trace}, there exists a morphism ${f' \in \Hom_{\calC}(V' \otimes \sigma(-k),V)}$ satisfying $\rmt_V(f' \circ f^k) \neq 0$.
 This means that, using again the notation introduced during the proof of Proposition \ref{P:univ_lin_cat}, we get
 \[
  \rmt_{V} \left( F_{\calC} \left( T_{s_{V',-k} \circ f^k,f' \circ r_{V',-k},g_0} \right) \right) = \zeta \psi(g_0,-k) \rmt_V \left(f' \circ f^k \right) \neq 0. \qedhere
 \]
\end{proof}

\section{2-Pants}\label{S:2-pants}

In this section we focus on a new family of $1$-morphisms of $\bfadCob_{\calC}$. We define the \textit{$2$-pant cobordism $P^2$} as the $2$-dimensional cobordism from $S^1 \sqcup S^1$ to $S^1$ whose support is given by $D^2$ minus two open balls of radius $\frac{1}{4}$ and center $(-\frac{1}{2},0,0)$ and $(+\frac{1}{2},0,0)$ respectively, and whose incoming and outgoing horizontal boundary identifications are induced by $\id_{S^1}$ through rescaling and translation. Then, let us consider $g,g' \in G$.

\begin{definition}\label{D:2-pants}
 The \textit{$(g,g')$-colored $2$-pant} $\bbP^2_{g,g'} : \bbS^1_g \disjun \bbS^1_{g'} \rightarrow \bbS^1_{g + g'}$ is the $1$-mor\-phism of $\bfadCob_{\calC}$ given by 
 \[
  \left( P^2,\varnothing,\omega_{g,g'},H_1(P^2;\R) \right)
 \]
 where $\omega_{g,g'}$ is the only $G$-coloring of $(P^2,\varnothing)$ which extends $\xi_g \sqcup \xi_{g'}$ and $\xi_{g + g'}$, and which vanishes on all relative homology classes contained in the southern hemisphere, as represented in Figure \ref{F:2-pant}. Similarly, the \textit{dual $(g,g')$-colored $2$-pant} $\smash{\overline{\bbP^2_{g,g'}} : \bbS^1_{g + g'} \rightarrow \bbS^1_g \disjun \bbS^1_{g'}}$ is the $1$-mor\-phism of $\bfadCob_{\calC}$ given by 
 \[
  \left( \overline{P^2},\varnothing,\omega_{g,g'},H_1(P^2;\R) \right).
 \]
\end{definition}

\begin{figure}[htb]\label{F:2-pant}
 \centering
 \includegraphics{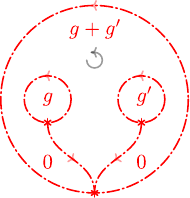}
 \caption{The $(g,g')$-colored $2$-pant $\bbP^2_{g,g'}$.}
\end{figure}

The first family of universal $\PGr$-graded linear functors we are going to consider can be described in terms of tensor product functors. First, for all $g,g' \in G$, let
\[
 \bbnabla_{g,g'} : \bbProj(\calC_g) \sqtimes \bbProj(\calC_{g'}) \rightarrow \bbProj(\calC_{g+g'})
\]
be the $\PGr$-graded linear functor sending every object $(V,V')$ of 
$\bbProj(\calC_g) \sqtimes \bbProj(\calC_{g'})$ to the object
$V \otimes V'$ of $\bbProj(\calC_{g+g'})$, and every degree $k + k'$ morphism
$f^k \otimes f'^{k'}$ of $\bbHom_{\calC}(V,V'') \otimes \bbHom_{\calC}(V',V''')$ to the degree $k + k'$ morphism
\[
 \left( \id_{V''} \otimes \left( ( \id_{V'''} \otimes \mu_{-k,-k'}) \circ (c_{\sigma(-k),V'''} \otimes \id_{\sigma(-k')}) \right) \right) \circ ( f^k \otimes f'^{k'} )
\]
of $\bbHom_{\calC}(V \otimes V',V'' \otimes V''')$. 
Next, for every object $(V,V') \in \bbProj(\calC_g) \sqtimes \bbProj(\calC_{g'})$, we consider the degree $0$ morphism
\[
 \bbeta^0_{g,g',(V,V')} \in \Hom_{\bbA^Z_{\calC}(\bbS^1_{g+g'})} \left( \bbD^2_{V \otimes V'},\bbP^2_{g,g'} \circ \left( \bbD^2_V \disjun \bbD^2_{V'} \right) \right),
\]
given by $[(\bbD^2 \times \bbI)_{(V,V')} \disjun \bbD^3_0]$, where the $2$-morphism
\[
 (\bbD^2 \times \bbI)_{(V,V')} : \bbD^2_{V \otimes V'} \Rightarrow \bbP^2_{g,g'} \circ \left( \bbD^2_V \disjun \bbD^2_{V'}, \right)
\]
of $\bfadCob_{\calC}$ is given by $(D^2 \times I,T_{(V,V')},\omega_{(V,V')},0)$ for the $\calC$-colored ribbon graph $T_{(V,V')} \subset D^2 \times I$ represented in Figure \ref{F:2-pant_lemma_1}, and for the unique compatible $G$-coloring $\omega_{(V,V')}$ of $(D^2 \times I,T_{(V,V')})$.

\begin{figure}[hbt]\label{F:2-pant_lemma_1}
	\centering
	\includegraphics{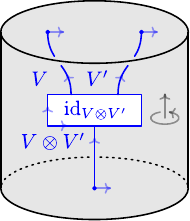}
	\caption{The $2$-morphism $(\bbD^2 \times \bbI)_{(V,V')}$ of $\bfadCob_{\calC}$.}
\end{figure}


\begin{proposition}\label{P:product}
 For all $g,g' \in G$ the degree $0$ morphisms $\bbeta^0_{g,g',(V,V')}$ define a $\PGr$-graded natural isomorphism
 \begin{center}
  \begin{tikzpicture}[descr/.style={fill=white}]
   \node (P0) at ({atan(-0.5)}:{3.25/(2*sin(atan(0.5)))}) {$\bbA^Z_{\calC}(\bbS^1_{g+g'})$};
   \node (P1) at (90:{3.25/2}) {$\bbProj(\calC_{g+g'})$};
   \node (P2) at ({atan(-0.5) + 2*90}:{3.25/(2*sin(atan(0.5)))}) {$\bbProj(\calC_g) \sqtimes \bbProj(\calC_{g'})$};
   \node (P3) at (3*90:{3.25/2}) {$\bbA^Z_{\calC}(\bbS^1_g \disjun \bbS^1_{g'})$};
   \node[above] at (0,0) {$\Downarrow$};
   \node[below] at (0,0) {$\bbeta_{g,g'}$};
   \draw
   (P1) edge[->] node[right,xshift=10pt] {$\bbF_{g+g'}$} (P0)
   (P2) edge[->] node[above,yshift=5pt] {$\bbnabla_{g,g'}$} (P1)
   (P2) edge[->] node[left,xshift=-10pt] {$\bfmu_{\bbS^1_g,\bbS^1_{g'}} \circ 
   (\bbF_{g} \sqtimes \bbF_{g'})$} (P3)
   (P3) edge[->] node[below,yshift=-5pt] {$\bbA^Z_{\calC} \left( \bbP^2_{g,g'} \right)$} (P0);
  \end{tikzpicture}
 \end{center}
\end{proposition}

\begin{proof}
 The result is established by showing that the degree $0$ morphisms
 \[
  \bbeta^0_{g,g',(V,V')} \in \Hom_{\bbA^Z_{\calC}(\bbS^1_{g+g'})} \left( \bbD^2_{V \otimes V'},\bbP^2_{g,g'} \circ \left( \bbD^2_V \disjun \bbD^2_{V'}, \right) \right)
 \]
 are natural with respect to degree $k$ morphisms $f^k \in \bbHom_{\calC}(V,V'')$ and degree $k'$ morphisms $\smash{f'^{k'} \in \bbHom_{\calC}(V',V''')}$, and that they are invertible for all objects $V \in \bbProj(\calC_g)$ and $V' \in \bbProj(\calC_{g'})$. For what concerns naturality, we need to show that
 \begin{align*}
  &\calD^{-2} \rmd(V_0)^2 \cdot \left( \bbA^Z_{\calC}(\bbP^2_{g,g'}) \right) \left( \bfmu_{\bbS^1_g,\bbS^1_{g'}} \left( \left[ \bbY_{f^k} \right] \otimes \left[ \bbY_{f'^{k'}} \right] \right) \right) \diamond \left[ (\bbD^2 \times \bbI)_{(V,V')} \disjun \bbD^3_0 \right] \\
  &\hspace{\parindent} = \calD^{-1} \rmd(V_0) \cdot \left[ (\bbD^2 \times \bbI)_{(V'',V''')} \disjun \bbD^3_0 \right] \diamond \left[ \bbY_{\bbnabla_{g,g'}(f^k \otimes f'^{k'})} \right]
 \end{align*}
 for every degree $k + k'$ morphism $f^k \otimes f'^{k'} \in \bbHom_{\calC}(V,V'') \otimes \bbHom_{\calC}(V',V''')$. This follows directly from the skein equivalence of Figure \ref{F:2-pant_lemma_2}, which gives a graphical representation of the equality we need to check.
 \begin{figure}[t]\label{F:2-pant_lemma_2}
  \centering
  \includegraphics{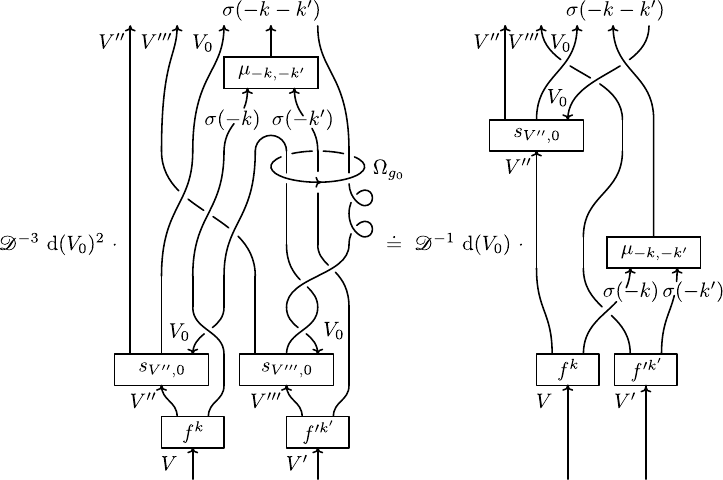}
  \caption{Skein equivalence witnessing the naturality of $\bbeta_{g,g'}$.} 
 \end{figure}
 Remark that we are using Remark \ref{R:independence_of_section} in order to assume that 
 \[
  s_{V'' \otimes V''',0} = ( \id_{V''} \otimes c_{V_0 \otimes V_0^*,V'''} ) \circ ( s_{V'',0} \otimes \id_{V'''} ).
 \]
 The skein equivalence is proved exactly like the one represented in Figure \ref{F:punctured_3-cylinder_lemma_1}: we use twice the compatibility condition between the $G$-structure and the $\PGr$-action of $\calC$ in order to turn every undercrossing of the $\sigma(-k')$-colored edge with the rest of the graph into an overcrossing. Next, the relative monoidality condition allows us to remove the surgery component, and the price to pay is a coefficient $\calD^{-2} \rmd(V_0)$. Then, thanks to isotopy and to the definition of the section $s_{V''',0}$, we can conclude. Invertibility is clear from the definition.
\end{proof}

We move on to study the dual $2$-pant cobordism, which gives rise to a family of universal $\PGr$-graded linear functors which can be described in terms of coends of internal Hom functors. 
This time, we have to distinguish two cases: either one of the two colors is generic, or both of them are critical. The two situations can be treated similarly, although the second one is considerably more complicated, as it forces us to deal with the non-semisimple part of $\calC$, which we control less effectively. We begin by studying the generic case.
First of all, for every $g \in G$ and every generic $g' \in G \smallsetminus X$, let
\[
 \bbDelta_{g,g'} : \bbProj(\calC_{g+g'}) \rightarrow \bbProj(\calC_g) \csqtimes \bbProj(\calC_{g'})
\]
be the $\PGr$-graded linear functor sending every object $V$ of $\bbProj(\calC_{g+g'})$ to the object 
\[ 
 \bigoplus_{i \in \rmI_{g'}} (V \otimes V_i^*,V_i) 
\]
of $\bbProj(\calC_g) \csqtimes \bbProj(\calC_{g'})$, and every degree $k$ morphism $f^k$ of $\bbHom_{\calC}(V,V')$ to the degree $k$ morphism
\[
 \left( \delta_{ij} \cdot (\bbnabla_{g+g',-g'}(f^k \otimes \id_{V_i^*}^0) \otimes \id^0_{V_i}) \right)_{(i,j) \in \rmI_{g'}^2}
\]
of
\[
 \bbHom_{\calC \csqtimes \calC} \left( \bigoplus_{i \in \rmI_{g'}} (V \otimes V_i^*,V_i),\bigoplus_{i \in \rmI_{g'}} (V' \otimes V_i^*,V_i) \right).
\]
Next, we consider a $3$-di\-men\-sion\-al cobordism with corners $W$ from $D^2 \sqcup D^2$ to $\smash{D^2 \cup_{S^1} \overline{P^2}}$ whose support is given by $I \times D^2 \subset \R^3$ minus two open balls of radius $\frac{1}{4}$ and centers $(0,0,0)$ and $(1,0,0)$.
Then, for every object $V \in \bbProj(\calC_{g+g'})$ and every $i \in \rmI_{g'}$, we consider the degree $0$ morphism
\[
 \bbeta^0_{g,g',V,i} \in \Hom_{\bbA^Z_{\calC}(\bbS^1_g \disjun \bbS^1_{g'})} \left( \bbD^2_{V \otimes V_i^*} \disjun \bbD^2_{V_i}, \overline{\bbP^2_{g,g'}} \circ \bbD^2_V \right)
\]
given by $[\bbW_{V,i} \disjun \bbD^3_0]$, where the $2$-morphism
\[
 \bbW_{V,i} : \bbD^2_{V \otimes V_i^*} \disjun \bbD^2_{V_i} \Rightarrow \overline{\bbP^2_{g,g'}} \circ \bbD^2_V
\]
 of $\bfadCob_{\calC}$ is given by $(W,T_{V,i},\omega_{V,i},0)$ for the $\calC$-colored ribbon graph $T_{V,i} \subset W$ represented in Figure \ref{F:coproduct_generic}, and for the unique compatible $G$-coloring $\omega_{V,i}$ of $(W,T_{V,i})$.

\begin{figure}[hbt]\label{F:coproduct_generic}
 \centering
 \includegraphics{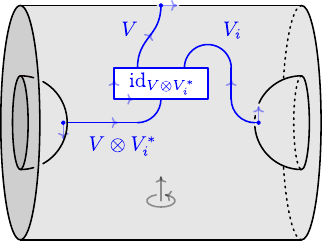}
 \caption{The $2$-morphism $\bbW_{V,i}$ of $\bfadCob_{\calC}$.} 
\end{figure}


\begin{proposition}\label{P:coproduct_generic}
 For every $g \in G$ and every generic $g' \in G \smallsetminus X$, the degree $0$ morphisms $\bbeta^0_{g,g',V,i}$ define a $\PGr$-graded natural isomorphism
 \begin{center}
  \begin{tikzpicture}[descr/.style={fill=white}]
   \node (P0) at ({atan(-0.5)}:{3.25/(2*sin(atan(0.5)))}) {$\hat{\bbA}^Z_{\calC}(\bbS^1_g \disjun \bbS^1_{g'})$};
   \node (P1) at (90:{3.25/2}) {$\bbProj(\calC_g) \csqtimes \bbProj(\calC_{g'})$};
   \node (P2) at ({atan(-0.5) + 2*90}:{3.25/(2*sin(atan(0.5)))}) {$\bbProj(\calC_{g+g'})$};
   \node (P3) at (3*90:{3.25/2}) {$\bbA^Z_{\calC}(\bbS^1_{g+g'})$};
   \node[above] at (0,0) {$\Downarrow$};
   \node[below] at (0,0) {$\bbeta_{g,g'}$};
   \draw
   (P1) edge[->] node[right,xshift=10pt] {$\hat{\bfmu}_{\bbS^1_g,\bbS^1_{g'}} \circ (\bbF_g \csqtimes \bbF_{g'})$} (P0)
   (P2) edge[->] node[above,yshift=5pt] {$\bbDelta_{g,g'}$} (P1)
   (P2) edge[->] node[left,xshift=-10pt] {$\bbF_{g+g'}$} (P3)
   (P3) edge[->] node[below,yshift=-5pt] {$\bbA^Z_{\calC} \left( \overline{\bbP^2_{g,g'}} \right)$} (P0);
  \end{tikzpicture}
 \end{center}
\end{proposition}

\begin{proof}
 The result is established by showing that the degree $0$ morphisms
 \[
  \left( \bbeta^0_{g,g',V,j} \right)_{(i,j) \in \{ 1 \} \times \rmI_{\mathrlap{g'}}} \in \Hom_{\hat{\bbA}^Z_{\calC}(\bbS^1_g \disjun \bbS^1_{g'})} \left( \bigoplus_{i \in \rmI_{g'}} \bbD^2_{V \otimes V_j^*} \disjun \bbD^2_{V_j}, \overline{\bbP^2_{g,g'}} \circ \bbD^2_V \right)
 \]
 are natural with respect to degree $k$ morphisms $f^k \in \bbHom_{\calC}(V,V')$, and that they are invertible for all objects $V \in \bbProj(\calC_{g + g'})$. For what concerns naturality, we need to show that
 \begin{align*}
  &\calD^{-2} \rmd(V_0)^2 \cdot \left[ \bbW_{V',j} \disjun \bbD^3_0 \right] \diamond \bfmu_{\bbS^1_g,\bbS^1_{g'}} \left( \left[ \bbY_{\bbnabla_{g+g',-g'}(f^k \otimes \id^0_{V_j^*})} \right] \otimes \left[ \bbY_{\id^0_{V_j}} \right] \right) \\
  &\hspace{\parindent} = \calD^{-1} \rmd(V_0) \cdot \left[ \id_{\overline{\bbP^2_{g,g'}}} \circ \bbY_{f^k} \right] \diamond \left[ \bbW_{V,j} \disjun \bbD^3_0 \right]
 \end{align*}
 for every $j \in \rmI_{g'}$ and every degree $k$ morphism $f^k \in \bbHom_{\calC}(V,V')$.
 This follows directly from Lemma \ref{L:unitarity_of_3-discs}. Next, for what concerns invertibility, we claim that, for every object $V \in \Proj(\calC_{g+g'})$, the inverse of the degree $0$ morphism
 \[
  \left( \left[ \bbW_{V,j} \disjun \bbD^3_0 \right] \right)_{(i,j) \in \{ 1 \} \times \rmI_{g'}}
 \]
 of 
 \[
  \Hom_{\hat{\bbA}^Z_{\calC}(\bbS^1_g \disjun \bbS^1_{g'})} \left( \bigoplus_{i \in \rmI_{g'}} \bbD^2_{V \otimes V_i^*} \disjun \bbD^2_{V_i}, \overline{\bbP^2_{g,g'}} \circ \bbD^2_V \right)
 \]
 is given by the degree $0$ morphism 
 \[
  \calD^{-1} \cdot
  \left( \rmd(V_i) \cdot \left[ \overline{\bbW_{V,i}} \disjun \bbD^3_0 \right] \right)_{(i,j) \in \rmI_{g'} \times \{ 1 \}}
 \]
 of 
 \[
  \Hom_{\hat{\bbA}^Z_{\calC}(\bbS^1_g \disjun \bbS^1_{g'})} \left( 
  \overline{\bbP^2_{g,g'}} \circ \bbD^2_V,
  \bigoplus_{i \in \rmI_{g'}} \bbD^2_{V \otimes V_i^*} \disjun \bbD^2_{V_i} \right),
 \] 
 where the $2$-morphism
\[
 \overline{\bbW_{V,i}} : \overline{\bbP^2_{g,g'}} \circ \bbD^2_V \Rightarrow \bbD^2_{V \otimes V_i^*} \disjun \bbD^2_{V_i}
\]
 of $\bfadCob_{\calC}$ is given by $(\overline{W},\overline{T_{V,i}},\omega_{V,i},0)$ for the $\calC$-colored ribbon graph $\overline{T_{V,i}} \subset \overline{W}$ obtained from $T_{V,i}$ by reversing the orientation of all edges and vertical boundaries of coupons.
 Indeed, on the one hand the equality
 \[
  \calD^{-1} \rmd(V_i) \cdot \left[ \overline{\bbW_{V,i}} \disjun \bbD^3_0 \right] \diamond \left[ \bbW_{V,j} \disjun \bbD^3_0 \right] 
  = \delta_{ij} \cdot \left[ \id_{\bbD^2_{V \otimes V_i^*}} \disjun \id_{\bbD^2_{V_i}} \disjun \bbD^3_0 \right]
 \]
 follows directly from Lemmas \ref{L:surgery_axioms} and \ref{L:2-sphere}. 
 \begin{figure}[t]\label{F:coproduct_generic_composition}
  \centering
  \includegraphics{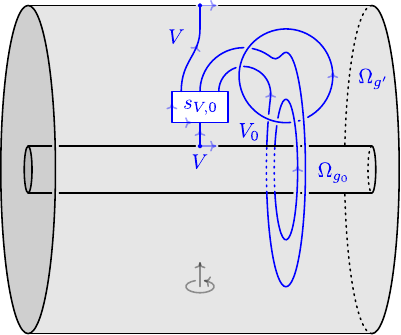}
  \caption{The $2$-mor\-phism $( \bbI \times \bbS^1 \times \bbI )_{g',V}$ of $\bfadCob_{\calC}$.}
 \end{figure}
 On the other hand, we have
 \[
 \sum_{h \in \rmI_{g'}} \rmd(V_h) \cdot \left[ \bbW_{V,h} \disjun \bbD^3_0 \right] \diamond \left[ \overline{\bbW_{V,h}} \disjun \bbD^3_0 \right] = \calD^{-1} \cdot \left[  ( \bbI \times \bbS^1 \times \bbI )_{g',V} \disjun \bbD^3_0 \right],
 \]
 where the $2$-morphism
\[
 ( \bbI \times \bbS^1 \times \bbI )_{g',V} : \overline{\bbP^2_{g,g'}} \circ \bbD^2_V \Rightarrow \overline{\bbP^2_{g,g'}} \circ \bbD^2_V
\]
 of $\bfadCob_{\calC}$ is given by $(I \times S^1 \times I,T_{g',V},\omega_{g',V},0)$ for the $\calC$-colored ribbon graph $T_{g',V} \subset I \times S^1 \times I$ represented in Figure \ref{F:coproduct_generic_composition}, and for the unique compatible $G$-coloring $\omega_{g',V}$ of $(I \times S^1 \times I,T_{g',V})$.
 Then, thanks to the relative modularity condition for $\zeta = \calD^2$, and thanks to the definition of the section $s_{V,0}$, we get
 \[
  \calD^{-1} \cdot \sum_{h \in \rmI_{g'}} \rmd(V_h) \cdot \left[ \bbW_{V,h} \disjun \bbD^3_0 \right] \diamond \left[ \overline{\bbW_{V,h}} \disjun \bbD^3_0 \right] =  \left[ \id^0_{\overline{\bbP^2_{g,g'}} \circ \bbD^2_V} \right]. \qedhere
 \]
\end{proof}

We move on to study the critical case. The strategy is the same as before, but this time the description is more complicated, because dominating sets for homogeneous subcategories of projective objects of critical degree are not completely reduced. This means we need some preparation, so let us start by fixing some notation. For all critical $x,x' \in X$, for every object $V \in \bbProj(\calC_{x+x'})$, and for every $(i,j) \in \rmI_{g_{x'} + x'}^2$, we denote with $T_{Vij}$ the $\calC$-colored ribbon graph represented in Figure \ref{F:uccidetemi}, and we set
\[
 p^0_{Vij} := \calD^{-2} \rmd(V_i) \cdot F_{\calC} \left( T_{Vij} \right) \otimes \varepsilon.
\]
Next, let us establish a key technical result for our characterization.

\begin{figure}[hbt]\label{F:uccidetemi}
 \centering
 \includegraphics{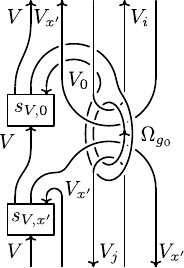}
 \caption{The $\calC$-colored ribbon graph $T_{Vij}$.}
\end{figure}

\begin{lemma}\label{L:uccidetemi}
 The degree $0$ morphisms 
 \[ 
  p^0_{Vij} \in \bbHom_{\calC}(V \otimes V_{x'} \otimes V_j^* \otimes V_j \otimes V_{x'}^*,V \otimes V_{x'} \otimes V_i^* \otimes V_i \otimes V_{x'}^*)
 \]
 satisfy
 \[
   \sum_{h \in \rmI_{g_{x'} + x'}} p^0_{Vih} \diamond p^0_{Vhj} = p^0_{Vij},
 \]
 and there exist some integer $N_{ij}$, some degree $k_n$ morphism 
 \[
  p^{k_n}_{Vijn} \in \bbHom_{\calC}(V \otimes V_{x'} \otimes V_j^*,V \otimes V_{x'} \otimes V_i^*)
 \]
 and some degree $-k_n$ morphism
 \[
  p'^{-k_n}_{ijn} \in \bbHom_{\calC}(V_j \otimes V_{x'}^*,V_i \otimes V_{x'}^*)
 \]
 for some $k_n \in \PGr$ and for every integer $1 \leqslant n \leqslant N_{ij}$, which satisfy
 \[
  p^0_{Vij} = \sum_{n=1}^{N_{ij}} \bbnabla_{x,x'}(p^{k_n}_{Vijn} \otimes p'^{-k_n}_{ijn}).
 \]
\end{lemma}

\begin{proof}
 If we denote with $T'_{Vij}$ the $\calC$-colored ribbon graph represented in Figure \ref{F:uccidetemi_composition}, then we have
 \[
  \sum_{h \in \rmI_{g_{x'} + x'}} p^0_{Vih} \diamond p^0_{Vhj} = \calD^{-4} \rmd(V_i) \cdot F_{\calC} \left( T'_{Vij} \right) \otimes \varepsilon.
 \]
 \begin{figure}[t]\label{F:uccidetemi_composition}
 \centering
 \includegraphics{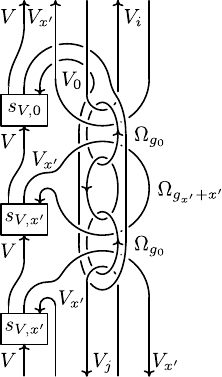}
 \caption{The $\calC$-colored ribbon graph $T'_{Vij}$.}
\end{figure}
 The relative modularity condition applied to the $\Omega_{g_{x'} + x'}$-colored component, together with the definition of the section $s_{V,x'}$ of $\id_V \otimes \rev_{V_{x'}}$, gives the desired equality. Furthermore, the fact that $p^0_{Vij}$ belongs to the image of $\bbnabla_{x,x'}$ follows directly from Remark \ref{R:critical_fusion}, which allows us to decompose 
 \[
  \id_{V_j} \otimes \id_{V_{x'}^*} \otimes \id_{V_{x'}} \otimes \id_{V_i^*} = \sum_{n=1}^{N_{ij}} r_n \circ s_n
 \]
 for some morphisms 
 \begin{align*}
  s_n \in \Hom_{\calC}(V_j \otimes V_{x'}^* \otimes V_{x'} \otimes V_i^*,V_{h_n} \otimes \sigma(k_n) \otimes V_0^*), \\
  r_n \in \Hom_{\calC}(V_{h_n} \otimes \sigma(k_n) \otimes V_0^*,V_j \otimes V_{x'}^* \otimes V_{x'} \otimes V_i^*).
 \end{align*}
  Then, the relative modularity condition applied to the $\Omega_{g_0}$-colored knot of Figure \ref{F:uccidetemi} gives the desired property.
\end{proof}

For all critical indices $x,x'$ in $X$ let
\[
 \bbDelta_{x,x'} : \bbProj(\calC_{x+x'}) \rightarrow \bbProj(\calC_x) \csqtimes \bbProj(\calC_{x'})
\]
be the $\PGr$-graded linear functor sending every object $V$ of $\bbProj(\calC_{x+x'})$ to the object
\[
 \im \left( \underline{p^0_{Vij}} \right)_{(i,j) \in \rmI_{g_{x'} + x'}^2}
\]
of $\bbProj(\calC_x) \csqtimes \bbProj(\calC_{x'})$ given by 
\[
 \underline{p^0_{Vij}} := \sum_{n=1}^{N_{Vij}} p^{k_n}_{Vijn} \otimes p'^{-k_n}_{ijn}
\]
and every degree $k$ morphism $f^k$ of $\bbHom_{\calC}(V,V')$ to the degree $k$ morphism 
\[
 \left( \underline{f^k_{ij}} \right)_{(i,j) \in \rmI_{g_{x'} + x'}^2} 
\]
of
\[
 \bbHom_{\calC \csqtimes \calC} \left( \im \left( \underline{p^0_{Vij}} \right)_{(i,j) \in \rmI_{g_{x'} + x'}^2},\im \left( \underline{p^0_{V'ij}} \right)_{(i,j) \in \rmI_{g_{x'} + x'}^2} \right)
\]
given by
\[
 \underline{f^k_{ij}} := \sum_{h \in \rmI_{g_{x'} + x'}}\underline{p^0_{V'ih}} \diamond \left( \bbnabla_{x+x',-x'} \left( f^k \otimes \id^0_{V_{x'} \otimes V_h^*} \right) \otimes \id^0_{V_h \otimes V_{x'}^*} \right) \diamond \underline{p^0_{Vhj}}
\]
Next, for every object $V \in \bbProj(\calC_{x+x'})$ and every $j \in \rmI_{g_{x'} + x'}$, we consider the degree $0$ morphism
\[
 \bbeta^0_{x,x',V,j} \in \Hom_{\bbA^Z_{\calC}(\bbS^1_x \disjun \bbS^1_{x'})} \left( \bbD^2_{V \otimes V_{x'} \otimes V_j^*} \disjun \bbD^2_{V_j \otimes V_{x'}^*}, \overline{\bbP^2_{x,x'}} \circ \bbD^2_V \right)
\]
given by
\[
 \calD^{-2} \rmd(V_0)^2 \cdot \sum_{i \in \rmI_{g_{x'} + x'}} \left[ \bbW_{x',V,i} \disjun \bbD^3_0 \right] \diamond \bfmu_{\bbS^1_x,\bbS^1_{x'}} \left( \left[ \bbY_{\underline{p^0_{Vij}}} \right] \right)
\]
where 
\[
 \calD^{-2} \rmd(V_0)^2 \cdot \left[ \bbY_{\underline{p^0_{Vij}}} \right] := \left( \bbF_x \sqtimes \bbF_{x'} \right) \left( \underline{p^0_{Vij}} \right),
\]
and where the $2$-morphism
\[
 \bbW_{x',V,i} : \bbD^2_{V \otimes V_{x'} \otimes V_i^*} \disjun \bbD^2_{V_i \otimes V_{x'}^*} \Rightarrow \overline{\bbP^2_{x,x'}} \circ \bbD^2_V
\]
of $\bfadCob_{\calC}$ is given by $(W,T_{x' \! ,V,i},\omega_{x' \! ,V,i},0)$ for the $\calC$-colored ribbon graph ${T_{x' \! ,V,i} \subset W}$ represented in Figure \ref{F:coproduct_critical}, where
\[
 s_{x',V,i} := \left( ( \id_V \otimes \id_{V_{x'}} \otimes \lev_{V_{x'}} ) \circ ( s_{V,x'} \otimes \id_{V_{x'}} ) \right) \otimes \id_{V_i^*},
\]
and for the unique compatible $G$-coloring $\omega_{x',V,i}$ of $(W,T_{x',V,i})$.

\begin{figure}[hbt]\label{F:coproduct_critical}
 \centering
 \includegraphics{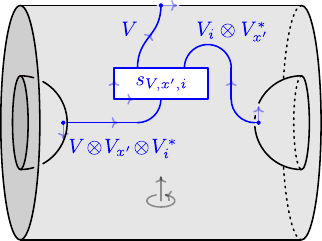}
 \caption{The $2$-morphism $\bbW_{x',V,i}$ of $\bfadCob_{\calC}$.} 
\end{figure}

\begin{proposition}\label{P:coproduct_critical}
 For all $x,x' \in X$ the degree $0$ morphisms $\bbeta^0_{x,x',V,i}$ define a $\PGr$-graded natural isomorphism
 \begin{center}
  \begin{tikzpicture}[descr/.style={fill=white}]
   \node (P0) at ({atan(-0.5)}:{3.25/(2*sin(atan(0.5)))}) {$\hat{\bbA}^Z_{\calC}(\bbS^1_x \disjun \bbS^1_{x'})$};
   \node (P1) at (90:{3.25/2}) {$\bbProj(\calC_x) \csqtimes \bbProj(\calC_{x'})$};
   \node (P2) at ({atan(-0.5) + 2*90}:{3.25/(2*sin(atan(0.5)))}) {$\bbProj(\calC_{x+x'})$};
   \node (P3) at (3*90:{3.25/2}) {$\bbA^Z_{\calC}(\bbS^1_{x+x'})$};
   \node[above] at (0,0) {$\Downarrow$};
   \node[below] at (0,0) {$\bbeta_{x,x'}$};
   \draw
   (P1) edge[->] node[right,xshift=10pt] {$\hat{\bfmu}_{\bbS^1_x,\bbS^1_{x'}} \circ (\bbF_x \csqtimes \bbF_{x'})$} (P0)
   (P2) edge[->] node[above,yshift=5pt] {$\bbDelta_{x,x'}$} (P1)
   (P2) edge[->] node[left,xshift=-10pt] {$\bbF_{x+x'}$} (P3)
   (P3) edge[->] node[below,yshift=-5pt] {$\bbA^Z_{\calC} \left( \overline{\bbP^2_{x,x'}} \right)$} (P0);
  \end{tikzpicture}
 \end{center}
\end{proposition}

\begin{proof}
 The result is established by showing that the degree $0$ morphisms
 \[
  \left( \bbeta^0_{x,x',V,j} \right)_{(i,j) \in \{ 1 \} \times \rmI_{g_{x'} + x'}}
 \]
 of
 \[
  \Hom_{\hat{\bbA}^Z_{\calC}(\bbS^1_x \disjun \bbS^1_{x'})} \left( \im \left( \calD^{-2} \rmd(V_0)^2 \cdot \left( \bfmu_{\bbS^1_x,\bbS^1_{x'}} \left( \left[ \bbY_{\underline{p^0_{Vij}}} \right] \right) \right)_{(i,j) \in \rmI^2_{g_{x'} + x'}} \right), \overline{\bbP^2_{x,x'}} \circ \bbD^2_V \right)
 \]
 are natural with respect to degree $k$ morphisms $f^k \in \bbHom_{\calC}(V,V')$, and that they are invertible for all objects $V \in \bbProj(\calC_{x + x'})$.
 For what concerns naturality, we need to show that
 \begin{align*}
  &\calD^{-2} \rmd(V_0)^2 \cdot \sum_{i \in \rmI_{g_{x'} + x'}} \left[ \bbW_{x',V',i} \disjun \bbD^3_0 \right] \diamond \bfmu_{\bbS^1_x,\bbS^1_{x'}} \left( \left[ \bbY_{\underline{f^k_{ij}}} \right] \right) \\
  &\hspace{\parindent} = \calD^{-3} \rmd(V_0)^3 \cdot \sum_{i \in \rmI_{g_{x'} + x'}} \left[ \id_{\overline{\bbP^2_{x,x'}}} \circ \bbY_{f^k} \right] \diamond \left[ \bbW_{x',V,i} \disjun \bbD^3_0 \right] \diamond \bfmu_{\bbS^1_x,\bbS^1_{x'}} \left( \left[ \bbY_{\underline{p^0_{Vij}}} \right] \right)
 \end{align*}
 for every $j \in \rmI_{g'}$ and every degree $k$ morphism $f^k \in \bbHom_{\calC}(V,V')$, where
 \[
  \calD^{-2} \rmd(V_0)^2 \cdot \left[ \bbY_{\underline{f^k_{ij}}} \right] := \left( \bbF_x \sqtimes \bbF_{x'} \right) \left( \underline{f^k_{ij}} \right).
 \]
 This follows directly from Lemma \ref{L:unitarity_of_3-discs}. Next, for what concerns invertibility, we claim that, for every object $V \in \Proj(\calC_{x + x'})$, the inverse of the degree $0$ morphism
 \[
  \calD^{-2} \rmd(V_0)^2 \cdot \left( \sum_{h \in \rmI_{g_{x'} + x'}} \left[ \bbW_{x',V,h} \disjun \bbD^3_0 \right] \diamond \bfmu_{\bbS^1_x,\bbS^1_{x'}} \left( \left[ \bbY_{\underline{p^0_{Vhj}}} \right] \right) \right)_{(i,j) \in \{ 1 \} \times \rmI_{g_{x'} + x'}}
 \]
 of 
 \[
  \Hom_{\hat{\bbA}^Z_{\calC}(\bbS^1_x \disjun \bbS^1_{x'})} \left( \im \left( \calD^{-2} \rmd(V_0)^2 \cdot \left( \bfmu_{\bbS^1_x,\bbS^1_{x'}} \left( \left[ \bbY_{\underline{p^0_{Vij}}} \right] \right) \right)_{(i,j) \in \rmI^2_{g_{x'} + x'}} \right), \overline{\bbP^2_{x,x'}} \circ \bbD^2_V \right)
 \]
 is given by the degree $0$ morphism 
 \[
  \calD^{-3} \rmd(V_0)^2 \cdot
  \left( \sum_{h \in \rmI_{g_{x'} + x'}} \rmd(V_h) \cdot \bfmu_{\bbS^1_x,\bbS^1_{x'}} \left( \left[ \bbY_{\underline{p^0_{Vih}}} \right] \right) \diamond \left[ \overline{\bbW_{x',V,h}} \disjun \bbD^3_0 \right] \right)_{(i,j) \in \rmI_{g_{x'} + x'} \times \{ 1 \}}
 \]
 of 
 \[
  \Hom_{\hat{\bbA}^Z_{\calC}(\bbS^1_x \disjun \bbS^1_{x'})} \left( \overline{\bbP^2_{x,x'}} \circ \bbD^2_V, \im \left( \calD^{-3} \rmd(V_0)^2 \cdot \left( \bfmu_{\bbS^1_x,\bbS^1_{x'}} \left( \left[ \bbY_{\underline{p^0_{Vij}}} \right] \right) \right)_{(i,j) \in \rmI^2_{g_{x'} + x'}} \right) \right),
 \] 
 where the $2$-morphism
\[
 \overline{\bbW_{x',V,i}} : \overline{\bbP^2_{x,x'}} \circ \bbD^2_V \Rightarrow \bbD^2_{V \otimes V_{x'} \otimes V_i^*} \disjun \bbD^2_{V_i \otimes V_{x'}^*}
\]
 of $\bfadCob_{\calC}$ is given by $(\overline{W},\overline{T_{x\mathrlap{'},V,i}},\omega_{x\mathrlap{'},V,i},0)$ for the $\calC$-colored ribbon graph ${\overline{T_{x\mathrlap{'},V,i}} \subset \overline{W}}$ obtained from $T_{x\mathrlap{'},V,i}$ by reversing the orientation of all edges and vertical boundaries of coupons, and by changing the color of the latter to $\id_{V \otimes V_{x'} \otimes V_i^*}$. Indeed, on the one hand the equality
 \[
  \calD^{-1} \rmd(V_i) \cdot \left[ \overline{\bbW_{x',V,i}} \disjun \bbD^3_0 \right] \diamond \left[ \bbW_{x',V,j} \disjun \bbD^3_0 \right] 
  = \calD^{-2} \rmd(V_0)^2  \cdot \bfmu_{\bbS^1_x,\bbS^1_{x'}} \left( \left[ \bbY_{\underline{p^0_{Vij}}} \right] \right)
 \]
 follows directly from Lemmas \ref{L:surgery_axioms}, \ref{L:2-sphere}, and \ref{L:uccidetemi}. On the other hand, we have
 \begin{align*}
  &\calD^{-2} \rmd(V_0)^2 \cdot \! \sum_{h,h' \in \rmI_{g_{x'} + x'}} \! \rmd(V_h) \cdot \left[ \bbW_{x',V,h'} \disjun \bbD^3_0 \right] \diamond \bfmu_{\bbS^1_x,\bbS^1_{x'}} \left( \left[ \bbY_{\underline{p^0_{Vh'h}}} \right] \right) \diamond \left[ \overline{\bbW_{x',V,h}} \disjun \bbD^3_0 \right] \\
  &\hspace{\parindent} = \calD^{-3} \cdot \left[  ( \bbI \times \bbS^1 \times \bbI )_{x',V} \disjun \bbD^3_0 \right],
 \end{align*}
 where the $2$-morphism
\[
 ( \bbI \times \bbS^1 \times \bbI )_{x',V} : \overline{\bbP^2_{x,x'}} \circ \bbD^2_V \Rightarrow \overline{\bbP^2_{x,x'}} \circ \bbD^2_V
\]
 of $\bfadCob_{\calC}$ is given by $(I \times S^1 \times I,T_{x',V},\omega_{x',V},0)$ for the $\calC$-colored ribbon graph $T_{x',V} \subset I \times S^1 \times I$ represented in Figure \ref{F:coproduct_critical_composition}, and for the unique compatible $G$-coloring $\omega_{x',V}$ of $(I \times S^1 \times I,T_{x',V})$.
 \begin{figure}[bt]\label{F:coproduct_critical_composition}
  \centering
  \includegraphics{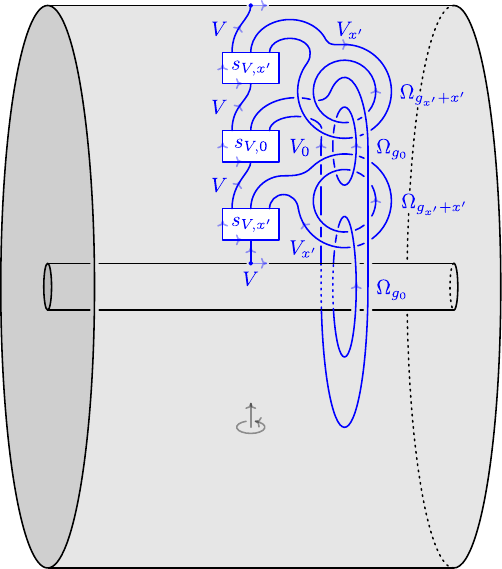}
  \caption{The $2$-mor\-phism $( \bbI \times \bbS^1 \times \bbI )_{x',V}$ of $\bfadCob_{\calC}$.}
 \end{figure}
 Then, using twice the relative modularity condition with $\zeta = \calD^2$, we get
 \begin{align*}
  &\calD^{-3} \rmd(V_0)^2 \cdot \! \sum_{h,h' \in \rmI_{g_{x'} + x'}} \! \rmd(V_h) \cdot \left[ \bbW_{x',V,h'} \disjun \bbD^3_0 \right] \diamond \bfmu_{\bbS^1_x,\bbS^1_{x'}} \left( \left[ \bbY_{\underline{p^0_{Vh'h}}} \right] \right) \diamond \left[ \overline{\bbW_{x',V,h}} \disjun \bbD^3_0 \right] \\
  &\hspace{\parindent} = \left[ \id^0_{\overline{\bbP^2_{x,x'}} \circ \bbD^2_V} \right].
 \end{align*}
\end{proof}

\section{2-Cylinders}\label{S:2-cylinders}

In this section we focus on a very simple family of $1$-morphisms of $\bfadCob_{\calC}$. In order to introduce it, let us consider $g,h \in G$.

\begin{definition}\label{D:2-cylinder}
 The \textit{$(g,h)$-colored $2$-cylinder} $(\bbI \times \bbS^1)_{g,h} : \bbS^1_g \rightarrow \bbS^1_g$ is the $1$-mor\-phism of $\bfadCob_{\calC}$ given by 
 \[
  \left( I \times S^1,\varnothing,\vartheta_{g,h},H_1(I \times S^1;\R) \right)
 \]
 where $\vartheta_{g,h}$ is the unique $G$-coloring of $(I \times S^1,\varnothing)$ extending $\xi_g$ which satisfies $\langle \vartheta_{g,h}, I \times A \rangle = h$ for the relative homology class $I \times A$ joining the base points of the two copies of $S^1$, as represented in Figure \ref{F:2-cylinder}.
\end{definition}

\begin{figure}[hbt]\label{F:2-cylinder}
 \centering
 \includegraphics{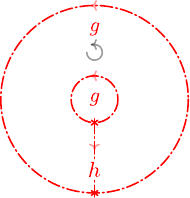}
 \caption{The $1$-mor\-phism $(\bbI \times \bbS^1)_{g,h}$ of $\bfadCob_{\calC}$.}
\end{figure}

The family of universal $\PGr$-graded linear functors we are going to discuss now can be described in terms of the bilinear map which governs the compatibility between the $G$-structure and the $\PGr$-action. First, for all $g,h \in G$, let
\[
 \bbI_{g,h} : \bbProj(\calC_g) \rightarrow \bbProj(\calC_g)
\]
be the $\PGr$-graded linear functor given by the identity on objects of $\bbProj(\calC_g)$, and by a rescaling map of a factor $\psi(h,-k)$ on degree $k$ morphisms of $\bbProj(\calC_g)$. Next, for all $g,h \in G$ and every object $V$ of $\bbProj(\calC_g)$, we consider the degree $0$ morphism
\[
 \bbeta_{g,h,V}^0 \in \Hom_{\bbA^Z_{\calC}(\bbS^1_g)} \left( \bbD^2_V, \left( \bbI \times \bbS^1 \right)_{g,h} \circ \bbD^2_V \right)
\] 
given by $[ ( \bbD^2 \times \bbI)_{V,h} ) \disjun \bbD^3_0 ]$, where the $2$-mor\-phism 
\[
 \left( \bbD^2 \times \bbI \right)_{V,h} : \bbD^2_V \Rightarrow \left( \bbI \times \bbS^1 \right)_{g,h} \circ \bbD^2_V
\] 
of $\bfadCob_{\calC}$ is given by $(D^2 \times I,P_V \times I,\omega_{V,h},0)$ for the $(\calC,G)$-coloring $P_V \times I,\omega_{V,h}$ of $D^2 \times I$ represented in Figure \ref{F:punctured_3-cylinder_lemma_1}.

\begin{proposition}\label{P:2-cylinder_isomorphism}
 For all $g,h \in G$ the degree $0$ morphisms $\bbeta_{g,h,V}^0$ define a $\PGr$-graded natural isomorphism
 \begin{center}
 	\begin{tikzpicture}[descr/.style={fill=white}]
 	\node (P0) at ({atan(-0.5)}:{3.25/(2*sin(atan(0.5)))}) {$\bbA^Z_{\calC}(\bbS^1_g)$};
 	\node (P1) at (90:{3.25/2}) {$\bbProj(\calC_g)$};
 	\node (P2) at ({atan(-0.5) + 2*90}:{3.25/(2*sin(atan(0.5)))}) {$\bbProj(\calC_g)$};
 	\node (P3) at (3*90:{3.25/2}) {$\bbA^Z_{\calC}(\bbS^1_g)$};
 	\node[above] at (0,0) {$\Downarrow$};
 	\node[below] at (0,0) {$\bbeta_{g,h}$};
 	\draw
 	(P1) edge[->] node[right,xshift=10pt] {$\bbF_g$} (P0)
 	(P2) edge[->] node[above,yshift=5pt] {$\bbI_{g,h}$} (P1)
 	(P2) edge[->] node[left,xshift=-10pt] {$\bbF_g$} (P3)
 	(P3) edge[->] node[below,yshift=-5pt] {$\bbA^Z_{\calC} \left( \left( \bbI \times \bbS^1 \right)_{g,h} \right)$} (P0);
 	\end{tikzpicture}
 \end{center}
\end{proposition}

\begin{proof}
 The result is established by showing that the degree $0$ morphisms
 \[
  \bbeta^0_{g,h,V} \in \Hom_{\bbA^Z_{\calC}(\bbS^1_g)} \left( \bbD^2_V,\left( \bbI \times \bbS^1 \right)_{g,h} \circ \bbD^2_V \right)
 \]
 are natural with respect to degree $k$ morphisms $f^k \in \bbHom_{\calC}(V,V')$, and that they are invertible for all objects $V \in \bbProj(\calC_g)$. For what concerns naturality, we need to show that
 \[
  \left[ \id_{\left( \bbI \times \bbS^1 \right)_{g,h}} \circ \bbY_{f^k} \right] \diamond \left[ (\bbD^2 \times \bbI)_{V,h} \disjun \bbD^3_0 \right] = \psi(h,-k) \cdot \left[ \left( \bbD^2 \times \bbI \right)_{V',h} \disjun \bbD^3_0 \right] \diamond \left[ \bbY_{f^k} \right]
 \]
 for every degree $k$ morphism $f^k \in \bbHom_{\calC}(V,V')$. In order to do this, we can appeal to Lemma \ref{L:triviality_lemma}, so let us specify $\smash{\overline{D^2}}$ as a $2$-di\-men\-sion\-al cobordism from $S^1$ to $\varnothing$, let us specify $D^3$ as a $3$-di\-men\-sion\-al cobordism from $\varnothing$ to $D^2 \cup_{S^1} \overline{D^2}$, and let us specify $S^2 \times I$ as a $3$-di\-men\-sion\-al cobordism from $(D^2 \cup_{S^1} (I \times S^1) \cup_{S^1} \overline{D^2}) \sqcup \overline{S^2}$ to $\varnothing$. Now the identity can be tested by considering all $V'' \in \calC_g$ and all
 \begin{gather*}
 	(T,\omega) \in \adSk \left( D^3;\id_{\varnothing},\overline{\bbD^2_{V''}} \circ \bbD^2_V \right), \\
 	(T',\omega') \in \adSk \left( S^2 \times I; \left( \overline{\bbD^2_{V''}} \circ \left( \bbI \times \bbS^1 \right)_{g,h} \circ \bbD^2_{V'} \right) \disjun \bbS^2_{-k},\id_{\varnothing} \right),
 \end{gather*}
 and by computing the Costantino-Geer-Patureau invariant $\CGP_{\calC}$ of the resulting closed $2$-mor\-phisms of $\bfadCob_{\calC}$. We can choose a surgery presentation composed of a single unknot with framing zero, whose computability can always be forced by performing generic or projective stabilization. Therefore, up to isotopy, skein equivalence, and multiplication by invertible scalars, we only need to check that 
 \[
  \rmt_V \left( F_{\calC}(T_{f,f',h+h'}) \right) = \psi(h,-k) \rmt_V \left( F_{\calC}(T_{f,f',h'}) \right)
 \]
 for all
 \[
  f' \in \Hom_{\calC} \left( V' \otimes V_0 \otimes \sigma(-k) \otimes V_0^*,V \right), \quad
 h' \in \left( \{ -h \} + (G \smallsetminus X) \right) \cap (G \smallsetminus X),
 \]
 where the $\calC$-colored ribbon graph $T_{f,f',h + h'}$ is represented in the left-hand part of Figure \ref{F:punctured_3-cylinder_lemma_4}. That same skein equivalence gives the identity we are looking for. Therefore, we only need to check invertibility, which follows directly from the definition.
\end{proof}

We are missing one last family of generating $1$-morphisms of $\bfadCob_{\calC}$ from our description: indeed,
for every generic $g \in G \smallsetminus X$, the $g$-colored $1$-sphere $\bbS^1_g$ is a dualizable object of $\bfadCob_{\calC}$, whose dual object is given by the $-g$-colored $1$-sphere $\bbS^2_{-g}$, with right evaluation and coevaluation given by the $(g,0)$-colored $2$-cylinders $\overline{\bbD^2} \circ \bbP^2_{g,-g} : \varnothing \rightarrow \bbS^1_g \disjun \bbS^1_{-g}$ and $\overline{\bbP^2_{-g,g}} \circ \bbD^2 : \varnothing \rightarrow \bbS^1_{-g} \disjun \bbS^1_g$. However, the latter
cannot be obtained from the previous $1$-morphisms, because $\bbD^2$ is not admissible. Fortunately, no additional work is required in order to figure out the universal $\PGr$-graded linear functor $\bbA^Z_{\calC}(\overline{\bbP^2_{-g,g}} \circ \bbD^2)$. Indeed, for every generic $g \in G \smallsetminus X$, let 
\[
 \bbDelta_{-g,g}(\one) : \Bbbk \rightarrow \bbProj(\calC_{-g}) \csqtimes \bbProj(\calC_g)
\]
be the $\PGr$-graded linear functor sending the unique object of $\Bbbk$ to the object
\[
 \bigoplus_{i \in \rmI_g} (V_i^*,V_i) 
\]
of $\bbProj(\calC_{-g}) \csqtimes \bbProj(\calC_g)$. Then, for every $i \in \rmI_g$, we consider the degree $0$ morphism
\[
\bbeta^0_{-g,g,i} \in \Hom_{\bbA^Z_{\calC}(\bbS^1_{-g} \disjun \bbS^1_g)} \left( \bbD^2_{V_i^*} \disjun \bbD^2_{V_i}, \overline{\bbP^2_{-g,g}} \circ \bbD^2 \right)
\]
given by $[\bbW_i \disjun \bbD^3_0]$, where the $2$-morphism
\[
\bbW_i : \bbD^2_{V_i^*} \disjun \bbD^2_{V_i} \Rightarrow \overline{\bbP^2_{-g,g}} \circ \bbD^2
\]
of $\bfadCob_{\calC}$ is given by $(W,T_i,\omega_i,0)$ for the $\calC$-colored ribbon graph $T_i \subset W$ represented in Figure \ref{F:counit_generic}, and for the unique compatible $G$-coloring $\omega_i$ of $(W,T_i)$.

\begin{figure}[hbt]\label{F:counit_generic}
  \centering
  \includegraphics{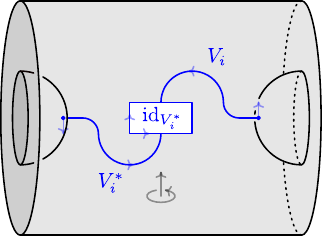}
  \caption{The $2$-morphism $\bbW_i$ of $\bfadCob_{\calC}$.} 
\end{figure}

Then, for every generic $g \in G \smallsetminus X$ we have a $\PGr$-graded natural isomorphism
\begin{center}
	\begin{tikzpicture}[descr/.style={fill=white}]
	\node (P0) at ({atan(-0.5)}:{3.25/(2*sin(atan(0.5)))}) {$\hat{\bbA}^Z_{\calC}(\bbS^1_{-g} \disjun \bbS^1_g)$};
	\node (P1) at (90:{3.25/2}) {$\bbProj(\calC_{-g}) \csqtimes \bbProj(\calC_g)$};
	\node (P2) at ({atan(-0.5) + 2*90}:{3.25/(2*sin(atan(0.5)))}) {$\Bbbk$};
	\node (P3) at ({atan(0.5) + 2*90}:{3.25/(2*sin(atan(0.5)))}) {$\bbA^Z_{\calC}(\varnothing)$};
	\node[above] at (0,0) {$\Downarrow$};
	\node[below] at (0,0) {$\bbeta_{-g,g}$};
	\draw
	(P1) edge[->] node[right,xshift=10pt] {$\hat{\bfmu}_{\bbS^1_{-g},\bbS^1_g} \circ (\bbF_{-g} \csqtimes \bbF_g)$} (P0)
	(P2) edge[->] node[above,yshift=5pt] {$\bbDelta_{-g,g}(\one)$} (P1)
	(P2) edge[->] node[left,xshift=-10pt] {$\bfepsilon$} (P3)
	(P3) edge[->] node[below,yshift=-5pt] {$\bbA^Z_{\calC} \left( \overline{\bbP^2_{-g,g}} \circ \bbD^2 \right)$} (P0);
	\end{tikzpicture}
\end{center}
Remark that the proof of this claim can be adapted directly from the proof of Proposition \ref{P:coproduct_generic}. The inverse of the degree $0$ morphism
\[
 \left( \left[ \bbW_j \disjun \bbD^3_0 \right] \right)_{(i,j) \in \{ 1 \} \times \rmI_g}
\]
of 
\[
\Hom_{\hat{\bbA}^Z_{\calC}(\bbS^1_{-g} \disjun \bbS^1_g)} \left( \bigoplus_{i \in \rmI_g} \bbD^2_{V_i^*} \disjun \bbD^2_{V_i}, \overline{\bbP^2_{-g,g}} \circ \bbD^2 \right)
\]
is given by the degree $0$ morphism 
\[
 \calD^{-1} \cdot
\left( \rmd(V_i) \cdot \left[ \overline{\bbW_i} \disjun \bbD^3_0 \right] \right)_{(i,j) \in \rmI_g \times \{ 1 \}}
\]
of 
\[
\Hom_{\hat{\bbA}^Z_{\calC}(\bbS^1_{-g} \disjun \bbS^1_g)} \left( 
\overline{\bbP^2_{-g,g}} \circ \bbD^2,
\bigoplus_{i \in \rmI_g} \bbD^2_{V_i^*} \disjun \bbD^2_{V_i} \right),
\] 
where the $2$-morphism
\[
\overline{\bbW_i} : \overline{\bbP^2_{-g,g}} \circ \bbD^2 \Rightarrow \bbD^2_{V_i^*} \disjun \bbD^2_{V_i}
\]
of $\bfadCob_{\calC}$ is given by $(\overline{W},\overline{T_i},\omega_i,0)$ for the $\calC$-colored ribbon graph $\overline{T_i} \subset \overline{W}$ obtained from $T_i$ by reversing the orientation of all edges and vertical boundaries of coupons. However this time, in order to obtain a computable surgery presentation of $\overline{W} \cup_{D^2 \sqcup D^2} W$ inside $I \times S^1 \times I$ with respect to the $(\calC,G)$-coloring 
\[
 \left( \overline{T_i} \cup_{P_{V_i^*} \sqcup P_{V_i}} T_i, \omega_i \cup_{\vartheta_{V_i^*} \sqcup \vartheta_{V_i}} \omega_i \right),
\] 
we first have to perform a generic stabilization of index $g$ along the generic curve $\{ \frac 12 \} \times S^1 \times \{ \frac 12 \}$, then we have to perfrom a projective stabilization of index $g_0$ along one of the resulting projective components, and only then we can slide a $V_0$-colored edge over the surgery knot. The rest of the proof is unchanged.

\section{Examples of computations}

In this section we provide a combinatorial description of universal $\PGr$-graded vector spaces associated with families of closed surfaces.

\begin{remark}
 The genus $0$ case follows immediately from Section \ref{S:2-discs}: indeed, suppose $\bbSigma = (\varSigma,P,0,\{ 0 \})$ is a closed connected $1$-morphism of $\bfadCob_{\calC}$ which admits a positive diffeomorphism $f : S^2 \rightarrow \varSigma$, and suppose $P \subset \varSigma$ is composed of a single positive point of color $F_{\calC}(P) \in \Proj(\calC_0)$. Then Proposition \ref{P:counit} implies
 \[
  \bbV^Z_{\calC}(\bbSigma) \cong \bbHom_{\calC}(\one,F_{\calC}(P)).
 \]
\end{remark}

We move on to the genus $1$ case, so let us begin by fixing our terminology. We say a cohomology class $\vartheta \in H_1(S^1 \times S^1;G)$ is \textit{generic} if the meridian $\{ (1,0) \} \times S^1$ is a generic curve with respect to $\vartheta$, which means its homology class $m$ satisfies $\langle \vartheta,m \rangle \in G \smallsetminus X$. A \textit{generic surface of genus $1$} is a closed connected $1$-morphism $\bbSigma = (\varSigma,\varnothing,\vartheta,\calL)$ of $\bfadCob_{\calC}$ which admits a \textit{generic identification} $f : S^1 \times S^1 \rightarrow \varSigma$, which is a positive diffeomorphism such that $f^*(\vartheta)$ is generic. This condition is equivalent to the existence of a generic curve $\gamma \subset \varSigma$.

\begin{proposition}
 If $\bbSigma = (\varSigma,\varnothing,\vartheta,\calL)$ is a generic surface of genus $1$, and if $f : S^1 \times S^1 \rightarrow \varSigma$ is a generic identification, then 
 \[
  \dim_{\Bbbk} \left( \bbV^Z_{\calC}(\bbSigma) \right) = \left| \rmI_{\langle f^*(\vartheta),m \rangle} \right|.
 \]
\end{proposition}

\begin{proof}
 Let us set
 \[
  g := \langle f^*(\vartheta),m \rangle \in G \smallsetminus X, \quad
  h := \langle f^*(\vartheta),\ell \rangle \in G,
 \]
 where $\ell$ denotes the homology class of the longitude $S^1 \times \{ (0,1) \}$. Then, we can decompose $\bbSigma$, up to isomorphism, as
 \[
  \overline{\bbD^2} \circ \overline{\bbP^2_{-g,g}} \circ \left( (\bbI \times \bbS^1)_{-g,-h} \disjun \id_{\bbS^1_g} \right) \circ \bbP^2_{-g,g} \circ \bbD^2.
 \]
 Therefore, thanks to Propositions \ref{P:counit}, \ref{P:product}, \ref{P:coproduct_generic}, and \ref{P:2-cylinder_isomorphism}, $\bbV^Z_{\calC}(\bbSigma)$ is isomorphic to
 \[
  \bigoplus_{i \in \rmI_g} \bbHom_{\calC} \left( \one,V_i^* \otimes V_i \right).
 \]
 This allows us to conclude.
%
%
\end{proof}

We say a $(\calC,G)$-coloring $(P,\vartheta)$ of $S^1 \times S^1$ is \textit{critical} if the $\calC$-colored ribbon set $P \subset S^1 \times S^1$ is composed of a single positive point of color $F_{\calC}(P) \in \Proj(\calC_0)$, and if the cohomology class $\vartheta \in H^1(S^1 \times S^1 \smallsetminus P;G)$ admits no generic curve. A \textit{critical surface of genus $1$} is a closed connected $1$-morphism $\bbSigma = (\varSigma,P,\vartheta,\calL)$ of $\bfadCob_{\calC}$ which admits a \textit{critical identification} $f : S^1 \times S^1 \rightarrow \varSigma$, which is a positive diffeomorphism such that $(f^{-1}(P),f^*(\vartheta))$ is critical. 
For all critical $x,y \in X$, for every object $V \in \Proj(\calC_0)$, and for every $(i,j) \in \rmI_{g_x + x}^2$, we denote with $T_{Vijy}$ the $\calC$-colored ribbon graph represented in Figure \ref{F:critical_torus}, and we set
\[
p^0_{Vijy} := \calD^{-2} \rmd(V_i) \cdot F_{\calC} \left( T_{Vijy} \right) \otimes \varepsilon.
\]

\begin{figure}[hbt]\label{F:critical_torus}
 \centering
 \includegraphics{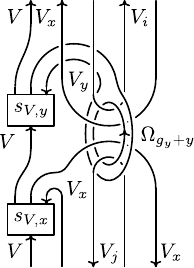}
 \caption{The $\calC$-colored ribbon graph $T_{Vijy}$.}
\end{figure}


\begin{proposition}
 If $\bbSigma = (\varSigma,P,\vartheta,\calL)$ is a critical surface of genus $1$, and if $f : S^1 \times S^1 \rightarrow \varSigma$ is a critical identification, then 
 \[
  \bbV^Z_{\calC}(\bbSigma) \cong \bbHom_{\hat{\calC}} \left( \one, \im \left( \left( p^0_{F_{\calC}(P) ij \langle f^*(\vartheta),\ell \rangle} \right)_{(i,j) \in \rmI_{g_{\langle f^*(\vartheta),m \rangle} + \langle f^*(\vartheta),m \rangle}^2}  \right) \right).
 \]
\end{proposition}

\begin{proof}
 Let us set
 \[
  V := F_{\calC}(P) \in \Proj(\calC_0), \quad
  x := \langle f^*(\vartheta),m \rangle \in X, \quad
  y := \langle f^*(\vartheta),\ell \rangle \in X.
 \]
 Then, we can decompose $\bbSigma$, up to isomorphism, as
 \[
  \overline{\bbD^2} \circ \overline{\bbP^2_{-x,x}} \circ \left( \id_{\bbS^1_{-x}} \disjun (\bbI \times \bbS^1)_{x,y} \right) \circ \bbP^2_{-x,x} \circ \bbD^2_V.
 \]
 But now remark that, if 
 \[
 p^0_{Vij0} = \sum_{n=1}^{N_{ij}} \bbnabla_{-x,x}(p^{k_n}_{Vij0n} \otimes p'^{-k_n}_{ij0n}),
 \]
 then
 \begin{align*}
  p^0_{Vijy} &= \sum_{n=1}^{N_{ij}} \psi(y,k_n) \cdot \bbnabla_{-x,x}(p^{k_n}_{Vij0n} \otimes p'^{-k_n}_{ij0n})
 \end{align*}
 for every $(i,j) \in \rmI_{g_x + x}^2$. Therefore, thanks to Propositions \ref{P:counit}, \ref{P:product}, \ref{P:coproduct_generic}, and \ref{P:2-cylinder_isomorphism}, we can conclude.
\end{proof}

Next, we discuss surfaces of genus $n > 1$, although we restrict to the generic case for our study. First of all, let us start by fixing our notation. If $\calG$ is an oriented trivalent graph, whose set of vertices is denoted $\calV$, and whose set of edges is denoted $\calE$, then, for every vertex $v \in \calV$, we denote with $\calE_v$ the set of germs of edges of $\calE$ which are incident to $v$, and we denote with $\varepsilon_v : \calE_v \rightarrow \{ +,- \}$ the function associating the sign $+$ with incoming germs of edges, and the sign $-$ with outgoing ones.
If $\calG$ is an ordered trivalent graph, and if $V : \calE \rightarrow \calC$ is a $\calC$-coloring of $\calE$, then we set
\[
 \bbH_v (V) := \bbHom_{\calC} \left( \one, \bigotimes_{e \in \calE_v} V(e)^{\varepsilon_v(e)} \right)
\]
for every $v \in \calV$, where $V(e)^+ := V(e)$ and $V^-(e) := V(e)^*$ for every $e \in \calE_v$. For our description, we will focus on a particular family of graphs. For every integer $n > 1$, we consider the oriented planar trivalent graph $\calG_n \subset \R^2 \times \{ 0 \} \subset \R^3$ represented in Figure \ref{F:trivalent_graph}. 

\begin{figure}[hbt]\label{F:trivalent_graph}
	\centering
	\includegraphics{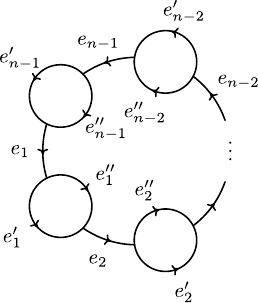}
	\caption{Trivalent graph of genus $n$.} 
\end{figure}

Let $\varSigma_n$ be the closed surface of genus $n$ obtained as the boundary of a standard tubular neighborhood $H_n$ of $\calG_n$ in $\R^3$. 
Let $m_i$ denote 
a positive meridian of the edge $e_i$, let $m'_i$ denote 
a positive meridian of the edge $e'_i$, and let $m''_i$ denote 
a positive meridian of the edge $e''_i$ for all integers $1 \leqslant i < n$. We use the same notation for the corresponding homology classes, with the exception of the meridians $m_i$, which all determine the same homology class $m_0$. We say a cohomology class $\vartheta_n \in H^1(\varSigma_n;G)$ is \textit{generic} if 
\[
 \langle \vartheta_n,m_0 \rangle, \langle \vartheta_n,m'_i \rangle, \langle \vartheta_n,m''_i \rangle \in G \smallsetminus X
\]
for all integers $1 \leqslant i < n$. A \textit{generic surface of genus $n$} is a closed connected $1$-morphism $\bbSigma = (\varSigma,\varnothing,\vartheta,\calL)$ of $\bfadCob_{\calC}$ which admits a \textit{generic identification} ${f : \varSigma_n \rightarrow \varSigma}$, which is a positive diffeomorphism such that $f^*(\vartheta)$ is generic. 
Then, let us give a sufficient condition for the genericity of surfaces of genus $n$ which requires some additional terminology. We say a generic $g \in G \smallsetminus X$ is \textit{doubly generic} if $\langle g \rangle \not\subset (g_1 + X) \cup (g_2 + X)$ for all $g_1,g_2 \in G$, where $\langle g \rangle$ denotes the cyclic subgroup of $G$ generated by $g$. Then, if $\varSigma$ is a $2$-dimensional cobordism, and if $\vartheta \in H^1(\varSigma;G)$ is a cohomology class, we say an embedded closed oriented curve $\gamma \subset \varSigma$ is \textit{doubly generic} with respect to $\vartheta$ if $\langle \vartheta,\gamma \rangle \in G \smallsetminus X$ is doubly generic.

\begin{lemma}\label{L:generic_spine}
 If $\bbSigma = (\varSigma,\varnothing,\vartheta,\calL)$ is a closed connected $1$-morphism of $\bfadCob_{\calC}$ of genus $n > 1$ which admits a doubly generic curve $\gamma \subset \varSigma$, then $\bbSigma$ is generic.
\end{lemma}

\begin{proof}
 The argument can be directly adapted from the proof of Proposition 6.5 of \cite{BCGP16}. Indeed, let us consider a positive diffeomorphism $f : \varSigma_n \rightarrow \varSigma$ sending the embedded closed oriented curve $m_1$ to $\gamma$, and 
 let us consider the integral basis 
 \[
  \{ m_0,m'_1,\ldots,m'_{n-1} \} \subset H_1(\varSigma_n;\Z)
 \]
 of $\ker(i_{\partial H_n *}) \subset H_1(\varSigma_n;\R)$ for the inclusion $i_{\partial H_n} : \partial H_n \hookrightarrow H_n$. Then, we claim that, for every integer $1 \leqslant i < n$, there exists some $k_i \in \Z$ such that 
 \[
  \langle f^*(\vartheta),m'_i + k_i \cdot m_0 \rangle, \langle f^*(\vartheta),m'_i + (k_i-1) \cdot m_0 \rangle \in G \smallsetminus X.
 \]
 This can be proven by contradiction. Indeed, suppose there exists some integer $1 \leqslant i < n$ such that either $\langle f^*(\vartheta),m'_i + k \cdot m_0 \rangle \in X$ or $\langle f^*(\vartheta),m'_i + (k-1) \cdot m_0 \rangle \in X$ for every $k \in \Z$. Then, since
 \begin{align*}
  k \langle f^*(\vartheta),m_0 \rangle &= - \langle f^*(\vartheta),m'_i \rangle + \langle f^*(\vartheta),m'_i + k \cdot m_0 \rangle \\
  & = - \langle f^*(\vartheta),m'_i - m_0 \rangle + \langle f^*(\vartheta),m'_i + (k-1) \cdot m_0 \rangle,
 \end{align*} 
 this means 
 \[
  k \langle f^*(\vartheta),m_0 \rangle \in (- \langle f^*(\vartheta),m'_i \rangle + X) \cup (- \langle f^*(\vartheta),m'_i - m_0 \rangle + X) 
 \]
 for every $k \in \Z$. This is a contradiction, because $\langle f^*(\vartheta),m_0 \rangle$ is doubly generic. Then, let us consider the integral basis 
 \[
  \{ m_0,m'_1 + k_1 \cdot m_0,\ldots,m'_{n-1} + k_{n-1} \cdot m_0 \} \subset H_1(\varSigma_n;\Z)
 \]
 of $\ker(i_{\partial H_n *}) \subset H_1(\varSigma_n;\R)$.
 It is clear from the definition that
 \[
  \langle f^*(\vartheta),m'_i + k_i \cdot m_0 \rangle, \langle f^*(\vartheta),m_0 - (m'_i + k_i \cdot m_0) \rangle \in G \smallsetminus X
 \]
 for every integer $1 \leqslant i < n$. Now let $f' : \varSigma_n \rightarrow \varSigma_n$ be an automorphism of $\varSigma_n$ satisfying $f'_*(m_0) = m_0$ and $f'_*(m'_i) = m'_i + k_i \cdot m_0$ for every integer $1 \leqslant i < n$, which exists because the mapping class group of $\varSigma_n$ acts transitively on the set of integral bases of any Lagrangian subspace of $H_1(\varSigma_n;\R)$. Then, we can set $(f \circ f')^*(\vartheta)$ is generic.
\end{proof}

If $\calG_n$ is generic with respect to a cohomology class $\vartheta_n \in H^1(\varSigma_n;\R)$, then we define the set $\rmCol(\calE_n,\vartheta_n)$ of fundamental $\vartheta_n$-compatible $\calC$-colorings of $\calE_n$ to be
\[
 \left\{ V : \calE_n \rightarrow \calC \Biggm| 
 \begin{array}{l}
  V(e_i) \in \Theta(\calC_{\langle \vartheta_n,m_i \rangle}), \\
  V(e'_i) \in \Theta(\calC_{\langle \vartheta_n,m'_i \rangle}), \\
  V(e''_i) \in \Theta(\calC_{\langle \vartheta_n,m''_i \rangle})
 \end{array} \Forall 1 \leqslant i < n \right\}.
\]

\begin{proposition}
 If $\bbSigma = (\varSigma,\varnothing,\vartheta,\calL)$ is a generic surface of genus $n > 1$, and if $f : \varSigma_n \rightarrow \varSigma$ is a generic identification, then 
 \[
  \bbV^Z_{\calC}(\bbSigma) \cong \bigoplus_{V \in \rmCol(\calE_n,f^*(\vartheta))} \bigotimes_{v \in \calV_n} \bbH_v(V).
  \]
\end{proposition}

\begin{proof}
 Let us set
 \[
  g_0 := \langle f^*(\vartheta),m_0 \rangle, \quad
  h_0 := \langle f^*(\vartheta),\ell_0 \rangle, \quad
  g'_i := \langle f^*(\vartheta),m'_i \rangle, \quad
  h'_i := \langle f^*(\vartheta),\ell'_i \rangle,
 \]
 where $\ell_0$ denotes the homology class determined by the cycle 
 \[
  \sum_{i=1}^{n-1} e_i + e'_i,
 \]
 and where $\ell'_i$ denotes the homology class determined by the cycle $e'_i - e''_i$ using the blackboard framing for every integer $1 \leqslant i < n$. Then, we can decompose $\bbSigma$, up to isomorphism, as
 \[
  \overline{\bbD^2} \circ \overline{\bbP^2_{-g_0,g_0}} \circ \left( (\bbI \times \bbS^1)_{-g_0,-h_0} \disjun \left( \bbJ_{g'_{n-1},h'_{n-1}} \circ \ldots \circ \bbJ_{g'_1,h'_1} \right) \right) \circ \bbP^2_{-g_0,g_0} \circ \bbD^2
 \]
 for the $1$-morphisms $\bbJ_{g'_i,h'_i} : \bbS^1_{g_0} \rightarrow \bbS^1_{g_0}$ of $\bfadCob_{\calC}$ represented Figure \ref{F:genus} for every integer $1 \leqslant i < n$.
 \begin{figure}[tb]\label{F:genus}
  \centering
  \includegraphics{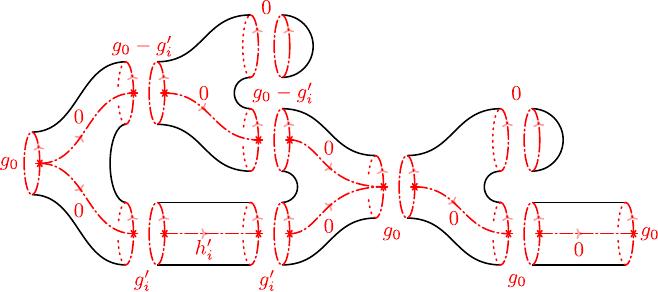}
  \caption{The $1$-morphism $\bbJ_{g'_i,h'_i}$ of $\bfadCob_{\calC}$.} 
 \end{figure}
 Therefore, thanks to Propositions \ref{P:counit}, \ref{P:product}, \ref{P:coproduct_generic}, and \ref{P:2-cylinder_isomorphism}, we can conclude.
\end{proof}

%% file: appendix_a.tex
%
%
%

\chapter{Unrolled quantum groups}\label{A:unrolled_quantum_groups}

In this appendix we construct the main families of examples of relative modular categories, which come from finite-dimensional weight representations of unrolled quantum groups at roots of unity. More precisely, we specialize the deformation parameter ${q}$ to ${e^{\frac{2 \pi i}{r}}}$ for some integer ${r}$ called the \textit{level} of the theory. We distinguish two cases: when ${r}$ is even, the representation theory is slightly more complicated, and explicit examples have been constructed so far only for the Lie algebra ${\sltwo}$. The resulting quantum invariants and TQFTs are very interesting, as it is this particular family that contains the abelian Reidemeister torsion. When ${r}$ is odd, we have a relative modular category associated with every simple Lie algebra ${\frakg}$. Many properties of these quantum groups are collected in \cite{CGP14}, \cite{CGP15}, \cite{BCGP16}, \cite{D15}, and \cite{GP13}.

\section{Even roots of unity}

For this section we will focus on the Lie algebra ${\sltwo}$. If ${q}$ is a formal parameter, then for every integer ${k \geqslant \ell \geqslant 0}$ we set
\[
 \{ k \} := q^k - q^{-k}, \quad [k] := \frac{\{ k \}}{\{ 1 \}}, \quad [k]! := [k][k-1]\cdots[1], \quad \sqbinom{k}{\ell} := \frac{[k]!}{[\ell]![k-\ell]!}.
\]
Let ${\calU_q \sltwo}$ denote the \textit{quantum group of ${\sltwo}$}, which is the ${\C(q)}$-algebra with generators ${K,K^{-1},E,F}$ and relations
\begin{gather*}
 K K^{-1} = K^{-1} K = 1, \quad K E K^{-1} = q^2 \cdot E, \quad K F K^{-1} = q^{-2} \cdot F, \\
 [E,F] = \frac{K - K^{-1}}{q - q^{-1}}.
\end{gather*}
Then ${\calU_q \sltwo}$ can be made into a Hopf algebra by setting
\begin{align*}
 \Delta(K) &= K \otimes K, & \varepsilon(K) &= 1, & S(K) &= K^{- 1}, \\
 \Delta(E) &= E \otimes K + 1 \otimes E, & \varepsilon(E) &= 0, & S(E) &= -E K^{-1}, \\
 \Delta(F) &= F \otimes 1 + K^{-1} \otimes F, & \varepsilon(F) &= 0, & S(F) &= - K F.
\end{align*}

Let us fix an even integer ${r \geqslant 4}$ with the further condition that ${r \not\equiv 0}$ modulo ${8}$, and let us specialize ${q}$ to ${e^{\frac{2 \pi i}{r}}}$. The \textit{unrolled quantum group of ${\sltwo}$} is the ${\C}$-algebra ${U^H_q \sltwo}$ obtained from ${\calU_q \sltwo}$ by adding the generator ${H}$ and the relations
\[
 [H,K] = 0, \quad [H,E] = 2 \cdot E, \quad [H,F] = - 2 \cdot F, \quad E^{\frac{r}{2}} = F^{\frac{r}{2}} = 0.
\]
Then ${U^H_q \sltwo}$ can be made into a Hopf algebra by setting
\[
 \Delta(H) =  H \otimes 1 + 1 \otimes H, \quad \varepsilon(H) = 0, \quad S(H) = -H,
\]
and we denote with ${U^H_q \frakh}$, with ${U^H_q \frakn_+}$, and with ${U^H_q \frakn_-}$ the subalgebras of ${U^H_q \sltwo}$ generated by ${H}$, by ${E}$, and by ${F}$ respectively.
Remark that ${U^H_q \sltwo}$ is a pivotal Hopf algebra, with pivotal element ${K^{\frac{r}{2} + 1}}$.

Let us move on to discuss the representation theory of ${U^H_q \sltwo}$. For every ${z \in \C}$ let us introduce the notation
\[
 q^z := e^{\frac{z 2 \pi i}{r}}, \quad \{ z \} := q^z - q^{-z}.
\]
A finite-dimensional ${U^H_q \sltwo}$-module ${V}$ is a \textit{weight module} if it is a semisimple ${U^H_q \frakh}$-module and if for every ${\lambda \in \C}$ and every ${v \in V}$ we have
\[
 \rho_V(H)(v) = \lambda \cdot v \quad \Rightarrow \quad \rho_V(K)(v) = q^{\lambda} \cdot v.
\]
If we denote with ${\calC}$ the full subcategory of the linear category of finite-dimensional ${U^H_q \sltwo}$-modules whose objects are weight modules, then ${\calC}$ can be made into a ribbon linear category as follows: if ${V}$ and ${V'}$ are objects of $\calC$ their braiding morphism is given by 
\[
 \begin{array}{rccc}
  c_{V,V'} : & V \otimes V' & \rightarrow & V' \otimes V \\
  & v \otimes v' & \mapsto & \tau_{V,V'}(R_{0,V,V'}((\rho_V \otimes \rho_{V'})(\Theta)(v \otimes v')))
 \end{array}
\]
for the linear maps ${R_{0,V,V'} : V \otimes V' \rightarrow V \otimes V'}$ and ${\tau_{V,V'} : V \otimes V' \rightarrow V' \otimes V}$ determined by
\[
 R_{0,V,V'}(v \otimes v') := q^{\frac{\lambda \lambda'}{2}} \cdot v \otimes v', \quad \tau_{V,V'}(v \otimes v') := v' \otimes v
\]
for all ${v \in V}$, ${v' \in V'}$ satisfying
\[
 \rho_V(H)(v) = \lambda \cdot v, \quad \rho_{V'}(H)(v') = \lambda' \cdot v',
\]
and for the element $\Theta \in U^H_q \frakn_+ \otimes U^H_q \frakn_-$ given by
\[
 \Theta := \sum_{b=0}^{\frac{r}{2}-1} \frac{\{ 1 \}^{b}}{[b]!} q^{\frac{b(b-1)}{2}} \cdot E^b \otimes F^b.
\]

If we set ${G := \C / 2 \Z}$, then ${\calC}$ supports the structure of a ${G}$-category: indeed, for every ${\gamma \in \C}$ we can define the homogeneous subcategory ${\calC_{[\gamma]}}$ to be the full subcategory of ${\calC}$ with objects given by weight modules whose weights are all of the form ${\gamma + 2k}$ for some ${k \in \Z}$. Furthermore, if we set 
\[
 \bar{r} := \frac{r}{\gcd \left( 2, \frac{r}{2} \right)},
\]
then we can define ${\PGr := \bar{r}\Z}$, and we have a free realization ${\sigma : \PGr \rightarrow \calC_{[0]}}$ mapping every ${k \in \PGr}$ to the object ${\sigma(k) \in \calC_{[0]}}$ given by the vector space ${\C}$ with ${U^H_q \sltwo}$-action specified by
\[
 \rho_{\sigma(k)}(H)(1) := k, \quad \rho_{\sigma(k)}(E)(1) := 0, \quad \rho_{\sigma(k)}(F)(1) := 0.
\]
Now the bilinear map 
\[
 \begin{array}{rccc}
  \psi : & G \times \PGr & \rightarrow & \C^* \\
  & ([\gamma],k) & \mapsto & q^{\gamma k}
 \end{array}
\]
satisfies ${c_{\sigma(k),V} \circ c_{V,\sigma(k)} = \psi([\gamma],k) \cdot \id_{V \otimes \sigma(k)}}$ for every ${\gamma \in \C}$, every ${V \in \calC_{[\gamma]}}$, and every ${k \in \PGr}$. If we consider the critical set ${X := \Z / 2 \Z}$ then, as proved in Theorem 5.2 of \cite{CGP15}, the category ${\calC_{[\gamma]}}$ is semisimple for every ${[\gamma] \in G \smallsetminus X}$. Therefore, the last ingredient we are missing is a projective trace. In order to define it, let us introduce typical ${U^H_q \sltwo}$-modules. First of all, we say a vector ${v_+}$ of a weight ${U^H_q\sltwo}$-module ${V}$ is a \textit{highest weight vector} if ${\rho_V(E)(v_+) = 0}$. Analogously, we say a vector ${v_-}$ of ${V}$ is a \textit{lowest weight vector} if ${\rho_V(F)(v_-) = 0}$. Then for every ${\lambda \in \C}$ there exists a simple weight ${U^H_q\sltwo}$-module ${V_{\lambda}}$ featuring a heighest weight vector of weight ${\lambda}$. Remark that this highest weight notation for simple ${U^H_q\sltwo}$-modules agrees with the one used in \cite{GPT09} and in \cite{GP13}, but differs from the middle weight notation used in \cite{CGP14} and in \cite{CGP15}. This module is unique up to isomorphism, and every simple ${U^H_q\sltwo}$-module is of this form, see Lemma 5.3 of \cite{CGP15}. Every such module also has a lowest weight vector, and it is called \textit{typical} if its lowest weight is given by ${\lambda - 2(r - 1)}$. If we consider the set ${\ddot{\C} := (\C \smallsetminus \Z) \cup (\{ r-1 \} + rZ)}$ then ${V_{\gamma}}$ is typical if and only if ${\gamma \in \ddot{\C}}$. Remark that if ${\gamma \in \C}$ satisfies ${[\gamma] \not\in X}$, then ${V_{\gamma}}$ is typical. Furthermore, thanks to Theorem 5.2 of \cite{CGP15} and Corollary 31 of \cite{GPT09}, every typical ${U^H_q\sltwo}$-module is projective and ambidextrous. Then, by combining Theorem 3.3.2 of \cite{GKP11} with Lemma 17 of \cite{GPV13}, there exists a non-zero trace on the ideal ${\Proj(\calC)}$ of projective objects of ${\calC}$ which is unique up to scalar. Here is the normalization we choose: for every ${\gamma \in \ddot{\C}}$ we set the projective dimension of ${V_{\gamma}}$ to be
\[
 \rmd(V_{\gamma}) := \frac{r \left\{ \gamma - (r - 1) \right\}}{\left\{ r \left( \gamma - (r - 1) \right) \right\}}.
\]

\begin{proposition}
 ${\calC}$ is a modular ${G}$-category relative to ${(\PGr,X)}$.
\end{proposition}

A proof of the relative pre-modularity of ${\calC}$ is provided by Section 6.3 of \cite{CGP14}, while the relative modularity condition is checked in Lemma A.4 of \cite{BCGP16}.

\section{Odd roots of unity}

Let ${\frakg}$ be a simple Lie algebra of rank ${n}$ and dimension ${2N + n}$, let ${B}$ be its Killing form, let ${\frakh}$ be a Cartan subalgebra of ${\frakg}$, let ${\Phi}$ be the corresponding root system, let ${\Phi_+}$ be a choice of a set of positive roots of ${\frakg}$, and let ${\{ \alpha_1, \ldots, \alpha_n \}}$ be an ordering of its set of simple roots. Let ${A = ( a_{ij} )_{1 \leqslant i,j \leqslant n}}$ be the corresponding Cartan matrix, which is the integral matrix given by
\[
 a_{ij} := \frac{2 B^*(\alpha_i,\alpha_j)}{B^*(\alpha_i,\alpha_i)},
\]
where ${B^*}$ is the symmetric bilinear form on ${\frakh^*}$ determined by the restriction of ${B}$ to ${\frakh}$ under the isomorphism which identifies a vector ${H \in \frakh}$ with the linear form ${B(H,\cdot) \in \frakh^*}$, and let ${\{ H_1,\ldots, H_n \}}$ be the basis of ${\frakh}$ determined by ${\alpha_j (H_i) = a_{ij}}$ for all indices ${1 \leqslant i,j \leqslant n}$. For every ${\alpha \in \Phi_+}$ we set
\[
 d_{\alpha} := \frac{B^*(\alpha,\alpha)}{\min \{ B^*(\alpha_i, \alpha_i) \mid 1 \leqslant i \leqslant n \}}
\]
and for every integer ${1 \leqslant i \leqslant n}$ we use the short notation ${d_i := d_{\alpha_i}}$. We denote with ${\langle \cdot,\cdot \rangle}$ the symmetric bilinear form on ${\frakh^*}$ determined by ${\langle \alpha_i,\alpha_j \rangle = d_i a_{ij}}$ for all indices ${1 \leqslant i,j \leqslant n}$, and we denote with ${\lambda_1,\ldots,\lambda_n}$ the corresponding fundamental dominant weights, which are the vectors of ${\frakh^*}$ determined by the condition ${\langle \lambda_i,\alpha_j \rangle = d_i \delta_{ij}}$ for all indices ${1 \leqslant i,j \leqslant n}$. We denote with ${\Lambda_R}$ the root lattice, which is the subgroup of ${\frakh^*}$ generated by simple roots, and we denote with ${\Lambda_W}$ the weight lattice, which is the subgroup of ${\frakh^*}$ generated by fundamental dominant weights. If ${q}$ is a formal parameter, then for every ${\alpha \in \Phi_+}$ we set ${q_{\alpha} := q^{d_{\alpha}}}$, for all integer ${k \geqslant \ell \geqslant 0}$ we set
\begin{gather*}
 \{ k \}_{\alpha} := q_{\alpha}^k - q_{\alpha}^{-k}, \quad [k]_{\alpha} := \frac{\{ k \}_{\alpha}}{\{ 1 \}_{\alpha}},
 \quad [k]_{\alpha}! := [k]_{\alpha}[k-1]_{\alpha}\cdots[1]_{\alpha}, \\
 \sqbinom{k}{\ell}_{\alpha} := \frac{[k]_{\alpha}!}{[\ell]_{\alpha}![k-\ell]_{\alpha}!},
\end{gather*}
and for every integer ${1 \leqslant i \leqslant n}$ we use the short notation 
\[
 q_i := q_{\alpha_i}, \quad \{ k \}_i := \{ k \}_{\alpha_i}, \quad [k]_i := [k]_{\alpha_i},
 \quad [k]_i! := [k]_{\alpha_i}!, \quad
 \sqbinom{k}{\ell}_i := \sqbinom{k}{\ell}_{\alpha_i}.
\]
Let ${\calU_q \frakg}$ denote the \textit{quantum group of ${\frakg}$}, which is the ${\C(q)}$-algebra with generators
\[
 \{ K_i,K_i^{-1},E_i,F_i \mid 1 \leqslant i \leqslant n \}
\]
and relations
\begin{gather*}
 K_i K_i^{-1} = K_i^{-1} K_i = 1, \quad [K_i,K_j] = 0, \\
 K_i E_{j} K_i^{-1} = q_i^{a_{ij}} \cdot E_{j}, \quad
 K_i F_{j} K_i^{-1} = q_i^{- a_{ij}} \cdot F_{j}, \\
 [E_i,F_j] = \delta_{ij} \cdot \frac{K_i - K_i^{-1}}{q_i-q_i^{-1}}
\end{gather*}
for all indices ${1 \leqslant i,j \leqslant n}$ and
\begin{gather*}
 \displaystyle \sum_{k=0}^{1 - a_{ij}} (-1)^{k} \sqbinom{1-a_{ij}}{k}_i \cdot E_{i}^{k} E_{j} E_{i}^{1-a_{ij}-k} = 0, \\
 \displaystyle \sum_{k=0}^{1 - a_{ij}} (-1)^{k} \sqbinom{1-a_{ij}}{k}_i \cdot F_{i}^{k} F_{j} F_{i}^{1-a_{ij}-k} = 0\phantom{,}
\end{gather*}
for all indices ${1 \leqslant i,j \leqslant n}$ satisfying ${i \neq j}$. Then ${\calU_q \frakg}$ can be made into a Hopf algebra by setting
\begin{align*}
 \Delta(K_i) &= K_i \otimes K_i, & \varepsilon(K_i) &= 1, & S(K_i) &= K_i^{- 1}, \\
 \Delta(E_i) &= E_i \otimes K_i + 1 \otimes E_i, & \varepsilon(E_i) &= 0, & S(E_i) &= -E_i K_i^{-1}, \\
 \Delta(F_i) &= F_i \otimes 1 + K_i^{-1} \otimes F_i, & \varepsilon(F_i) &= 0, & S(F_i) &= - K_i F_i
\end{align*}
for every integer ${1 \leqslant i \leqslant n}$. Let ${W}$ be the Weyl group of ${\frakg}$ associated with ${\frakh}$, which is the subgroup of ${\GL(\frakh^*)}$ generated by reflections
\[
 \begin{array}{rccc}
  s_i : & \frakh^* & \rightarrow & \frakh^* \\
  & \alpha_j & \mapsto & \alpha_j - a_{ij} \cdot \alpha_i
 \end{array}
\]
for every integer ${1 \leqslant i \leqslant n}$, and let ${w_0 \in W}$ be the unique element corresponding to a word of maximal length in the generators. The choice of a decomposition ${w_0 = s_{i_1} \circ \cdots \circ s_{i_N}}$ determines a total order on the set of positive roots
\[
 \Phi_+ = \{ \alpha_{i_1},s_{i_1}(\alpha_{i_2}),\ldots,(s_{i_1} \circ \cdots \circ s_{i_{N-1}})(\alpha_{i_N}) \}. 
\]
For every integer ${1 \leqslant i \leqslant n}$ let us consider the automorphism ${T_i}$ of ${\calU_q \frakg}$ determined by 
\begin{align*}
 T_i(K_j) &:= K_j K_i^{-a_{ij}}, \\
 T_i(E_j) &:= 
 \begin{cases}
  - F_iK_i \phantom{- K_i^{-1}E_i} \hspace{150pt} & i = j, \\
  \displaystyle \sum_{k=0}^{-a_{ij}} (-1)^{k} \frac{q_i^{a_{ij} + k}}{[k]_i![-a_{ij}-k]_i!} \cdot E_i^{k} E_j E_i^{-a_{ij}-k} & i \neq j,
 \end{cases} \\
 T_i(F_j) &:= 
 \begin{cases}
  - K_i^{-1}E_i \phantom{- F_iK_i} \hspace{150pt} & i = j, \\
  \displaystyle \sum_{k=0}^{-a_{ij}} (-1)^{-a_{ij}-k} \frac{q_i^{k}}{[k]_i![-a_{ij}-k]_i!} \cdot F_i^{k} F_j F_i^{-a_{ij}-k} & i \neq j.
 \end{cases}
\end{align*}
Then for every integer ${1 \leqslant k \leqslant N}$ we set 
\[
 \beta_k := (s_{i_1} \circ \cdots \circ s_{i_{k-1}})(\alpha_{i_k}) \in \Phi_+
\]
and
\[
 E_{\beta_k} := (T_{i_1} \circ \cdots \circ T_{i_{k-1}})(E_{i_k}) \in \calU_q \frakg, \quad 
 F_{\beta_k} := (T_{i_1} \circ \cdots \circ T_{i_{k-1}})(F_{i_k}) \in \calU_q \frakg.
\]

Let us fix an odd integer ${r \geqslant 3}$ with the further condition that ${r \not\equiv 0}$ modulo ${3}$ if ${\frakg = \frakg_2}$, and let us specialize ${q}$ to ${e^{\frac{2 \pi i}{r}}}$. The \textit{unrolled quantum group of ${\frakg}$} is the ${\C}$-algebra ${U^H_q \frakg}$ obtained from ${\calU_q \frakg}$ by adding generators
\[
 \{ H_i \mid 1 \leqslant i \leqslant n \}
\]
and relations
\[
 [H_i,H_j] = [H_i,K_j] = 0, \quad [H_i,E_j] = a_{ij} \cdot E_j, \quad [H_i,F_j] = - a_{ij} \cdot F_j, \quad E_{\alpha}^r = F_{\alpha}^r = 0
\]
for all indices ${1 \leqslant i,j \leqslant n}$ and all positive roots ${\alpha \in \Phi_+}$. Then ${U^H_q \frakg}$ can be made into a Hopf algebra by setting
\[
 \Delta(H_i) =  H_i \otimes 1 + 1 \otimes H_i, \quad \varepsilon(H_i) = 0, \quad S(H_i) = -H_i
\]
for every integer ${1 \leqslant i \leqslant n}$, and we denote with ${U^H_q \frakh}$, with ${U^H_q \frakn_+}$, and with ${U^H_q \frakn_-}$ the subalgebras of ${U^H_q \frakg}$ generated by ${\{ H_i \mid 1 \leqslant i \leqslant n \}}$, by ${\{ E_i \mid 1 \leqslant i \leqslant n \}}$, and by ${\{ F_i \mid 1 \leqslant i \leqslant n \}}$ respectively. Remark that ${U^H_q \frakg}$ is a pivotal Hopf algebra, with pivotal element ${K_{2 \cdot \rho}^{1-r}}$ determined by
\[
 \rho := \frac 12 \cdot \sum_{k=1}^N \beta_k \in \Lambda_W,
\]
where for every
\[
 \lambda = \sum_{i=1}^n \ell_i \cdot \alpha_i \in \Lambda_R
\]
we use the notation
\[
 K_{\lambda} := \prod_{i=1}^n K_i^{\ell_i}.
\]

Let us move on to discuss the representation theory of ${U^H_q \frakg}$. For every ${z \in \C}$ let us introduce the notation
\[
 q^z := e^{\frac{z 2 \pi i}{r}}, \quad \{ z \} := q^z - q^{-z}.
\]
A finite-dimensional ${U^H_q \frakg}$-module ${V}$ is a \textit{weight module} if it is a semisimple ${U^H_q \frakh}$-module and if for every weight ${\lambda \in \frakh^*}$ and every vector ${v \in V}$ we have
\[
 \rho_V(H_i)(v) = \lambda(H_i) \cdot v \quad \Forall 1 \leqslant i \leqslant n \quad \Rightarrow \quad \rho_V(K_i)(v) = q_i^{\lambda(H_i)} \cdot v \quad \Forall 1 \leqslant i \leqslant n,
\]
where we are identifying ${\frakh}$ with the corresponding linear subspace of ${U^H_q \frakh}$ in the obvious way. If we denote with $\calC$ the full subcategory of the linear category of finite-dimensional $U^H_q \frakg$-modules whose objects are weight modules, then ${\calC}$ can be made into a ribbon linear category as follows: first of all, a pivotal element is given by $K_{2 \cdot \rho}^{1-r}$, where the choice of the exponent is explained in Remark 4 of \cite{GP18}. Furthermore, if $V$ and $V'$ are objects of $\calC$, their braiding morphism is given by 
\[
 \begin{array}{rccc}
  c_{V,V'} : & V \otimes V' & \rightarrow & V' \otimes V \\
  & v \otimes v' & \mapsto & \tau_{V,V'}(R_{0,V,V'}((\rho_V \otimes \rho_{V'})(\Theta)(v \otimes v')))
 \end{array}
\]
for the linear maps ${R_{0,V,V'} : V \otimes V' \rightarrow V \otimes V'}$ and ${\tau_{V,V'} : V \otimes V' \rightarrow V' \otimes V}$ determined by
\[
 R_{0,V,V'}(v \otimes v') := q^{\langle \lambda,\lambda' \rangle} \cdot v \otimes v', \quad \tau_{V,V'}(v \otimes v') := v' \otimes v
\]
for all ${v \in V}$, ${v' \in V'}$ satisfying
\[
 \rho_V(H_i)(v) = \lambda(H_i) \cdot v, \quad \rho_{V'}(H_i)(v') = \lambda'(H_i) \cdot v'
\]
for every integer ${1 \leqslant i \leqslant n}$, and for the element $\Theta \in U^H_q \frakn_+ \otimes U^H_q \frakn_-$ given by
\[
 \Theta := \sum_{b_1,\ldots,b_N=0}^{r-1} 
 \left( \prod_{k=1}^N \frac{\{ 1 \}_{\beta_k}^{b_k}}{[b_k]_{\beta_k}!} q_{\beta_k}^{\frac{b_k(b_k-1)}{2}} \right) \cdot 
 \left( \prod_{k=1}^N E_{\beta_k}^{b_k} \right) \otimes 
 \left( \prod_{k=1}^N F_{\beta_k}^{b_k} \right).
\]
Then Theorem 4 of \cite{GP18} shows that $\calC$ is a ribbon category.

If we set ${G := \frakh^* / \Lambda_R}$, then ${\calC}$ supports the structure of a ${G}$-category: indeed, for every ${\gamma \in \frakh^*}$ we can define the homogeneous subcategory ${\calC_{[\gamma]}}$ to be the full subcategory of ${\calC}$ with objects given by modules whose weights are all of the form ${\gamma + \lambda}$ for some ${\lambda \in \Lambda_R}$. Furthermore, if we set ${\PGr := \Lambda_R \cap (r \cdot \Lambda_W)}$, then we have a free realization ${\sigma : \PGr \rightarrow \calC_{[0]}}$ mapping every ${\kappa \in \PGr}$ to the object ${\sigma(\kappa) \in \calC_{[0]}}$ given by the vector space ${\C}$ with ${U^H_q \frakg}$-action specified by
\[
 \rho_{\sigma(\kappa)}(H_i)(1) := \kappa(H_i), \quad \rho_{\sigma(\kappa)}(E_i)(1) := 0, \quad \rho_{\sigma(\kappa)}(F_i)(1) := 0
\]
for every integer ${1 \leqslant i \leqslant n}$. 
Now the bilinear map 
\[
 \begin{array}{rccc}
  \psi : & G \times \PGr & \rightarrow & \C^* \\
  & ([\gamma],\kappa) & \mapsto & q^{2\langle \gamma,\kappa \rangle}
 \end{array}
\]
satisfies ${c_{\sigma(\kappa),V} \circ c_{V,\sigma(\kappa)} = \psi([\gamma],\kappa) \cdot \id_{V \otimes \sigma(\kappa)}}$ for every ${\gamma \in \frakh^*}$, every ${V \in \calC_{[\gamma]}}$, and every ${\kappa \in \PGr}$. If we consider the critical set
\[
 X := \{ [\xi] \in \frakh^* / \Lambda_R \mid \Exists \alpha \in \Phi_+ : 2 \langle \alpha,\xi \rangle \in \Z \}
\]
then, as explained in Section 7 of \cite{CGP14}, the category ${\calC_{[\gamma]}}$ is semisimple for every ${[\gamma] \in G \smallsetminus X}$. Therefore, the last ingredient we are missing is a projective trace. In order to define it, let us introduce typical ${U^H_q\frakg}$-modules. First of all, we say a vector ${v_+}$ of a ${U^H_q\frakg}$-module ${V}$ is a \textit{highest weight vector} if ${\rho_V(E_i)(v_+) = 0}$ for every integer ${1 \leqslant i \leqslant n}$. Analogously, we say a vector ${v_-}$ of ${V}$ is a \textit{lowest weight vector} if ${\rho_V(F_i)(v_-) = 0}$ for every integer ${1 \leqslant i \leqslant n}$. Then for every weight ${\lambda \in \frakh^*}$ there exists a simple finite-dimensional weight $U^H_q\frakg$-module ${V_{\lambda}}$ featuring a heighest weight vector of weight ${\lambda}$. This module is unique up to isomorphism, and every simple ${U^H_q\frakg}$-module is of this form, see Proposition 33 of \cite{GP13}. Every such module also has a lowest weight vector, and it is called \textit{typical} if its lowest weight is given by ${\lambda - 2(r-1) \cdot \rho}$. If we consider the set 
\[
 \ddot{\frakh}^* := \{ \gamma \in \frakh^* \mid 2 \langle \alpha,\gamma + \rho \rangle + m \langle \alpha,\alpha \rangle \not\in r \Z \ \Forall \alpha \in \Phi_+, \ \Forall 1 \leqslant m \leqslant r-1 \}
\]
then, thanks to Proposition 34 of \cite{GP13}, ${V_{\gamma}}$ is typical if and only if ${\gamma \in \ddot{\frakh}^*}$. Remark that if ${\gamma \in \frakh^*}$ satisfies ${2 \langle \alpha,\gamma \rangle \not\in \Z}$ for every ${\alpha \in \Phi}$, then ${\gamma \in \ddot{\frakh}^*}$. This means that if ${\gamma \in \frakh^*}$ satisfies ${[\gamma] \not\in X}$, then ${V_{\gamma}}$ is typical. Furthermore, thanks to Lemma 7.1 of \cite{CGP14} and Theorem 38 of \cite{GP13}, every typical $U^H_q\frakg$-module is projective and ambidextrous. Then, by combining Theorem 3.3.2 of \cite{GKP11} with Lemma 17 of \cite{GPV13}, there exists a non-zero trace on the ideal $\Proj(\calC)$ of projective objects of $\calC$ which is unique up to scalar. Here is the normalization we choose: for every ${\gamma \in \ddot{\frakh}^*}$ we set the projective dimension of ${V_{\gamma}}$ to be
\[
 \rmd(V_{\gamma}) := \prod_{\alpha \in \Phi_+} \frac{r \{ \langle \gamma - (r-1) \cdot \rho,\alpha \rangle \}}{\{ r\langle \gamma - (r-1) \cdot \rho,\alpha \rangle \}}.
\]
Remark that this normalization is ${r^N}$ times the one given in \cite{GP13}.

\begin{proposition}
 ${\calC}$ is a modular ${G}$-category relative to ${(\PGr,X)}$.
\end{proposition}

A proof of the relative pre-modularity of ${\calC}$ is provided by Theorem 7.2 of \cite{CGP14}, while the relative modularity condition is checked in \cite{DGP18}.


%% file: appendix_b.tex
%
%
%

\chapter{Manifolds and cobordisms with corners}\label{A:cobordisms_with_corners}

In this appendix we fix our notation for manifolds and cobordisms with corners. We recall we are assuming all manifolds in this memoir are oriented, except when explicitly stated otherwise, so our terminology will reflect this assumption. In order to induce an orientation on the boundary, we will use the ``outward normal first'' convention.

\section{Manifolds with corners}\label{S:manifolds_with_corners}

This section is just an oriented version of Section 1 of \cite{J68} and of Section 2 of \cite{L00}. Consider the manifolds with boundary 
\[
 \R_+ := \{ x \in \R \mid x \geq 0 \}, \quad\R_- := \{ x \in \R \mid x \leq 0 \}
\]
oriented as submanifolds of ${\R}$. We say a map ${f}$ from ${V \subset \R_-^{m_-} \times \R_+^{m_+}}$ to ${\R^n}$ is \textit{smooth} if there exists an open neighborhood ${W}$ of ${V}$ in ${\R^{m_- + m_+}}$ together with a smooth map ${g : W \rightarrow \R^n}$ such that ${\restr{g}{V} = f}$. Let us fix now an ${m}$-\-di\-men\-sion\-al topological manifold with boundary ${X}$. A \textit{chart at ${x \in X}$} is a pair ${(U,\varphi)}$ where ${U}$ is an open subset of ${X}$ containing ${x}$, and where ${\varphi : U \rightarrow \R_-^{m_-} \times \R_+^{m_+}}$ is a homeomorphism onto its image with ${m_- + m_+ = m}$. Two charts ${(U,\varphi)}$ and ${(U',\varphi')}$ are \textit{compatible} if the transition function
\[
 \restr{\varphi' \circ \varphi^{-1}}{\varphi(U \cap U')} : \varphi(U \cap U') \rightarrow \varphi'(U \cap U')
\]
is a diffeomorphism, and they are \textit{co-oriented} if the Jacobian determinant of the transition function is positive. An \textit{atlas with corners} for ${X}$ is a collection of compatible co-oriented charts containing a chart at ${x}$ for every ${x \in X}$. A \textit{smooth structure with corners} on a topological manifold with boundary is an atlas with corners ${\calA}$ which is maximal under inclusion. Given an atlas with corners for ${X}$, there exists a unique maximal smooth structure with corners containing it.

\begin{definition}
 A \textit{manifold with corners} is a topological manifold with boundary ${X}$ equipped with a smooth structure with corners. 
\end{definition}

If ${X}$ is a manifold with corners and if ${\calA}$ is its smooth structure with corners, then we denote with ${\overline{X}}$ the \textit{opposite manifold with corners}, which is obtained from ${X}$ by replacing ${\calA}$ with the smooth structure with corners composed of all charts outside of ${\calA}$ which are compatible with charts in ${\calA}$. Other notations we will sometimes use for ${X}$ and ${\overline{X}}$ are ${+X}$ and ${-X}$, or ${(-1)X}$ and ${(+1)X}$. A map ${f : X \rightarrow Y}$ between manifolds with corners is \textit{smooth at ${x \in X}$} if it is smooth when read in any pair of charts at ${x}$ and at ${f(x)}$ respectively. A \textit{smooth map} ${f : X \rightarrow Y}$ between manifolds with corners is a map which is smooth at ${x}$ for every ${x \in X}$. The  \textit{tangent space at ${x \in X}$} is the vector space ${T_xX}$ of derivations on the algebra ${\calC^{\infty}(x)}$ of germs of smooth real-valued functions at ${x}$, and the \textit{differential} ${d_xf : T_xX \rightarrow T_{f(x)}Y}$ at ${x \in X}$ of a smooth map ${f : X \rightarrow Y}$ between manifolds with corners is defined using the algebra homomorphism ${f^* : \calC^{\infty}(f(x)) \rightarrow \calC^{\infty}(x)}$. A smooth map ${f : X \rightarrow Y}$ between manifolds with corners is an \textit{immersion} if ${d_xf}$ is injective for every ${x \in X}$, and it is and \textit{embedding} if it is also a homeomorphism onto its image. If ${X}$ and ${Y}$ are manifolds with corners of the same dimension, then we say an embedding ${f : X \hookrightarrow Y}$ is \textit{positive} if it preserves the orientations, and we say it is \textit{negative} otherwise. A \textit{submanifold with corners} of a manifold with corners ${X}$ is a topological submanifold ${Y}$ of ${X}$ together with a smooth structure with corners making the inclusion ${i : Y \hookrightarrow X}$ into a smooth embedding. The \textit{index ${\ind(x)}$} of a point ${x \in X}$ is the number of vanishing coordinates of ${\varphi(x)}$ for any chart ${(U,\varphi)}$ at ${x}$. A \textit{connected face} of a manifold with corners ${X}$ is the closure of a connected component of the set ${\{ x \in X \mid \ind(x) = 1 \}}$.

\begin{definition}
 A manifold with corners ${X}$ is a \textit{manifold with faces} if every point ${x \in X}$ belongs to exactly ${\ind(x)}$ distinct connected faces of ${X}$.
\end{definition}

A \textit{face} of a manifold with faces ${X}$ is a disjoint union of connected faces. It can be made into a manifold with faces itself by equipping it with the orientation induced using the boundary convention. Remark that, if ${X}$ and ${Y}$ are manifolds with faces, then so is ${X \times Y}$.

\begin{definition}
 An \textit{$m$-di\-men\-sion\-al ${\two}$-man\-i\-fold} is an ${m}$-di\-men\-sion\-al manifold with faces ${X}$ together with a decomposition of its boundary into a union of faces
 \[
  \partial X = (-1)^{m-1} \left( \overline{\partial^{\rmh}_- X} \cup \partial^{\rmh}_+ X \right) \cup \overline{\partial^{\rmv}_- X} \cup \partial^{\rmv}_+ X
 \]
 such that:
 \begin{enumerate}
  \item ${\partial^{\rmh}_- X \cap \partial^{\rmh}_+ X = \partial^{\rmv}_- X \cap \partial^{\rmv}_+ X = \varnothing}$;
  \item ${\partial^2_{\varepsilon_{\rmh} \varepsilon_{\rmv}} X := \partial^{\rmh}_{\varepsilon_{\rmh}} X \cap \partial^{\rmv}_{\varepsilon_{\rmv}} X}$ is a possibly empty face of both ${\partial^{\rmh}_{\varepsilon_{\rmh}} X}$ and ${\partial^{\rmv}_{\varepsilon_{\rmv}} X}$, which we orient as a face of ${\varepsilon_{\rmv} \partial^{\rmh}_{\varepsilon_{\rmh}} X}$ for all signs ${\varepsilon_{\rmh},\varepsilon_{\rmv} \in \{ +,- \}}$.
 \end{enumerate}
 We call ${\partial^{\rmh}_- X}$ the \textit{incoming horizontal boundary}, we call ${\partial^{\rmh}_+ X}$ the \textit{outgoing horizontal boundary}, we call ${\partial^{\rmv}_- X}$ the \textit{incoming vertical boundary}, and we call ${\partial^{\rmv}_+ X}$ the \textit{outgoing vertical boundary} of ${X}$.
\end{definition}

See Figure \ref{F:two-manifold_example} for an example of a ${\two}$-man\-i\-fold structure on the 2-di\-men\-sion\-al manifold with faces ${I \times I}$ with
\begin{gather*}
 \partial^{\rmh}_- (I \times I) := I \times \{ 0 \}, \quad \partial^{\rmh}_+ (I \times I) := I \times \{ 1 \}, \\
 \partial^{\rmv}_- (I \times I) := \{ 0 \} \times I, \quad \partial^{\rmv}_+ (I \times I) := \{ 1 \} \times I.
\end{gather*}

\begin{figure}[htb]\label{F:two-manifold_example}
 \centering
 \includegraphics{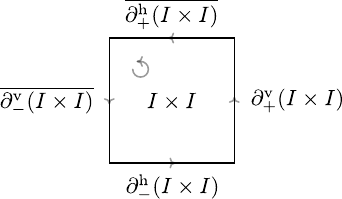}
 \caption{${\two}$-Manifold structure on ${I \times I}$.}
\end{figure}

An \textit{isomorphism of ${\two}$-man\-i\-folds} is a positive diffeomorphism of manifolds with faces which preserves horizontal and vertical boundaries. Remark that if ${Y}$ is manifold with boundary, then every decomposition ${\partial Y = \overline{\partial_- Y} \cup \partial_+ Y}$ arising from a partition of the set of connected components of ${\partial Y}$ determines multiple ${\two}$-man\-i\-fold structures on ${Y}$. We say ${Y}$ is a \textit{horizontal manifold with boundary} when we interpret all of its boundary as vertical, thus taking its horizontal boundary to be empty, and we say ${Y}$ is a \textit{vertical manifold with boundary} when we do the opposite. Furthermore, remark that every closed manifold admits a unique ${\two}$-man\-i\-fold structure. Now if ${X}$ is a ${\two}$-man\-i\-fold, then ${\partial^{\rmh}_{\varepsilon_{\rmh}} X}$ is naturally a horizontal manifold with boundary, ${\partial^{\rmv}_{\varepsilon_{\rmv}} X}$ is naturally a vertical manifold with boundary, and ${\partial^2_{\varepsilon_{\rmh}\varepsilon_{\rmv}} X}$ is a closed manifold for all signs ${\varepsilon_{\rmh},\varepsilon_{\rmv} \in \{ +,- \}}$. Remark also that if ${X}$ is a horizontal ${n}$-di\-men\-sion\-al manifold with boundary, and if ${Y}$ is a vertical ${m-n}$-di\-men\-sion\-al manifold with boundary, then the product ${X \times Y}$ is naturally an ${m}$-di\-men\-sion\-al ${\two}$-man\-i\-fold with
\[
 \partial^{\rmh}_{\varepsilon_{\rmh}} (X \times Y) := (-1)^n \left( X \times \partial_{\varepsilon_{\rmh}} Y \right), \quad 
 \partial^{\rmv}_{\varepsilon_{\rmv}} (X \times Y) := \partial_{\varepsilon_{\rmh}} X \times Y.
\]

\section{Collars}\label{S:collars}

In this section we discuss collars of boundary identifications. The reason we need to do this is the following: while the operation of gluing two manifolds with faces along a face produces a well-defined topological manifold, if we want this manifold to carry a canonical smooth structure with corners we need to specify collars. If we do not, we have infinitely many smooth structures with corners on the glued manifold extending the original ones, and they are all diffeomorphic, although unfortunately diffeomorphisms between them are non-canonical. We point out that the reason this problem does not emerge when constructing categories of cobordisms for 2+1-TQFTs is that morphisms in this context are given by diffeomorphism classes of cobordisms. In the 2-categorical  setting of 1+1+1-ETQFTs, however, things are more complicated: while 2-morphisms are still given by diffeomorphism classes of cobordisms with corners, we need to use 1-morphisms given by honest cobordisms, and not just diffeomorphism classes, or else we lose the ability of composing 2-morphisms. This can be taken care of simply by equipping cobordisms with collars for their boundary identifications, which is precisely what we set out to do now. We will give definitions directly in the setting of ${\two}$-man\-i\-folds in order to treat coherently horizontal and vertical gluing operations. Just like in \cite{L00}, we can give a categorical description of structure of a ${\two}$-man\-i\-fold. Indeed, let ${\two}$ denote the two-fold cartesian product of the category associated with the poset ${\{ 0,1 \}}$, and let ${\Man}$ denote the category whose objects are manifold with corners, and whose morphisms are embeddings. Then every ${\two}$-man\-i\-fold ${X}$ determines a family of functors ${X_{\varepsilon_{\rmh} \varepsilon_{\rmv}} : \two \rightarrow \Man}$, one for every choice of signs ${\varepsilon_{\rmh},\varepsilon_{\rmv} \in \{ +,- \}}$, which map every object ${a = (a_{\rmh},a_{\rmv})}$ of ${\two}$ to the manifold with faces
\[
 X_{\varepsilon_{\rmh} \varepsilon_{\rmv}}(a) := 
 \begin{cases}
  \partial^2_{\varepsilon_{\rmh} \varepsilon_{\rmv}} X & (a_{\rmh},a_{\rmv}) = (0,0), \\
  \partial^{\rmh}_{\varepsilon_{\rmh}} X & (a_{\rmh},a_{\rmv}) = (0,1), \\
  \partial^{\rmv}_{\varepsilon_{\rmv}} X & (a_{\rmh},a_{\rmv}) = (1,0), \\
  X & (a_{\rmh},a_{\rmv}) = (1,1), \\
 \end{cases}
\]
and which map every morphism ${a \leqslant b : a \rightarrow b}$ of ${\two}$ to the corresponding inclusion 
\[
 X_{\varepsilon_{\rmh} \varepsilon_{\rmv}}(a \leqslant b) : X_{\varepsilon_{\rmh} \varepsilon_{\rmv}}(a) \hookrightarrow X_{\varepsilon_{\rmh} \varepsilon_{\rmv}}(b).
\]
Then, if we consider the involution ${*}$ on the set ${\{ +,- \}}$ determined by ${+^* := -}$, the following is just an oriented version of Lemma 2.1.6 of \cite{L00}.

\begin{lemma}\label{L:existence_of_collars}
 For every ${\two}$-man\-i\-fold ${X}$ and for all signs ${\varepsilon_{\rmh},\varepsilon_{\rmv} \in \{ +,- \}}$ there exists a functor ${C_{\varepsilon_{\rmh} \varepsilon_{\rmv}} : \two \rightarrow \Man}$ mapping every object ${a = (a_{\rmh},a_{\rmv})}$ of ${\two}$ to the manifold with faces 
 \[
  \R_{\varepsilon_{\rmv}^*}^{1-a_{\rmv}} \times X_{\varepsilon_{\rmh} \varepsilon_{\rmv}}(a) \times \R_{\varepsilon_{\rmh}^*}^{1-a_{\rmh}},
 \]
 and mapping every morphism ${a \leqslant b : a \rightarrow b}$ of ${\two}$ to a positive embedding
 \[
  C_{\varepsilon_{\rmh} \varepsilon_{\rmv}}(a \leqslant b) : \R_{\varepsilon_{\rmv}^*}^{1-a_{\rmv}} \times X_{\varepsilon_{\rmh} \varepsilon_{\rmv}}(a) \times \R_{\varepsilon_{\rmh}^*}^{1-a_{\rmh}} \hookrightarrow \R_{\varepsilon_{\rmv}^*}^{1-b_{\rmv}} \times X_{\varepsilon_{\rmh} \varepsilon_{\rmv}}(b) \times \R_{\varepsilon_{\rmh}^*}^{1-b_{\rmh}}
 \]
 whose restriction to ${\R_{\varepsilon_{\rmv}^*}^{1-b_{\rmv}} \times X_{\varepsilon_{\rmh} \varepsilon_{\rmv}}(a) \times \R_{\varepsilon_{\rmh}^*}^{1-b_{\rmh}}}$ coincides with the inclusion 
 \[
  {\id_{\R_{\varepsilon_{\rmv}^*}^{1-b_{\rmv}}}} \times X_{\varepsilon_{\rmh} \varepsilon_{\rmv}}(a \leqslant b) \times {\id_{\R_{\varepsilon_{\rmh}^*}^{1-b_{\rmh}}}}.
 \]
\end{lemma}

If ${X}$ is a ${\two}$-man\-i\-fold, then a \textit{compatible system of boundary identifications} around ${\partial^2_{\varepsilon_{\rmh} \varepsilon_{\rmv}} X}$ is a functor ${Y_{\varepsilon_{\rmh} \varepsilon_{\rmv}} : \two \rightarrow \Man}$ together with a positive natural isomorphism ${f_{\varepsilon_{\rmh} \varepsilon_{\rmv}} : Y_{\varepsilon_{\rmh} \varepsilon_{\rmv}} \Rightarrow X_{\varepsilon_{\rmh} \varepsilon_{\rmv}}}$, where positive means the diffeomorphism ${f_{\varepsilon_{\rmh} \varepsilon_{\rmv}}(a)}$ is positive for every object ${a}$ of ${\two}$. If ${(Y_{\varepsilon_{\rmh} \varepsilon_{\rmv}},f_{\varepsilon_{\rmh} \varepsilon_{\rmv}})}$ is a compatible system of boundary identifications around ${\partial^2_{\varepsilon_{\rmh} \varepsilon_{\rmv}} X}$, then a \textit{compatible system of collars} for ${(Y_{\varepsilon_{\rmh} \varepsilon_{\rmv}},f_{\varepsilon_{\rmh} \varepsilon_{\rmv}})}$ is a functor ${F_{\varepsilon_{\rmh} \varepsilon_{\rmv}} : \two \rightarrow \Man}$ mapping every object ${a = (a_{\rmh},a_{\rmv})}$ of ${\two}$ to the manifold with faces 
 \[
  \R_{\varepsilon_{\rmv}^*}^{1-a_{\rmv}} \times Y_{\varepsilon_{\rmh} \varepsilon_{\rmv}}(a) \times \R_{\varepsilon_{\rmh}^*}^{1-a_{\rmh}},
 \]
 and mapping every morphism ${a \leqslant b : a \rightarrow b}$ of ${\two}$ to a positive embedding
 \[
  F_{\varepsilon_{\rmh} \varepsilon_{\rmv}}(a \leqslant b) : \R_{\varepsilon_{\rmv}^*}^{1-a_{\rmv}} \times Y_{\varepsilon_{\rmh} \varepsilon_{\rmv}}(a) \times \R_{\varepsilon_{\rmh}^*}^{1-a_{\rmh}} \hookrightarrow \R_{\varepsilon_{\rmv}^*}^{1-b_{\rmv}} \times Y_{\varepsilon_{\rmh} \varepsilon_{\rmv}}(b) \times \R_{\varepsilon_{\rmh}^*}^{1-b_{\rmh}}
 \]
 whose restriction to ${\R_{\varepsilon_{\rmv}^*}^{1-b_{\rmv}} \times Y_{\varepsilon_{\rmh} \varepsilon_{\rmv}}(a) \times \R_{\varepsilon_{\rmh}^*}^{1-b_{\rmh}}}$ coincides with the embedding
 \[
  {\id_{\R_{\varepsilon_{\rmv}^*}^{1-b_{\rmv}}}} \times Y_{\varepsilon_{\rmh} \varepsilon_{\rmv}}(a \leqslant b) \times {\id_{\R_{\varepsilon_{\rmh}^*}^{1-b_{\rmh}}}}.
 \]
We say the embedding 
\[
 F_{\varepsilon_{\rmh} \varepsilon_{\rmv}}((a_{\rmh},a_{\rmv}) \leqslant (1,1)) : \R_{\varepsilon_{\rmv}^*}^{1-a_{\rmv}} \times Y_{\varepsilon_{\rmh} \varepsilon_{\rmv}}(a_{\rmh},a_{\rmv}) \times \R_{\varepsilon_{\rmh}^*}^{1-a_{\rmh}} \hookrightarrow Y_{\varepsilon_{\rmh} \varepsilon_{\rmv}}(1,1)
\]
is a \textit{collar} for the \textit{boundary identification} 
\[
 (f_{\varepsilon_{\rmh} \varepsilon_{\rmv}})_{(a_{\rmh},a_{\rmv})} : Y_{\varepsilon_{\rmh} \varepsilon_{\rmv}}(a_{\rmh},a_{\rmv}) \hookrightarrow X_{\varepsilon_{\rmh} \varepsilon_{\rmv}}(a_{\rmh},a_{\rmv}).
\]
Remark that every compatible system of boundary identifications admits a compatible system of collars. Indeed, if ${(Y_{\varepsilon_{\rmh} \varepsilon_{\rmv}},f_{\varepsilon_{\rmh} \varepsilon_{\rmv}})}$ is a compatible system of boundary identifications around ${\partial^2_{\varepsilon_{\rmh} \varepsilon_{\rmv}} X}$, and if ${C_{\varepsilon_{\rmh} \varepsilon_{\rmv}}}$ is the functor given by Lemma \ref{L:existence_of_collars}, then for every morphism ${a \leqslant b : a \rightarrow b}$ of ${\two}$ we can define the embedding ${F_{\varepsilon_{\rmh} \varepsilon_{\rmv}}(a \leqslant b)}$ as
\[
 \left( {\id_{\R_{\varepsilon_{\rmv}^*}^{1-b_{\rmv}}}} \times f^{-1}_{\varepsilon_{\rmh} \varepsilon_{\rmv}}(b) \times {\id_{\R_{\varepsilon_{\rmh}^*}^{1-b_{\rmh}}}} \right) \circ 
 C_{\varepsilon_{\rmh} \varepsilon_{\rmv}}(a \leqslant b) \circ 
 \left( {\id_{\R_{\varepsilon_{\rmv}^*}^{1-a_{\rmv}}}} \times f_{\varepsilon_{\rmh} \varepsilon_{\rmv}}(a) \times {\id_{\R_{\varepsilon_{\rmh}^*}^{1-a_{\rmh}}}} \right).
\]

\section{Gluing}\label{S:gluing}

We are now ready to define horizontal and vertical gluing operations. Let ${X}$ and ${X'}$ be ${m}$-di\-men\-sion\-al ${\two}$-man\-i\-folds, let ${Y^{\rmv}}$ be an ${m-1}$-di\-men\-sion\-al vertical manifold with boundary, let ${Y^{\rmh}}$ be an ${m-1}$-di\-men\-sion\-al horizontal manifold with boundary, let 
\[
 f^{\rmv} : Y^{\rmv} \rightarrow \partial^{\rmv}_+ X, \quad 
 {f'}^{\rmv} : Y^{\rmv} \rightarrow \partial^{\rmv}_- X', \quad
 f^{\rmh} : Y^{\rmh} \rightarrow \partial^{\rmh}_+ X, \quad 
 {f'}^{\rmh} : Y^{\rmh} \rightarrow \partial^{\rmv}_- X 
\]
be boundary identifications, and let 
\begin{gather*}
 F^{\rmv} : \R_- \times Y^{\rmv} \hookrightarrow X, \quad 
 {F'}^{\rmv} : \R_+ \times Y^{\rmv} \hookrightarrow X', \\
 F^{\rmh} : Y^{\rmh} \times \R_- \hookrightarrow X, \quad 
 {F'}^{\rmh} : Y^{\rmh} \times \R_+ \hookrightarrow X' 
\end{gather*}
be collars for ${f^{\rmv}}$, ${{f'}^{\rmv}}$, ${f^{\rmh}}$, and ${{f'}^{\rmh}}$. Then we denote with ${X \cup_{Y^{\rmv}} X'}$ and with ${X \cup_{Y^{\rmh}} X'}$ the ${m}$-di\-men\-sion\-al ${\two}$-man\-i\-folds obtained by gluing horizontally ${X}$ to ${X'}$ along ${Y^{\rmv}}$ using the gluing data ${(f^{\rmv},{f'}^{\rmv},F^{\rmv},{F'}^{\rmv})}$, and by gluing vertically ${X}$ to ${X'}$ along ${Y^{\rmh}}$ using the gluing data ${(f^{\rmh},{f'}^{\rmh},F^{\rmh},{F'}^{\rmh})}$. More precisely, ${X \cup_{Y^{\rmv}} X'}$ and ${X \cup_{Y^{\rmh}} X'}$ are obtained, as topological manifolds, from the disjoint union
\[
 X \sqcup X' := (\{ -1 \} \times X) \cup (\{ +1 \} \times X')
\]
by identifying ${\{ -1 \} \times \partial^{\rmv}_+ X}$ to ${\{ +1 \} \times \partial^{\rmv}_- X'}$ using ${{f'}^{\rmv} \circ (f^{\rmv})^{-1}}$, and by identifying ${\{ -1 \} \times \partial^{\rmh}_+ X}$ to ${\{ +1 \} \times \partial^{\rmh}_- X'}$ using ${{f'}^{\rmh} \circ (f^{\rmh})^{-1}}$. These quotient spaces come with natural topological embeddings 
\begin{gather*}
 i_{X}^{\rmv} : X \hookrightarrow X \cup_{Y^{\rmv}} X', \quad i_{X'}^{\rmv} : X' \hookrightarrow X \cup_{Y^{\rmv}} X', \\
 i_{X}^{\rmh} : X \hookrightarrow X \cup_{Y^{\rmh}} X', \quad i_{X'}^{\rmh} : X' \hookrightarrow X \cup_{Y^{\rmh}} X'.
\end{gather*}
Their smooth structures are the only ones making the restrictions 
\begin{gather*}
 \restr{i_{X}^{\rmv}}{X \smallsetminus \partial^{\rmv}_+ X} : X \smallsetminus \partial^{\rmv}_+ X \hookrightarrow X \cup_{Y^{\rmv}} X', \quad
 \restr{i_{X'}^{\rmv}}{X' \smallsetminus \partial^{\rmv}_- X'} : X' \smallsetminus \partial^{\rmv}_- X' \hookrightarrow X \cup_{Y^{\rmv}} X', \\
 \restr{i_{X}^{\rmh}}{X \smallsetminus \partial^{\rmh}_+ X} : X \smallsetminus \partial^{\rmh}_+ X \hookrightarrow X \cup_{Y^{\rmh}} X', \quad
 \restr{i_{X'}^{\rmh}}{X' \smallsetminus \partial^{\rmh}_- X'} : X' \smallsetminus \partial^{\rmh}_- X' \hookrightarrow X \cup_{Y^{\rmh}} X'
\end{gather*}
and the maps
\begin{gather*}
 \begin{array}{rccc}
  F^{\rmv} \cup_{Y^{\rmv}} {F'}^{\rmv} : & \R \times Y^{\rmv} & \rightarrow & X \cup_{Y^{\rmv}} X' \\
  & (t,y^{\rmv}) & \rightarrow & \begin{cases}
                                  [-1,F^{\rmv}(t,y^{\rmv})] & t \leq 0 \\
                                  [+1,{F'}^{\rmv}(t,y^{\rmv})] & t \geq 0
                                 \end{cases}
 \end{array} \\
 \begin{array}{rccc}
  F^{\rmh} \cup_{Y^{\rmh}} {F'}^{\rmh} : & Y^{\rmh} \times \R & \rightarrow & X \cup_{Y^{\rmh}} X' \\
  & (y^{\rmh},t) & \rightarrow & \begin{cases}
                                  [-1,F^{\rmh}(y^{\rmh},t)] & t \leq 0 \\
                                  [+1,{F'}^{\rmh}(y^{\rmh},t)] & t \geq 0
                                 \end{cases}
 \end{array}
\end{gather*}
into smooth embeddings. The boundary of ${X \cup_{Y^{\rmv}} X'}$ is decomposed as
\begin{gather*}
 \partial^{\rmh}_{\varepsilon_{\rmh}} (X \cup_{Y^{\rmv}} X') := \partial^{\rmh}_{\varepsilon_{\rmh}} X \cup_{\partial_{\varepsilon_{\rmh}} Y^{\rmv}} \partial^{\rmh}_{\varepsilon_{\rmh}} X', \\
 \partial^{\rmv}_- (X \cup_{Y^{\rmv}} X') := i_{X}^{\rmv}(\partial^{\rmv}_- X), \quad 
 \partial^{\rmv}_+ (X \cup_{Y^{\rmv}} X') := i_{X'}^{\rmv}(\partial^{\rmv}_+ X')
\end{gather*}
with ${\varepsilon_{\rmh} \in \{ +,- \}}$, and the boundary of ${X \cup_{Y^{\rmh}} X'}$ is decomposed as
\begin{gather*}
 \partial^{\rmh}_- (X \cup_{Y^{\rmh}} X') := i_{X}^{\rmh}(\partial^{\rmh}_- X), \quad 
 \partial^{\rmh}_+ (X \cup_{Y^{\rmh}} X') := i_{X'}^{\rmh}(\partial^{\rmh}_+ X'), \\
 \partial^{\rmv}_{\varepsilon_{\rmv}} (X \cup_{Y^{\rmh}} X') := \partial^{\rmv}_{\varepsilon_{\rmv}} X \cup_{\partial_{\varepsilon_{\rmv}} Y^{\rmh}} \partial^{\rmv}_{\varepsilon_{\rmv}} X'
\end{gather*}
with ${\varepsilon_{\rmv} \in \{ +,- \}}$. In the following, the embeddings ${i^{\rmv}_X}$, ${i^{\rmv}_{X'}}$, ${i^{\rmh}_X}$, and ${i^{\rmh}_X}$ will be always suppressed from the notation. Furthermore, whenever we will have ${\two}$-man\-i\-folds ${X''}$ and ${X'''}$, and isomorphisms ${g : X \rightarrow X''}$ and ${g' : X' \rightarrow X'''}$, then we will denote with ${X'' \cup_{Y^{\rmv}} X'''}$ and with ${X'' \cup_{Y^{\rmh}} X'''}$ the ${\two}$-man\-i\-folds determined by the gluing datas ${(g \circ f^{\rmv},g \circ f'^{\rmv},g \circ F'^{\rmv},g' \circ F'^{\rmv})}$ and ${(g \circ f^{\rmh},g \circ f'^{\rmh},g \circ F'^{\rmh},g' \circ F'^{\rmh})}$, and we will denote with 
\[
 g \cup_{Y^{\rmv}} g' : X \cup_{Y^{\rmv}} X' \rightarrow X'' \cup_{Y^{\rmv}} X''', \quad g \cup_{Y^{\rmh}} g' : X \cup_{Y^{\rmh}} X' \rightarrow X'' \cup_{Y^{\rmh}} X'''
\]
the isomorphisms induced by ${g}$ and ${g'}$.

\section{Cobordisms with corners}\label{S:cobordisms}

This section provides a reference for the notation we use for cobordisms. All the manifolds considered here are equipped with a smooth structure with possibly empty corners, and we recall that, according to our conventions, this means an orientation has been fixed. We also recall that \textit{horizontal} and \textit{vertical} manifolds with boundary are manifolds with boundary equipped with a ${\two}$-man\-i\-fold structure, which in this case simply amounts to a partition of the set of connected components of the boundary, leading to a distinction between incoming and outgoing boundary.

\begin{definition}\label{D:cobordism}
 If ${\varGamma}$ and ${\varGamma'}$ are ${d-2}$-di\-men\-sion\-al closed manifolds, then a \textit{${d-1}$-di\-men\-sion\-al cobordism ${\varSigma}$ from ${\varGamma}$ to ${\varGamma'}$} is a 5-tuple 
 \[
  (\varSigma,f_{\varSigma_-},f_{\varSigma_+},F_{\varSigma_-},F_{\varSigma_+})
 \]
 where: 
 \begin{enumerate}
  \item ${\varSigma}$ is a ${d-1}$-di\-men\-sion\-al compact horizontal manifold with boundary, called the \textit{support};
  \item ${f_{\varSigma_-} : \varGamma \rightarrow \partial_- \varSigma}$ and ${f_{\varSigma_+} : \varGamma' \rightarrow \partial_+ \varSigma}$ are positive diffeomorphisms called the \textit{incoming} and the \textit{outgoing boundary identification} respectively;
  \item ${F_{\varSigma_-} : \R_+ \times \varGamma \hookrightarrow \varSigma}$ and ${F_{\varSigma_+} : \R_- \times \varGamma' \hookrightarrow \varSigma}$ are collars for ${f_{\varSigma_+}}$ and ${f_{\varSigma_-}}$ respectively.
 \end{enumerate}
\end{definition}

We always suppress boundary identifications and collars from the notation. Two cobordisms ${\varSigma}$ and ${\varSigma'}$ from ${\varGamma}$ to ${\varGamma'}$ are \textit{isomorphic} if there exists a positive diffeomorphism ${f : \varSigma \rightarrow \varSigma'}$ such that:
\begin{enumerate}
 \item ${f \circ f_{\varSigma_-} = f_{\varSigma'_-}}$ and ${f \circ f_{\varSigma_+} = f_{\varSigma'_+}}$;
 \item ${f \circ F_{\varSigma_-}}$ agrees with ${F_{\varSigma'_-}}$ in a neighborhood of ${\{ 0 \} \times \varGamma}$;
 \item ${f \circ F_{\varSigma_+}}$ agrees with ${F_{\varSigma'_+}}$ in a neighborhood of ${\{ 0 \} \times \varGamma'}$.
\end{enumerate}
Such a diffeomorphism ${f}$ is called an \textit{isomorphism of cobordisms}.

\begin{remark}\label{R:trivial_cobordism}
 Every ${d-2}$-di\-men\-sion\-al closed manifold ${\varGamma}$ determines a ${d-1}$-di\-men\-sion\-al cobordism ${I \times \varGamma}$ from ${\varGamma}$ to itself called the \textit{trivial cobordism on ${\varGamma}$}, which is given by
 \[
  (I \times \varGamma,(0,\id_{\varGamma}),(1,\id_{\varGamma}),F_{I_-} \times {\id_{\varGamma}},F_{I_+} \times {\id_{\varGamma}}),
 \]
 where ${(0,\id_{\varGamma}) : \varGamma \rightarrow \{ 0 \} \times \varGamma}$ and ${(1,\id_{\varGamma}) : \varGamma \rightarrow \{ 1 \} \times \varGamma}$ are the obvious diffeomorphisms induced by $\id_{\varGamma}$, and where ${F_{I_-} : \R_+ \hookrightarrow I}$ and ${F_{I_+} : \R_- \hookrightarrow I}$ are some embeddings we fix here once and for all satisfying:
 \begin{enumerate}
  \item ${F_{I_-}(t_+) = t_+}$ for every ${t_+ \leq \frac{1}{3}}$;
  \item ${F_{I_+}(t_-) = 1 + t_-}$ for every ${t_- \geq -\frac{1}{3}}$.
 \end{enumerate}
\end{remark}


Figure \ref{F:trivial_cobordism_over_circle} represents the trivial cobordism ${I \times S^1}$ on the circle ${S^1}$, together with its boundary identifications $f_{(I \times S^1)_-} = (0,\id_{S^1})$ and $f_{(I \times S^1)_+} = (1,\id_{S^1})$.

\begin{figure}[htb]\label{F:trivial_cobordism_over_circle}
 \centering
 \includegraphics{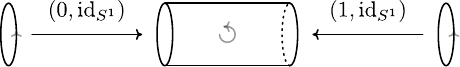}
 \caption{Boundary identifications for the cobordism ${I \times S^1}$.}
\end{figure}

\begin{definition}\label{D:cobordism_with_corners}
 If ${\varSigma}$ and ${\varSigma'}$ are ${d-1}$-di\-men\-sion\-al cobordisms from ${\varGamma}$ to ${\varGamma'}$, then a \textit{$d$-di\-men\-sion\-al cobordism with corners ${M}$ from ${\varSigma}$ to ${\varSigma'}$} is a 9-tuple 
 \[
  (M,f_{M^{\rmh}_-},f_{M^{\rmh}_+},f_{M^{\rmv}_-},f_{M^{\rmv}_+},F_{M^{\rmh}_-},F_{M^{\rmh}_+},F_{M^{\rmv}_-},F_{M^{\rmv}_+}),
 \]
 where:
 \begin{enumerate}
  \item ${M}$ is a ${d}$-di\-men\-sion\-al compact ${\two}$-man\-i\-fold, called the \textit{support};
  \item $f_{M^{\rmh}_-} : \varSigma \rightarrow \partial^{\rmh}_- M$ and $f_{M^{\rmh}_+} : \varSigma' \rightarrow \partial^{\rmh}_+ M$ are positive diffeomorphisms called the \textit{incoming} and the \textit{outgoing horizontal boundary identification}, and $f_{M^{\rmv}_-} : \varGamma \times I \rightarrow \partial^{\rmv}_- M$ and $f_{M^{\rmv}_+} : \varGamma' \times I \rightarrow \partial^{\rmv}_+ M$ are positive diffeomorphisms called the \textit{incoming} and the \textit{outgoing vertical boundary identification} respectively;
  \item ${\smash{F_{M^{\rmh}_-} : \varSigma \times \R_+ \hookrightarrow M}}$, ${\smash{\ \ F_{M^{\rmh}_+} : \varSigma' \times \R_- \hookrightarrow M}}$, ${\smash{\ \ F_{M^{\rmv}_-} : \R_- \times \varGamma \times I \hookrightarrow M}}$, and ${\smash{F_{M^{\rmv}_+} : \R_+ \times \varGamma' \times I \hookrightarrow M}}$ are collars for ${f_{M^{\rmh}_-}}$, ${f_{M^{\rmh}_+}}$, ${f_{M^{\rmv}_-}}$, and ${f_{M^{\rmv}_+}}$ respectively.
 \end{enumerate}
 These data satisfy the following conditions, using the notation of Remark \ref{R:trivial_cobordism}:
 \begin{enumerate}
  \item The assignment
   \begin{gather*}
    Y_{--}(0,0) := \varGamma, \quad Y_{--}(0,1) := \varSigma, \quad Y_{--}(1,0) := \varGamma \times I, \quad Y_{--}(1,1) := M \\
    Y_{--}((0,0) \leqslant (0,1)) := f_{\varSigma_-}, \quad 
    Y_{--}((0,0) \leqslant (1,0)) := (\id_{\varGamma},0), \\
    \left( f_{--} \right)_{(0,1)} := f_{M^{\rmh}_-}, \quad
    \left( f_{--} \right)_{(1,0)} := f_{M^{\rmv}_-}, \quad 
    \left( f_{--} \right)_{(1,1)} := \id_M, 
   \end{gather*}
   can be completed to a compatible system of boundary identifications ${(Y_{--},f_{--})}$ around ${\partial^2_{--} M}$, and the assignment
   \begin{gather*}
    F_{--}((0,0) \leqslant (0,1)) := F_{\varSigma_-} \times {\id_{\R_+}}, \quad 
    F_{--}((0,0) \leqslant (1,0)) := {\id_{\R_+}} \times {\id_{\varGamma}} \times F_{I_-}, \\
    F_{--}((0,1) \leqslant (1,1)) := F_{M^{\rmh}_-}, \quad 
    F_{--}((1,0) \leqslant (1,1)) := F_{M^{\rmv}_-}
   \end{gather*}
   can be completed to a compatible system of collars ${F_{--}}$ for ${(Y_{--},f_{--})}$;
  \item The assignment
   \begin{gather*}
    Y_{-+}(0,0) := \varGamma', \quad Y_{-+}(0,1) := \varSigma, \quad Y_{-+}(1,0) := \varGamma' \times I, \quad Y_{-+}(1,1) := M \\
    Y_{-+}((0,0) \leqslant (0,1)) := f_{\varSigma_+}, \quad 
    Y_{-+}((0,0) \leqslant (1,0)) := (\id_{\varGamma'},0), \\
    \left( f_{-+} \right)_{(0,1)} := f_{M^{\rmh}_-}, \quad
    \left( f_{-+} \right)_{(1,0)} := f_{M^{\rmv}_+}, \quad 
    \left( f_{-+} \right)_{(1,1)} := \id_M, 
   \end{gather*}
   can be completed to a compatible system of boundary identifications ${(Y_{-+},f_{-+})}$ around ${\partial^2_{-+} M}$, and the assignment
   \begin{gather*}
    F_{-+}((0,0) \leqslant (0,1)) := F_{\varSigma_+} \times {\id_{\R_+}}, \quad 
    F_{-+}((0,0) \leqslant (1,0)) := {\id_{\R_-}} \times {\id_{\varGamma'}} \times F_{I_-}, \\
    F_{-+}((0,1) \leqslant (1,1)) := F_{M^{\rmh}_-}, \quad 
    F_{-+}((1,0) \leqslant (1,1)) := F_{M^{\rmv}_+}
   \end{gather*}
   can be completed to a compatible system of collars ${F_{-+}}$ for ${(Y_{-+},f_{-+})}$;
  \item The assignment
   \begin{gather*}
    Y_{+-}(0,0) := \varGamma, \quad Y_{+-}(0,1) := \varSigma', \quad Y_{+-}(1,0) := \varGamma \times I, \quad Y_{+-}(1,1) := M \\
    Y_{+-}((0,0) \leqslant (0,1)) := f_{\varSigma'_-}, \quad 
    Y_{+-}((0,0) \leqslant (1,0)) := (\id_{\varGamma},1), \\
    \left( f_{+-} \right)_{(0,1)} := f_{M^{\rmh}_+}, \quad
    \left( f_{+-} \right)_{(1,0)} := f_{M^{\rmv}_-}, \quad 
    \left( f_{+-} \right)_{(1,1)} := \id_M, 
   \end{gather*}
   can be completed to a compatible system of boundary identifications ${(Y_{+-},f_{+-})}$ around ${\partial^2_{+-} M}$, and the assignment
   \begin{gather*}
    F_{+-}((0,0) \leqslant (0,1)) := F_{\varSigma'_-} \times {\id_{\R_-}}, \quad 
    F_{+-}((0,0) \leqslant (1,0)) := {\id_{\R_+}} \times {\id_{\varGamma}} \times F_{I_+}, \\
    F_{+-}((0,1) \leqslant (1,1)) := F_{M^{\rmh}_+}, \quad 
    F_{+-}((1,0) \leqslant (1,1)) := F_{M^{\rmv}_-}
   \end{gather*}
   can be completed to a compatible system of collars ${F_{+-}}$ for ${(Y_{+-},f_{+-})}$;
  \item The assignment
   \begin{gather*}
    Y_{++}(0,0) := \varGamma', \quad Y_{++}(0,1) := \varSigma', \quad Y_{++}(1,0) := \varGamma' \times I, \quad Y_{++}(1,1) := M \\
    Y_{++}((0,0) \leqslant (0,1)) := f_{\varSigma'_+}, \quad 
    Y_{++}((0,0) \leqslant (1,0)) := (\id_{\varGamma'},1), \\
    \left( f_{++} \right)_{(0,1)} := f_{M^{\rmh}_+}, \quad
    \left( f_{++} \right)_{(1,0)} := f_{M^{\rmv}_+}, \quad 
    \left( f_{++} \right)_{(1,1)} := \id_M, 
   \end{gather*}
   can be completed to a compatible system of boundary identifications ${(Y_{++},f_{++})}$ around ${\partial^2_{++} M}$, and the assignment
   \begin{gather*}
    F_{++}((0,0) \leqslant (0,1)) := F_{\varSigma'_+} \times {\id_{\R_-}}, \quad 
    F_{++}((0,0) \leqslant (1,0)) := {\id_{\R_-}} \times {\id_{\varGamma'}} \times F_{I_+}, \\
    F_{++}((0,1) \leqslant (1,1)) := F_{M^{\rmh}_+}, \quad 
    F_{++}((1,0) \leqslant (1,1)) := F_{M^{\rmv}_+}
   \end{gather*}
   can be completed to a compatible system of collars ${F_{++}}$ for ${(Y_{++},f_{++})}$.
 \end{enumerate}
\end{definition}

As for cobordisms, we always suppress boundary identifications and collars from the notation. Two cobordisms with corners ${M}$ and ${M'}$ from ${\varSigma}$ to ${\varSigma'}$ are \textit{isomorphic} if there exists an isomorphism of oriented ${\two}$-man\-i\-folds ${f : M \rightarrow M'}$ such that:
\begin{enumerate}
 \item ${f \circ f_{M^{\rmh}_-} = f_{M'^{\rmh}_-}}$ and ${f \circ f_{M^{\rmh}_+} = f_{M'^{\rmh}_+}}$;
 \item ${f \circ f_{M^{\rmv}_-} = f_{M'^{\rmv}_-}}$ and ${f \circ f_{M^{\rmv}_+} = f_{M'^{\rmv}_+}}$;
 \item ${f \circ F_{M_{--}}}$ and ${f \circ F_{M_{+-}}}$ agree with ${F_{M'_{--}}}$ and ${F_{M'_{+-}}}$ in a neighborhood of ${\{ 0 \} \times \varGamma \times \{ 0 \}}$;
 \item ${f \circ F_{M_{-+}}}$ and ${f \circ F_{M_{++}}}$ agree with ${F_{M'_{+-}}}$ and ${F_{M'_{++}}}$ in a neighborhood of ${\{ 0 \} \times \varGamma' \times \{ 0 \}}$.
\end{enumerate}
Such an ${f}$ is called an \textit{isomorphism of cobordisms with corners}.

\begin{remark}\label{R:trivial_cobordism_with_corners}
 Every ${d-1}$-di\-men\-sion\-al cobordism ${\varSigma}$ from ${\varGamma}$ to ${\varGamma'}$ determines a ${d}$-di\-men\-sion\-al cobordism with corners ${\varSigma \times I}$ from ${\varSigma}$ to itself called the \textit{trivial cobordism with corners on ${\varSigma}$}, which is given by
 \begin{gather*}
  (\varSigma \times I,(\id_{\varSigma},0),(\id_{\varSigma},1),f_{\varSigma_-} \times {\id_I},f_{\varSigma_+} \times {\id_I}, \\
  {\id_{\varSigma}} \times F_{I_-},{\id_{\varSigma}} \times F_{I_+},F_{\varSigma_-} \times {\id_I},F_{\varSigma_+} \times {\id_I}),
 \end{gather*}
 where again we are using the notation of Remark \ref{R:trivial_cobordism}.
\end{remark}

Figure \ref{F:trivial_cobordism_with_corners_over_annulus} represents the trivial cobordism with corners ${I \times S^1 \times I}$ on the annulus ${I \times S^1}$, which is itself the trivial cobordism on the circle ${S^1}$, together with its boundary identifications $f_{(I \times S^1 \times I)^{\rmh}_-} = (\id_{I \times S_1},0)$, $f_{(I \times S^1 \times I)^{\rmh}_+} = (\id_{I \times S_1},1)$, $f_{(I \times S^1 \times I)^{\rmv}_-} = (0,\id_{S_1 \times I})$, and $f_{(I \times S^1 \times I)^{\rmv}_+} = (1,\id_{S_1 \times I})$.

\begin{figure}[htb]\label{F:trivial_cobordism_with_corners_over_annulus}
 \centering
 \includegraphics{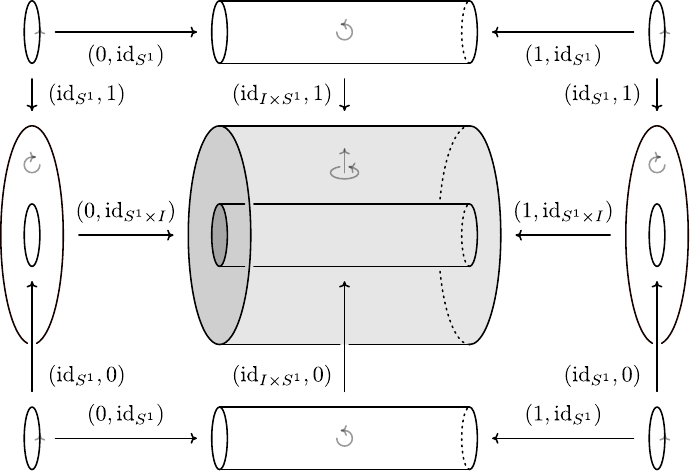}
 \caption{Boundary identifications for the cobordism with corners ${I \times S^1 \times I}$.}
\end{figure}

\begin{definition}\label{D:gluing}
 If ${\varSigma}$ is a ${d-1}$-di\-men\-sion\-al cobordism from ${\varGamma}$ to ${\varGamma'}$, and if ${\varSigma'}$ is a ${d-1}$-di\-men\-sion\-al cobordism from ${\varGamma'}$ to ${\varGamma''}$, then the \textit{gluing of ${\varSigma}$ to ${\varSigma'}$ along ${\varGamma}$} is the ${d-1}$-di\-men\-sion\-al cobordism ${\varSigma \cup_{\varGamma'} \varSigma'}$ from ${\varGamma}$ to ${\varGamma''}$ given by
 \[
  (\varSigma \cup_{\varGamma'} \varSigma',f_{\varSigma_-},f_{\varSigma'_+},F_{\varSigma_-},F_{\varSigma'_+}) 
 \]
 with gluing data ${(f_{\varSigma_+},f_{\varSigma'_-},F_{\varSigma_+},F_{\varSigma'_-})}$.
\end{definition}


\begin{definition}\label{D:horizontal_gluing}
 If ${\varSigma}$ and ${\varSigma''}$ are ${d-1}$-di\-men\-sion\-al cobordisms from ${\varGamma}$ to ${\varGamma'}$, if ${\varSigma'}$ and ${\varSigma'''}$ are ${d-1}$-di\-men\-sion\-al cobordisms from ${\varGamma'}$ to ${\varGamma''}$, if ${M}$ is a ${d}$-di\-men\-sion\-al cobordism with corners from ${\varSigma}$ to ${\varSigma''}$, and if 
 ${M'}$ is a ${d}$-di\-men\-sion\-al cobordism with corners from ${\varSigma'}$ to ${\varSigma'''}$, then the \textit{horizontal gluing of ${M}$ to ${M'}$ along ${\varGamma' \times I}$} is the ${d}$-di\-men\-sion\-al cobordism with corners ${M \cup_{\varGamma' \times I} M'}$ from ${\varSigma \cup_{\varGamma'} \varSigma'}$ to ${\varSigma'' \cup_{\varGamma'} \varSigma'''}$ given by
 \begin{gather*}
  (M \cup_{\varGamma' \times I} M',f_{M^{\rmh}_-} \cup_{\varGamma'} f_{M'^{\rmh}_-}, f_{M^{\rmh}_+} \cup_{\varGamma'} f_{M'^{\rmh}_+}, f_{M^{\rmv}_-}, f_{M'^{\rmv}_+}, \\
  F_{M^{\rmh}_-} \cup_{\varGamma' \times \R_+} F_{M'^{\rmh}_-}, F_{M^{\rmh}_+} \cup_{\varGamma' \times \R_-} F_{M'^{\rmh}_+}, F_{M^{\rmv}_-}, F_{M'^{\rmv}_+})
 \end{gather*}
 with gluing data ${(f_{M^{\rmv}_+},f_{M'^{\rmv}_-},F_{M^{\rmv}_+},F_{M'^{\rmv}_-})}$.
\end{definition}

In order to define vertical gluings, we need to make a choice. Indeed, vertical boundary identifications are more rigid than horizontal ones, because the form of their sources is fixed. This means we need to specify how to identify the gluing of two copies of a trivial vertical cobordism with a single copy. To this end, let us fix here once and for all a diffeomorphism ${u : I \rightarrow I}$ satisfying:
\begin{enumerate}
 \item ${u(t) = \frac{t}{2}}$ for all ${t \leq \frac{1}{4}}$;
 \item ${u(t) = \frac{t+1}{2}}$ for all ${t \geq \frac{3}{4}}$; 
 \item ${u(\frac 12) = \frac 12}$.
\end{enumerate}
Then, if ${X}$ is a horizontal manifold with boundary, and if ${(X \times I) \cup_X (X \times I)}$ is the ${\two}$-man\-i\-fold otained by gluing vertically ${(X \times I)}$ to itself along ${X}$ using the gluing data ${((\id_X,1),(\id_X,0),F_{I_+},F_{I_-})}$, we denote with ${u_X}$ the isomorphism
\[
 \begin{array}{rccc}
  u_X : & X \times I & \rightarrow & (X \times I) \cup_X (X \times I) \\
  & (x,t) & \mapsto & \begin{cases}
                       [-1,(x,2u(t))] & t \leq \frac 12, \\
                       [+1,(x,2u(t)-1)] & t \geq \frac 12.
                      \end{cases}
 \end{array}
\]

\begin{definition}\label{D:vertical_gluing}
 If ${\varSigma}$, ${\varSigma'}$, and ${\varSigma''}$ are ${d-1}$-di\-men\-sion\-al cobordisms from ${\varGamma}$ to ${\varGamma'}$, if ${M}$ is a ${d}$-di\-men\-sion\-al cobordism with corners from ${\varSigma}$ to ${\varSigma'}$, and if ${M'}$ is a ${d}$-di\-men\-sion\-al cobordism with corners from ${\varSigma'}$ to ${\varSigma''}$, then the \textit{vertical gluing of ${M}$ to ${M'}$ along ${\varSigma'}$} the ${d}$-di\-men\-sion\-al cobordism with corners ${M \cup_{\varSigma'} M'}$ from ${\varSigma}$ to ${\varSigma''}$ given by
 \begin{gather*}
  (M \cup_{\varSigma'} M',f_{M^{\rmh}_-},f_{M'^{\rmh}_+},
  (f_{M^{\rmv}_-} \cup_{\varGamma} f_{M'^{\rmv}_-}) \circ u_{\varGamma},
  (f_{M^{\rmv}_+} \cup_{\varGamma'} f_{M'^{\rmv}_+}) \circ u_{\varGamma'}, \\
  F_{M^{\rmh}_-},F_{M'^{\rmh}_+},
  (F_{M^{\rmv}_-} \cup_{\R_+ \times \varGamma} F_{M'^{\rmv}_-}) \circ u_{\R_+ \times \varGamma},
  (F_{M^{\rmv}_+} \cup_{\R_- \times \varGamma'} F_{M'^{\rmv}_+}) \circ u_{\R_- \times \varGamma'})
 \end{gather*}
 with gluing data ${(f_{M^{\rmh}_+}, f_{M'^{\rmh}_-}, F_{M^{\rmh}_+}, F_{M'^{\rmh}_-})}$.
\end{definition}

The operation of disjoint union extends very naturally to cobordisms and cobordisms with corners, and the reader can easily figure out how to obtain boundary identifications and collars from those of components. However, we point out a small detail which is essentially negligible until the introduction of gradings: when ${X}$ and ${X'}$ are topological spaces, then their disjoint unions 
\begin{gather*}
 X \sqcup X' = (\{ -1 \} \times X) \cup (\{ +1 \} \times X'), \\
 X' \sqcup X = (\{ -1 \} \times X') \cup (\{ +1 \} \times X)\phantom{,}
\end{gather*}
are actually given by different spaces, unless ${X = X'}$. Of course, we have a natural map 
\[
 \begin{array}{rccc}
  \tau_{X,X'} : & X \sqcup X' & \rightarrow & X' \sqcup X \\
  & (-1,x) & \mapsto & (+1,x), \\
  & (+1,x') & \mapsto & (-1,x')
 \end{array}
\]
which realizes a homeomorphism, but the associated cobordisms reversing the order of disjoint unions do play a role, although a minor one, in our construction.

\begin{definition}\label{D:flip_cobordism}
 If ${\varGamma}$ and ${\varGamma'}$ are ${d-2}$-di\-men\-sion\-al closed manifolds, then we denote with ${I \ttimes (\varGamma \sqcup \varGamma')}$ the ${d-1}$-di\-men\-sion\-al cobordism from ${\varGamma \sqcup \varGamma'}$ to ${\varGamma' \sqcup \varGamma}$ given by
 \[
  (I \times (\varGamma \sqcup \varGamma'),(0,\id_{\varGamma \sqcup \varGamma'}),(1,\tau_{\varGamma',\varGamma}),F_{I_-} \times {\id_{\varGamma \sqcup \varGamma'}},F_{I_+} \times \tau_{\varGamma',\varGamma}).
 \]
\end{definition}

Figure \ref{F:flip_cobordism} represents schematically the cobordism ${I \ttimes (\varGamma \sqcup \varGamma')}$ together with its boundary identifications.

\begin{figure}[htb]\label{F:flip_cobordism}
 \centering
 \includegraphics{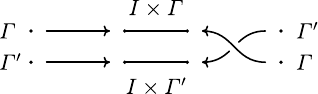}
 \caption{Schematic representation of the cobordism ${I \ttimes (\varGamma \sqcup \varGamma')}$.}
\end{figure}

Analogously, we have cobordisms with corners reversing the order of disjoint unions.

\begin{definition}\label{D:flip_cobordism_with_corners}
 If $\varSigma$ is a $d-1$-di\-men\-sion\-al cobordism from $\varGamma$ to $\varGamma''$, and if $\varSigma'$ is a $d-1$-di\-men\-sion\-al cobordism from $\varGamma'$ to $\varGamma'''$, then we denote with $(\varSigma \sqcup \varSigma') \ttimes I$ the $d$-di\-men\-sion\-al cobordism from $\varSigma \tsqcup \varSigma'$ to $\varSigma' \tsqcup \varSigma$ given by 
 \begin{gather*}
  ((\varSigma \sqcup \varSigma') \times I, (\id_{\varSigma \sqcup \varSigma'},0), (\tau_{\varSigma',\varSigma},1), f_{(\varSigma \sqcup \varSigma')_-} \times {\id_I}, ( f_{(\varSigma \sqcup \varSigma')_+} \circ \tau_{\varGamma''',\varGamma''} ) \times {\id_I}, \\  
  (\id_{\varSigma \sqcup \varSigma'},F_{I_-}), (\tau_{\varSigma',\varSigma},F_{I_+}), (F_{(\varSigma \sqcup \varSigma')_-} \times {\id_I}, ( F_{(\varSigma \sqcup \varSigma')_+} \circ \tau_{\varGamma''',\varGamma''} ) \times {\id_I} ),
 \end{gather*}
 where $\varSigma \tsqcup \varSigma'$ denotes the $d-1$-di\-men\-sion\-al cobordism from $\varGamma \sqcup \varGamma'$ to $\varGamma''' \sqcup \varGamma''$ given by
 \[
  (\varSigma \sqcup \varSigma',f_{(\varSigma \sqcup \varSigma')_-}, f_{(\varSigma \sqcup \varSigma')_+} \circ \tau_{\varGamma''',\varGamma''},F_{(\varSigma \sqcup \varSigma')_-},F_{(\varSigma \sqcup \varSigma')_+} \circ \tau_{\varGamma''',\varGamma''}),
 \]
 and where $\varSigma' \tsqcup \varSigma$ denotes the $d-1$-di\-men\-sion\-al cobordism from $\varGamma \sqcup \varGamma'$ to $\varGamma''' \sqcup \varGamma''$ given by
 \[
  (\varSigma' \sqcup \varSigma,f_{(\varSigma' \sqcup \varSigma)_-} \circ \tau_{\varGamma,\varGamma'}, f_{(\varSigma' \sqcup \varSigma)_+},F_{(\varSigma' \sqcup \varSigma)_-} \circ \tau_{\varGamma,\varGamma'},F_{(\varSigma' \sqcup \varSigma)_+}).
 \]
\end{definition}

Figure \ref{F:flip_cobordism_with_corners} represents schematically the cobordism with corners ${(\varSigma \sqcup \varSigma') \ttimes I}$ together with its boundary identifications.

\begin{figure}[htb]\label{F:flip_cobordism_with_corners}
 \centering
 \includegraphics{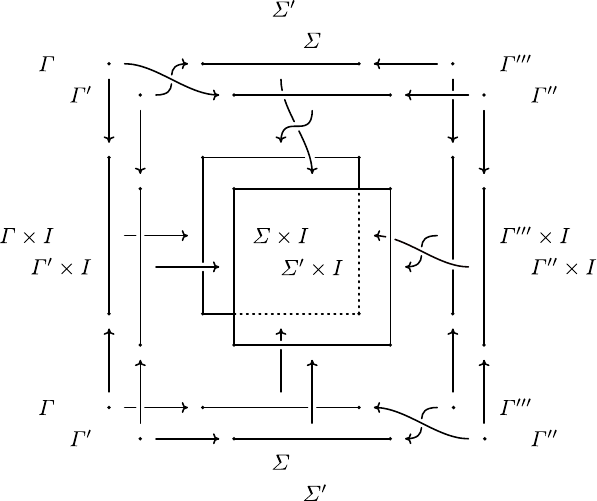}
 \caption{Schematic representation of the cobordism with corners ${(\varSigma \sqcup \varSigma') \ttimes I}$.}
\end{figure}

%% file: appendix_c.tex
\chapter{Signature defects}\label{A:Maslov}

In this appendix we discuss Lagrangian subspaces and Maslov indices in the case of surfaces with boundary. This is done in order to equip cobordisms with decorations whose role is to fix the framing anomaly associated with quantum invariants of Witten-Reshetikhin-Turaev type, which is related to signature defects studied by Wall in \cite{W69}. Throughout this chapter, the term manifold will stand for compact oriented topological manifold, and coefficients for homology and cohomology groups will always be real.

\section{Intersection pairings}

Let us consider an ${n}$-\-di\-men\-sion\-al manifold with boundary ${X}$. For every integer ${0 \leqslant k \leqslant n}$, the \textit{relative intersection pairing} ${\pitchfork_X^{\partial} : H_k(X) \times H_{n-k}(X,\partial X) \rightarrow \R}$ is defined by the formula
\[
 a \pitchfork_X^{\partial} b := \langle D_X^{-1}(b),a \rangle
\]
for all relative homology classes ${a \in H_k(X)}$ and ${b \in H_{n-k}(X,\partial X)}$, where
\[
 \begin{array}{rccc}
  D_X : & H^k(X) & \rightarrow & H_{n-k}(X,\partial X) \\
  & \varphi & \mapsto & [X] \cappr \varphi
 \end{array}
\]
is the Poincaré duality isomorphism, and where ${[X] \in H_n(X, \partial X)}$ is the fundamental class of ${X}$. Remark that, since ${D_X}$ is an isomorphism, and since the universal coefficient pairing is non-degenerate, relative intersection pairings are always non-degenerate. For every integer ${0 \leqslant k \leqslant n}$, the \textit{absolute intersection pairing} ${\pitchfork_X : H_k(X) \times H_{n-k}(X) \rightarrow \R}$ is defined by the formula
\[
 a \pitchfork_X b := a \pitchfork_X^{\partial} q_{\partial X *}(b)
\]
for all homology classes ${a \in H_k(X)}$ and ${b \in H_{n-k}(X)}$, where 
\[
 q_{\partial X *} : H_{n-k}(X) \rightarrow H_{n-k}(X,\partial X)
\]
is induced by inclusion. Remark that, since ${\pitchfork_X^{\partial}}$ is non-degenerate, we immediately have
\[
 H_k(X)^{\perp} = \ker(q_{\partial X *}) = \im(i_{\partial X *}) \subset H_{n-k}(X),
\]
where 
\[
 i_{\partial X *} : H_{n-k}(\partial X ) \rightarrow H_{n-k}(X)
\]
is induced by inclusion. Now let us fix for the rest of this section an ${n}$-\-di\-men\-sion\-al manifold with boundary ${X}$, and let us establish a few technical results about intersection pairings which will be used later in order to discuss Lagrangian subspaces.

\begin{lemma}\label{L:(anti)symmetry}
 The absolute intersection pairing ${\pitchfork_X}$ satisfies
 \[
  a \pitchfork_X b = (-1)^{k(n-k)} b \pitchfork_X a
 \]
 for all homology classes ${a \in H_k(X)}$ and ${b \in H_{n-k}(X)}$.
\end{lemma}

\begin{proof}
 Naturality of relative cap products gives 
 \[
  [X] \cappr q_{\partial X}^*(\varphi) = q_{\partial X*}([X] \cappr \varphi)
 \]
 for every relative cohomology class ${\varphi \in H^k(X,\partial X)}$. This means
 \begin{align*}
  a \pitchfork_X b 
  &= \langle D_X^{-1}(a),q_{\partial X*}(b) \rangle \\
  &= \langle q_{\partial X}^*(D_X^{-1}(a)),b \rangle \\
  &= \langle D_X^{-1}(q_{\partial X}^*(a)),b \rangle \\
  &= \langle D_X^{-1}(b) \cuppr D_X^{-1}(q_{\partial X}^*(a)),[X] \rangle \\
  &= (-1)^{k(n-k)} \langle D_X^{-1}(q_{\partial X}^*(a)) \cuppr D_X^{-1}(b),[X] \rangle \\
  &= (-1)^{k(n-k)} \langle D_X^{-1}(b),q_{\partial X*}(a) \rangle \\
  &= (-1)^{k(n-k)} b \pitchfork_X a
 \end{align*}
 for all homology classes ${a \in H_k(X)}$ and ${b \in H_{n-k}(X)}$.
\end{proof}

We move on to study how relative intersection pairings behave under boundary homomorphisms coming from long exact sequences in homology.

\begin{lemma}\label{L:boundary}
 If $i_{\partial X *} : H_k(\partial X) \rightarrow H_k(X)$ is induced by inclusion, and if $\partial_* : H_{n-k}(X,\partial X) \rightarrow H_{n-k-1}(\partial X)$ comes from the long exact sequence of the pair $(X,\partial X)$, then
 \[
  i_{\partial X *}(a) \pitchfork_X^{\partial} b = (-1)^k a \pitchfork_{\partial X} \partial_* (b)  
 \]
 for all relative homology classes ${a \in H_k(\partial X)}$ and ${b \in H_{n-k}(X,\partial X)}$.
\end{lemma}

\begin{proof}
 If we adopt the notation
 \[
  [X] \cappr \varphi = \varphi \left( \restr{[X]}{[0,\ldots,k]} \right) \cdot \restr{[X]}{[k,\ldots,n]},
 \]
 then we have
 \[
  \partial_*([X] \cappr \varphi) 
  = \sum_{i=0}^{n-k} (-1)^i \varphi \left( \restr{[X]}{[0,\ldots,k]} \right) \cdot \restr{[X]}{[k,\ldots,k \hat{+} i,\ldots,n]}.
 \]
 for every relative cohomology class ${\varphi \in H^k(X,A)}$. But since the fundamental class of ${\partial X}$ satisfies ${[\partial X] = \partial_*[X]}$, we have
 \begin{align*}
  [\partial X] \cappr i_{\partial X}^*(\varphi) 
  &= \sum_{i=0}^k (-1)^i \varphi \left( \restr{[X]}{[0,\ldots,\hat{i},\ldots,k+1]} \right) \cdot \restr{[X]}{[k+1,\ldots,n]} \\
  &+ \sum_{i=k+1}^n (-1)^i \varphi \left( \restr{[X]}{[0,\ldots,k]} \right) \cdot \restr{[X]}{[k,\ldots,\hat{i},\ldots,n]} \\
  &= (-1)^k \varphi \left( \restr{[X]}{[0,\ldots,k]} \right) \cdot \restr{[X]}{[k+1,\ldots,n]} \\
  &+ \sum_{i=k+1}^n (-1)^i \varphi \left( \restr{[X]}{[0,\ldots,k]} \right) \cdot \restr{[X]}{[k,\ldots,\hat{i},\ldots,n]} \\
  &= (-1)^k \sum_{i=0}^{n-k} (-1)^i \varphi \left( \restr{[X]}{[0,\ldots,k]} \right) \cdot \restr{[X]}{[k,\ldots,k \hat{+} i,\ldots,n]}
 \end{align*}
 where the second equality follows from
 \[
   \sum_{i=0}^{k+1} (-1)^i \varphi \left( \restr{[X]}{[0,\ldots,\hat{i},\ldots,k+1]} \right) = 0.
 \]
 In other words, the diagram
 \begin{center}
  \begin{tikzpicture}[descr/.style={fill=white}]
   \node (P0) at (135:{1.5*sqrt(2)}) {$H^k(X)$};
   \node (P1) at (45:{1.5*sqrt(2)}) {$H^k(\partial X)$};
   \node (P2) at (-135:{1.5*sqrt(2)}) {$H_{n-k}(X,\partial X)$};
   \node (P3) at (-45:{1.5*sqrt(2)}) {$H_{n-k-1}(\partial X)$};
   \draw
   (P0) edge[->] node[above] {$i_{\partial X}^*$} (P1)
   (P1) edge[->] node[right] {$(-1)^k \cdot D_{\partial X}$} (P3)
   (P0) edge[->] node[left] {$D_X$} (P2)
   (P2) edge[->] node[below] {$\partial_*$} (P3);
  \end{tikzpicture}
 \end{center}
 is commutative. This gives
 \begin{align*}
  i_{\partial X *}(a) \pitchfork_X^{\partial} b 
  &= \langle D_X^{-1}(b),i_{\partial X *}(a) \rangle \\
  &= \langle i_{\partial X}^*(D_X^{-1}(b)),a \rangle \\
  &= (-1)^k \langle D_{\partial X}^{-1}(\partial_*(b)),a \rangle \\
  &= (-1)^k a \pitchfork_{\partial X} \partial_*(b)
 \end{align*}
 for all relative homology classes ${a \in H_k(\partial X)}$ and ${b \in H_{n-k}(X,\partial X)}$.
\end{proof}

Now let us fix for the rest of this section a closed separating ${n-1}$-\-di\-men\-sion\-al submanifold ${Y_0}$ of ${X}$ disjoint from ${\partial X}$. Then ${Y_0}$ induces a decomposition ${X = X_- \cup X_+}$ for some codimension ${0}$ submanifolds ${X_-}$ and ${X_+}$ of ${X}$ having boundaries ${\partial X_- = \overline{Y_-} \cup Y_0}$ and ${\partial X_+ = \overline{Y_0} \cup Y_+}$ respectively, where ${Y_-}$ and ${Y_+}$ satisfy ${\overline{Y_-} \cup Y_+ = \partial X}$ and ${\overline{Y_-} \cap Y_+ = \varnothing}$. Remark we do not exclude the case ${\partial X = \varnothing}$, and so ${Y_-}$ and ${Y_+}$ may be empty, even simultaneously.

\begin{lemma}\label{L:decompositions}
 If 
 \begin{gather*}
  i_{X_{\pm} *} : H_k(X_{\pm}) \rightarrow H_k(X), \\
  q_{X_{\mp} *} : H_{n-k}(X,\partial X) \rightarrow H_{n-k}(X,X_{\mp} \cup \partial X), \\
  e_{X_{\mp} *} : H_{n-k}(X_{\pm},\partial X_{\pm}) \rightarrow H_{n-k}(X,X_{\mp} \cup \partial X)
 \end{gather*}
 are induced by inclusion, then 
 \[
  i_{X_{\pm} *}(a) \pitchfork_X^{\partial} b = a \pitchfork_{X_{\pm}}^{\partial} e_{X_{\mp} *}^{-1}(q_{X_{\mp} *}(b))  
 \]
 for all relative homology classes ${a \in H_k(X_{\pm})}$ and ${b \in H_{n-k}(X,\partial X)}$.
\end{lemma}

\begin{proof}
 Naturality of relative cap products gives 
 \[
  e_{X_{\mp} *}([X_{\pm}] \cappr i_{X_{\pm}}^*(\varphi)) = e_{X_{\mp} *}([X_{\pm}]) \cappr \varphi, \quad q_{X_{\mp} *}([X] \cappr \varphi) = q_{X_{\mp} *}([X]) \cappr \varphi
  \]
  for every cohomology class ${\varphi \in H^k(X)}$. But since the fundamental class of ${X_{\pm}}$ satisfies ${[X_{\pm}] = e^{-1}_{X_{\mp} *}(q_{X_{\mp} *}([X]))}$, we have
 \begin{align*}
  [X_{\pm}] \cappr i_{X_{\pm}}^*(\varphi) &= e_{X_{\mp} *}^{-1}(e_{X_{\mp} *}([X_{\pm}]) \cappr \varphi) \\
  &= e_{X_{\mp} *}^{-1}(q_{X_{\mp} *}([X]) \cappr \varphi) \\
  &= e_{X_{\mp} *}^{-1}(q_{X_{\mp} *}([X] \cappr \varphi)),
 \end{align*} 
 which implies
 \begin{align*}
  i_{X_{\pm} *}(a) \pitchfork_X^{\partial} b 
  &= \langle D_X^{-1}(b),i_{X_{\pm} *}(a) \rangle \\
  &= \langle i_{X_{\pm}}^*(D_X^{-1}(b)),a \rangle \\
  &= \langle D_{X_{\pm}}^{-1}(e_{X_{\mp} *}^{-1}(q_{X_{\mp} *}(b))),a \rangle \\
  &= a \pitchfork_{X_{\pm}}^{\partial} e_{X_{\mp} *}^{-1}(q_{X_{\mp} *}(b))
 \end{align*} 
 for all relative homology classes ${a \in H_k(X_{\pm})}$ and ${b \in H_{n-k}(X,\partial X)}$.
\end{proof}

\begin{remark}\label{R:isometry}
 If $i_{X_{\pm}} : X_{\pm} \hookrightarrow X$ denotes inclusion, then Lemma \ref{L:decompositions} immediately implies that ${i_{X_{\pm} *}}$ preserves absolute intersection pairings. Indeed, if $q_{\partial X} : (X,\varnothing) \hookrightarrow (X,\partial X)$ and $q_{\partial X_{\pm}} : (X_{\pm},\varnothing) \hookrightarrow (X_{\pm},\partial X_{\pm})$ denote inclusions, then
 \begin{align*}
  i_{X_{\pm} *}(a) \pitchfork_X i_{X_{\pm} *}(b)
  &= i_{X_{\pm} *}(a) \pitchfork_X^{\partial} q_{\partial X *}(i_{X_{\pm} *}(b)) \\
  &= a \pitchfork_{X_{\pm}}^{\partial} q_{\partial X_{\pm} *}(b) \\
  &= a \pitchfork_{X_{\pm}} b,
 \end{align*}
 for all homology classes ${a \in H_k(X_{\pm})}$ and ${b \in H_{n-k}(X_{\pm})}$.
\end{remark}

Now, we will characterize some annihilators for the absolute intersection pairing ${\pitchfork_X}$ which are related to the decomposition ${X = X_- \cup X_+}$. Let ${i_{X_{\pm}} : X_{\pm} \hookrightarrow X}$, ${i_{Y_0} : Y_0 \hookrightarrow X}$, and ${i_{Y_{\pm}} : Y_{\pm} \hookrightarrow X}$ denote inclusions.

\begin{lemma}\label{L:annihilators}
 The absolute intersection pairing ${\pitchfork_X}$ satisfies
 \begin{gather*}
  \im(i_{X_{\pm *}})^{\perp} = \im(i_{X_{\mp} *}) + \im(i_{Y_{\pm} *}), \\
  \left( \im(i_{X_{- *}}) + \im(i_{X_{+ *}}) \right)^{\perp} = \im(i_{Y_0 *}) + \im(i_{Y_- *}) + \im(i_{Y_+ *}).
 \end{gather*}
\end{lemma}

\begin{proof}
 Lemma \ref{L:decompositions} gives
 \begin{align*}
  i_{X_{\pm}*}(a) \pitchfork_X b 
  &= i_{X_{\pm}*}(a) \pitchfork_X^{\partial} q_{\partial X *}(b) \\ 
  &= a \pitchfork_{X_{\pm}}^{\partial} e_{X_{\mp} *}^{-1}(q_{X_{\mp} *}(q_{\partial X *}(b)))
 \end{align*}
 for all homology classes ${a \in H_k(X_{\pm})}$ and ${b \in H_{n-k}(X)}$, where 
 \begin{gather*}
  q_{\partial X *} : H_{n-k}(X) \rightarrow H_{n-k}(X,\partial X), \\
  q_{X_{\mp} *} : H_{n-k}(X,\partial X) \rightarrow H_{n-k}(X,X_{\mp} \cup \partial X), \\
  e_{X_{\mp} *} : H_{n-k}(X_{\pm},\partial X_{\pm}) \rightarrow H_{n-k}(X,X_{\mp} \cup \partial X)
 \end{gather*}
 are induced by inclusion. Therefore, since ${\pitchfork_{X_{\pm}}^{\partial}}$ is non-degenerate, we have 
 \begin{align*}
  \im(i_{X_{\pm *}})^{\perp} 
  &= \ker(e_{X_{\mp} *}^{-1} \circ q_{X_{\mp} *} \circ q_{\partial X *}) \\
  &= \ker(q_{X_{\mp} *} \circ q_{\partial X *}) \\
  &= \im(i_{X_{\mp} *}) + \im(i_{Y_{\pm} *}),
 \end{align*}
 where the last equality follows from the long exact sequence of the pair $(X,X_{\mp} \cup Y_{\pm})$. This gives the first annihilator, and immediately determines also the second one, since
 \begin{align*}
  \left( \im(i_{X_{- *}}) + \im(i_{X_{+ *}}) \right)^{\perp}
  &= \im(i_{X_{- *}})^{\perp} \cap \im(i_{X_{+ *}})^{\perp} \\
  &= \left( \im(i_{X_+ *}) + \im(i_{Y_- *}) \right) \cap \left( \im(i_{X_- *}) + \im(i_{Y_+ *}) \right) \\
  &= \left( \left( \im(i_{X_+ *}) + \im(i_{Y_- *}) \right) \cap \im(i_{X_- *}) \right) + \im(i_{Y_+ *}) \\
  &= \left( \im(i_{X_+ *}) \cap \im(i_{X_- *}) \right) + \im(i_{Y_- *}) + \im(i_{Y_+ *}) \\
  &= \im(i_{Y_0 *}) + \im(i_{Y_- *}) + \im(i_{Y_+ *}),
 \end{align*}
 where the third equality follows from the inclusion ${\im(i_{Y_+ *}) \subset \im(i_{X_+ *})}$, where the fourth equality follows from the inclusion ${\im(i_{Y_- *}) \subset \im(i_{X_- *})}$, and where the fifth equality follows from the Mayer-Vietoris sequence associated with ${X_+}$ and ${X_-}$.
\end{proof}

\section{Lagrangian subspaces}

In this section we introduce and study Lagragian subspaces of the middle homology of a ${4m+2}$-\-di\-men\-sion\-al manifold with boundary. Let us fix a finite-dimensional real vector space ${H}$ equipped with an antisymmetric bilinear form ${\omega}$. A subspace ${A \subset H}$ is \textit{isotropic} if it satisfies ${A \subset A^{\perp}}$.

\begin{definition}\label{D:Lagrangian}
 An isotropic subspace ${\calL}$ of ${H}$ is \textit{Lagrangian} if ${\calL = \calL^{\perp}}$.
\end{definition}

When ${\omega}$ is non-degenerate, ${H}$ is called a \textit{symplectic space}. For instance, if ${\varSigma}$ is a closed ${4m+2}$-\-di\-men\-sion\-al manifold, then its middle homology ${H_{2m+1}(\varSigma)}$ is a symplectic space once equipped with the intersection pairing ${\pitchfork_{\varSigma}}$.

\begin{proposition}\label{P:Lagrangian}
 If ${M}$ is a ${4m+3}$-\-di\-men\-sion\-al manifold, and if 
 \[
  i_{\partial M *} : H_{2m+1}(\partial M) \rightarrow H_{2m+1}(M)
 \]
 is induced by inclusion, then ${\ker(i_{\partial M *})}$ is a Lagrangian subspace of the symplectic space ${H_{2m+1}(\partial M)}$.
\end{proposition}

\begin{proof}
 If ${\partial_* : H_{2m+2}(M,\partial M) \rightarrow H_{2m+1}(\partial M)}$ comes from the long exact sequence of the pair ${(M,\partial M)}$, then ${\ker(i_{\partial M *}) = \im(\partial_*)}$. But now Lemma \ref{L:boundary} gives the equality
 \[
  i_{\partial M *}(a) \pitchfork_M^{\partial} b = - a \pitchfork_{\partial M} \partial_*(b)
 \]
 for all relative homology classes ${a \in H_{2m+1}(\partial M)}$ and ${b \in H_{2m+2}(M,\partial M)}$. This means ${\ker(i_{\partial M *})^{\perp} = \im (\partial_*)^{\perp} = \ker(i_{\partial M *})}$, where the second equality follows from the non-degeneracy of ${\pitchfork_M^{\partial}}$.
\end{proof}

Now let us fix a ${4m+2}$-\-di\-men\-sion\-al manifold ${\varSigma}$, and let ${\varGamma_0}$ be a closed separating ${4m+1}$-\-di\-men\-sion\-al submanifold of ${\varSigma}$ disjoint from ${\partial \varSigma}$. Then ${\varGamma_0}$ induces a decomposition ${\varSigma = \varSigma_- \cup \varSigma_+}$ for some codimension ${0}$ submanifolds ${\varSigma_-}$ and ${\varSigma_+}$ of ${\varSigma}$ having boundaries ${\partial \varSigma_- = \overline{\varGamma_-} \cup \varGamma_0}$ and ${\partial \varSigma_+ = \overline{\varGamma_0} \cup \varGamma_+}$ respectively, where ${\varGamma_-}$ and ${\varGamma_+}$ satisfy ${\overline{\varGamma_-} \cup \varGamma_+ = \partial \varSigma}$ and ${\overline{\varGamma_-} \cap \varGamma_+ = \varnothing}$. Let ${i_{\varSigma_{\pm} *} : H_{2m+1}(\varSigma_{\pm}) \rightarrow H_{2m+1}(\varSigma)}$ be induced by inclusion.

\begin{proposition}\label{P:sum_of_Lagrangians}
 If $\calL_- \subset H_{2m+1}(\varSigma_-)$ and $\calL_+ \subset H_{2m+1}(\varSigma_+)$ are Lagrangian subspaces, then $\calL := i_{\varSigma_- *}(\calL_-) + i_{\varSigma_+ *}(\calL_+) \subset H_{2m+1}(\varSigma)$ is a Lagrangian subspace.
\end{proposition}

\begin{proof}
 Let 
 \[
  i_{\varGamma_0 *} : H_{2m+1}(\varGamma_0) \rightarrow H_{2m+1}(\varSigma), \quad
  i_{\varGamma_{\pm} *} : H_{2m+1}(\varGamma_{\pm}) \rightarrow H_{2m+1}(\varSigma)
 \]
 be induced by inclusion. We start by remarking that 
 \[
  \calL^{\perp} = i_{\varSigma_- *}(\calL_-)^{\perp} \cap i_{\varSigma_+ *}(\calL_+)^{\perp},
 \] 
 and that
 \begin{align*}
  i_{\varSigma_{\pm} *}(\calL_{\pm})^{\perp} 
  &= i_{\varSigma_{\pm} *}(\calL_{\pm})^{\perp \perp \perp} \\
  &= \left( i_{\varSigma_{\pm} *}(\calL_{\pm}) + \im(i_{\varGamma_- *}) + \im(i_{\varGamma_+ *}) \right)^{\perp} \\
  &= \left( i_{\varSigma_{\pm} *}(\calL_{\pm}) + \im(i_{\varGamma_{\mp} *}) \right)^{\perp},
 \end{align*}
 where the second equality follows from the inclusion
 \[
  \im(i_{\varGamma_0 *}) + \im(i_{\varGamma_{\pm} *}) = i_{\varSigma_{\pm} *} \left( H_{2m+1}(\varSigma_{\pm})^{\perp} \right) \subset i_{\varSigma_{\pm} *} \left( \calL_{\pm}^{\perp} \right) = i_{\varSigma_{\pm} *} \left( \calL_{\pm} \right).
 \]
 This implies
 \[ 
  \im( i_{\varGamma_0 *} ) + \im( i_{\varGamma_- *} ) + \im( i_{\varGamma_+ *} ) \subset i_{\varSigma_{\pm} *}(\calL_{\pm}) + \im(i_{\varGamma_{\mp} *}),
 \]
 which means 
 \[
  i_{\varSigma_{\pm} *}(\calL_{\pm})^{\perp} = \left( i_{\varSigma_{\pm} *}(\calL_{\pm}) + \im(i_{\varGamma_{\mp} *}) \right)^{\perp} \subset \left( \im( i_{\varGamma_0 *} ) + \im( i_{\varGamma_- *} ) + \im( i_{\varGamma_+ *} ) \right)^{\perp},
 \]
 and the second annihilator in Lemma \ref{L:annihilators} gives
 \begin{align*}
  \left( \im( i_{\varGamma_0 *} ) + \im( i_{\varGamma_- *} ) + \im( i_{\varGamma_+ *} ) \right)^{\perp} 
  &= \left( \im(i_{\varSigma_- *}) + \im (i_{\varSigma_+ *}) \right)^{\perp \perp} \\
  &= \im(i_{\varSigma_- *}) + \im (i_{\varSigma_+ *}) + \im( i_{\varGamma_- *} ) + \im( i_{\varGamma_+ *} ) \\
  &= \im(i_{\varSigma_- *}) + \im(i_{\varSigma_+ *}).
 \end{align*}
 Thus
 \begin{align*}
  i_{\varSigma_{\pm} *}(\calL_{\pm})^{\perp}
  &= i_{\varSigma_{\pm} *}(\calL_{\pm})^{\perp} \cap \left( \im(i_{\varSigma_{\pm} *}) + \im(i_{\varSigma_{\mp} *}) \right) \\
  &= \left( i_{\varSigma_{\pm} *}(\calL_{\pm})^{\perp} \cap \im(i_{\varSigma_{\pm} *}) \right) + \im(i_{\varSigma_{\mp} *}) \\
  &= i_{\varSigma_{\pm} *} \left( \calL_{\pm}^{\perp} \right) + \im(i_{\varSigma_{\mp} *}) \\
  &= i_{\varSigma_{\pm} *} \left( \calL_{\pm} \right) + \im(i_{\varSigma_{\mp} *}),
 \end{align*}
 where the second equality follows from the inclusion 
 \[
  \im(i_{\varSigma_{\mp} *}) \subset \im(i_{\varSigma_{\pm} *})^{\perp} \subset i_{\varSigma_{\pm} *}(\calL_{\pm})^{\perp} 
 \]
 given by the first annihilator in Lemma \ref{L:annihilators}, and where the third equality follows from Remark \ref{R:isometry}. Therefore
 \begin{align*}
  \calL^{\perp}
  &= \left( i_{\varSigma_- *}(\calL_-) + \im \left( i_{\varSigma_+ *} \right) \right) \cap \left( i_{\varSigma_+ *}(\calL_+) + \im \left( i_{\varSigma_- *} \right) \right) \\
  &= i_{\varSigma_- *}(\calL_-) + \left( \im \left( i_{\varSigma_+ *} \right) \cap \left( i_{\varSigma_+ *}(\calL_+) + \im \left( i_{\varSigma_- *} \right) \right) \right) \\
  &= i_{\varSigma_- *}(\calL_-) + i_{\varSigma_+ *}(\calL_+) + \left( \im \left( i_{\varSigma_+ *} \right) \cap \im \left( i_{\varSigma_- *} \right) \right) \\
  &= i_{\varSigma_- *}(\calL_-) + i_{\varSigma_+ *}(\calL_+) + \im \left( i_{\varGamma_0 *} \right) \\
  &= i_{\varSigma_- *}(\calL_-) + i_{\varSigma_+ *}(\calL_+) \\
  &= \calL,
 \end{align*}
 where the second equality follows from the inclusion ${i_{\varSigma_- *}(\calL_-) \subset \im \left( i_{\varSigma_- *} \right)}$, where the third equality follows from the inclusion ${i_{\varSigma_+ *}(\calL_+) \subset \im \left( i_{\varSigma_+ *} \right)}$, and where the fourth equality follows from the Mayer-Vietoris sequence associated with ${\varSigma_-}$ and ${\varSigma_+}$.
\end{proof}

Let ${H}$ be a symplectic space, let ${\omega}$ be its non-degenerate antisymmetric bilinear form, and let ${A}$ be an isotropic subspace of ${H}$. Then the space
\[
 H | A := \left. A^{\perp} \right/ A,
\]
equipped with the antisymmetric bilinear form ${\omega | A}$ induced by ${\omega}$, is a symplectic space too. If ${B}$ is a subspace of ${H}$, then we say the subspace
\[
 B | A :=  \left. \left( (B + A) \cap A^{\perp} \right) \right/ A
\]
of ${H | A}$ is obtained from ${B}$ by \textit{contraction along $A$}.

\begin{lemma}\label{L:contraction}
 If ${\calL}$ is a Lagrangian subspace of a symplectic space ${H}$, then for every isotropic subspace ${A \subset H}$ the contraction ${\calL | A}$ is a Lagrangian subspace of ${H | A}$. 
\end{lemma}

\begin{proof}
 For all ${x,x',a,a' \in H}$ we have
 \[
  \omega(x+a,x'+a') = \omega(x,x') + \omega(x+a,a') + \omega(a,x'+a') - \omega(a,a').
 \]
 If ${x,x' \in \calL}$, if ${a,a' \in A}$, and if ${x+a, x'+a' \in A^{\perp}}$, then all terms on the right-hand side of the equality vanish. Therefore we have ${\calL | A \subset (\calL | A)^{\perp}}$. For the opposite inclusion, we have
 \[ 
  \left( (\calL + A) \cap A^{\perp} \right)^{\perp} 
  = (\calL + A)^{\perp} + A^{\perp \perp}
  = (\calL^{\perp} \cap A^{\perp}) + A 
  \subset \calL^{\perp} + A 
  = \calL + A 
 \]
 Therefore
 \[
  (\calL | A)^{\perp} = \left. \left( \left( (\calL + A) \cap A^{\perp} \right)^{\perp} \cap A^{\perp} \right) \right/ A \subset \left. \left( (\calL + A) \cap A^{\perp} \right) \right/ A = \calL | A. \qedhere
 \]
\end{proof}

Now let us fix a ${4m+3}$-\-di\-men\-sion\-al manifold ${M}$, and let ${\varGamma}$ be a closed separating ${4m+1}$-\-di\-men\-sion\-al submanifold of ${\partial M}$. Then ${\varGamma}$ induces a decomposition ${\partial M = \varSigma_- \cup \varSigma_+}$ for some codimension ${0}$ submanifolds ${\varSigma_-}$ and ${\varSigma_+}$ of ${\partial M}$ having boundaries ${\partial \varSigma_- = \varGamma}$ and ${\partial \varSigma_+ = \overline{\varGamma}}$ respectively. Let 
\[
 i_{\varSigma_{\pm} *} : H_{2m+1}(\varSigma_{\pm}) \rightarrow H_{2m+1}(\partial M), \quad 
 i_{\partial M *} : H_{2m+1}(\partial M) \rightarrow H_{2m+1}(M),
\]
be induced by inclusion.

\begin{proposition}\label{P:Lagrangian_relations}
 If ${\calL_-}$ is a Lagrangian subspace of ${H_{2m+1}(\varSigma_-)}$, then
 \[
  \calL_+ := \{ x_+ \in H_{2m+1}(\varSigma_+) \mid i_{\partial M *}(i_{\varSigma_+ *}(x_+)) \in i_{\partial M *}(i_{\varSigma_-*}(\calL_-)) \}
 \] 
 is a Lagrangian subspace of ${H_{2m+1}(\varSigma_+)}$.
\end{proposition}

\begin{proof}
 Thanks to Proposition \ref{P:Lagrangian}, we know ${\ker(i_{\partial M *})}$ is a Lagrangian subspace of the symplectic space ${H_{2m+1}(\partial M)}$. Moreover, thanks to Remark \ref{R:isometry}, we also know ${i_{\varSigma_- *}}$ preserves absolute intesection pairings, so ${i_{\varSigma_- *}(\calL_-)}$ is an isotropic subspace of ${H_{2m+1}(\partial M)}$. This means we can define the contraction ${\ker(i_{\partial M *}) | i_{\varSigma_- *}(\calL_-)}$ which, thanks to Lemma \ref{L:contraction}, is a Lagrangian subspace of ${H_{2m+1}(\partial M) | i_{\varSigma_- *}(\calL_-)}$. Now let ${i_{\varGamma_{\pm} *} : H_{2m+1}(\varGamma) \rightarrow H_{2m+1}(\varSigma_{\pm})}$ be induced by inclusion. We claim that ${H_{2m+1}(\partial M) | i_{\varSigma_- *}(\calL_-)}$ is isomorphic, as a symplectic space, to ${H_{2m+1}(\varSigma_+) / \im(i_{\varGamma_+ *})}$, and that, under this isomorphism, the contraction ${\ker (i_{\partial M *}) | i_{\varSigma_- *}(\calL_-)}$ corresponds to ${\calL_+ / \im(i_{\varGamma_+ *})}$. Remark that this claim clearly allows us to conclude, so let us prove it. We saw during the proof of Proposition \ref{P:sum_of_Lagrangians} that we have
 \[
  i_{\varSigma_- *}(\calL_-)^{\perp} = i_{\varSigma_- *}(\calL_-) + \im(i_{\varSigma_+ *}).
 \]
 This yields isomorphisms of symplectic spaces
 \begin{align*}
  H_{2m+1}(\partial M) | i_{\varSigma_- *}(\calL_-) 
  &\cong \left. \im(i_{\varSigma_+ *}) \right/ \left( i_{\varSigma_- *}(\calL_-) \cap \im(i_{\varSigma_+ *}) \right) \\
  &= \left. \im(i_{\varSigma_+ *}) \right/ \im(i_{\varSigma_+ *} \circ i_{\varGamma_+ *}) \\
  &\cong \left. H_{2m+1}(\varSigma_+) \right/ \im(i_{\varGamma_+ *}),
 \end{align*}
 where the second one follows from Remark \ref{R:isometry}. Furthermore, we have
 \begin{align*}
  \left( \ker(i_{\partial M *}) + i_{\varSigma_- *}(\calL_-) \right) 
  &\cap i_{\varSigma_- *}(\calL_-)^{\perp} \\
  &=  \left( \ker(i_{\partial M *}) + i_{\varSigma_- *}(\calL_-) \right) \cap \left( i_{\varSigma_- *}(\calL_-) + \im(i_{\varSigma_+ *}) \right) \\
  &= \left( \ker(i_{\partial M *}) \cap \left( i_{\varSigma_- *}(\calL_-) + \im(i_{\varSigma_+ *}) \right) \right) + i_{\varSigma_- *}(\calL_-).
 \end{align*}
 But now, by the very definition of ${\calL_*}$, we have
 \begin{align*}
  \ker(i_{\partial M *}) \cap \left( i_{\varSigma_- *}(\calL_-) + \im(i_{\varSigma_+ *}) \right) = i_{\varSigma_- *}(\calL_-) + i_{\varSigma_+ *}(\calL_+).
 \end{align*}
 This means
 \begin{align*}
  \ker(i_{\partial M *}) | i_{\varSigma_- *}(\calL_-) 
  &\cong \left. i_{\varSigma_+ *}(\calL_+)) \right/ \left( i_{\varSigma_- *}(\calL_-) \cap i_{\varSigma_+ *}(\calL_+) \right) \\
  &= \left. i_{\varSigma_+ *}(\calL_+) \right/ \im(i_{\varSigma_+ *} \circ i_{\varGamma_+ *}) \\
  &\cong \left. \calL_+ \right/ \im(i_{\varGamma_+ *}),
 \end{align*}
 where again the second isomorphism follows from Remark \ref{R:isometry}. This proves our claim.
\end{proof}

\section{Maslov indices}\label{S:Maslov}

In this section we define Maslov indices of Lagrangian subspaces, which control the behaviour of the signature of ${4m}$-\-di\-men\-sion\-al manifolds under the operation of gluing along codimension ${0}$ submanifolds of their boundary. Let us fix for this section a finite-dimensional real vector space ${H}$ equipped with an antisymmetric bilinear form ${\omega}$, and let ${\calL_1}$, ${\calL_2}$, and ${\calL_3}$ be three Lagrangian subspaces of ${H}$. Every element of the vector space
\[
 W(\calL_1,\calL_2,\calL_3) := \frac{\calL_1 \cap (\calL_2 + \calL_3)}{(\calL_1 \cap \calL_2) + (\calL_1 \cap \calL_3)}
\]
is represented by some ${a_1 \in \calL_1}$ which is equal to a sum ${a_2 + a_3}$ for ${a_2 \in \calL_2}$ and ${a_3 \in \calL_3}$. Now consider the bilinear form
\[
 \begin{array}{rccc}
  \langle \cdot,\cdot \rangle_{\calL_1,\calL_2,\calL_3} : & W(\calL_1,\calL_2,\calL_3) \times W(\calL_1,\calL_2,\calL_3) & \rightarrow & \R \\
  & ([a_1],[b_2 + b_3]) & \mapsto & \omega(a_1,b_2)
 \end{array}
\]
where ${a_1 \in \calL_1}$, ${b_2 \in \calL_2}$, and ${b_3 \in \calL_3}$. Remark that this map is indeed well-defined: on one hand, if ${a_1 = a_2 + a_3}$ with ${a_i \in \calL_1 \cap \calL_i}$, and if ${b_1 = b_2 + b_3}$ with ${b_i \in \calL_i}$ for every integer ${1 \leqslant i \leqslant 3}$, then
\[
 \omega(a_1,b_2) = \omega(a_2 + a_3,b_2) = \omega(a_3,b_2) = \omega(a_3,b_1 - b_3) = \omega(a_3,b_1) = 0.
\]
On the other hand, if $a_1 = a_2 + a_3$ with $a_i \in \calL_i$, and if $b_1 = b_2 + b_3$ with $b_i \in \calL_1 \cap \calL_i$ for every integer $1 \leqslant i \leqslant 3$, then we immediately have ${\omega(a_1,b_2) = 0}$. Moreover, the bilinear form $\langle \cdot,\cdot \rangle_{\calL_1,\calL_2,\calL_3}$ is symmetric: indeed, if $a_1 = a_2 + a_3$ with $a_i \in \calL_i$, and if $b_1 = b_2 + b_3$ with $b_i \in \calL_i$ for every integer $1 \leqslant i \leqslant 3$, then we have
\[
 0 = \omega(a_1,b_1) = \omega(a_3,b_2) + \omega(a_2,b_3) = \omega(a_3,b_2) - \omega(b_3,a_2),
\]
which implies
\[
 \omega(a_1,b_2) = \omega(a_3,b_2) = \omega(b_3,a_2) = \omega(b_1,a_2).
\]
The symmetric form ${\langle \cdot,\cdot \rangle_{\calL_1,\calL_2,\calL_3}}$ is called the \textit{Maslov form associated with the Lagrangian subspaces ${\calL_1}$, ${\calL_2}$, and ${\calL_3}$ of $H$}.

\begin{definition}\label{D:Maslov}
 If ${\calL_1}$, ${\calL_2}$, and ${\calL_3}$ are Lagrangian subspaces of ${H}$, their \textit{Maslov index} ${\mu(\calL_1,\calL_2,\calL_3)}$ is the signature of the Maslov form ${\langle \cdot,\cdot \rangle_{\calL_1,\calL_2,\calL_3}}$.
\end{definition}

\begin{lemma}
 The Maslov index satisfies
 \[
  \mu(\calL_{\sigma(1)},\calL_{\sigma(2)},\calL_{\sigma(3)}) = \sgn(\sigma) \mu(\calL_1,\calL_2,\calL_3)
 \]
 for all Lagrangian subspaces ${\calL_1}$, ${\calL_2}$, and ${\calL_3}$ of ${H}$, and for every permutation ${\sigma}$ in ${\frakS_3}$.
\end{lemma}

\begin{proof}
 We have
\[
 \langle [a_1],[b_3 + b_2] \rangle_{\calL_1,\calL_3,\calL_2} = \omega(a_1,b_3) = - \omega(a_1,b_2) = - \langle [a_1],[b_2 + b_3] \rangle_{\calL_1,\calL_2,\calL_3}
\]
for all $[a_1] = [a_2 + a_3],[b_1] = [b_2 + b_3] \in W(\calL_1,\calL_2,\calL_3)$. This accounts for the equality $\mu(\calL_1,\calL_3,\calL_2) = - \mu(\calL_1,\calL_2,\calL_3)$. Furthermore, the space $W(\calL_1,\calL_2,\calL_3)$ is isomorphic to $W(\calL_2,\calL_1,\calL_3)$, with every $[a_1] = [a_2 + a_3] \in W(\calL_1,\calL_2,\calL_3)$ corresponding to $[a_2] = [a_1 - a_3] \in W(\calL_2,\calL_1,\calL_3)$. Therefore, we have
\[
 \langle [a_2],[b_1 - b_3] \rangle_{\calL_2,\calL_1,\calL_3} = \omega(a_2,b_1) = - \omega(a_1,b_2) = - \langle [a_1],[b_2 + b_3] \rangle_{\calL_1,\calL_2,\calL_3}
\]
for all $[a_1] = [a_2 + a_3],[b_1] = [b_2 + b_3] \in W(\calL_1,\calL_2,\calL_3)$. This accounts for the equality $\mu(\calL_2,\calL_1,\calL_3) = - \mu(\calL_1,\calL_2,\calL_3)$.
\end{proof}

A small remark on conventions: our definition of the space ${W(\calL_1,\calL_2,\calL_3)}$ agrees with the one of \cite{W69}, while the one given in \cite{T94} corresponds to our ${W(\calL_3,\calL_1,\calL_2)}$. However, our symmetric form ${\langle \cdot,\cdot \rangle_{\calL_1,\calL_2,\calL_3}}$ is the opposite of the one given in \cite{W69}, while it agrees with the one given in \cite{T94}. Therefore, our definition of ${\mu(\calL_1,\calL_2,\calL_3)}$ turns out to coincide with the one of \cite{T94}, while if we wanted to match our formulas with the ones coming from \cite{W69}, we would need to place a minus sign in front of every occurence of the Maslov index.

Now, to put Maslov indices into perspective, let us quote Wall's famous non-additivity result, whose proof can be found in \cite{W69}. The signature ${\sigma(W)}$ of a ${4m}$-\-di\-men\-sion\-al manifold ${W}$ is the signature of its middle absolute intersection pairing ${\pitchfork_W : H_{2m}(W) \times H_{2m}(W) \rightarrow \R}$, which is a symmetric bilinear form on ${H_{2m}(W)}$. Now let us fix a ${4m}$-\-di\-men\-sion\-al manifold ${W}$, and let ${M_0}$ be a separating ${4m-1}$-\-di\-men\-sion\-al submanifold of ${W}$ with boundary ${\partial M_0 = M_0 \cap \partial W}$. Then ${M_0}$ induces a decomposition ${W = W_- \cup W_+}$ for some codimension ${0}$ submanifolds ${W_-}$ and ${W_+}$ of ${W}$ having boundaries ${\partial W_- = \overline{M_-} \cup M_0}$ and ${\partial W_+ = \overline{M_0} \cup M_+}$ respectively, where ${M_-}$ and ${M_+}$ satisfy ${\partial W = \overline{M_-} \cup M_+}$ and ${\overline{M_-} \cap M_+ = \partial M_- = \partial M_0 = \partial M_+ = \varSigma}$.

\begin{theorem}\label{T:Wall}
 Let
 \[
  i_{\varSigma_0 *} : H_{2m-1}(\varSigma) \rightarrow H_{2m-1}(M_0), \quad
  i_{\varSigma_{\pm} *} : H_{2m-1}(\varSigma) \rightarrow H_{2m-1}(M_{\pm})
 \]
 be induced by inclusion. Then
 \[
  \sigma(W) = \sigma(W_-) + \sigma(W_+) + \mu(\calL_-,\calL_0,\calL_+),
 \]
 where ${\calL_0 := \ker(i_{\varSigma_0 *})}$, and ${\calL_{\pm} := \ker(i_{\varSigma_{\pm} *})}$.
\end{theorem}

A final remark: the manifestation of the Maslov index in the composition of 2-morphisms of the 2-category ${\bfCob_{\calC}}$ of decorated cobordisms of dimension 1+1+1 is indeed related to the non-additivity of the signature of 4-manifolds. Indeed, Lagrangian subspaces and signature defects are the shadow of certain manifolds whose boundaries are related to morphisms of ${\bfCob_{\calC}}$, and which leave no trace on objects. The idea is roughly the following: every object ${\bbGamma}$ of ${\bfCob_{\calC}}$ secretly carries along a collection of topological discs whose boundary is identified with ${\varGamma}$. We remember nothing about these discs simply because there is essentially a unique way of identifying a topological circle to a connected 1-manifold preserving orientations. Analogously, every 1-morphism ${\bbSigma : \bbGamma \rightarrow \bbGamma'}$ of ${\bfCob_{\calC}}$ secretly carries along a collection of topological handlebodies whose boundary is identified with the closed surface ${\hat{\varSigma}}$ obtained from ${\varSigma}$ by gluing in all the discs bounding ${\varGamma}$ and ${\varGamma'}$. The only thing we remember about these handlebodies is the Lagrangians they leave in the first homology of ${\varSigma}$. Lastly, every 2-morphism ${\bbM : \bbSigma \Rightarrow \bbSigma'}$ of ${\bfCob_{\calC}}$ secretly carries along a topological 4-manifold whose boundary is identified with the closed 3-manifold ${\hat{M}}$ obtained from ${M}$ by gluing in all the handlebodies bounding the closed surfaces ${\hat{\varSigma}}$ and ${\hat{\varSigma}'}$. The only thing we remember about these 4-manifolds is their signature.

The reader might wonder then why there is a minus sign in front of the Maslov index in Definition \ref{D:vertical_composition}, while the correction term in Theorem \ref{T:Wall} appears with a plus sign. The question is related to orientations: when we compose vertically 2-morphisms ${\bbM : \bbSigma \Rightarrow \bbSigma'}$ and ${\bbM' : \bbSigma' \Rightarrow \bbSigma''}$ of ${\bfCob_{\calC}}$, let us see who plays who with respect to the notation of Theorem \cite{W69}. First of all, we assign the roles for ${W_-}$ and ${W_+}$, which are played by the bounding 4-manifolds attached to ${M}$ and to ${M'}$ respectively. Now, remark that ${W_-}$ is supposed to induce the opposite orientation on ${M_-}$. This means that ${\overline{M_-}}$ has to be played by ${\hat{M}}$, while ${\overline{M_0}}$ is played by the bounding collection of handlebodies attached to ${\varSigma'}$, and ${M_+}$ is played by ${\hat{M}'}$. Lastly, since ${\varSigma}$ needs to be oriented as the boundary of ${M_-}$, ${M_0}$, and ${M_+}$ at the same time, we need to set ${\overline{\varSigma} = \hat{\varSigma}'}$. Summing up everything, when we compose vertically 2-morphisms ${\bbM : \bbSigma \Rightarrow \bbSigma'}$ and ${\bbM' : \bbSigma' \Rightarrow \bbSigma''}$ of ${\bfCob_{\calC}}$, we look at Lagrangian subspaces of the first homology of ${\overline{\varSigma'}}$, and the orientation reversal operation reverses intersection pairings too. This explains the minus sign in front of the Maslov index in Definition \ref{D:vertical_composition}.


%% file: appendix_d.tex
%
%
%

\chapter{Symmetric monoidal 2-categories}\label{A:symmetric_monoidal}

This appendix collects definitions of higher categorical structures related to symmetric monoidal $2$-cat\-e\-go\-ries, which are all gathered here for convenience. We also recall an important coherence result for symmetric monoidal $2$-cat\-e\-go\-ries, due to Schommer-Pries, which allows us to forget about most of their coherence data.

\section{2-Categories}

We begin by fixing our terminology and notation for $2$-cat\-e\-go\-ries. The concept was introduced for the first time in \cite{B67} under the name bicategory. Other references are provided by \cite{KS74} and \cite{G74}, or by the more recent and concise \cite{L98}, although the main source for all the material presented in this chapter is \cite{S11}.

\begin{definition}\label{D:2-category}
 A \textit{$2$-cat\-e\-go\-ry} $\bcalC$ (sometimes called a \textit{weak $2$-cat\-e\-go\-ry}, or a \textit{bicategory}) is given by:
 \begin{enumerate}
  \item A class $\bcalC$, whose elements are called \textit{objects} (our notation is abusive and confuses $2$-cat\-e\-go\-ries with their classes of objects);
  \item A category $\bcalC(x,y)$ for all $x,y \in \bcalC$, whose objects are called \textit{$1$-mor\-phisms}, whose morphisms are called \textit{$2$-mor\-phisms}, and whose composition is called \textit{vertical composition} (we represent 1-mor\-phisms as arrows $f : x \rightarrow y$, 2-mor\-phisms as double-struck arrows $\alpha : f \Rightarrow g$, and we denote vertical composition of 2-mor\-phisms $\alpha : f \Rightarrow g$ and $\beta : g \Rightarrow h$ with $\beta \ast \alpha : f \Rightarrow h$);
  \item A distinguished 1-mor\-phism $\id_x \in \bcalC(x,x)$ for every $x \in \bcalC$, called \textit{identity};
  \item A functor $C_{x,y,z} : \bcalC(y,z) \times \bcalC(x,y) \rightarrow \bcalC(x,z)$ for all $x,y,z \in \bcalC$, called \textit{horizontal composition functor} (we denote horizontal composition of 1-mor\-phisms $f : x \rightarrow y$ and $g : y \rightarrow z$ with $g \circ f : x \rightarrow z$, and of 2-mor\-phisms $\alpha : f \Rightarrow h$ and $\beta : g \Rightarrow i$ with $\beta \circ \alpha : g \circ f \Rightarrow i \circ h$);
  \item Natural isomorphisms
   \begin{gather*}
    \hspace{2\parindent} 
    \lambda_{x,y} : C_{x,y,y} \circ (\id_y, \id_{\bcalC(x,y)}) \Rightarrow \id_{\bcalC(x,y)}, \\
    \hspace{2\parindent} 
    \rho_{x,y} : C_{x,x,y} \circ (\id_{\bcalC(x,y)},\id_x) \Rightarrow \id_{\bcalC(x,y)}
   \end{gather*}
   for all $x,y \in \bcalC$, called \textit{left} and \textit{right unitor} respectively (we denote their components with $\lambda_f : \id_y \circ f \Rightarrow f$ and with $\rho_f : f \circ \id_x \Rightarrow f$ for every $1$-mor\-phism $f : x \rightarrow y$ of $\bcalC$);
  \item A natural isomorphism
   \[
    \hspace{2\parindent} 
    \alpha_{x,y,z,a} : C_{x,y,a} \circ (C_{y,z,a} \times \id_{\bcalC(x,y)}) \Rightarrow C_{x,z,a} \circ (\id_{\bcalC(z,a)} \times C_{x,y,z})
   \]
   for all $x,y,z,a \in \bcalC$, called \textit{associator} (we denote its components with $\alpha_{h,g,f} : (h \circ g) \circ f \Rightarrow h \circ (g \circ f)$ for all $1$-mor\-phisms $f : x \rightarrow y$, $g : y \rightarrow z$, and $h : z \rightarrow a$ of $\bcalC$).
 \end{enumerate}
 These data satisfy the following conditions:
 \begin{enumerate}
  \item The diagram of $2$-mor\-phisms of $\bcalC$
   \begin{center}
    \begin{tikzpicture}[descr/.style={fill=white}]
     \node (P0) at (90:{3.5/sqrt(3)}) {$g \circ ({\id_y} \circ f)$};
     \node (P1) at (90+120:{3.5/sqrt(3)}) {$(g \circ {\id_y}) \circ f$};
     \node (P2) at (90+2*120:{3.5/sqrt(3)}) {$g \circ f$};
     \draw
     (P1) edge[-implies,double equal sign distance] node[left,xshift=-5pt] {$\alpha_{g,\id_y,f}$} (P0)
     (P1) edge[-implies,double equal sign distance] node[below,yshift=-5pt] {$\rho_g \circ \id_f$} (P2)
     (P0) edge[-implies,double equal sign distance] node[right,xshift=5pt] {$\id_g \circ \lambda_f$} node[right,xshift=10pt] {$\phantom{\alpha_{g,\id_y,f}}$} (P2);
    \end{tikzpicture}
   \end{center}
   is commutative for all $1$-mor\-phisms $f : x \rightarrow y$ and $g : y \rightarrow z$ of $\bcalC$.
  \item The diagram of $2$-mor\-phisms of $\bcalC$
   \begin{center}
    \begin{tikzpicture}[descr/.style={fill=white}]
     \node (P0) at (-90:{3.5/sqrt((1-cos(72))^2+sin(72)^2)}) {$(i \circ h) \circ (g \circ f)$};
     \node (P1) at (-90+72:{3.5/sqrt((1-cos(72))^2+sin(72)^2)}) {$i \circ (h \circ (g \circ f))$};
     \node (P2) at (-90+2*72:{3.5/sqrt((1-cos(72))^2+sin(72)^2)}) {$i \circ ((h \circ g) \circ f)$};
     \node (P3) at (-90+3*72:{3.5/sqrt((1-cos(72))^2+sin(72)^2)}) {$(i \circ (h \circ g)) \circ f$};
     \node (P4) at (-90+4*72:{3.5/sqrt((1-cos(72))^2+sin(72)^2)}) {$((i \circ h) \circ g) \circ f$};
     \draw
     (P4) edge[-implies,double equal sign distance] node[left,xshift=-5pt,yshift=-5pt] {$\alpha_{i \circ h,g,f}$} (P0)
     (P0) edge[-implies,double equal sign distance] node[right,xshift=5pt,yshift=-5pt] {$\alpha_{i,h,g \circ f}$} (P1)
     (P2) edge[-implies,double equal sign distance] node[right,xshift=5pt] {$\id_i \circ \alpha_{h,g,f}$} node[right,xshift=5pt] {$\phantom{\alpha_{i,h,g} \circ \id_f}$} (P1)
     (P3) edge[-implies,double equal sign distance] node[above,yshift=5pt] {$\alpha_{i,h \circ g,f}$} (P2)
     (P4) edge[-implies,double equal sign distance] node[left,xshift=-5pt] {$\alpha_{i,h,g} \circ \id_f$} node[left,xshift=-5pt] {$\phantom{\id_i \circ \alpha_{h,g,f}}$} (P3);
    \end{tikzpicture}
   \end{center}
   is commutative for all $1$-mor\-phisms $f : x \rightarrow y$, $g : y \rightarrow z$, $h : z \rightarrow a$, and $i : a \rightarrow b$ of $\bcalC$.
 \end{enumerate}
\end{definition}


We move on to introduce $2$-func\-tors, which we can interpret as $1$-mor\-phisms between $2$-cat\-e\-go\-ries.

\begin{definition}\label{D:2-functor}
 A \textit{$2$-func\-tor} $\bfF : \bcalC \rightarrow \bcalC'$ between a pair of $2$-categories $\bcalC$, $\bcalC'$ is given by:
 \begin{enumerate}
  \item A function $\bfF : \bcalC \rightarrow \bcalC'$ (again, by abuse of notation, we confuse $2$-func\-tors with their functions on classes of objects);
  \item A functor $\bfF_{x,y} : \bcalC(x,y) \rightarrow \bcalC'(\bfF(x),\bfF(y))$ for all $x,y \in \bcalC$ (we denote its images with $\bfF(f) : \bfF(x) \rightarrow \bfF(y)$ and with $\bfF(\alpha) : \bfF(f) \Rightarrow \bfF(g)$ for all $1$-mor\-phisms $f,g : x \rightarrow y$ and every $2$-mor\-phism $\alpha : f \Rightarrow g$ of $\bcalC$);
  \item A distinguished invertible $2$-mor\-phism $\bfF_x : \id_{\bfF(x)} \Rightarrow \bfF(\id_x)$ of $\bcalC'$ for every $x \in \bcalC$;
  \item A natural isomorphism 
   \[
    \hspace{2\parindent} 
    \bfF_{x,y,z} : C_{\bfF(x),\bfF(y),\bfF(z)} \circ (\bfF_{y,z} \times \bfF_{x,y}) \Rightarrow \bfF_{x,z} \circ C_{x,y,z}
   \]
   for all $x,y,z \in \bcalC$ (we denote with $\bfF_{g,f} : \bfF(g) \circ \bfF(f) \Rightarrow \bfF(g \circ f)$ its compontents for all $1$-mor\-phisms $f : x \rightarrow y$ and $g : y \rightarrow z$ of $\bcalC$).
 \end{enumerate} 
 These data satisfy the following conditions:
 \begin{enumerate}
  \item The diagrams of $2$-mor\-phisms of $\bcalC'$ 
   \begin{center}
    \begin{tikzpicture}[descr/.style={fill=white}]
     \node (P0) at (45:{3.1/sqrt(2)}) {$\bfF({\id_y} \circ f)$};
     \node (P1) at (45+90:{3.1/sqrt(2)}) {$\bfF(\id_y) \circ \bfF(f)$};
     \node (P2) at (45+2*90:{3.1/sqrt(2)}) {${\id_{\bfF(y)}} \circ \bfF(f)$};
     \node (P3) at (45+3*90:{3.1/sqrt(2)}) {$\bfF(f)$};
     \draw
     (P1) edge[-implies,double equal sign distance] node[above,yshift=5pt] {$\bfF_{\id_y,f}$} (P0)
     (P2) edge[-implies,double equal sign distance] node[left,xshift=-5pt] {$\bfF_y \circ \id_{\bfF(f)}$} node[left,xshift=-5pt] {$\phantom{\bfF(\lambda_f)}$} (P1)
     (P2) edge[-implies,double equal sign distance] node[below,yshift=-5pt] {$\lambda_{\bfF(f)}$} (P3)
     (P0) edge[-implies,double equal sign distance] node[right,xshift=5pt] {$\bfF(\lambda_f)$} node[right,xshift=5pt] {$\phantom{\bfF_y \circ \id_{\bfF(f)}}$} (P3);
    \end{tikzpicture}
    \begin{tikzpicture}[descr/.style={fill=white}]
     \node (P0) at (45:{3.1/sqrt(2)}) {$\bfF(f \circ {\id_x})$};
     \node (P1) at (45+90:{3.1/sqrt(2)}) {$\bfF(f) \circ \bfF(\id_x)$};
     \node (P2) at (45+2*90:{3.1/sqrt(2)}) {$\bfF(f) \circ {\id_{\bfF(x)}}$};
     \node (P3) at (45+3*90:{3.1/sqrt(2)}) {$\bfF(f)$};
     \draw
     (P1) edge[-implies,double equal sign distance] node[above,yshift=5pt] {$\bfF_{f,\id_x}$} (P0)
     (P2) edge[-implies,double equal sign distance] node[left,xshift=-5pt] {$\id_{\bfF(f)} \circ \bfF_x$} node[left,xshift=-5pt] {$\phantom{\bfF(\rho_f)}$} (P1)
     (P2) edge[-implies,double equal sign distance] node[below,yshift=-5pt] {$\rho_{\bfF(f)}$} (P3)
     (P0) edge[-implies,double equal sign distance] node[right,xshift=5pt] {$\bfF(\rho_f)$} node[right,xshift=5pt] {$\phantom{\id_{\bfF(f)} \circ \bfF_x}$} (P3);
    \end{tikzpicture}
   \end{center}
   are commutative for every $1$-mor\-phism $f : x \rightarrow y$ of $\bcalC$; 
  \item The diagram of $2$-mor\-phisms of $\bcalC'$
   \begin{center}
    \begin{tikzpicture}[descr/.style={fill=white}]
     \node (P0) at (180:3.1) {$(\bfF(h) \circ \bfF(g)) \circ \bfF(f)$};
     \node[left,xshift=25pt] (P1) at (180-60:3.1) {$\bfF(h) \circ (\bfF(g) \circ \bfF(f))$};
     \node[right,xshift=-25pt] (P2) at (180-2*60:3.1) {$\bfF(h) \circ \bfF(g \circ f)$};
     \node (P3) at (180-3*60:3.1) {$\bfF(h \circ (g \circ f))$};
     \node[right,xshift=-25pt] (P4) at (180-4*60:3.1) {$\bfF((h \circ g) \circ f)$};
     \node[left,xshift=25pt] (P5) at (180-5*60:3.1) {$\bfF(h \circ g) \circ \bfF(f)$};
     \node (PH1) at (180-60:3.1) {$\hspace{50pt} \vphantom{\bfF(h) \circ (\bfF(g) \circ \bfF(f))}$};
     \node (PH2) at (180-2*60:3.1) {$\hspace{50pt} \vphantom{\bfF(h) \circ \bfF(g \circ f)}$};
     \node (PH4) at (180-4*60:3.1) {$\hspace{50pt} \vphantom{\bfF((h \circ g) \circ f)}$};
     \node (PH5) at (180-5*60:3.1) {$\hspace{50pt} \vphantom{\bfF(h \circ g) \circ \bfF(f)}$};
     \draw
     (PH2) edge[-implies,double equal sign distance] node[right,xshift=5pt] {$\bfF_{h,g \circ f}$} (P3)
     (P0) edge[-implies,double equal sign distance] node[left,xshift=-5pt] {$\alpha_{\bfF(h),\bfF(g),\bfF(f)}$} (PH1) 
     (PH1) edge[-implies,double equal sign distance] node[above,yshift=5pt] {$\id_{\bfF(h)} \circ \bfF_{g,f}$} (PH2)
     (P0) edge[-implies,double equal sign distance] node[left,xshift=-5pt] {$\bfF_{h,g} \circ \id_{\bfF(f)}$} node[left,xshift=-5pt] {$\phantom{\bfF(\alpha_{h,g,f})}$} (PH5)
     (PH5) edge[-implies,double equal sign distance] node[below,yshift=-5pt] {$\bfF_{h \circ g,f}$} (PH4)
     (PH4) edge[-implies,double equal sign distance] node[right,xshift=5pt] {$\bfF(\alpha_{h,g,f})$} node[right,xshift=5pt] {$\phantom{\bfF_{h,g} \circ \id_{\bfF(f)}}$} (P3);
    \end{tikzpicture}
   \end{center}
   is commutative for all $1$-mor\-phisms $f : x \rightarrow y$, $g : y \rightarrow z$, and $h : z \rightarrow a$ of $\bcalC$.
 \end{enumerate}
\end{definition}

Next, let us introduce $2$-mor\-phisms between $2$-func\-tors, which are called $2$-trans\-for\-ma\-tions.

\begin{definition}\label{D:2-transformation}
 A \textit{$2$-trans\-for\-ma\-tion} $\bfsigma : \bfF \Rightarrow \bfF'$ between a pair of $2$-func\-tors $\bfF,\bfF' : \bcalC \to \bcalC'$ is given by:
 \begin{enumerate}
  \item A 1-mor\-phism $\bfsigma_x : \bfF(x) \rightarrow \bfF'(x)$ of $\bcalC'$ for every $x \in \bcalC$, called \textit{component};
  \item A natural isomorphism 
   \[
    \hspace{2\parindent} 
    \bfsigma_{x,y} : (\bfsigma_x)^* \circ \bfF'_{x,y} \Rightarrow (\bfsigma_y)_* \circ \bfF_{x,y}
   \]
   for all $x,y \in \bcalC$ (we denote its components with $\bfsigma_f : \bfF'(f) \circ \bfsigma_x \Rightarrow \bfsigma_y \circ \bfF(f)$ for every $1$-mor\-phism $f : x \rightarrow y$ of $\bcalC$).
 \end{enumerate}
 These data satisfy the following conditions:
 \begin{enumerate}
  \item The diagram of $2$-mor\-phisms of $\bcalC'$
   \begin{center}
    \begin{tikzpicture}[descr/.style={fill=white}]
     \node (P0) at (-90:{2.7/sqrt((1-cos(72))^2+sin(72)^2)}) {$\bfF'(\id_x) \circ \bfsigma_x$};
     \node (P1) at (-90+72:{2.7/sqrt((1-cos(72))^2+sin(72)^2)}) {$\bfsigma_x \circ \bfF(\id_x)$};
     \node (P2) at (-90+2*72:{2.7/sqrt((1-cos(72))^2+sin(72)^2)}) {$\bfsigma_x \circ \id_{\bfF(x)}$};
     \node (P3) at (-90+3*72:{2.7/sqrt((1-cos(72))^2+sin(72)^2)}) {$\bfsigma_x$};
     \node (P4) at (-90+4*72:{2.7/sqrt((1-cos(72))^2+sin(72)^2)}) {$\id_{\bfF'(x)} \circ \bfsigma_x$};
     \draw
     (P4) edge[-implies,double equal sign distance] node[left,xshift=-5pt,yshift=-5pt] {$\bfF'_x \circ \id_{\bfsigma_x}$} (P0)
     (P0) edge[-implies,double equal sign distance] node[right,xshift=5pt,yshift=-5pt] {$\bfsigma_{\id_x}$} (P1)
     (P4) edge[-implies,double equal sign distance] node[left,xshift=-5pt] {$\lambda_{\bfsigma_x}$} node[left,xshift=-5pt] {$\phantom{\id_{\bfsigma_x} \circ \bfF_x}$} (P3)
     (P3) edge[-implies,double equal sign distance] node[above,yshift=5pt] {$\rho^{-1}_{\bfsigma_x}$} (P2)
     (P2) edge[-implies,double equal sign distance] node[right,xshift=5pt] {$\id_{\bfsigma_x} \circ \bfF_x$} (P1);
    \end{tikzpicture}
   \end{center}
   is commutative for every $x \in \bcalC$. 
  \item The diagram of $2$-mor\-phisms of $\bcalC'$
   \begin{center}
    \begin{tikzpicture}[descr/.style={fill=white}]
     \node (P0) at (0:{2.7/(2*sin(22.5))}) {$\bfsigma_z \circ (\bfF(g) \circ \bfF(f))$};
     \node (P1) at (45:{2.7/(2*sin(22.5))}) {$(\bfsigma_z \circ \bfF(g)) \circ \bfF(f)$};
     \node (P2) at (2*45:{2.7/(2*sin(22.5))}) {$(\bfF'(g) \circ \bfsigma_y) \circ \bfF(f)$};
     \node (P3) at (3*45:{2.7/(2*sin(22.5))}) {$\bfF'(g) \circ (\bfsigma_y \circ \bfF(f))$};
     \node (P4) at (4*45:{2.7/(2*sin(22.5))}) {$\bfF'(g) \circ (\bfF'(f) \circ \bfsigma_x)$};
     \node (P5) at (5*45:{2.7/(2*sin(22.5))}) {$(\bfF'(g) \circ \bfF'(f)) \circ \bfsigma_x$};
     \node (P6) at (6*45:{2.7/(2*sin(22.5))}) {$\bfF'(g \circ f) \circ \bfsigma_x$};
     \node (P7) at (7*45:{2.7/(2*sin(22.5))}) {$\bfsigma_z \circ \bfF(g \circ f)$};
     \draw
     (P5) edge[-implies,double equal sign distance] node[left,xshift=-5pt] {$\alpha_{\bfF'(g),\bfF'(f),\bfsigma_x}$} (P4)
     (P4) edge[-implies,double equal sign distance] node[left,xshift=-5pt] {$\id_{\bfF'(g)} \circ \bfsigma_f$} node[left,xshift=-5pt] {$\phantom{\alpha_{\bfsigma_z,\bfF(g),\bfF(f)}}$} (P3)
     (P3) edge[-implies,double equal sign distance] node[left,xshift=-10pt,yshift=5pt] {$\alpha_{\bfF'(g),\bfsigma_y,\bfF(f)}^{-1}$} (P2)
     (P2) edge[-implies,double equal sign distance] node[right,xshift=10pt,yshift=5pt] {$\bfsigma_g \circ \id_{\bfF(f)}$} (P1)
     (P1) edge[-implies,double equal sign distance] node[right,xshift=5pt] {$\alpha_{\bfsigma_z,\bfF(g),\bfF(f)}$} (P0)
     (P0) edge[-implies,double equal sign distance] node[right,xshift=5pt] {$\id_{\bfsigma_z} \circ \bfF_{g,f}$} node[right,xshift=5pt] {$\phantom{\alpha_{\bfF'(g),\bfF'(f),\bfsigma_x}}$} (P7)
     (P5) edge[-implies,double equal sign distance] node[left,xshift=-10pt,yshift=-5pt] {$\bfF'_{g,f} \circ \id_{\bfsigma_x}$} (P6)
     (P6) edge[-implies,double equal sign distance] node[right,xshift=10pt,yshift=-5pt] {$\bfsigma_{g \circ f}$} (P7);
    \end{tikzpicture}
   \end{center}
   is commutative for all $1$-mor\-phisms $f : x \rightarrow y$ and $g : y \rightarrow z$ of $\bcalC$.
 \end{enumerate}
\end{definition}

Now, we introduce what can be thought of as a notion of $3$-mor\-phism between $2$-trans\-for\-ma\-tions: $2$-mod\-i\-fi\-ca\-tions.

\begin{definition}\label{D:2-modification}
 A \textit{$2$-mod\-i\-fi\-ca\-tion} $\bfGamma : \bfsigma \Rrightarrow \bfsigma'$ between a pair of 2-trans\-for\-ma\-tions $\bfsigma,\bfsigma' : \bfF \Rightarrow \bfF'$ between $2$-func\-tors $\bfF,\bfF' : \bcalC \rightarrow \bcalC'$ is given by a 2-mor\-phism $\bfGamma_x : \bfsigma_x \Rightarrow \bfsigma'_x$ of $\bcalC'$ for every $x \in \bcalC$, called \textit{component}, such that the diagram of 2-mor\-phisms of $\bcalC'$
 \begin{center}
  \begin{tikzpicture}[descr/.style={fill=white}]
   \node (P0) at (0:{3.0/sqrt(2)}) {$\bfsigma'_y \circ \bfF(f)$};
   \node (P1) at (90:{3.0/sqrt(2)}) {$\bfF'(f) \circ \bfsigma'_x$};
   \node (P2) at (2*90:{3.0/sqrt(2)}) {$\bfF'(f) \circ \bfsigma_x$};
   \node (P3) at (3*90:{3.0/sqrt(2)}) {$\bfsigma_y \circ \bfF(f)$};
   \node (PH0) at (0:{3.0/sqrt(2)}) {$\phantom{\bfF'(f) \circ \bfsigma_x}$};
   \node (PH2) at (2*90:{3.0/sqrt(2)}) {$\phantom{\bfsigma'_y \circ \bfF(f)}$};
   \draw
   (P1) edge[-implies,double equal sign distance] node[right,xshift=5pt,yshift=5pt] {$\bfsigma'_f$} node[right,xshift=5pt,yshift=5pt] {$\phantom{\id_{\bfF'(f)} \circ \bfGamma_x}$} (P0)
   (P2) edge[-implies,double equal sign distance] node[left,xshift=-5pt,yshift=5pt] {$ \hspace{2\parindent}\id_{\bfF'(f)} \circ \bfGamma_x$} (P1)
   (P2) edge[-implies,double equal sign distance] node[left,xshift=-5pt,yshift=-5pt] {$\bfsigma_f$} node[left,xshift=-5pt,yshift=-5pt] {$\phantom{\bfGamma_y \circ \id_{\bfF(f)}}$} (P3)
   (P3) edge[-implies,double equal sign distance] node[right,xshift=5pt,yshift=-5pt] {$\bfGamma_y \circ \id_{\bfF(f)}$} (P0);
  \end{tikzpicture}
 \end{center}
 is commutative for every $1$-mor\-phism $f : x \rightarrow y$ of $\bcalC$.
\end{definition}

The simplest examples of all these morphisms are provided by identities. Indeed, remark that the identity $2$-func\-tor associated with a $2$-cat\-e\-go\-ry $\bcalC$ can be readily defined using only identity functions, functors, and 2-mor\-phisms of $\bcalC$. Analogously, the identity $2$-trans\-for\-ma\-tion associated with a  $2$-func\-tor $\bfF : \bcalC \rightarrow \bcalC'$ can be easily defined using identity 1-mor\-phisms and left and right unitors in $\bcalC'$, while the identity $2$-mod\-i\-fi\-ca\-tion associated with a $2$-trans\-for\-ma\-tion $\bfsigma : \bfF \Rightarrow \bfF'$ between $2$-func\-tors $\bfF, \bfF' : \bcalC \rightarrow \bcalC'$ only requires the use of identity 2-mor\-phisms in $\bcalC'$.
Next, all these morphisms can be composed in various ways. First of all, we define composition of $2$-func\-tors.

\begin{definition}\label{D:comp_of_2-functors}
 The \textit{composition $\bfF' \circ \bfF : \bcalC \rightarrow \bcalC''$ of $2$-func\-tors $\bfF : \bcalC \rightarrow \bcalC'$ and $\bfF' : \bcalC' \rightarrow \bcalC''$} between $2$-cat\-e\-go\-ries $\bcalC$, $\bcalC'$, $\bcalC''$ is obtained by specifying:
 \begin{enumerate}
  \item $\bfF' \circ \bfF : \bcalC \rightarrow \bcalC''$ between classes of objects;
  \item $\bfF'_{\bfF(x),\bfF(y)} \circ \bfF_{x,y}$ as $(\bfF' \circ \bfF)_{x,y} : \bcalC(x,y) \rightarrow \bcalC''(\bfF'(\bfF(x)),\bfF'(\bfF(y)))$ for all $x,y \in \bcalC$;
  \item $\bfF'(\bfF_x) \circ \bfF'_{\bfF(x)}$ as $(\bfF' \circ \bfF)_x : \id_{\bfF'(\bfF(x))} \Rightarrow \bfF'(\bfF(\id_x))$ for every $x \in \bcalC$;
  \item $\bfF'(\bfF_{g,f}) \circ \bfF'_{\bfF(g),\bfF(f)}$ as $(\bfF' \circ \bfF)_{g,f} : \bfF'(\bfF(g)) \circ \bfF'(\bfF(f)) \Rightarrow \bfF'(\bfF(g \circ f))$ for all $1$-mor\-phisms $f : x \rightarrow y$ and $g : y \rightarrow z$ of $\bcalC$;
 \end{enumerate}
\end{definition}

Then, we have two operations, called \textit{left} and \textit{right whiskering}, which mix $2$-func\-tors with $2$-trans\-for\-ma\-tions and $2$-mod\-i\-fi\-ca\-tions.

\begin{definition}\label{D:left_whisker_2-transf}
 The \textit{left whiskering $\bfF'' \triangleright \bfsigma : \bfF'' \circ \bfF \Rightarrow \bfF'' \circ \bfF'$ of a $2$-trans\-for\-ma\-tion $\bfsigma : \bfF \Rightarrow \bfF'$ with a $2$-func\-tor $\bfF'' : \bcalC' \to \bcalC''$} for $2$-func\-tors $\bfF,\bfF' : \bcalC \to \bcalC'$ is obtained by specifying:
 \begin{enumerate}
  \item $\bfF''(\bfsigma_x)$ as the component $(\bfF'' \triangleright \bfsigma)_x : \bfF''(\bfF(x)) \rightarrow \bfF''(\bfF'(x))$ for every $x \in \bcalC$;
  \item $\bfF''^{-1}_{\bfsigma_y,\bfF(f)} \ast \bfF''(\bfsigma_f) \ast \bfF''_{\bfF'(f),\bfsigma_x}$ as 
   \[
    \hspace{2\parindent} 
    (\bfF'' \triangleright \bfsigma)_f : \bfF''(\bfF'(f)) \circ \bfF''(\bfsigma_x) \Rightarrow \bfF''(\bfsigma_y) \circ \bfF''(\bfF(f))
   \]
   for every $1$-mor\-phism $f : x \rightarrow y$ of $\bcalC$.
 \end{enumerate}
\end{definition}

\begin{definition}\label{D:right_whisker_2-transf}
 The \textit{right whiskering $\bfsigma \triangleleft \bfF : \bfF' \circ \bfF \Rightarrow \bfF'' \circ \bfF$ of a $2$-trans\-for\-ma\-tion $\bfsigma : \bfF' \Rightarrow \bfF''$ with a $2$-func\-tor $\bfF : \bcalC \to \bcalC'$} for $2$-func\-tors $\bfF',\bfF'' : \bcalC' \to \bcalC''$ is obtained by specifying:
 \begin{enumerate}
  \item $\bfsigma_{\bfF(x)}$ as the component $(\bfsigma \triangleleft \bfF)_x : \bfF'(\bfF(x)) \rightarrow \bfF''(\bfF(x))$ for every $x \in \bcalC$;
  \item $\bfsigma_{\bfF(f)}$ as $(\bfsigma \triangleleft \bfF)_f : \bfF''(\bfF(f)) \circ \bfsigma_{\bfF(x)} \Rightarrow \bfsigma_{\bfF(y)} \circ \bfF'(\bfF(f))$ for every $1$-mor\-phism $f : x \rightarrow y$ of $\bcalC$.
 \end{enumerate}
\end{definition}

\begin{definition}\label{D:left_whisker_2-modif}
 The \textit{left whiskering $\bfF'' \triangleright \bfGamma : \bfF'' \triangleright \bfsigma \Rrightarrow \bfF'' \triangleright \bfsigma'$ of a $2$-mod\-i\-fi\-ca\-tion $\bfGamma : \bfsigma \Rightarrow \bfsigma'$ with a $2$-func\-tor $\bfF'' : \bcalC' \to \bcalC''$} for $2$-trans\-for\-ma\-tions ${\bfsigma,\bfsigma' : \bfF \Rightarrow \bfF'}$ between $2$-func\-tors $\bfF,\bfF' : \bcalC \to \bcalC'$ is obtained by specifying $\bfF''(\bfGamma_x)$ as the component $(\bfF'' \triangleright \bfGamma)_x : \bfF''(\bfsigma_x) \Rightarrow \bfF''(\bfsigma'_x)$ for every $x \in \bcalC$.
\end{definition}

\begin{definition}\label{D:right_whisker_2-modif}
 The \textit{right whiskering $\bfGamma \triangleleft \bfF : \bfsigma \triangleleft \bfF \Rrightarrow \bfsigma' \triangleleft \bfF$ of a $2$-mod\-i\-fi\-ca\-tion $\bfGamma : \bfsigma \Rightarrow \bfsigma'$ with a $2$-func\-tor $\bfF : \bcalC \to \bcalC'$} for $2$-trans\-for\-ma\-tions ${\bfsigma,\bfsigma' : \bfF' \Rightarrow \bfF''}$ between $2$-func\-tors $\bfF',\bfF'' : \bcalC' \to \bcalC''$ is obtained by specifying $\bfGamma_{\bfF(x)}$ as the component $(\bfGamma \triangleleft \bfF)_x : \bfsigma_{\bfF(x)} \Rightarrow \bfsigma'_{\bfF(x)}$ for every $x \in \bcalC$.
\end{definition}

Next, we define horizontal composition of $2$-trans\-for\-ma\-tions, which involves associators of the target $2$-cat\-e\-go\-ry.

\begin{definition}\label{D:hor_comp_of_2-transf}
 The \textit{horizontal composition $\bfsigma' \circ \bfsigma : \bfF \Rightarrow \bfF''$ of $2$-trans\-for\-ma\-tions $\bfsigma : \bfF \Rightarrow \bfF'$ and $\bfsigma' : \bfF' \Rightarrow \bfF''$} between $2$-func\-tors $\bfF,\bfF',\bfF'' : \bcalC \to \bcalC'$ is obtained by specifying:
 \begin{enumerate}
  \item $\bfsigma'_x \circ \bfsigma_x$ as the component $(\bfsigma' \circ \bfsigma)_x : \bfF(x) \rightarrow \bfF''(x)$ for every $x \in \bcalC$;
  \item $\alpha^{-1}_{\bfsigma'_y,\bfsigma_y,\bfF(f)} \ast \left( \id_{\bfsigma'_y} \circ \bfsigma_f \right) \ast \alpha_{\bfsigma'_y,\bfF'(f),\bfsigma_x} \ast \left( \bfsigma'_f \circ \id_{\bfsigma_x} \right) \ast \alpha^{-1}_{\bfF''(f),\bfsigma'_x,\bfsigma_x}$ as
   \[
    \hspace{2\parindent} 
    (\bfsigma' \circ \bfsigma)_f : \bfF''(f) \circ (\bfsigma'_x \circ \bfsigma_x) \Rightarrow (\bfsigma'_y \circ \bfsigma_y) \circ \bfF(f)
   \]
   for every $1$-mor\-phism $f : x \rightarrow y$ of $\bcalC$.
 \end{enumerate}
\end{definition}

Then, we define horizontal and vertical composition of $2$-trans\-for\-ma\-tions simply using horizontal and vertical composition of 2-mor\-phisms in the target $2$-cat\-e\-go\-ry.

\begin{definition}\label{D:hor_comp_of_2-modif}
 The \textit{horizontal composition $\bfGamma' \circ \bfGamma : \bfsigma' \circ \bfsigma \Rrightarrow \bfsigma''' \circ \bfsigma''$ of $2$-mod\-i\-fi\-ca\-tions $\bfGamma : \bfsigma \Rrightarrow \bfsigma''$ and $\bfGamma' : \bfsigma' \Rrightarrow \bfsigma'''$} between $2$-trans\-for\-ma\-tions $\bfsigma,\bfsigma'' : \bfF \Rightarrow \bfF'$ and $\bfsigma',\bfsigma''' : \bfF' \Rightarrow \bfF''$ between $2$-func\-tors $\bfF,\bfF',\bfF'' : \bcalC \to \bcalC'$ is obtained by specifying $\bfGamma'_x \circ \bfGamma_x$ as the component $(\bfGamma' \circ \bfGamma)_x : \bfsigma'_x \circ \bfsigma_x \Rightarrow \bfsigma'''_x \circ \bfsigma''_x$ for every $x \in \bcalC$.
\end{definition}

\begin{definition}\label{D:vert_comp_of_2-modif}
 The \textit{vertical composition $\bfGamma' \ast \bfGamma : \bfsigma \Rrightarrow \bfsigma''$ of $2$-mod\-i\-fi\-ca\-tions $\bfGamma : \bfsigma \Rrightarrow \bfsigma'$ and $\bfGamma' : \bfsigma' \Rrightarrow \bfsigma''$} between $2$-trans\-for\-ma\-tions $\bfsigma,\bfsigma',\bfsigma'' : \bfF \Rightarrow \bfF'$ between $2$-func\-tors $\bfF, : \bcalC \to \bcalC'$ is obtained by specifying $\bfGamma'_x \ast \bfGamma_x$ as the component $(\bfGamma' \ast \bfGamma)_x : \bfsigma_x \Rightarrow \bfsigma''_x$ for every $x \in \bcalC$.
\end{definition}

Now, as explained in Definition A.9 of \cite{S11}, $2$-func\-tors from $\bcalC$ to $\bcalC'$, together with $2$-trans\-for\-ma\-tions and $2$-mod\-i\-fi\-ca\-tions between them, can be arranged into a $2$-cat\-e\-go\-ry, which we denote $\twoCat(\bcalC,\bcalC')$. 
Remark that our notation hints to the fact that there exists a notion of $3$-cat\-e\-go\-ry, and that $2$-cat\-e\-go\-ries, together with $2$-func\-tors, $2$-trans\-for\-ma\-tions, and $2$-mod\-i\-fi\-ca\-tions between them, provide an example, see \cite{GPS95}. We move on to recall an important general result, belonging to a family of statements known as \textit{coherence theorems}, which essentially tells us we have the right to assume triviality of coherence data for $2$-cat\-e\-go\-ries. In order to explain what we mean precisely, we first need to give a few more definitions. We say two objects $x$ and $y$ of a $2$-cat\-e\-go\-ry $\bcalC$ are \textit{equivalent} if there exist $1$-mor\-phisms $f : x \rightarrow y$ and $g : y \rightarrow x$ together with invertible 2-mor\-phisms of the form $\eta : \id_x \Rightarrow g \circ f$ and $\varepsilon : f \circ g \Rightarrow \id_{y}$, in which case both $f$ and $g$ are called \textit{equivalences}. If moreover $\eta$ and $\varepsilon$ satisfy $({\id_g} \circ \varepsilon) \ast (\eta \circ {\id_g}) = \id_g$ and $(\varepsilon \circ {\id_f}) \ast ({\id_f} \circ \eta) = \id_f$, then we say 
$f$ and $g$ are \textit{adjoint equivalences}, with \textit{unit} $\eta$ and \textit{counit} $\varepsilon$. Remark that every equivalence $f : x \rightarrow y$ between objects $x$ and $y$ of $\bcalC$ admits an adjoint equivalence, as proved in Proposition A.27 of \cite{S11}. This notion of equivalence is said to be \textit{internal}, because it applies to objects of a $2$-cat\-e\-go\-ry. However, there is also an \textit{external} notion of equivalence, which applies directly to $2$-cat\-e\-go\-ries.
We say two $2$-cat\-e\-go\-ries $\bcalC$ and $\bcalC'$ are \textit{equivalent} if there exist $2$-func\-tors $\bfF : \bcalC \rightarrow \bcalC'$ and $\bfF' : \bcalC' \rightarrow \bcalC$ such that $\bfF' \circ \bfF$ is equivalent to $\id_{\bcalC}$ in $\twoCat(\bcalC,\bcalC)$, and such that $\bfF \circ \bfF'$ is equivalent to $\id_{\bcalC'}$ in $\twoCat(\bcalC',\bcalC')$, in which case both $\bfF$ and $\bfF'$ are called \textit{equivalences}. Our goal now is to make sure we can replace any $2$-cat\-e\-go\-ry with a nicer one, so let us say what a nice $2$-cat\-e\-go\-ry is.

\begin{definition}\label{D:strict_2-cat}
 A $2$-cat\-e\-go\-ry $\bcalC$ is \textit{strict} if $\lambda_{x,y}$, $\rho_{x,y}$, and $\alpha_{x,y,z}$ are identity natural transformations for all $x,y,z \in \bcalC$.
\end{definition}

The following result is a direct consequence of MacLane's famous original coherence theorem from Chapter 7 of \cite{M78}. See for instance Corollary A.18 of \cite{S11} for a discussion, but see also the short paper \cite{L98} for a very concise proof.

\begin{theorem}\label{T:coherence_for_2-cat}
 Every $2$-cat\-e\-go\-ry is equivalent to a strict $2$-cat\-e\-go\-ry.
\end{theorem}

This is why, outside of this chapter, we always assume unitors and associators of $2$-cat\-e\-go\-ries to be identities. More precisely, whenever we consider a $2$-cat\-e\-go\-ry, we secretly replace it with a strict version of itself without changing our notation. Remark that, since monoidal categories can be realized as $2$-cat\-e\-go\-ries with a single object, Theorem \ref{T:coherence_for_2-cat} can in particular be seen as a coherence result for monoidal categories. We also have a similar notion for $2$-func\-tors.

\begin{definition}\label{D:strict_2-funct}
 A $2$-func\-tor $\bfF : \bcalC \rightarrow \bcalC'$ is \textit{strict} if $\bfF_x$ and $\bfF_{g,f}$ are identity $2$-mor\-phisms of $\bcalC'$ for every $x \in \bcalC$ and for all $1$-mor\-phisms $f : x \rightarrow y$ and $g : y \rightarrow z$ of $\bcalC$.
\end{definition}

Then, as explained in Section A.2 of \cite{S11}, we also have a coherence result for $2$-func\-tors, although we do not make use of it in our construction. Loosely speaking, we have a functor which associates with every $2$-cat\-e\-go\-ry $\bcalC$ an equivalent strict $2$-cat\-e\-go\-ry, called the \textit{strictification} of $\bcalC$, and which associates with every $2$-func\-tor $\bfF$ between $2$-cat\-e\-go\-ries $\bcalC$ and $\bcalC'$ a strict $2$-func\-tor, called the \textit{strictification} of $\bfF$, between the strictifications of $\bcalC$ and $\bcalC'$. This means we can always assume coherence data of $2$-func\-tors to be trivial.

\section{Symmetric monoidal structures}

In this section we introduce symmetric monoidal $2$-cat\-e\-go\-ries. Their definition, at least in the form given here made, its first appearance in \cite{S16}, but see also Section 2.1 of \cite{S11} for a full account of the history behind the concept. 

\begin{definition}\label{D:symmetric_monoidal_2-category}
 A \textit{symmetric monoidal $2$-cat\-e\-go\-ry} is a $2$-cat\-e\-go\-ry $\bcalC$ together with:
 \begin{enumerate}
  \item A distinguished object $\one \in \bcalC$ called \textit{tensor unit};
  \item A $2$-func\-tor $\otimes : \bcalC \times \bcalC \rightarrow \bcalC$ called \textit{tensor product};
  \item Equivalence $2$-trans\-for\-ma\-tions
   \begin{gather*}
    \hspace{2\parindent} 
    \bflambda : \otimes \circ (\one,\id_{\bcalC}) \Rightarrow \id_{\bcalC}, \\
    \hspace{2\parindent} 
    \bfrho : \otimes \circ (\id_{\bcalC},\one) \Rightarrow \id_{\bcalC}
   \end{gather*}
   called \textit{left} and \textit{right unitor} respectively, with specified 
   adjoint equivalence $2$-trans\-for\-ma\-tions $\bflambda^*$ and $\bfrho^*$, where the $2$-func\-tors $(\one,\id_{\bcalC}) : \bcalC \rightarrow \bcalC \times \bcalC$ and $(\id_{\bcalC},\one) : \bcalC \rightarrow \bcalC \times \bcalC$ are induced by $\id_{\bcalC}$;
  \item An equivalence $2$-trans\-for\-ma\-tion 
   \[
    \hspace{2\parindent}     
    \bfalpha : \otimes \circ (\otimes \times \id_{\bcalC}) \Rightarrow \otimes \circ (\id_{\bcalC} \times \otimes)
   \]
   called \textit{associator}, with a specified adjoint equivalence $2$-trans\-for\-ma\-tion $\bfalpha^*$;
  \item An equivalence $2$-trans\-for\-ma\-tion 
   \[
    \hspace{2\parindent}
    \twobr : \otimes \Rightarrow \otimes \circ \bfT
   \]
   called \textit{braiding}, with a specified adjoint equivalence $2$-trans\-for\-ma\-tions $\twobr^*$, where the $2$-func\-tor $\bfT : \bcalC \times \bcalC \rightarrow \bcalC \times \bcalC$ transposes the two copies of $\bcalC$;
  \item Invertible $2$-mod\-i\-fi\-ca\-tions $\bfLambda$, $\bfP$, $\bfM$, $\bfPi$, $\bfA$, $\bfB$, and $\bfSigma$, whose components have the form
   \begin{center}
    \begin{tikzpicture}[descr/.style={fill=white}]
     \node (P0) at (90:{2.5/sqrt(3)}) {$\one \otimes (x \otimes y)$};
     \node (P1) at (90+120:{2.5/sqrt(3)}) {$(\one \otimes x) \otimes y$};
     \node (P2) at (90+2*120:{2.5/sqrt(3)}) {$x \otimes y$};
     \node (PH) at (90+2*120:{2.5/sqrt(3)}) {$\phantom{(\one \otimes x) \otimes y}$};
     \node[above] at (0,0) {$\Downarrow$};
     \node[below] at (0,0) {$\bfLambda_{x,y}$};
     \draw
     (P1) edge[->] node[left,xshift=-5pt] {$\bfalpha_{\one,x,y}$} (P0)
     (P0) edge[->] node[right,xshift=5pt] {$\bflambda_{x \otimes y}$} (P2)
     (P1) edge[->] node[below,yshift=-5pt] {$\bflambda_{x} \otimes \id_y \vphantom{\bfrho_{x \otimes y}}$} (P2);
    \end{tikzpicture}
    \begin{tikzpicture}[descr/.style={fill=white}]
     \node (P0) at (90:{2.5/sqrt(3)}) {$x \otimes (y \otimes \one)$};
     \node (P1) at (90+120:{2.5/sqrt(3)}) {$(x \otimes y) \otimes \one$};
     \node (P2) at (90+2*120:{2.5/sqrt(3)}) {$x \otimes y$};
     \node (PH) at (90+2*120:{2.5/sqrt(3)}) {$\phantom{(x \otimes y) \otimes \one}$};
     \node[above] at (0,0) {$\Downarrow$};
     \node[below] at (0,0) {$\bfP_{x,y}$};
     \draw
     (P1) edge[->] node[left,xshift=-5pt] {$\bfalpha_{x,y,\one}$} (P0)
     (P0) edge[->] node[right,xshift=5pt] {$\id_x \otimes \bfrho_{y}$} (P2)
     (P1) edge[->] node[below,yshift=-5pt] {$\bfrho_{x \otimes y} \vphantom{\bflambda_{x} \otimes \id_y}$} (P2);
    \end{tikzpicture}
    \begin{tikzpicture}[descr/.style={fill=white}]
     \node (P0) at (90:{2.5/sqrt(3)}) {$x \otimes (\one \otimes y)$};
     \node (P1) at (90+120:{2.5/sqrt(3)}) {$(x \otimes \one) \otimes y$};
     \node (P2) at (90+2*120:{2.5/sqrt(3)}) {$x \otimes y$};
     \node (PH) at (90+2*120:{2.5/sqrt(3)}) {$\phantom{(x \otimes \one) \otimes y}$};
     \node[above] at (0,0) {$\Downarrow$};
     \node[below] at (0,0) {$\bfM_{x,y}$};
     \draw
     (P1) edge[->] node[left,xshift=-5pt] {$\bfalpha_{x,\one,y}$} (P0)
     (P0) edge[->] node[right,xshift=5pt] {$\id_x \otimes \bflambda_{y}$} (P2)
     (P1) edge[->] node[below,yshift=-5pt] {$\bfrho_{x} \otimes \id_y$} (P2);
    \end{tikzpicture}
   \end{center}
   \begin{center}
    \begin{tikzpicture}[descr/.style={fill=white}]
     \node (P0) at (-90:{2.5/sqrt((1-cos(72))^2+sin(72)^2)}) {$(x \otimes y) \otimes (z \otimes a)$};
     \node (P1) at (-90+72:{2.5/sqrt((1-cos(72))^2+sin(72)^2)}) {$x \otimes (y \otimes (z \otimes a))$};
     \node[right,xshift=-25pt] (P2) at (-90+2*72:{2.5/sqrt((1-cos(72))^2+sin(72)^2)}) {$x \otimes ((y \otimes z) \otimes a)$};
     \node[left,xshift=25pt] (P3) at (-90+3*72:{2.5/sqrt((1-cos(72))^2+sin(72)^2)}) {$(x \otimes (y \otimes z)) \otimes a$};
     \node (P4) at (-90+4*72:{2.5/sqrt((1-cos(72))^2+sin(72)^2)}) {$((x \otimes y) \otimes z) \otimes a$};
     \node (PH2) at (-90+2*72:{2.5/sqrt((1-cos(72))^2+sin(72)^2)}) {$\hspace{50pt} \vphantom{x \otimes ((y \otimes z) \otimes a)}$};
     \node (PH3) at (-90+3*72:{2.5/sqrt((1-cos(72))^2+sin(72)^2)}) {$\hspace{50pt} \vphantom{(x \otimes (y \otimes z)) \otimes a}$};
     \node[above] at (0,0) {$\Downarrow$};
     \node[below] at (0,0) {$\bfPi_{x,y,z,a}$};
     \draw
     (P4) edge[->] node[left,xshift=-5pt,yshift=-5pt] {$\bfalpha_{x \otimes y,z,a}$} (P0)
     (P0) edge[->] node[right,xshift=5pt,yshift=-5pt] {$\bfalpha_{x,y,z \otimes a}$} (P1)
     (PH2) edge[->] node[right,xshift=5pt] {$\id_x \otimes \bfalpha_{y,z,a}$} node[right,xshift=5pt] {$\phantom{\bfalpha_{x,y,z} \otimes \id_a}$} (P1)
     (PH3) edge[->] node[above,yshift=5pt] {$\bfalpha_{x,y \otimes z,a}$} (PH2)
     (P4) edge[->] node[left,xshift=-5pt] {$\bfalpha_{x,y,z} \otimes \id_a$} node[left,xshift=-5pt] {$\phantom{\id_x \otimes \bfalpha_{y,z,a}}$} (PH3);
   \end{tikzpicture}
  \end{center}
  \begin{center}
    \begin{tikzpicture}[descr/.style={fill=white}]
     \node (P0) at (0:2.5) {$y \otimes (z \otimes x)$};
     \node (P1) at (60:2.5) {$(y \otimes z) \otimes x$};
     \node (P2) at (2*60:2.5) {$x \otimes (y \otimes z)$};
     \node (P3) at (3*60:2.5) {$(x \otimes y) \otimes z$};
     \node (P4) at (4*60:2.5) {$(y \otimes x) \otimes z$};
     \node (P5) at (5*60:2.5) {$y \otimes (x \otimes z)$};
     \node at (0,0) {$\Downarrow$};
     \node[below,yshift=-5pt] at (0,0) {$\bfA_{x,y,z}$};
     \draw
     (P3) edge[->] node[left,xshift=-5pt] {$\bfalpha_{x,y,z}$} (P2)
     (P2) edge[->] node[above,yshift=5pt] {$\twobr_{x,y \otimes z}$} (P1)
     (P1) edge[->] node[right,xshift=5pt] {$\bfalpha_{y,z,x}$} (P0)
     (P3) edge[->] node[left,xshift=-5pt] {$\twobr_{x,y} \otimes \id_z$} (P4)
     (P4) edge[->] node[below,yshift=-5pt] {$\bfalpha_{y,x,z}$} (P5)
     (P5) edge[->] node[right,xshift=5pt] {$\id_y \otimes \twobr_{x,z}$} (P0);
    \end{tikzpicture}
   \end{center}
   \begin{center}
    \begin{tikzpicture}[descr/.style={fill=white}]
     \node (P0) at (0:2.5) {$(z \otimes x) \otimes y$};
     \node (P1) at (60:2.5) {$z \otimes (x \otimes y)$};
     \node (P2) at (2*60:2.5) {$(x \otimes y) \otimes z$};
     \node (P3) at (3*60:2.5) {$x \otimes (y \otimes z)$};
     \node (P4) at (4*60:2.5) {$x \otimes (z \otimes y)$};
     \node (P5) at (5*60:2.5) {$(x \otimes z) \otimes y$};
     \node[above] at (0,0) {$\Downarrow$};
     \node[below] at (0,0) {$\bfB_{x,y,z}$};
     \draw
     (P3) edge[->] node[left,xshift=-5pt] {$\bfalpha^*_{x,y,z}$} (P2)
     (P2) edge[->] node[above,yshift=5pt] {$\twobr_{x \otimes y,z}$} (P1)
     (P1) edge[->] node[right,xshift=5pt] {$\bfalpha^*_{z,x,y}$} (P0)
     (P3) edge[->] node[left,xshift=-5pt] {$\id_x \otimes \twobr_{y,z}$} (P4)
     (P4) edge[->] node[below,yshift=-5pt] {$\bfalpha^*_{y,z,x}$} (P5)
     (P5) edge[->] node[right,xshift=5pt] {$\twobr_{x,z} \otimes \id_y$} (P0);
    \end{tikzpicture}
   \end{center}
   \begin{center}
    \begin{tikzpicture}[descr/.style={fill=white}]
     \node (P0) at (90:{2.5/sqrt(3)}) {$y \otimes x$};
     \node (P1) at (90+120:{2.5/sqrt(3)}) {$x \otimes y$};
     \node (P2) at (90+2*120:{2.5/sqrt(3)}) {$x \otimes y$};
     \node[above] at (0,0) {$\Downarrow$};
     \node[below] at (0,0) {$\bfSigma_{x,y}$};
     \draw
     (P1) edge[->] node[left,xshift=-5pt] {$\twobr_{x,y}$} (P0)
     (P0) edge[->] node[right,xshift=5pt] {$\twobr_{y,x}$} (P2)
     (P1) edge[->] node[below,yshift=-5pt] {$\id_{x \otimes y}$} (P2);
    \end{tikzpicture}
   \end{center}
  for all $x,y,z,a \in \bcalC$;
 \end{enumerate}
 These data satisfy ten conditions in the form of commuting polytopes of 2-mod\-i\-fi\-ca\-tions, which can be obtained from Definitions 4.4 through 4.8 of \cite{S16}, as explained in Definition 2.3 of \cite{S11}.
\end{definition}

Remark that Definitions 4.4 through 4.8 of \cite{S16} actually give a total of twelve relations, but, as pointed out in Definition C.1 of \cite{S11}, two of them are redundant. Remark also that all sources and targets for the $2$-mod\-i\-fi\-ca\-tions $\bfLambda$, $\bfP$, $\bfM$, $\bfPi$, $\bfA$, $\bfB$, and $\bfSigma$ can be explicitly written in terms of tensor unit, tensor product, left and right unitor, associator, and braiding, by perfoming compositions, whiskerings, and cartesian products, although this takes up some space, which is the reason why we did not do it. Also, see Section 2.2 of \cite{S11} for an explanation of how to make sense of the commuting polytopes of 2-mod\-i\-fi\-ca\-tions of \cite{S16}. Now, let us introduce symmetric monoidal $2$-func\-tors between symmetric monoidal $2$-cat\-e\-go\-ries.

\begin{definition}\label{D:symmetric_monoidal_2-functor}
 A \textit{symmetric monoidal $2$-func\-tor} is a $2$-func\-tor $\bfF : \bcalC \rightarrow \bcalC'$ between symmetric monoidal $2$-cat\-e\-go\-ries $\bcalC$ and $\bcalC'$ together with:
 \begin{enumerate}
  \item An equivalence 1-mor\-phism $\bfiota : \one' \rightarrow \bfF(\one)$ of $\bcalC'$ with a specified adjoint equivalence 1-mor\-phism $\bfiota^*$ of $\bcalC'$;
  \item An equivalence $2$-trans\-for\-ma\-tion 
   \[
    \hspace{2\parindent} 
    \bfchi : \otimes' \circ (\bfF \times \bfF) \Rightarrow \bfF \circ \otimes
   \]
   with a specified adjoint equivalence $2$-trans\-for\-ma\-tion $\bfchi^*$;
 \item Invertible $2$-mod\-i\-fi\-ca\-tions $\bfGamma$, $\bfDelta$, $\bfOmega$, and $\bfTheta$ whose components have the form
  \begin{center}
   \begin{tikzpicture}[descr/.style={fill=white}]
    \node (P0) at (45:{2.7/sqrt(2)}) {$\bfF(\one \otimes x)$};
    \node (P1) at (45+90:{2.7/sqrt(2)}) {$\bfF(\one) \otimes' \bfF(x)$};
    \node (P2) at (45+2*90:{2.7/sqrt(2)}) {$\one' \otimes' \bfF(x)$};
    \node (P3) at (45+3*90:{2.7/sqrt(2)}) {$\bfF(x)$};
    \node[above] at (0,0) {$\Downarrow$};
    \node[below] at (0,0) {$\bfGamma_x$};
    \draw
    (P1) edge[->] node[above,yshift=5pt] {$\bfchi_{\one,x}$} (P0)
    (P2) edge[->] node[left,xshift=-5pt] {$\bfiota \otimes' \id_{\bfF(x)}$} node[left,xshift=-5pt] {$\phantom{\bfF(\bflambda_x)}$} (P1)
    (P2) edge[->] node[below,yshift=-5pt] {$\bflambda'_{\bfF(x)}$} (P3)
    (P0) edge[->] node[right,xshift=5pt] {$\bfF(\bflambda_x)$}node[right,xshift=5pt] {$\phantom{\bfiota \otimes' \id_{\bfF(x)}}$} (P3);
   \end{tikzpicture}
  \end{center}
  \begin{center}
   \begin{tikzpicture}[descr/.style={fill=white}]
    \node (P0) at (45:{2.7/sqrt(2)}) {$\bfF(x \otimes \one)$};
    \node (P1) at (45+90:{2.7/sqrt(2)}) {$\bfF(x) \otimes' \bfF(\one)$};
    \node (P2) at (45+2*90:{2.7/sqrt(2)}) {$\bfF(x) \otimes' \one'$};
    \node (P3) at (45+3*90:{2.7/sqrt(2)}) {$\bfF(x)$};
    \node[above] at (0,0) {$\Downarrow$};
    \node[below] at (0,0) {$\bfDelta_{x}$};
    \draw
    (P1) edge[->] node[above,yshift=5pt] {$\bfchi_{x,\one}$} (P0)
    (P2) edge[->] node[left,xshift=-5pt] {$\id_{\bfF(x)} \otimes' \bfiota$} node[left,xshift=-5pt] {$\phantom{\bfF(\bfrho_x)}$} (P1)
    (P2) edge[->] node[below,yshift=-5pt] {$\bfrho'_{\bfF(x)}$} (P3)
    (P0) edge[->] node[right,xshift=5pt] {$\bfF(\bfrho_x)$}node[right,xshift=5pt] {$\phantom{\id_{\bfF(x)} \otimes' \bfiota}$} (P3);
   \end{tikzpicture}
  \end{center}
  \begin{center}
   \begin{tikzpicture}[descr/.style={fill=white}]
    \node (P0) at (0:2.7) {$\bfF(x \otimes (y \otimes z))$};
    \node[right,xshift=-25pt] (P1) at (60:2.7) {$\bfF(x) \otimes' \bfF(y \otimes z)$};
    \node[left,xshift=25pt] (P2) at (2*60:2.7) {$\bfF(x) \otimes' (\bfF(y) \otimes' \bfF(z))$};
    \node (P3) at (3*60:2.7) {$(\bfF(x) \otimes' \bfF(y)) \otimes' \bfF(z)$};
    \node[left,xshift=25pt] (P4) at (4*60:2.7) {$\bfF(x \otimes y) \otimes' \bfF(z)$};
    \node[right,xshift=-25pt] (P5) at (5*60:2.7) {$\bfF((x \otimes y) \otimes z)$};
    \node (PH0) at (0:2.7) {$\phantom{(\bfF(x) \otimes' \bfF(y)) \otimes' \bfF(z)}$};
    \node (PH1) at (60:2.7) {$\hspace{50pt} \vphantom{\bfF(x) \otimes' \bfF(y \otimes z)}$};
    \node (PH2) at (2*60:2.7) {$\hspace{50pt} \vphantom{\bfF(x) \otimes' (\bfF(y) \otimes' \bfF(z))}$};
    \node (PH4) at (4*60:2.7) {$\hspace{50pt} \vphantom{\bfF(x \otimes y) \otimes' \bfF(z)}$};
    \node (PH5) at (5*60:2.7) {$\hspace{50pt} \vphantom{\bfF((x \otimes y) \otimes z)}$};
    \node[above] at (0,0) {$\Downarrow$};
    \node[below] at (0,0) {$\bfOmega_{x,y,z}$};
    \draw
    (P3) edge[->] node[left,xshift=-5pt] {$\bfalpha'_{\bfF(x),\bfF(y),\bfF(z)}$} (PH2)
    (PH2) edge[->] node[above,yshift=5pt] {$\id_{\bfF(x)} \otimes' \bfchi_{y,z}$} (PH1)
    (PH1) edge[->] node[right,xshift=5pt] {$\bfchi_{x,y \otimes z}$} node[right,xshift=5pt] {$\phantom{\bfalpha'_{\bfF(x),\bfF(y),\bfF(z)}}$} (P0)
    (P3) edge[->] node[left,xshift=-5pt] {$\bfchi_{x,y} \otimes' \id_{\bfF(z)}$} node[left,xshift=-5pt] {$\phantom{\bfF(\bfalpha_{x,y,z})}$} (PH4)
    (PH4) edge[->] node[below,yshift=-5pt] {$\bfchi_{x \otimes y,z}$} (PH5)
    (PH5) edge[->] node[right,xshift=5pt] {$\bfF(\bfalpha_{x,y,z})$} (P0);
   \end{tikzpicture}
  \end{center}
  \begin{center}
   \begin{tikzpicture}[descr/.style={fill=white}]
    \node (P0) at (0:{2.7/sqrt(2)}) {$\bfF(y \otimes x)$};
    \node (P1) at (90:{2.7/sqrt(2)}) {$\bfF(x \otimes y)$};
    \node (P2) at (2*90:{2.7/sqrt(2)}) {$\bfF(x) \otimes' \bfF(y)$};
    \node (P3) at (3*90:{2.7/sqrt(2)}) {$\bfF(y) \otimes' \bfF(x)$};
    \node (PH0) at (0:{2.7/sqrt(2)}) {$\phantom{\bfF(x) \otimes' \bfF(y)}$};
    \node[above] at (0,0) {$\Downarrow$};
    \node[below] at (0,0) {$\bfTheta_{x,y}$};
    \draw
    (P1) edge[->] node[right,xshift=5pt,yshift=5pt] {$\bfF(\twobr_{x,y})$} (P0)
    (P2) edge[->] node[left,xshift=-5pt,yshift=5pt] {$\bfchi_{x,y}$} node[left,xshift=-5pt,yshift=5pt] {$\phantom{\bfF(\twobr_{x,y})}$} (P1)
    (P2) edge[->] node[left,xshift=-5pt,yshift=-5pt] {$\twobr'_{\bfF(x),\bfF(y)}$} (P3)
    (P3) edge[->] node[right,xshift=5pt,yshift=-5pt] {$\bfchi_{y,x}$} (P0);
   \end{tikzpicture}
  \end{center} 
  for all $x,y,z \in \bcalC$.
 \end{enumerate}
 These data satisfy five conditions in the form of commuting polytopes of 2-mod\-i\-fi\-ca\-tions, which can be obtained from Equations (HTA1) and (HTA2) of \cite{GPS95}, and (BHA1), (BHA2), and (SHA1) of \cite{M00}, as explained in Definition 2.5 of \cite{S11}.
\end{definition}

Next, we define symmetric monoidal $2$-trans\-for\-ma\-tions between symmetric monoidal $2$-func\-tors.

\begin{definition}\label{D:symmetric_monoidal_2-transformation}
 A \textit{symmetric monoidal $2$-trans\-for\-ma\-tion} is a $2$-trans\-for\-ma\-tion $\bfsigma : \bfF \Rightarrow \bfF'$ between symmetric monoidal $2$-func\-tors $\bfF,\bfF' : \bcalC \to \bcalC'$ together with:
 \begin{enumerate}
  \item An invertible 2-mor\-phism $\varphi : \bfsigma_{\one} \circ \bfiota \Rightarrow \bfiota'$;
  \item An invertible 2-mod\-i\-fi\-ca\-tion $\bfXi$ whose components have the form
   \begin{center}
    \begin{tikzpicture}[descr/.style={fill=white}]
     \node (P0) at (-90:{2.7/sqrt((1-cos(72))^2+sin(72)^2)}) {$\bfF(x \otimes y)$};
     \node (P1) at (-90+72:{2.7/sqrt((1-cos(72))^2+sin(72)^2)}) {$\bfF'(x \otimes y)$};
     \node[right,xshift=-25pt] (P2) at (-90+2*72:{2.7/sqrt((1-cos(72))^2+sin(72)^2)}) {$\bfF'(x) \otimes' \bfF'(y)$};
     \node[left,xshift=25pt] (P3) at (-90+3*72:{2.7/sqrt((1-cos(72))^2+sin(72)^2)}) {$\bfF'(x) \otimes' \bfF(y)$};
     \node (P4) at (-90+4*72:{2.7/sqrt((1-cos(72))^2+sin(72)^2)}) {$\bfF(x) \otimes' \bfF(y)$};
     \node (PH2) at (-90+2*72:{2.7/sqrt((1-cos(72))^2+sin(72)^2)}) {$\hspace{50pt} \vphantom{\bfF'(x) \otimes' \bfF'(y)}$};
     \node (PH3) at (-90+3*72:{2.7/sqrt((1-cos(72))^2+sin(72)^2)}) {$\hspace{50pt} \vphantom{\bfF'(x) \otimes' \bfF(y)}$};
     \node[above] at (0,0) {$\Downarrow$};
     \node[below] at (0,0) {$\bfXi_{x,y}$};
     \draw
     (P4) edge[->] node[left,xshift=-5pt,yshift=-5pt] {$\bfchi_{x,y}$} (P0)
     (P0) edge[->] node[right,xshift=5pt,yshift=-5pt] {$\bfsigma_{x \otimes y}$} (P1)
     (PH2) edge[->] node[right,xshift=5pt] {$\bfchi'_{x,y}$} node[right,xshift=5pt] {$\phantom{\bfsigma_x \otimes' \id_{\bfF(y)}}$} (P1)
     (PH3) edge[->] node[above,yshift=5pt] {$\id_{\bfF'(x)} \otimes' \bfsigma_y$} (PH2)
     (P4) edge[->] node[left,xshift=-5pt] {$\bfsigma_x \otimes' \id_{\bfF(y)}$} (PH3);
   \end{tikzpicture}
  \end{center}
  for all $x,y \in \bcalC$.
 \end{enumerate}
 These data satisfy four conditions in the form of commuting polytopes of 2-mod\-i\-fi\-ca\-tions, which can be obtained from Equations (MBTA1), (MBTA2), and (MBTA3) of \cite{S11}, and (BTA1) of \cite{M00}, as explained in Definition 2.7 of \cite{S11}.
\end{definition}

We also have symmetric monoidal $2$-mod\-i\-fi\-ca\-tions between symmetric monoidal $2$-trans\-for\-ma\-tions.

\begin{definition}\label{D:symmetric_monoidal_2-modification}
 A \textit{symmetric monoidal $2$-mod\-i\-fi\-ca\-tion} is a $2$-mod\-i\-fi\-ca\-tion $\bfGamma : \bfsigma \Rrightarrow \bfsigma'$ between symmetric monoidal $2$-trans\-for\-ma\-tions $\bfsigma,\bfsigma' : \bfF \Rightarrow \bfF'$ for symmetric monoidal $2$-func\-tors $\bfF,\bfF' : \bcalC \to \bcalC'$ satisfying two condititions  in the form of commuting polytopes of 2-mod\-i\-fi\-ca\-tions, which can be found in Equations (BMBM1) and (BMBM2) of \cite{S11}, as explained in Definition 2.8 of \cite{S11}.
\end{definition}


Just like before, we can easily define identities for symmetric monoidal $2$-func\-tors, for symmetric monoidal $2$-trans\-for\-ma\-tions, and for symmetric monoidal $2$-mod\-i\-fi\-ca\-tions. Next, we can define composition of symmetric monoidal $2$-func\-tors as explained in Definition 2.12 of \cite{S11}, and we can define whiskering of symmetric monoidal $2$-func\-tors with symmetric monoidal $2$-trans\-for\-ma\-tions and symmetric monoidal $2$-mod\-i\-fi\-ca\-tions as explained in Definition 2.13 of \cite{S11}. Furthermore, we can define horizontal composition of symmetric monoidal $2$-trans\-for\-ma\-tions as explained in Definition 2.9 of \cite{S11}, and we can define horizontal and vertical composition of symmetric monoidal $2$-mod\-i\-fi\-ca\-tions as explained in Definitions 2.10 and 2.11 of \cite{S11}. In particular, for every pair of symmetric monoidal $2$-cat\-e\-go\-ries $\bcalC$ and $\bcalC'$, symmetric monoidal $2$-func\-tors from $\bcalC$ to $\bcalC'$, together with symmetric monoidal $2$-trans\-for\-ma\-tions and symmetric monoidal $2$-mod\-i\-fi\-ca\-tions between them, form a $2$-cat\-e\-go\-ry, which we denote $\smtwoCat(\bcalC,\bcalC')$. This construction is needed in order to state an important coherence theorem for symmetric monoidal $2$-cat\-e\-go\-ries due to Schommer-Pries. Indeed, the definitions given in this section can look quite scary. For instance, we would happily forget about the coherence data for symmetric monoidal $2$-cat\-e\-go\-ries, if we could, just like we did for standard $2$-cat\-e\-go\-ries. Luckily enough, for the most part we can. In order to state this result, we need an external notion of equivalence for symmetric monoidal $2$-cat\-e\-go\-ries. We say two symmetric monoidal $2$-cat\-e\-go\-ries $\bcalC$ and $\bcalC'$ are \textit{equivalent} if there exist symmetric monoidal $2$-func\-tors $\bfF : \bcalC \rightarrow \bcalC'$ and $\bfF' : \bcalC' \rightarrow \bcalC$ such that $\bfF' \circ \bfF$ is equivalent to $\id_{\bcalC}$ in $\smtwoCat(\bcalC,\bcalC')$, and such that $\bfF \circ \bfF'$ is equivalent to $\id_{\bcalC'}$ in $\smtwoCat(\bcalC',\bcalC')$, in which case both $\bfF$ and $\bfF'$ are called \textit{equivalences}.
Then, we need to say what a nice symmetric monoidal $2$-cat\-e\-go\-ry is. The next definition corresponds to Definition 2.28 of \cite{S11}.

\begin{definition}\label{D:quasi-strict}
 A symmetric monoidal $2$-cat\-e\-go\-ry $\bcalC$ is \textit{quasi-strict} if:
 \begin{enumerate}
  \item $\bcalC$ is strict;
  \item $\bflambda$, $\bfrho$, $\bfalpha$, $\bflambda^*$, $\bfrho^*$, and $\bfalpha^*$ are identity $2$-trans\-for\-ma\-tions;
  \item $\bfLambda$, $\bfP$, $\bfM$, $\bfPi$, $\bfA$, $\bfB$, and $\bfSigma$ are identity $2$-mod\-i\-fi\-ca\-tions;
  \item $\twobr_{x,y} : x \otimes y \rightarrow y \otimes x$ is an identity $1$-mor\-phism of $\bcalC$ for all $x, y \in \bcalC$ satisfying either $x = \one$ or $y = \one$;
  \item $\otimes_{(h,i),(f,g)} : (h \otimes i) \circ (f \otimes g) \Rightarrow (h \circ f) \otimes (i \circ g)$ is an identity $2$-mor\-phism of $\bcalC$ for all $1$-mor\-phisms $f : x \rightarrow z$, $g : y \rightarrow a$, $h : z \rightarrow b$, and $i : a \rightarrow c$ of $\bcalC$ satisfying either $g = \twobr_{y_1,y_2}$ or $h = \twobr_{z_1,z_2}$;
  \item $\twobr_{f,g} : (g \otimes f) \circ \twobr_{x,y} \Rightarrow \twobr_{z,a} \circ (f \otimes g)$ is an identity $2$-mor\-phism of $\bcalC$ for all $1$-mor\-phisms $f : x \rightarrow z$ and $g : y \rightarrow a$ of $\bcalC$ satisfying either $f = \id_x$ or $g = \id_y$.
 \end{enumerate}
\end{definition}

Remark that if $\bcalC$ is quasi-strict, then its tensor product $\otimes$ is \textit{cubical}, meaning that $\otimes_{(x,y)} : \id_{x \otimes y} \Rightarrow \id_x \otimes \id_y$ is an identity $2$-mor\-phism of $\bcalC$ for all $x,y \in \bcalC$, and that $\otimes_{(h,i),(f,g)} : (h \otimes i) \circ (f \otimes g) \Rightarrow (h \circ f) \otimes (i \circ g)$ is an identity $2$-mor\-phism of $\bcalC$ for all $1$-mor\-phisms $f : x \rightarrow z$, $g : y \rightarrow a$, $h : z \rightarrow b$, and $i : a \rightarrow c$ of $\bcalC$ satisfying either $g = \id_y$ or $h = \id_z$. Then, let us state a result which is contained in Theorem 2.96 of \cite{S11}.

\begin{theorem}\label{T:coherence_for_sym_mon_2-cat}
 Every symmetric monoidal $2$-cat\-e\-go\-ry is equivalent to a quasi-strict symmetric monoidal $2$-cat\-e\-go\-ry.
\end{theorem}

This explains why, outside of this chapter, we always assume unitors, associators, and all structural $2$-modifications of symmetric monoidal $2$-cat\-e\-go\-ries to be identities. Again, what we are actually doing is, whenever we consider a symmetric monoidal $2$-cat\-e\-go\-ry, to secretly replace it with a quasi-strict version of itself without changing our notation. For completion, we mention there exists a coherence result for symmetric monoidal $2$-func\-tors too, also due to Schommer-Pries, which is given by Theorem 2.97 of \cite{S11}, although we do not use it here.


%% file: appendix_e.tex
%
%
%

\chapter{Complete linear and graded linear categories}\label{A:co_lin_&_gr_lin_cat}

In this appendix we introduce our target $2$-cat\-e\-go\-ries for ETQFTs and for $\PGr$-graded ETQFTs: the symmetric monoidal $2$-cat\-e\-go\-ry $\coCat_{\Bbbk}$ of complete linear categories, and the symmetric monoidal $2$-cat\-e\-go\-ry $\coCat_{\Bbbk}^{\PGr}$ of complete $\PGr$-graded linear categories. We also introduce non-complete analogues, denoted $\bfCat_{\Bbbk}$ and $\bfCat_{\Bbbk}^{\PGr}$ respectively, and we spend some time explaining the relations between all the different flavors of these constructions. The goal is to allow us to work in the simpler non-complete settings, while still specifying the complete versions of these symmetric monoidal $2$-cat\-e\-go\-ries as targets for ETQFTs and $\PGr$-graded ETQFTs.

\section{Linear and graded linear categories}\label{S:lin_&_gr_lin_cat}

We begin this section by defining a family of symmetric monoidal $2$-cat\-e\-go\-ries which will provide the basis for all the following constructions. To this end, let us fix a monoidal category $\calK$, and let us remark that $\calK$-en\-riched categories, together with $\calK$-en\-riched functors and $\calK$-en\-riched natural transformations between them, can be arranged into a $2$-cat\-e\-go\-ry $\bfCat_{\calK}$ whose horizontal and vertical composition are given by horizontal composition of $\calK$-en\-riched functors, and by horizontal and vertical composition of $\calK$-en\-riched natural transformations, as explained in Section 1.2 of \cite{K82}. Now Section 1.3 of \cite{K82} provides the notion of unit $\calK$-en\-riched category and, if $\calK$ is also equipped with a symmetric braiding $c$, Section 1.4 of \cite{K82} provides the notion of $\calK$-en\-riched tensor product of $\calK$-en\-riched categories. Furthermore, it is easy to see that the braiding $c$ on $\calK$ induces a $c$-braiding of $\calK$-en\-riched categories.

\begin{definition}\label{D:sym_mon_2-cat_of_rich_cat}
 The \textit{symmetric monoidal $2$-cat\-e\-go\-ry of $\calK$-en\-riched categories} is the symmetric monoidal $2$-cat\-e\-go\-ry obtained from the $2$-cat\-e\-go\-ry $\bfCat_{\calK}$ by specifying the unit $\calK$-en\-riched category as tensor unit, the $\calK$-en\-riched tensor product as tensor product, and the $c$-braiding of $\calK$-en\-riched categories as braiding.
\end{definition}

For what concerns enriched category theory, we will need nothing more than the very basic notions we just recalled in the previous definition. Indeed, we will specialize the symmetric monoidal category $\calK$ only in two of the simplest possible ways. Here is the first.

\begin{definition}\label{D:sym_mon_2-cat_of_lin_cat}
 The \textit{symmetric monoidal $2$-cat\-e\-go\-ry of linear categories} is the symmetric monoidal $2$-cat\-e\-go\-ry $\bfCat_{\Bbbk}$ obtained from Definition \ref{D:sym_mon_2-cat_of_rich_cat} by taking $\calK$ to be the monoidal category $\Vect_{\Bbbk}$ equipped with the trivial braiding $\tau$.
\end{definition}

A $\Vect_{\Bbbk}$-enriched category $A$ is referred to as a \textit{linear category}, a $\Vect_{\Bbbk}$-enriched functor $F : A \rightarrow A'$ between linear categories $A$ and $A'$ is referred to as a \textit{linear functor}, while a $\Vect_{\Bbbk}$-enriched natural transformation $\eta : F \Rightarrow F'$ between linear functors $F,F' : A \rightarrow A'$ is nothing more than an ordinary natural transformation. The tensor unit of $\bfCat_{\Bbbk}$ is denoted $\Bbbk$, the tensor product of $\bfCat_{\Bbbk}$ is denoted $\sqtimes$, and the braiding of $\bfCat_{\Bbbk}$ is denoted $\trbr$.

\begin{remark}\label{R:comp_tens_prod_lin_cat}
 Let us spell out some of the consequences of having specified the trivial braiding $\tau$ on $\Vect_{\Bbbk}$. First of all, this determines standard compositions in tensor products of linear categories. Indeed, for all linear categories $A$ and $A'$, and for all 
 \[
  f \otimes f' \in \Hom_A(x,y) \otimes \Hom_{A'}(x',y'), \quad g \otimes g' \in \Hom_A(y,z) \otimes \Hom_{A'}(y',z'),
 \]
 we have
 \[
  (g \otimes g') \circ (f \otimes f') = (g \circ f) \otimes (g' \circ f').
 \]
 Furthermore, this makes $\sqtimes$ into a strict $2$-func\-tor on the nose. Indeed, for all linear functors $F : A \rightarrow A''$, $F' : A' \rightarrow A'''$, $F'' : A'' \rightarrow A^{\fourth}$, and $F''' : A''' \rightarrow A^{\fifth}$, and for every morphism
 \[
  f \otimes f' \in \Hom_A(x,y) \otimes \Hom_{A'}(x',y')
 \]
 we have
 \[
  (F'' \sqtimes F''') \left( (F \sqtimes F')(f \otimes f') \right) = 
  (F'' \circ F)(f) \otimes (F''' \circ F')(f').
 \]
\end{remark}

Now, let us move on the the second specialization of $\calK$ we will consider, which is just a very mild generalization of the first one. Let us begin by fixing an abelian group $\PGr$, and let us denote with $\PGr$ also the discrete category over $\PGr$. Furthermore, if $\Z^*$ denotes the multiplicative group $\{ +1,-1 \}$ of invertible integers, then let us fix a bilinear map $\gamma : \PGr \times \PGr \rightarrow \Z^*$, which will be used in order to define potentially non-trivial symmetric braidings. Let us start with some preliminary definitions: a \textit{$\PGr$-graded vector space $\bbV$} is a functor from $\PGr$ to $\Vect_{\Bbbk}$, and for every $k \in \PGr$ the \textit{space $\bbV^k$ of degree $k$ vectors} of $\bbV$ is the vector space $\bbV(k)$. A \textit{$\PGr$-graded linear map $\bbf : \bbV \rightarrow \bbV'$} between $\PGr$-graded vector spaces $\bbV$ and $\bbV'$ is a natural transformation from $\bbV$ to $\bbV'$, and for every $k \in \PGr$ the \textit{degree $k$ component $\bbf^k : \bbV^k \rightarrow \bbV'^k$} of $\bbf : \bbV \rightarrow \bbV'$ is the linear map $\bbf_k : \bbV(k) \rightarrow \bbV'(k)$. The \textit{tensor unit $\Bbbk$} of $\PGr$-graded vector spaces is the $\PGr$-graded vector space defined by
\[
 \Bbbk^k := \begin{cases}
            \Bbbk & k = 0 \\
            \{ 0 \} & k \neq 0.
           \end{cases}
\]
for all $k \in \PGr$. The \textit{tensor product $\bbV \otimes \bbV'$} of $\PGr$-graded vector spaces $\bbV$ and $\bbV'$ is the $\PGr$-graded vector space defined by
\[
 (\bbV \otimes \bbV')^k := \bigoplus_{k' \in \PGr} \bbV^{k - k'} \otimes \bbV'^{k'}.
\]
for all $k \in \PGr$. The \textit{$\gamma$-braiding morphism $c^{\gamma}_{\bbV,\bbV'} : \bbV \otimes \bbV' \rightarrow \bbV' \otimes \bbV$} of $\PGr$-graded vector spaces $\bbV$ and $\bbV'$ is the $\PGr$-graded linear map given by 
\[
  (c^{\gamma}_{\bbV,\bbV'})^k := \bigoplus_{k' \in \PGr} \gamma(k-k',k') \cdot \tau_{\bbV^{k - k'},\bbV'^{k'}}
\]
for all $k \in \PGr$, where $\tau$ is the trivial braiding of vector spaces. Now, we have everything in place in order to define the monoidal category $\Vect_{\Bbbk}^{\PGr}$ of $\PGr$-graded vector spaces, which is the monoidal category whose objects are $\PGr$-graded vector spaces, whose morphisms are $\PGr$-graded linear maps, and whose monoidal structure is given by the tensor unit and by the tensor product we just introduced. Remark that $\Vect_{\Bbbk}^{\PGr}$ can be made into a symmetric monoidal category by equipping it with the $\gamma$-braiding $c^{\gamma}$.

\begin{definition}\label{D:sym_mon_2-cat_of_gr_lin_cat}
 The \textit{symmetric monoidal $2$-cat\-e\-go\-ry of $\PGr$-graded linear categories} is the symmetric monoidal $2$-cat\-e\-go\-ry $\bfCat_{\Bbbk}^{\PGr}$ obtained from Definition \ref{D:sym_mon_2-cat_of_rich_cat} by taking $\calK$ to be the monoidal category $\Vect_{\Bbbk}^{\PGr}$ equipped with the $\gamma$-braiding $c^{\gamma}$.
\end{definition}

A $\Vect_{\Bbbk}^{\PGr}$-enriched category $\bbA$ is referred to as a \textit{$\PGr$-graded linear category}, a $\Vect_{\Bbbk}^{\PGr}$-enriched functor $\bbF : \bbA \rightarrow \bbA'$ between $\PGr$-graded linear categories $\bbA$ and $\bbA'$ is referred to as a \textit{$\PGr$-graded linear functor}, and a $\Vect_{\Bbbk}^{\PGr}$-enriched natural transformation $\bbeta : \bbF \Rightarrow \bbF'$ between $\PGr$-graded linear functors $\bbF,\bbF' : \bbA \rightarrow \bbA'$ is referred to as a \textit{$\PGr$-graded natural transformation}. If $\bbA$ is a $\PGr$-graded linear category, then for every $k \in \PGr$ and for all $x,y \in \bbA$ the \textit{space of degree $k$ morphisms $f^k$ from $x$ to $y$} is the vector space $\Hom_{\bbA}^k(x,y)$ of degree $k$ vectors of $\Hom_{\bbA}(x,y)$, and the notation $f^k \in \Hom_{\bbA}(x,y)$ stands for $f^k \in \Hom_{\bbA}^k(x,y)$. If ${\bbF : \bbA \rightarrow \bbA'}$ is a $\PGr$-graded linear functor, then for every $k \in \PGr$ the image under $\bbF$ of a degree $k$ morphism $f^k \in \Hom_{\bbA}(x,y)$ is denoted $\bbF(f^k) \in \Hom_{\bbA'}(\bbF(x),\bbF(y))$. If ${\bbeta : \bbF \Rightarrow \bbF'}$ is a $\PGr$-graded natural transformation between $\PGr$-graded linear functors ${\bbF,\bbF' : \bbA \rightarrow \bbA'}$, then the component of $\bbeta$ associated with an object $x \in \bbA$ is denoted $\bbeta^0_x \in \Hom_{\bbA'}(\bbF(x),\bbF'(x))$. The tensor unit of $\bfCat_{\Bbbk}^{\PGr}$ is denoted $\Bbbk$, the tensor product of $\bfCat_{\Bbbk}^{\PGr}$ is denoted $\sqtimes$, and the braiding of $\bfCat_{\Bbbk}^{\PGr}$ is denoted $\gmmbr$.

\begin{remark}\label{R:comp_tens_prod_gr_lin_cat}
 Let us spell out some of the consequences of having specified the $\gamma$-braiding $c^{\gamma}$ on $\Vect_{\Bbbk}^{\PGr}$. First of all, this affects compositions in tensor products of $\PGr$-graded linear categories. Indeed, for all $\PGr$-graded linear categories $\bbA$ and $\bbA'$, and for all 
 \[
  f^k \otimes f'^{k'} \in \Hom_{\bbA}(x,y) \otimes \Hom_{\bbA'}(x',y'), \quad
  g^{\ell} \otimes g'^{\ell'} \in \Hom_{\bbA}(y,z) \otimes \Hom_{\bbA'}(y',z'),
 \]
 we have
 \[
  (g^{\ell} \otimes g'^{\ell'}) \circ (f^k \otimes f'^{k'}) = \gamma(k,\ell') \cdot (g^{\ell} \circ f^k) \otimes (g'^{\ell'} \circ f'^{k'}).
 \]
 Furthermore, this prevents $\sqtimes$ from being a strict $2$-func\-tor in general. Indeed, for all $\PGr$-graded linear functors $\bbF : \bbA \rightarrow \bbA''$, $\bbF' : \bbA' \rightarrow \bbA'''$, $\bbF'' : \bbA'' \rightarrow \bbA^{\fourth}$, and $\bbF''' : \bbA''' \rightarrow \bbA^{\fifth}$, and for every 
 \[
  f^k \otimes f'^{k'} \in \Hom_{\bbA}(x,y) \otimes \Hom_{\bbA'}(x',y')
 \]
 we have
 \[
  (\bbF'' \sqtimes \bbF''') \left( (\bbF \sqtimes \bbF')(f^k \otimes f'^{k'}) \right) = 
  \gamma(k,k') \cdot (\bbF'' \circ \bbF)(f^k) \otimes (\bbF''' \circ \bbF')(f'^{k'}).
 \]
\end{remark}

Linear categories and $\PGr$-graded linear categories are of course very closely related. First of all, every $\PGr$-graded linear category naturally determines a linear category.

\begin{definition}\label{D:underlying_linear_category}
 The \textit{underlying linear category} of a $\PGr$-graded linear category $\bbA$ is the linear category $A$ whose set of objects is given by $\bbA$, and whose space of morphisms from $x \in \bbA$ to $y \in \bbA$ is given by $\Hom_{\bbA}^0(x,y)$. 
\end{definition}

Conversely, every linear category naturally determines a $\PGr$-graded linear category, once provided with an additional piece of structure: an \textit{action of the abelian group $\PGr$ on a linear category $A$} is a monoidal functor $R : \PGr \rightarrow \End_{\Bbbk}(A)$, where $\End_{\Bbbk}(A)$ denotes the category of linear endofunctors of $A$, which is a monoidal category with tensor unit given by the identity functor of $A$, and with tensor product given by composition. If $R$ is an action of $\PGr$ on $A$, then we denote with $R^k$ the linear endofunctor $R(k) \in \End_{\Bbbk}(A)$ for every $k \in \PGr$.

\begin{definition}\label{D:graded_extension}
 The \textit{$\PGr$-graded extension} of a linear category $A$ with respect to an action $R$ of $\PGr$ on $A$ is the $\PGr$-graded linear category $R(A)$ whose objects are $x \in A$, whose degree $k$ morphisms from $x$ to $y$ are $f^k \in \Hom_A(x,R^{-k}(y))$, whose identities are coherence morphisms
 \[
  \id_x^0 := \varepsilon_x \in \Hom_A(x,R^{0}(x)),
 \]
 for every $x \in A$, and whose compositions are
 \[
  g^{\ell} \diamond f^k := (\mu_{-\ell,-k})_z \circ R^{-k}(g^{\ell}) \circ f^k \in \Hom_A(x,R^{- k - \ell}(z))
 \]
 for all $f^k \in \Hom_A(x,R^{-k}(y))$ and $g^{\ell} \in \Hom_A(y,R^{-\ell}(z))$.
\end{definition}

If $R$ is an action of $\PGr$ on $A$, and if $x$ and $y$ are objects of $A$, then the notation $\bbHom_A(x,y)$ stands for $\Hom_{R(A)}(x,y)$.

\begin{example}\label{E:gr_cat_of_gr_vector_spaces}
 An action $R$ of $\PGr$ on $\Vect_{\Bbbk}^{\PGr}$ is determined by the right translation endofunctors $R^k := R_{-k}^* \in \End_{\Bbbk}(\Vect_{\Bbbk}^{\PGr})$ mapping every $\PGr$-graded vector space $\bbV$ to $\bbV \circ R_{-k}$, where $R_{-k} : \PGr \rightarrow \PGr$ is the functor defined by $R_{-k} (\ell) := \ell - k$. We denote with $\bbVect_{\Bbbk}^{\PGr}$ the $\PGr$-graded extension $R(\Vect_{\Bbbk}^{\PGr})$.
\end{example}

\section{Complete linear categories}\label{S:co_lin_cat}

In this section we introduce our target symmetric monoidal $2$-category for ETQFTs. In order to do this, let us give some preliminary definitions. We say a linear category is \textit{complete}\footnote{What we call a complete linear category is usually called Cauchy complete, see for instance \cite{BDSV15}. We adopt a loose terminology because this is the only kind of completion we consider.} if it is closed under finite direct sums, and if all its idempotents are split. When a linear category is not complete, there is a standard procedure, called \textit{completion}, which produces a complete linear category out of it. The idea is very natural, and simply consists in adding by hand direct sums of objects, and subobjects corresponding to images of idempotent morphisms. We do this in multiple steps.

First, we define the \textit{additive completion $\Mat(A)$ of a linear category $A$} as the linear category whose objects are
\[
 \bigoplus_{i \in I} x_i := \left( \begin{array}{c}
                                    \vdots \\ x_i \\ \vdots
                                   \end{array} \right)
\]
with $x_i \in A$ for every $i$ in some finite ordered set $I$, and whose morphisms from $\smash{\displaystyle \bigoplus_{i \in I} x_i}$ to $\smash{\displaystyle \bigoplus_{i \in J} y_i}$ are
\[
 \left( f_{ij} \right)_{(i,j) \in J \times I} := 
 \left( \begin{array}{ccc}
         \ddots & \vdots & \bdots \\
         \cdots & f_{ij} & \cdots \\
         \bdots & \vdots & \ddots
        \end{array} \right)
\]
with $f_{ij} \in \Hom_{A}(x_j,y_i)$ for every $(i,j) \in J \times I$. The identity of an object $\smash{\displaystyle \bigoplus_{i \in I} x_i \in \Mat(A)}$ is the morphism
\[
 \left( \delta_{ij} \cdot \id_{x_i} \right)_{(i,j) \in I^2} \in \End_{\Mat(A)} \left( \bigoplus_{i \in I} x_i \right),
\] 
and the composition of a morphism
\[
 \left( f_{ij} \right)_{(i,j) \in J \times I} 
 \in \Hom_{\Mat(A)} \left( \bigoplus_{i \in I} x_i, \bigoplus_{i \in J} y_i \right)
\]
with a morphism
\[
 \left( g_{ij} \right)_{(i,j) \in K \times J} 
 \in \Hom_{\Mat(A)} \left( \bigoplus_{i \in J} y_i, \bigoplus_{i \in K} z_i \right)
\]
is the morphism
\[
 \left( \sum_{h \in J} g_{ih} \circ f_{hj} \right)_{(i,j) \in K \times I} 
 \in \Hom_{\Mat(A)} \left( \bigoplus_{i \in I} x_i, \bigoplus_{i \in K} z_i \right).
\]
Remark that every object of $A$ can be naturally interpreted as an object of $\Mat(A)$ given by a direct sum with a single summand.

Next, we define the \textit{idempotent completion $\Kar(A)$ of a linear category $A$} as the linear category whose objects are $\im(p) := (x,p)$ with $x \in A$ and $p \in \End_{A}(x)$ satisfying $p \circ p = p$, and whose morphisms from $\im(p)$ to $\im(q)$ are $f \in \Hom_{A}(x,y)$ satisfying $f \circ p = f = q \circ f$. The identity of an object $\im(p) \in \Kar(A)$ is the idempotent morphism $p \in \End_{\Kar(A)}(\im(p))$, and composition is directly inherited from composition in $A$. Remark again that every object of $A$ can be naturally interpreted as an object of $\Kar(A)$ by confusing it with the image of its identity.

\begin{definition}\label{D:completion}
 The \textit{completion $\hat{A}$ of a linear category $A$} is the linear category $\Kar(\Mat(A))$.
\end{definition}

Every object of the completion $\hat{A}$ of a linear category $A$ has the form
\[
 \im \left( \left( p_{ij} \right)_{(i,j) \in I^2} \right) :=
 \left( \left( \begin{array}{c}
                \vdots \\ x_i \\ \vdots
               \end{array} \right),
 \left( \begin{array}{ccc}
         \ddots & \vdots & \bdots \\
         \cdots & p_{ij} & \cdots \\
         \bdots & \vdots & \ddots
        \end{array} \right) \right)
\]
for some finite ordered set $I$, for some family $\{ x_i \in A \mid i \in I \}$, and for some family $\{ p_{ij} \in \Hom_{A}(x_j,x_i) \mid (i,j) \in I^2 \}$
satisfying
\[
 \sum_{h \in I} p_{ih} \circ p_{hj} = p_{ij}
\]
for every $(i,j) \in I^2$. Analogously, every morphism of the completion $\hat{A}$ from $\im ((p_{ij})_{(i,j) \in I^2})$ to $\im ((q_{ij})_{(i,j) \in J^2})$ has the form
\[
 \left( f_{ij} \right)_{(i,j) \in J \times I} :=
 \left( \begin{array}{ccc}
         \ddots & \vdots & \bdots \\
         \cdots & f_{ij} & \cdots \\
         \bdots & \vdots & \ddots
        \end{array} \right)
\] 
for some family $\{ f_{ij} \in \Hom_{A}(x_j,y_i) \mid (i,j) \in J \times I \}$ satisfying
\[
 \sum_{h \in I} f_{ih} \circ p_{hj} = f_{ij} =
 \sum_{h \in J} q_{ih} \circ f_{hj}
\]
for every $(i,j) \in J \times I$.

Every linear functor $F : A \rightarrow A'$ extends to a linear functor $\hat{F} : \hat{A} \rightarrow \hat{A}'$, called the \textit{completion of $F$}, which sends every object 
\[
 \im \left( \left( p_{ij} \right)_{(i,j) \in I^2} \right)
\]
of $\hat{A}$ to the object 
\[
 \im \left( \left( F \left( p_{ij} \right) \right)_{(i,j) \in I^2} \right)
\]
of $\hat{A}'$, and wich sends every morphism 
\[
 \left( f_{ij} \right)_{(i,j) \in J \times I}
\]
of $\Hom_{\hat{A}} ( \im ( ( p_{ij} )_{(i,j) \in I^2} ),\im ( ( q_{ij} )_{(i,j) \in I^2} ) )$ to the morphism 
\[
 \left( F \left( f_{ij} \right) \right)_{(i,j) \in J \times I}
\]
of $\Hom_{\hat{A}'} ( \im ( ( F(p_{ij} ) )_{(i,j) \in I^2} ),\im ( ( F(q_{ij} ) )_{(i,j) \in J^2} ) )$. Similarly, every natural transformation $\eta : F \Rightarrow F'$ extends to a natural transformation $\hat{\eta} : \hat{F} \Rightarrow \hat{F}'$, called the \textit{completion of $\eta$}, which associates with every object 
\[
 \im \left( \left( p_{ij} \right)_{(i,j) \in I^2} \right)
\]
of $\hat{A}$ the morphism
\[
 \left( \sum_{h \in I} F' \left( p_{ih} \right) \circ \eta_{x_h} \circ F \left( p_{hj} \right) \right)_{(i,j) \in I^2}
\]
of $\Hom_{\hat{A}'} ( \im ( ( F ( p_{ij} ) )_{(i,j) \in I^2} ), \im ( ( F'(p_{ij} ) )_{(i,j) \in I^2} ) )$.

\begin{definition}\label{D:2-cat_of_co_lin_cat}
 The \textit{$2$-cat\-e\-go\-ry of complete linear categories} is the full sub-$2$-cat\-e\-go\-ry $\coCat_{\Bbbk}$ of $\bfCat_{\Bbbk}$ whose objects are complete linear categories.
\end{definition}

\begin{proposition}\label{P:co_strict_2-funct}
 Completion defines a strict $2$-func\-tor $\cmpl : \bfCat_{\Bbbk} \rightarrow \coCat_{\Bbbk}$.
\end{proposition}

\begin{proof}
 The proof is straight-forward. First of all, for every linear functor $F : A \rightarrow A'$, we have $\id_{\cmpl(F)} = \cmpl(\id_F)$. Next, for all natural transformations $\eta' : F' \Rightarrow F''$ and $\eta : F \Rightarrow F'$ between linear functors $F,F',F'' : A \rightarrow A'$, we have $\cmpl(\eta') \ast \cmpl(\eta) = \cmpl(\eta' \ast \eta)$,
 because for every object $\im ( ( p_{ij} )_{(i,j) \in I^2} )$ of $\hat{A}$, and for every $(i,j) \in I^2$, we have
 \begin{align*}
  &\sum_{\mathclap{h,h',h'' \in I}} F'' \left( p_{ih''} \right) \circ \eta'_{x_{h''}} \circ F' \left( p_{h''h'} \right) \circ F' \left( p_{h'h} \right) \circ \eta_{x_h} \circ F \left( p_{hj} \right) \\
  &\hspace{\parindent} = \sum_{\mathclap{h,h'' \in I}} F'' \left( p_{ih''} \right) \circ \eta'_{x_{h''}} \circ F' \left( p_{h''h} \right) \circ \eta_{x_h} \circ F \left( p_{hj} \right) \\
  &\hspace{\parindent} = \sum_{\mathclap{h,h'' \in I}} F'' \left( p_{ih''} \right) \circ F'' \left( p_{h''h} \right) \circ \eta'_{x_h} \circ \eta_{x_h} \circ F \left( p_{hj} \right) \\
  &\hspace{\parindent} = \sum_{h \in I} F'' \left( p_{ih} \right) \circ \eta'_{x_h} \circ \eta_{x_h} \circ F \left( p_{hj} \right).
 \end{align*}
 This proves $\cmpl_{A,A'}$ is a functor for all linear categories $A$ and $A'$. Now, for every linear category $A$, we have $\id_{\cmpl(A)} = \cmpl(\id_A)$, and for all linear functors $F' : A' \rightarrow A''$ and $F : A \rightarrow A'$, we have $\cmpl(F') \circ \cmpl(F) = \cmpl(F' \circ F)$.
 This means we can set $\cmpl_A$ and $\cmpl_{F',F}$ to be identity natural tranformations for every linear category $A$ and for all linear functors $F' : A' \rightarrow A''$ and $F : A \rightarrow A'$.
\end{proof}


\begin{remark}\label{R:2-completion}
 If $\bcalC$ is a $2$-category, then every $2$-func\-tor $\bfF : \bcalC \rightarrow \bfCat_{\Bbbk}$ induces a $2$-func\-tor $\hat{\bfF} : \bcalC \rightarrow \coCat_{\Bbbk}$, called the \textit{completion} of $\bfF$, which is defined as the composition $\cmpl \circ \bfF$. Similarly, every $2$-trans\-for\-ma\-tion $\bfsigma : \bfF \Rightarrow \bfF'$ between $2$-func\-tors $\bfF,\bfF' : \bcalC \rightarrow \bfCat_{\Bbbk}$ induces a $2$-trans\-for\-ma\-tion $\hat{\bfsigma} : \hat{\bfF} \Rightarrow \hat{\bfF}'$, called the \textit{completion} of $\bfsigma$, which is defined as the left whiskering $\cmpl \triangleright \bfsigma$. Analogously, every $2$-mod\-i\-fi\-ca\-tion $\bfGamma : \bfsigma \Rrightarrow \bfsigma'$ between $2$-trans\-for\-ma\-tions $\bfsigma,\bfsigma' : \bfF \Rightarrow \bfF'$ between $2$-func\-tors $\bfF,\bfF' : \bcalC \rightarrow \bfCat_{\Bbbk}$ induces a $2$-mod\-i\-fi\-ca\-tion $\hat{\bfGamma} : \hat{\bfsigma} \Rrightarrow \hat{\bfsigma}'$, called the \textit{completion} of $\bfGamma$, which is defined as the left whiskering $\cmpl \triangleright \bfGamma$.
\end{remark}

We are now ready to define the target for our candidate ETQFT.

\begin{definition}\label{D:sym_mon_2-cat_of_co_lin_cat}
 The \textit{symmetric monoidal $2$-cat\-e\-go\-ry of complete linear categories} is the symmetric monoidal $2$-cat\-e\-go\-ry obtained from Definition \ref{D:2-cat_of_co_lin_cat} by specifying as tensor unit the completion $\hat{\Bbbk}$ of the tensor unit $\Bbbk$ of $\bfCat_{\Bbbk}$, by specifying as tensor product the completion $\csqtimes$ of the tensor product $\sqtimes$ of $\bfCat_{\Bbbk}$, and by specifying as braiding the completion $\hat{\trbr}$ of the braiding $\trbr$ of $\bfCat_{\Bbbk}$.
\end{definition}

Remark that, although $\coCat_{\Bbbk}$ is a full sub-$2$-cat\-e\-go\-ry of $\bfCat_{\Bbbk}$, it is not a symmetric monoidal sub-$2$-cat\-e\-go\-ry, because the two symmetric monoidal structures do not agree.

\begin{proposition}\label{P:completion}
 The strict $2$-functor $\cmpl : \bfCat_{\Bbbk} \rightarrow \coCat_{\Bbbk}$ is symmetric mo\-n\-oi\-dal.
\end{proposition}

The proof of this result is elementary but rather long, so we postpone it to Section \ref{S:proofs}. 
Moving on, let us make an important remark: although in this memoir we use $\coCat_{\Bbbk}$ as target for ETQFTs, we actually spend most of our time working inside $\bfCat_{\Bbbk}$, because this makes arguments easier. However, this means the notion of equivalence inside $\bfCat_{\Bbbk}$ is too strong for our purposes. Indeed, we have to consider as equivalent all linear categories whose completions are equivalent.

\begin{definition}\label{D:Morita_equivalence_for_lin_cat}
 Two linear categories $A$ and $A'$ are \textit{Morita equivalent} if their completions $\hat{A}$ and $\hat{A}'$ are equivalent, and a linear functor $F : A \rightarrow A'$ is a \textit{Morita equivalence} if its completion $\hat{F} : \hat{A} \rightarrow \hat{A}'$ is an equivalence.
\end{definition}

Now we give a very useful Morita equivalence criterion. We recall that we say a set ${D = \{ x_i \in A \mid i \in \rmI \}}$ of objects of a linear category $A$ is a \textit{dominating set} if for every $x \in A$ there exist $x_{i_1}, \ldots, x_{i_m} \in D$ and $r_j \in \Hom_{A}(x_{i_j},x)$ and ${s_j \in \Hom_{A}(x,x_{i_j})}$ for every integer $1 \leqslant j \leqslant m$ satisfying
\[
 \id_x = \sum_{j=1}^m r_j \circ s_j,
\]
in which case we also say \textit{$D$ dominates $A$}. We also recall that a linear functor $F : A \rightarrow A'$ is \textit{fully faithful} if, for all objects $x,y \in A$, the induced map from $H_A(x,y)$ to $H_{A'}(F(x),F(y))$ is a linear isomorphism, and that it is \textit{essentially surjective} if every object $x' \in A'$ is isomorphic to $F(x)$ for some object $x \in A$. Then, as explained in Section 1.11 of \cite{K82}, every linear functor which is fully faithful and essentially surjective defines an equivalence in $\bfCat_{\Bbbk}$.

\begin{proposition}\label{P:Morita_equivalence}
 Let $F : A \rightarrow A'$ be a fully faithful linear functor. If the set of objects $F(A)$ dominates $A'$, then $F$ is a Morita equivalence.
\end{proposition}

\begin{proof}
 Let us consider an object $x' \in A'$ and a decomposition
 \[
  \id_{x'} = \sum_{j=1}^m r'_j \circ s'_j
 \]
 with $r'_j \in \Hom_{A'}(F(x_j),x')$, with $s'_j \in \Hom_{A'}(x',F(x_j))$, and with $x_j \in A$ for every $j \in I := \{ 1, \ldots, m \}$. Then $\im ( ( s'_i \circ r'_j )_{(i,j) \in I^2} )$ is an object of $\hat{A}'$, because 
 \[
  \sum_{h=1}^m s'_i \circ r'_h \circ s'_h \circ r'_j = s'_i \circ r'_j.
 \]
 Furthermore, $x'$ is isomorphic to $\im ( ( s'_i \circ r'_j )_{(i,j) \in I^2} )$ as objects of $\hat{A}'$ through
 \begin{gather*}
  \left( s'_i \right)_{(i,j) \in I \times \{ 1 \}} \in \Hom_{\hat{A}'} \left( x, \im \left( \left( s'_i \circ r'_j \right)_{(i,j) \in I^2} \right) \right),
  \\
  \left( r'_j \right)_{(i,j) \in \{ 1 \} \times I} \in \Hom_{\hat{A}'} \left( \im \left( \left( s'_i \circ r'_j \right)_{(i,j) \in I^2} \right), x \right).
 \end{gather*}
 Indeed we have equalities 
 \begin{gather*}
  \left( s'_i \circ r'_j \right)_{(i,j) \in I^2} \circ \left( s'_i \right)_{(i,j) \in I \times \{ 1 \}} = 
  \left( s'_i \right)_{(i,j) \in I \times \{ 1 \}}, \\
  \left( r'_j \right)_{(i,j) \in \{ 1 \} \times I} \circ \left( s'_i \circ r'_j \right)_{(i,j) \in I^2} = 
  \left( r'_j \right)_{(i,j) \in \{ 1 \} \times I}, \\
  \left( r'_j \right)_{(i,j) \in \{ 1 \} \times I} \circ \left( s'_i \right)_{(i,j) \in I \times \{ 1 \}} =
  \id_x, \\
  \left( s'_i \right)_{(i,j) \in I \times \{ 1 \}} \circ \left( r'_j \right)_{(i,j) \in \{ 1 \} \times I} = 
  \left( s'_i \circ r'_j \right)_{(i,j) \in I^2}.
 \end{gather*}
 This means $x'$ is isomorphic, as an object of $\hat{A}$, to the image under the functor $\hat{F}$ of the object
 \[
  \im \left( \left( F^{-1} \left( s'_i \circ r'_j \right) \right)_{(i,j) \in I^2} \right) \in \hat{A}. \qedhere
 \]
\end{proof}

The notion of Morita equivalence given in Definition \ref{D:Morita_equivalence_for_lin_cat} applies to objects of the $2$-category $\bfCat_{\Bbbk}$, but it can actually be generalized. For our construction, we need to extend it to $2$-functors with target $\bfCat_{\Bbbk}$. In order to do this, let us consider a $2$-category $\bcalC$.


\begin{definition}\label{D:Morita_equivalence_for_2-funct}
 Two $2$-func\-tors $\bfF, \bfF' : \bcalC \rightarrow \bfCat_{\Bbbk}$ are \textit{Morita equivalent} if their completions $\hat{\bfF}, \hat{\bfF}' : \bcalC \rightarrow \coCat_{\Bbbk}$ are equivalent, and a $2$-transformation $\bfsigma : \bfF \Rightarrow \bfF'$ between $2$-functors $\bfF, \bfF' : \bcalC \rightarrow \bfCat_{\Bbbk}$ is a \textit{Morita equivalence} if its completion $\hat{\bfsigma} : \hat{\bfF} \Rightarrow \hat{\bfF}'$ is an equivalence $2$-transformation.
\end{definition}

\section{Complete graded linear categories}\label{S:co_gr_lin_cat}

In this section we introduce our target symmetric monoidal $2$-category for $\PGr$-graded ETQFTs. We do this by providing graded analogues for all the constructions and results of Section \ref{S:co_lin_cat}, with the only difference that this time we allow non-trivial symmetric braidings. Let us start with some preliminary definitions. We say an object $x$ of a $\PGr$-graded linear category $\bbA$ is the \textit{$k$-suspension of an object $y$ of $\bbA$} for some $k \in \PGr$ if there exists a degree $k$ invertible morphism $f^k \in \Hom_{\bbA}(x,y)$, and we say it is the \textit{direct sum in $\bbA$ of a finite family of objects $\{ x_i \in \bbA \mid i \in I \}$} if it is the direct sum of $\{ x_i \in \bbA \mid i \in I \}$ in the underlying linear category $A$ of $\bbA$. We say a degree $0$ morphism of a $\PGr$-graded linear category $\bbA$ is an \textit{idempotent of $\bbA$} if it is an idempotent morphism of the underlying linear category $A$ of $\bbA$. Then, we say a $\PGr$-graded linear category $\bbA$ is \textit{complete} if it is closed under suspensions and finite direct sums, and if all its idempotents are split. When a $\PGr$-graded linear category is not complete, there is a standard procedure, called \textit{completion}, which produces a complete $\PGr$-graded linear category out of it. The idea is again very natural, and simply consists in adding by hand suspensions and direct sums of objects, as well as subobjects corresponding to images of idempotent morphisms. We do this in multiple steps.

First, we define the \textit{suspension completion $\Deg(\bbA)$ of a $\PGr$-graded linear category $\bbA$} as the $\PGr$-graded linear category whose objects are $S^{\ell}(x) := (x,\ell)$ with $x \in \bbA$ and $\ell \in \PGr$, and whose degree $k$ morphisms from $S^{\ell}(x)$ to $S^m(y)$ are ${f^{k + \ell - m} \in \Hom_{\bbA}(x,y)}$. Identities and compositions are directly inherited from $\bbA$. Remark that every object of $\bbA$ can be naturally interpreted as an object of $\Deg(\bbA)$ by confusing it with its $0$ suspension.

Next, we define the \textit{additive completion $\Mat(\bbA)$ of a $\PGr$-graded linear category $\bbA$} as the $\PGr$-graded linear category whose objects are
\[
 \bigoplus_{i \in I} x_i := \left( \begin{array}{c}
                                    \vdots \\ x_i \\ \vdots
                                   \end{array} \right)
\]
with $x_i \in \bbA$ for every $i$ in some finite ordered set $I$, and whose degree $k$ morphisms~from $\smash{\displaystyle \bigoplus_{i \in I} x_i}$ to $\smash{\displaystyle \bigoplus_{i \in J} y_i}$ are
\[
 \left( f^k_{ij} \right)_{(i,j) \in J \times I} := 
 \left( \begin{array}{ccc}
         \ddots & \vdots & \bdots \\
         \cdots & f^k_{ij} & \cdots \\
         \bdots & \vdots & \ddots
        \end{array} \right)
\]
with $f^k_{ij} \in \Hom_{\bbA}(x_j,y_i)$ for every $(i,j) \in J \times I$. The identity of an object $\smash{\displaystyle \bigoplus_{i \in I} x_i \in \Mat(\bbA)}$ is the degree $0$ morphism
\[
 \left( \delta_{ij} \cdot \id^0_{x_i} \right)_{(i,j) \in I^2} \in \End_{\Mat(\bbA)} \left( \bigoplus_{i \in I} x_i \right),
\] 
and the composition of a degree $k$ morphism
\[
 \left( f^k_{ij} \right)_{(i,j) \in J \times I} 
 \in \Hom_{\Mat(\bbA)} \left( \bigoplus_{i \in I} x_i, \bigoplus_{i \in J} y_i \right)
\]
with a degree $\ell$ morphism
\[
 \left( g^{\ell}_{ij} \right)_{(i,j) \in K \times J} 
 \in \Hom_{\Mat(\bbA)} \left( \bigoplus_{i \in J} y_i, \bigoplus_{i \in K} z_i \right)
\]
is the degree $k + \ell$ morphism
\[
 \left( \sum_{h \in J} g^{\ell}_{ih} \circ f^k_{hj} \right)_{(i,j) \in K \times I} 
 \in \Hom_{\Mat(\bbA)} \left( \bigoplus_{i \in I} x_i, \bigoplus_{i \in K} z_i \right).
\]
Remark that every object of $\bbA$ can be naturally interpreted as an object of $\Mat(\bbA)$ given by a direct sum with a single summand.

Lastly, we define the \textit{idempotent completion $\Kar(\bbA)$ of a $\PGr$-graded linear category $\bbA$} as the $\PGr$-graded linear category whose objects are $\im(p^0) := (x,p^0)$ with $x \in \bbA$ and $p^0 \in \End_{\bbA}(x)$ satisfying ${p^0 \circ p^0 = p^0}$, and whose degree $k$ morphisms from $\im(p^0)$ to $\im(q^0)$ are $f^k \in \Hom_{\bbA}(x,y)$ satisfying $f^k \circ p^0 = f^k = q^0 \circ f^k$. The identity of an object $\im(p^0) \in \Kar(\bbA)$ is the idempotent degree $0$ morphism $p^0 \in \End_{\bbA}(\im(p^0))$, and composition is directly inherited from composition in $\bbA$. Remark again that every object of $\bbA$ can be naturally interpreted as an object of $\Kar(\bbA)$ by confusing it with the image of its identity.

\begin{definition}\label{D:graded_completion}
 The \textit{completion $\hat{\bbA}$ of a $\PGr$-graded linear category $\bbA$} is the $\PGr$-graded linear category $\Kar(\Mat(\Deg(\bbA)))$.
\end{definition}

Every object of the completion $\hat{\bbA}$ of a $\PGr$-graded linear category $\bbA$ has the form
\[
 \im \left( \left( p^{\ell_j - \ell_i}_{ij} \right)_{(i,j) \in I^2} \right) :=
 \left( \left( \begin{array}{c}
                \vdots \\ (x_i,\ell_i) \\ \vdots
               \end{array} \right),
 \left( \begin{array}{ccc}
         \ddots & \vdots & \bdots \\
         \cdots & p^{\ell_j - \ell_i}_{ij} & \cdots \\
         \bdots & \vdots & \ddots
        \end{array} \right) \right)
\]
for some finite ordered set $I$, for some family $\{ \ell_i \in \PGr \mid i \in I \}$, for some family $\{ x_i \in \bbA \mid i \in I \}$, and for some family $\{ p^{\ell_j - \ell_i}_{ij} \in \Hom_{\bbA}(x_j,x_i) \mid (i,j) \in I^2 \}$ satisfying
\[
 \sum_{h \in I} p^{\ell_h - \ell_i}_{ih} \circ p^{\ell_j - \ell_h}_{hj} = p^{\ell_j - \ell_i}_{ij}
\]
for every $(i,j) \in I^2$. Analogously, every degree $k$ morphism of the completion $\hat{\bbA}$ from $\im ( ( p^{\ell_j - \ell_i}_{ij} )_{(i,j) \in I^2} )$ to $\im ( ( q^{m_j - m_i}_{ij} )_{(i,j) \in J^2} )$ has the form
\[
 \left( f^{k + \ell_j - m_i}_{ij} \right)_{(i,j) \in J \times I} :=
 \left( \begin{array}{ccc}
         \ddots & \vdots & \bdots \\
         \cdots & f^{k + \ell_j - m_i}_{ij} & \cdots \\
         \bdots & \vdots & \ddots
        \end{array} \right)
\] 
for some family $\{ f^{k + \ell_j - m_i}_{ij} \in \Hom_{\bbA}(x_j,y_i) \mid (i,j) \in J \times I \}$ satisfying
\[
 \sum_{h \in I} f^{k + \ell_h - m_i}_{ih} \circ p^{\ell_j - \ell_h}_{hj} = f^{k + \ell_j - m_i}_{ij} =
 \sum_{h \in J} q^{m_h - m_i}_{ih} \circ f^{k + \ell_j - m_h}_{hj}
\]
for every $(i,j) \in J \times I$.

Every $\PGr$-graded linear functor $\bbF : \bbA \rightarrow \bbA'$ extends to a $\PGr$-graded linear functor $\hat{\bbF} : \hat{\bbA} \rightarrow \hat{\bbA}'$, called the \textit{completion of $\bbF$}, which sends every object 
\[
 \im \left( \left( p^{\ell_j - \ell_i}_{ij} \right)_{(i,j) \in I^2} \right)
\]
of $\hat{\bbA}$ to the object 
\[
 \im \left( \left( \bbF \left( p^{\ell_j - \ell_i}_{ij} \right) \right)_{(i,j) \in I^2} \right)
\]
of $\hat{\bbA}'$, and wich sends every degree $k$ morphism 
\[
 \left( f^{k + \ell_j - m_i}_{ij} \right)_{(i,j) \in J \times I} 
\]
of $\Hom_{\hat{\bbA}} ( \im ( ( p^{\ell_j - \ell_i}_{ij} )_{(i,j) \in I^2} ),\im ( ( q^{m_j - m_i}_{ij} )_{(i,j) \in I^2} ) )$ to the degree $k$ morphism 
\[
 \left( \bbF \left( f^{k + \ell_j - m_i}_{ij} \right) \right)_{(i,j) \in J \times I}
\]
of $\Hom_{\hat{\bbA}'} ( \im ( ( \bbF ( p^{\ell_j - \ell_i}_{ij} ) )_{(i,j) \in I^2} ),\im ( ( \bbF ( q^{m_j - m_i}_{ij} ) )_{(i,j) \in J^2} ) )$. Similarly, every $\PGr$-grad\-ed natural transformation $\bbeta : \bbF \Rightarrow \bbF'$ extends to a $\PGr$-grad\-ed natural transformation $\hat{\bbeta} : \hat{\bbF} \Rightarrow \hat{\bbF}'$, called the \textit{completion of $\bbeta$}, which associates with every object 
\[
 \im \left( \left( p^{\ell_j - \ell_i}_{ij} \right)_{(i,j) \in I^2} \right)
\]
of $\hat{\bbA}$ the degree $0$ morphism
\[
 \left( \sum_{h \in I} \bbF' \left( p^{\ell_h - \ell_i}_{ih} \right) \circ \bbeta^0_{x_h} \circ \bbF \left( p^{\ell_j - \ell_h}_{hj} \right) \right)_{(i,j) \in I^2}
\]
of $\Hom_{\hat{\bbA}'} ( \im ( ( \bbF ( p^{\ell_j - \ell_i}_{ij} ) )_{(i,j) \in I^2} ), \im ( ( \bbF' ( p^{\ell_j - \ell_i}_{ij} ) )_{(i,j) \in I^2} ) )$.

\begin{definition}\label{D:2-cat_of_co_gr_lin_cat}
 The \textit{$2$-cat\-e\-go\-ry of complete $\PGr$-graded linear categories} is the full sub-$2$-cat\-e\-go\-ry $\coCat_{\Bbbk}^{\PGr}$ of $\bfCat_{\Bbbk}^{\PGr}$ whose objects are complete $\PGr$-graded linear categories.
\end{definition}

\begin{proposition}\label{P:gr_co_strict_2-funct}
 Completion defines a strict $2$-func\-tor $\grcmpl : \bfCat_{\Bbbk}^{\PGr} \rightarrow \coCat_{\Bbbk}^{\PGr}$.
\end{proposition}

\begin{proof}
 The proof is straight-forward. First of all, for every $\PGr$-graded linear functor $\bbF : \bbA \rightarrow \bbA'$, we have $\id_{\grcmpl(\bbF)} = \grcmpl(\id_{\bbF})$. Next, for all $\PGr$-graded natural transformations $\bbeta' : \bbF' \Rightarrow \bbF''$ and $\bbeta : \bbF \Rightarrow \bbF'$ between $\PGr$-graded linear functors $\bbF,\bbF',\bbF'' : \bbA \rightarrow \bbA'$, we have $\grcmpl(\bbeta') \ast \grcmpl(\bbeta) = \grcmpl(\bbeta' \ast \bbeta)$,
 because for every object $\im ( ( p^{\ell_j - \ell_i}_{ij} )_{(i,j) \in I^2} )$ of $\hat{\bbA}$, and for every $(i,j) \in I^2$, we have
 \begin{align*}
  &\sum_{\mathclap{h,h',h'' \in I}} \bbF'' \left( p^{\ell_{h''} - \ell_i}_{ih''} \right) \circ \bbeta'^0_{x_{h''}} \circ \bbF' \left( p^{\ell_{h'} - \ell_{h''}}_{h''h'} \right) \circ \bbF' \left( p^{\ell_h - \ell_{h'}}_{h'h} \right) \circ \bbeta^0_{x_h} \circ \bbF \left( p^{\ell_j - \ell_h}_{hj} \right) \\
  &\hspace{\parindent} = \sum_{\mathclap{h,h'' \in I}} \bbF'' \left( p^{\ell_{h''} - \ell_i}_{ih''} \right) \circ \bbeta'^0_{x_{h''}} \circ \bbF' \left( p^{\ell_h - \ell_{h''}}_{h''h} \right) \circ \bbeta^0_{x_h} \circ \bbF \left( p^{\ell_j - \ell_h}_{hj} \right) \\
  &\hspace{\parindent} = \sum_{\mathclap{h,h'' \in I}} \bbF'' \left( p^{\ell_{h''} - \ell_i}_{ih''} \right) \circ \bbF'' \left( p^{\ell_h - \ell_{h''}}_{h''h} \right) \circ \bbeta'^0_{x_h} \circ \bbeta^0_{x_h} \circ \bbF \left( p^{\ell_j - \ell_h}_{hj} \right) \\
  &\hspace{\parindent} = \sum_{h \in I} \bbF'' \left( p^{\ell_h - \ell_i}_{ih} \right) \circ \bbeta'^0_{x_h} \circ \bbeta^0_{x_h} \circ \bbF \left( p^{\ell_j - \ell_h}_{hj} \right).
 \end{align*}
 This proves $\grcmpl_{\bbA,\bbA'}$ is a functor for all $\PGr$-graded linear categories $\bbA$ and $\bbA'$. Now, for every $\PGr$-graded linear category $\bbA$, we have $\id_{\grcmpl(\bbA)} = \grcmpl(\id_{\bbA})$, and for all $\PGr$-graded linear functors $\bbF' : \bbA' \rightarrow \bbA''$ and $\bbF : \bbA \rightarrow \bbA'$, we have $\grcmpl(\bbF') \circ \grcmpl(\bbF) = \grcmpl(\bbF' \circ \bbF)$.
 This means we can set $\grcmpl_{\bbA}$ and $\grcmpl_{\bbF',\bbF}$ to be identity natural tranformations for every $\PGr$-graded linear category $\bbA$ and for all $\PGr$-graded linear functors $\bbF' : \bbA' \rightarrow \bbA''$ and $\bbF : \bbA \rightarrow \bbA'$.
\end{proof}

\begin{remark}\label{R:2-gr-completion}
 If $\bcalC$ is a $2$-category, then every $2$-func\-tor $\bfF : \bcalC \rightarrow \smash{\bfCat_{\Bbbk}^{\PGr}}$ induces a $2$-func\-tor $\hat{\bfF} : \bcalC \rightarrow \smash{\coCat_{\Bbbk}^{\PGr}}$, called the \textit{completion} of $\bfF$, which is defined as the composition $\grcmpl \circ \bfF$. Similarly, every $2$-trans\-for\-ma\-tion $\bfsigma : \bfF \Rightarrow \bfF'$ between $2$-func\-tors $\bfF,\bfF' : \bcalC \rightarrow \smash{\bfCat_{\Bbbk}^{\PGr}}$ induces a $2$-trans\-for\-ma\-tion $\hat{\bfsigma} : \hat{\bfF} \Rightarrow \hat{\bfF}'$, called the \textit{completion} of $\bfsigma$, which is defined as the left whiskering $\grcmpl \triangleright \bfsigma$. Analogously, every $2$-mod\-i\-fi\-ca\-tion $\bfGamma : \bfsigma \Rrightarrow \bfsigma'$ between $2$-trans\-for\-ma\-tions $\bfsigma,\bfsigma' : \bfF \Rightarrow \bfF'$ between $2$-func\-tors $\bfF,\bfF' : \bcalC \rightarrow \smash{\bfCat_{\Bbbk}^{\PGr}}$ induces a $2$-mod\-i\-fi\-ca\-tion $\hat{\bfGamma} : \hat{\bfsigma} \Rrightarrow \hat{\bfsigma}'$, called the \textit{completion} of $\bfGamma$, which is defined as the left whiskering $\grcmpl \triangleright \bfGamma$.
\end{remark}

We are now ready to define the target for our graded ETQFT.

\begin{definition}\label{D:sym_mon_2-cat_of_co_gr_lin_cat}
 The \textit{symmetric monoidal $2$-cat\-e\-go\-ry of complete $\PGr$-graded linear categories} is the symmetric monoidal $2$-cat\-e\-go\-ry obtained from Definition \ref{D:2-cat_of_co_gr_lin_cat} by specifying as tensor unit the completion $\hat{\Bbbk}$ of the tensor unit $\Bbbk$ of $\bfCat_{\Bbbk}^{\PGr}$, by specifying as tensor product the completion $\csqtimes$ of the tensor product $\sqtimes$ of $\bfCat_{\Bbbk}^{\PGr}$, and by specifying as braiding the completion $\hat{\bfc}^{\gamma}$ of the braiding $\gmmbr$ of $\bfCat_{\Bbbk}$.
\end{definition}

Remark that, although $\coCat_{\Bbbk}^{\PGr}$ is a full sub-$2$-cat\-e\-go\-ry of $\bfCat_{\Bbbk}^{\PGr}$, it is not a symmetric monoidal sub-$2$-cat\-e\-go\-ry, because the two symmetric monoidal structures do not agree.

\begin{proposition}\label{P:gr_completion}
 The strict $2$-functor $\grcmpl : \bfCat_{\Bbbk}^{\PGr} \rightarrow \coCat_{\Bbbk}^{\PGr}$ is symmetric monoidal.
\end{proposition}

The proof of this result is elementary but rather long, so we postpone it to Section \ref{S:proofs}. 
Moving on, let us make an important remark: although in this memoir we use $\coCat_{\Bbbk}^{\PGr}$ as target for $\PGr$-graded ETQFTs, we actually spend most of our time working inside $\bfCat_{\Bbbk}^{\PGr}$, because this makes arguments easier. However, this means the notion of equivalence inside $\bfCat_{\Bbbk}^{\PGr}$ is too strong for our purposes. Indeed, we have to consider as equivalent all $\PGr$-graded linear categories whose completions are equivalent.

\begin{definition}\label{D:gr_Morita_equivalence_for_gr_lin_cat}
 Two $\PGr$-graded linear categories $\bbA$ and $\bbA'$ are \textit{$\PGr$-Morita equivalent} if their completions $\hat{\bbA}$ and $\hat{\bbA}'$ are equivalent, and a $\PGr$-graded linear functor $\bbF : \bbA \rightarrow \bbA'$ is a \textit{$\PGr$-Morita equivalence} if its completion $\hat{\bbF} : \hat{\bbA} \rightarrow \hat{\bbA}'$ is an equivalence.
\end{definition}

Now we give a very useful $\PGr$-Morita equivalence criterion. We say a set ${D = \{ x_i \in A \mid i \in \rmI \}}$ of objects of a $\PGr$-graded linear category $\bbA$ is a \textit{dominating set} if for every $x \in \bbA$ there exist $x_{i_1}, \ldots, x_{i_m} \in D$ and $r^{k_j}_j \in \Hom_{\bbA}(x_{i_j},x)$ and $s^{-k_j}_j \in \Hom_{\bbA}(x,x_{i_j})$ for every integer $1 \leqslant j \leqslant m$ satisfying
\[
 \id^0_x = \sum_{j=1}^m r^{k_j}_j \circ s^{-k_j}_j,
\]
in which case we also say \textit{$D$ dominates $\bbA$}. We also recall that a $\PGr$-graded linear functor $\bbF : \bbA \rightarrow \bbA'$ is \textit{fully faithful} if, for all objects $x,y \in \bbA$, the induced map from $\Hom_{\bbA}(x,y)$ to $\Hom_{\bbA'}(\bbF(x),\bbF(y))$ is a $\PGr$-graded linear isomorphism, and that it is \textit{essentially surjective} if every object $x' \in \bbA'$ is isomorphic, in the underlying linear category $A'$, to $\bbF(x)$ for some object $x \in \bbA$. Then, as explained in Section 1.11 of \cite{K82}, every linear functor which is fully faithful and essentially surjective defines an equivalence in $\bfCat_{\Bbbk}^{\PGr}$.

\begin{proposition}\label{P:gr_Morita_equivalence}
 Let $\bbF : \bbA \rightarrow \bbA'$ be a fully faithful $\PGr$-graded linear functor. If the set of objects $\bbF(\bbA)$ dominates $\bbA'$, then $\bbF$ is a $\PGr$-Morita equivalence.
\end{proposition}

\begin{proof}
 Let us consider an object $x' \in \bbA'$ and a decomposition
 \[
  \id^0_{x'} = \sum_{j=1}^m r'^{k_j}_j \circ s'^{-k_j}_j
 \]
 with $r'^{k_j}_j \in \Hom_{\bbA'}(\bbF(x_j),x')$, with $s'^{-k_j}_j \in \Hom_{\bbA'}(x',\bbF(x_j))$, and with $x_j \in \bbA$ for every $j \in I:= \{ 1, \ldots, m \}$. Then $\im ( ( s'^{-k_i}_i \circ r'^{k_j}_j )_{(i,j) \in I^2} )$ is an object of $\hat{\bbA}'$, because 
 \[
  \sum_{h=1}^m s'^{-k_i}_i \circ r'^{k_h}_h \circ s'^{-k_h}_h \circ r'^{k_j}_j = s'^{-k_i}_i \circ r'^{k_j}_j.
 \]
 Furthermore, $x'$ is isomorphic to $\im ( ( s'^{-k_i}_i \circ r'^{k_j}_j )_{(i,j) \in I^2} )$ as objects of $\hat{\bbA}'$ through
 \begin{gather*}
  \left( s'^{-k_i}_i \right)_{(i,j) \in I \times \{ 1 \}} \in \Hom_{\hat{\bbA}'} \left( x, \im \left( \left( s'^{-k_i}_i \circ r'^{k_j}_j \right)_{(i,j) \in I^2} \right) \right),
  \\
  \left( r'^{k_j}_j \right)_{(i,j) \in \{ 1 \} \times I} \in \Hom_{\hat{\bbA}'} \left( \im \left( \left( s'^{-k_i}_i \circ r'^{k_j}_j \right)_{(i,j) \in I^2} \right), x \right).
 \end{gather*}
 Indeed we have equalities 
 \begin{gather*}
  \left( s'^{-k_i}_i \circ r'^{k_j}_j \right)_{(i,j) \in I^2} \circ \left( s'^{-k_i}_i \right)_{(i,j) \in I \times \{ 1 \}} = 
  \left( s'^{-k_i}_i \right)_{(i,j) \in I \times \{ 1 \}}, \\
  \left( r'^{k_j}_j \right)_{(i,j) \in \{ 1 \} \times I} \circ \left( s'^{-k_i}_i \circ r'^{k_j}_j \right)_{(i,j) \in I^2} = 
  \left( r'^{k_j}_j \right)_{(i,j) \in \{ 1 \} \times I}, \\
  \left( r'^{k_j}_j \right)_{(i,j) \in \{ 1 \} \times I} \circ \left( s'^{-k_i}_i \right)_{(i,j) \in I \times \{ 1 \}} =
  \id^0_x, \\
  \left( s'^{-k_i}_i \right)_{(i,j) \in I \times \{ 1 \}} \circ \left( r'^{k_j}_j \right)_{(i,j) \in \{ 1 \} \times I} = 
  \left( s'^{-k_i}_i \circ r'^{k_j}_j \right)_{(i,j) \in I^2}.
 \end{gather*}
 This means $x'$ is isomorphic, as an object of $\hat{\bbA}$, to the image under the functor $\hat{\bbF}$ of the object
 \[
  \im \left( \left( \bbF^{-1} \left( s'^{-k_i}_i \circ r'^{k_j}_j \right) \right)_{(i,j) \in I^2} \right) \in \hat{\bbA}. \qedhere
 \]
\end{proof}

The notion of $\PGr$-Morita equivalence given in Definition \ref{D:gr_Morita_equivalence_for_gr_lin_cat} applies to objects of the $2$-category $\bfCat_{\Bbbk}^{\PGr}$, but it can actually be generalized. For our construction, we need to extended it to $2$-functors with target $\bfCat_{\Bbbk}^{\PGr}$. In order to do this, let us consider a $2$-category $\bcalC$.

\begin{definition}\label{D:gr_Morita_equivalence_for_2-funct}
 Two $2$-func\-tors $\bfF, \bfF' : \bcalC \rightarrow \bfCat_{\Bbbk}^{\PGr}$ are \textit{$\PGr$-Morita equivalent} if their completions $\hat{\bfF}, \hat{\bfF}' : \bcalC \rightarrow \coCat_{\Bbbk}^{\PGr}$ are equivalent, and a $2$-transformation $\bfsigma : \bfF \Rightarrow \bfF'$ between $2$-functors $\bfF, \bfF' : \bcalC \rightarrow \bfCat_{\Bbbk}^{\PGr}$ is a \textit{$\PGr$-Morita equivalence} if its completion $\hat{\bfsigma} : \hat{\bfF} \Rightarrow \hat{\bfF}'$ is an equivalence $2$-transformation.
\end{definition}

\section{Proofs}\label{S:proofs}


In this section we prove the results we announced in the previous two. Both proofs essentially boil down to some standard multilinear algebra. In particular, at their core lies the following remark: for all $\bfm,\bfn \in \N$, $M := (m_1,\ldots,m_{\bfm}) \in \N^{\bfm}$, and $n := (n_1,\ldots,n_{\bfn}) \in \N^{\bfn}$, if we set $m := m_1 + \ldots + m_{\bfm}$ and $n := n_1 + \ldots + n_{\bfn}$, then we have a natural linear isomorphism
\[
 f_{M,N} : \Bbbk^{M \times N} \rightarrow \Bbbk^{m \times n},
\]
where $\Bbbk^{M \times N}$ denotes the vector space of $\bfm \times \bfn$-matrices whose $(\bfi,\bfj)$-th entry is an $m_{\bfm} \times n_{\bfn}$-matrix with coefficients in $\Bbbk$, and where $\Bbbk_{m \times n}$ denotes the vector space of $m \times n$-matrices with coefficients in $\Bbbk$. Furthermore, these natural linear isomorphisms preserve matrix product, meaning we have
\[
 f_{L,M}(Y)f_{M,N}(X) = f_{L,N}(Y X)
\]
for all $Y \in \Bbbk^{L \times M}$ and $X \in \Bbbk^{M \times N}$.

\begin{proof}[Proof of Proposition \ref{P:completion}]
 First of all, by definition, $\cmpl(\Bbbk)$ is the tensor unit of $\coCat_{\Bbbk}$, so we can set $\bfiota$ to be the identity linear functor. Now, for all linear categories $A$ and $A'$, let us explain how to construct a linear functor
 \[
  \bfchi_{A,A'} : \cmpl(A) \csqtimes \cmpl(A') \rightarrow \cmpl(A \sqtimes A').
 \]
 First of all, every object of the completion $\cmpl(A) \csqtimes \cmpl(A')$ has the form
 \[
  \im \left( \left( \underline{\bfp_{\bfi \bfj}} \right)_{(\bfi,\bfj) \in \bfI^2} \right) :=
  \left( \left( \begin{array}{c}
                \vdots \\ \underline{\bfx_{\bfi}} \\ \vdots
                \end{array} \right),
  \left( \begin{array}{ccc}
          \ddots & \vdots & \bdots \\
          \cdots & \underline{\bfp_{\bfi \bfj}} & \cdots \\
          \bdots & \vdots & \ddots
         \end{array} \right) \right)
 \]
 for some finite ordered set $\bfI$, for some family
 \[ 
  \left\{ \underline{\bfx_{\bfi}} \in \cmpl(A) \sqtimes \cmpl(A') \bigm| \bfi \in \bfI \right\}
 \]
 of the form
 \[
  \underline{\bfx_{\bfi}} := \left( \im \left( \left( p_{\bfi ij} \right)_{(i,j) \in I^2_{\bfi} } \right),\im \left( \left( p'_{\bfi i'j'} \right)_{(i',j') \in I'^2_{\bfi}} \right) \right),
 \]
 where
 \[
  \im \left( \left( p_{\bfi ij} \right)_{(i,j) \in I^2_{\bfi} } \right) := 
  \left( \left( \begin{array}{c}
                 \vdots \\ x_{\bfi i} \\ \vdots
                \end{array} \right),
  \left( \begin{array}{ccc}
          \ddots & \vdots & \bdots \\
          \cdots & p_{\bfi ij} & \cdots \\
          \bdots & \vdots & \ddots
         \end{array} \right) \right)
 \]
 is an object of $\cmpl(A)$ and
 \[
  \im \left( \left( p'_{\bfi i'j'} \right)_{(i',j') \in I'^2_{\bfi} } \right) := 
  \left( \left( \begin{array}{c}
                 \vdots \\ x'_{\bfi i'} \\ \vdots
                \end{array} \right),
  \left( \begin{array}{ccc}
          \ddots & \vdots & \bdots \\
          \cdots & p'_{\bfi i'j'} & \cdots \\
          \bdots & \vdots & \ddots
         \end{array} \right) \right)
 \]
 is an object of $\cmpl(A')$ for every $\bfi \in \bfI$, and for some family
 \[
  \left\{ \underline{\bfp_{\bfi \bfj}} \in \Hom_{\cmpl(A) \sqtimes \cmpl(A')} \left( \underline{\bfx_{\bfj}},\underline{\bfx_{\bfi}} \right) \Bigm| (\bfi,\bfj) \in \bfI^2 \right\}
 \]
 of the form
 \[
  \underline{\bfp_{\bfi \bfj}} := \sum_{\bfn = 1}^{n_{\bfp_{\bfi \bfj}}} \left( p_{\bfi \bfj \bfn ij} \right)_{(i,j) \in I_{\bfi} \times I_{\bfj}} \otimes \left( p'_{\bfi \bfj \bfn i'j'} \right)_{(i',j') \in I'_{\bfi} \times I'_{\bfj}}
 \]
 satisfying
 \[
  \sum_{\bfh \in \bfI} \underline{\bfp_{\bfi \bfh}} \circ \underline{\bfp_{\bfh \bfj}} = \underline{\bfp_{\bfi \bfj}}
 \]
 for every $(\bfi,\bfj) \in \bfI^2$, where
 \[
  \left( p_{\bfi \bfj \bfn ij} \right)_{(i,j) \in I_{\bfi} \times I_{\bfj}} := 
  \left( \begin{array}{ccc}
          \ddots & \vdots & \bdots \\
          \cdots & p_{\bfi \bfj \bfn ij} & \cdots \\
          \bdots & \vdots & \ddots
         \end{array} \right)
 \]
 is a morphism of 
 \[
  \Hom_{\cmpl(A)} \left( \im \left( \left( p_{\bfj ij} \right)_{(i,j) \in I_{\bfj}^2 } \right),\im \left( \left( p_{\bfi ij} \right)_{(i,j) \in I_{\bfi}^2 } \right) \right)
 \]
 and 
 \[
  \left( p'_{\bfi \bfj \bfn i'j'} \right)_{(i',j') \in I'_{\bfi} \times I'_{\bfj}} := 
  \left( \begin{array}{ccc}
          \ddots & \vdots & \bdots \\
          \cdots & p'_{\bfi \bfj \bfn i'j'} & \cdots \\
          \bdots & \vdots & \ddots
         \end{array} \right)
 \]
 is a morphism of 
 \[ 
  \Hom_{\cmpl(A')} \left( \im \left( \left( p'_{\bfj i'j'} \right)_{(i',j') \in I'^2_{\bfj} } \right),\im \left( \left( p'_{\bfi i'j'} \right)_{(i',j') \in I'^2_{\bfi} } \right) \right)
 \]
 for every $(\bfi,\bfj) \in \bfI^2$ and $1 \leqslant \bfn \leqslant n_{\bfp_{\bfi \bfj}}$. 
 Analogously, every morphism of the completion $\cmpl(A) \csqtimes \cmpl(A')$ from $\im ((\underline{\bfp_{\bfi \bfj}})_{(\bfi, \bfj) \in \bfI^2})$ to $\im ((\underline{\bfq_{\bfi \bfj}})_{(\bfi,\bfj) \in \bfJ^2})$ has the form
 \[
  \left( \underline{\bff_{\bfi \bfj}} \right)_{(\bfi,\bfj) \in \bfJ \times \bfI} :=
  \left( \begin{array}{ccc}
          \ddots & \vdots & \bdots \\
          \cdots & \underline{\bff_{\bfi \bfj}} & \cdots \\
          \bdots & \vdots & \ddots
         \end{array} \right)
 \] 
 for some family
 \[
  \left\{ \underline{\bff_{\bfi \bfj}} \in \Hom_{\cmpl(A) \sqtimes \cmpl(A')} \left( \underline{\bfx_{\bfj}},\underline{\bfy_{\bfi}} \right) \Bigm| (\bfi,\bfj) \in \bfJ \times \bfI \right\}
 \]
 of the form
 \[
  \underline{\bff_{\bfi \bfj}} := \sum_{\bfn = 1}^{n_{\bff_{\bfi \bfj}}} \left( f_{\bfi \bfj \bfn ij} \right)_{(i,j) \in J_{\bfi} \times I_{\bfj}} \otimes \left( f'_{\bfi \bfj \bfn i'j'} \right)_{(i',j') \in J'_{\bfi} \times I'_{\bfj}}
 \]
 satisfying
 \[
  \sum_{\bfh \in \bfI} \underline{\bff_{\bfi \bfh}} \circ \underline{\bfp_{\bfh \bfj}} = \underline{\bff_{\bfi \bfj}} = \sum_{\bfh \in \bfJ} \underline{\bfq_{\bfi \bfh}} \circ \underline{\bff_{\bfh \bfj}}
 \]
 for every $(\bfi,\bfj) \in \bfJ \times \bfI$, where
 \[
  \left( f_{\bfi \bfj \bfn ij} \right)_{(i,j) \in J_{\bfi} \times I_{\bfj}} :=
  \left( \begin{array}{ccc}
          \ddots & \vdots & \bdots \\
          \cdots & f_{\bfi \bfj \bfn ij} & \cdots \\
          \bdots & \vdots & \ddots
         \end{array} \right)
 \]
 is a morphism of 
 \[
  \Hom_{\cmpl(A)} \left( \im \left( \left( p_{\bfj ij} \right)_{(i,j) \in I_{\bfj}^2 } \right),\im \left( \left( q_{\bfi ij} \right)_{(i,j) \in J_{\bfi}^2 } \right) \right)
 \]
 and 
 \[
  \left( f'_{\bfi \bfj \bfn i'j'} \right)_{(i',j') \in J'_{\bfi} \times I'_{\bfj}} := 
  \left( \begin{array}{ccc}
          \ddots & \vdots & \bdots \\
          \cdots & f'_{\bfi \bfj \bfn i'j'} & \cdots \\
          \bdots & \vdots & \ddots
         \end{array} \right)
 \]
 is a morphism of 
 \[ 
  \Hom_{\cmpl(A')} \left( \im \left( \left( p'_{\bfj i'j'} \right)_{(i',j') \in I'^2_{\bfj} } \right),\im \left( \left( q'_{\bfi i'j'} \right)_{(i',j') \in J'^2_{\bfi} } \right) \right)
 \]
 for every $(\bfi,\bfj) \in \bfJ \times \bfI$ and $1 \leqslant \bfn \leqslant n_{\bff_{\bfi \bfj}}$. 
 Then, we define the image under $\bfchi_{A,A'}$ of an object
 \[
  \im \left( \left( \underline{\bfp_{\bfi \bfj}} \right)_{(\bfi,\bfj) \in \bfI^2} \right)
 \]
 of $\cmpl(A) \csqtimes \cmpl(A')$ to be the object
 \[
  \im \left( \left( \underline{p_{ij}} \right)_{(i,j) \in I^2} \right) :=  
  \left( \left( \begin{array}{c}
                 \vdots \\ \underline{x_i} \\ \vdots
                \end{array} \right),
  \left( \begin{array}{ccc}
          \ddots & \vdots & \bdots \\
          \cdots & \underline{p_{ij}} & \cdots \\
          \bdots & \vdots & \ddots
         \end{array} \right) \right)
 \]
 of $\cmpl(A \sqtimes A')$ given by the finite ordered set
 \[
  I := \bigcup_{\bfi \in \bfI} \left( \{ \bfi \} \times I_{\bfi} \times I'_{\bfi} \right)
 \]
 of indices, by the objects
 \[
  \underline{x_{(\bfi,i,i')}} := \left( x_{\bfi i},x'_{\bfi i'} \right)
 \]
 of $A \sqtimes A'$, and by the morphisms
 \[
  \underline{p_{(\bfi,i,i')(\bfj,j,j')}} := \sum_{\bfn=1}^{n_{\bfp_{\bfi \bfj}}} p_{\bfi \bfj \bfn ij} \otimes p'_{\bfi \bfj \bfn i'j'}
 \]
 of $\Hom_{A \sqtimes A'} ( \underline{x_{(\bfj,j,j')}},\underline{x_{(\bfi,i,i')}} )$, and we define the image under $\bfchi_{A,A'}$ of a morphism
 \[
  \left( \underline{\bff_{\bfi \bfj}} \right)_{(\bfi,\bfj) \in \bfJ \times \bfI}
 \]
 of $\cmpl(A) \csqtimes \cmpl(A')$ to be the morphism
 \[
  \left( \underline{f_{ij}} \right)_{(i,j) \in J \times I} :=  
  \left( \begin{array}{ccc}
          \ddots & \vdots & \bdots \\
          \cdots & \underline{f_{ij}} & \cdots \\
          \bdots & \vdots & \ddots
         \end{array} \right)
 \]
 of $\cmpl(A \sqtimes A')$ given by the morphisms
 \[
  \underline{f_{(\bfi,i,i')(\bfj,j,j')}} := \sum_{\bfn=1}^{n_{\bff_{\bfi \bfj}}} f_{\bfi \bfj \bfn ij} \otimes f'_{\bfi \bfj \bfn i'j'}
 \]
 of $\Hom_{\cmpl(A \sqtimes A')} ( \underline{x_{(\bfj,j,j')}},\underline{y_{(\bfi,i,i')}} )$. It can be checked that this assignment actually defines an essentially surjective fully faithful linear functor. Now, remark that for all linear functors $F : A \rightarrow A''$ and $F' : A' \rightarrow A'''$ we have
 \[
  \cmpl(F \sqtimes F') \circ \bfchi_{A,A'} = \bfchi_{A'',A'''} \circ \left( \cmpl(F) \csqtimes \cmpl(F') \right),
 \]
 so we can set $\bfchi_{F,F'}$ to be the identity natural transformation. Then, we only need to specify 2-modifications for the symmetric monoidal structure of $\cmpl$. In the quasi-strict versions of $\bfCat_{\Bbbk}$ and $\coCat_{\Bbbk}$ we have, for all linear categories $A$, $A'$, and $A''$, the equalities
 \begin{gather*}
  \bfchi_{\Bbbk,A} = \id_{\cmpl(A)}, \quad \bfchi_{A,\Bbbk} = \id_{\cmpl(A)}, \\
  \bfchi_{A,A' \sqtimes A''} \circ \left( \id_{\cmpl(A)} \csqtimes \bfchi_{A',A''} \right) = \bfchi_{A \sqtimes A',A''} \circ \left( \bfchi_{A,A'} \csqtimes \id_{\cmpl(A'')} \right), \\
  \bfchi_{A',A} \circ \hat{\trbr}_{\cmpl(A),\cmpl(A')} = \cmpl(\trbr_{A,A'}) \circ \bfchi_{A,A'},
 \end{gather*}
 so we can set $\bfGamma_A$, $\bfDelta_A$, $\bfOmega_{A,A',A''}$, and $\bfTheta_{A,A'}$ to be identity natural transformations.
\end{proof}

The proof of the analogue graded result is exactly the same, but with gradings. Let us give it for completeness.

\begin{proof}[Proof of Proposition \ref{P:gr_completion}]
 First of all, by definition, $\grcmpl(\Bbbk)$ is the tensor unit of $\coCat_{\Bbbk}^{\PGr}$, so we can set $\bfiota$ to be the identity $\PGr$-graded linear functor. Now, for all $\PGr$-graded linear categories $\bbA$ and $\bbA'$, let us explain how to construct a $\PGr$-graded linear functor
 \[
  \bfchi_{\bbA,\bbA'} : \grcmpl(\bbA) \csqtimes \grcmpl(\bbA') \rightarrow \grcmpl(\bbA \sqtimes \bbA').
 \]
 First of all, every object of the completion $\grcmpl(\bbA) \csqtimes \grcmpl(\bbA')$ has the form
 \[
  \im \left( \left( \underline{\bfp_{\bfi \bfj}}^{\bfl_{\bfj} - \bfl_{\bfi}} \right)_{(\bfi,\bfj) \in \bfI^2} \right) :=
  \left( \left( \begin{array}{c}
                \vdots \\ \left( \underline{\bfx_{\bfi}},\bfl_{\bfi} \right) \\ \vdots
                \end{array} \right),
  \left( \begin{array}{ccc}
          \ddots & \vdots & \bdots \\
          \cdots & \underline{\bfp_{\bfi \bfj}}^{\bfl_{\bfj} - \bfl_{\bfi}} & \cdots \\
          \bdots & \vdots & \ddots
         \end{array} \right) \right)
 \]
 for some finite ordered set $\bfI$, for some family $\{ \bfl_{\bfi} \in Z \mid \bfi \in \bfI \}$, for some family
 \[ 
  \left\{ \underline{\bfx_{\bfi}} \in \grcmpl(\bbA) \sqtimes \grcmpl(\bbA') \bigm| \bfi \in \bfI \right\}
 \]
 of the form
 \[
  \underline{\bfx_{\bfi}} := \left( \im \left( \left( p^{\ell_{\bfi j} - \ell_{\bfi i}}_{\bfi ij} \right)_{(i,j) \in I^2_{\bfi} } \right),\im \left( \left( p'^{\ell'_{\bfi j'} - \ell'_{\bfi i'}}_{\bfi i'j'} \right)_{(i',j') \in I'^2_{\bfi}} \right) \right),
 \]
 where
 \[
  \im \left( \left( p^{\ell_{\bfi j} - \ell_{\bfi i}}_{\bfi ij} \right)_{(i,j) \in I^2_{\bfi} } \right) := 
  \left( \left( \begin{array}{c}
                 \vdots \\ \left( x_{\bfi i}, \ell_{\bfi i} \right) \\ \vdots
                \end{array} \right),
  \left( \begin{array}{ccc}
          \ddots & \vdots & \bdots \\
          \cdots & p^{\ell_{\bfi j} - \ell_{\bfi i}}_{\bfi ij} & \cdots \\
          \bdots & \vdots & \ddots
         \end{array} \right) \right)
 \]
 is an object of $\grcmpl(\bbA)$ and
 \[
  \im \left( \left( p'^{\ell'_{\bfi j'} - \ell'_{\bfi i'}}_{\bfi i'j'} \right)_{(i',j') \in I'^2_{\bfi} } \right) := 
  \left( \left( \begin{array}{c}
                 \vdots \\ \left( x'_{\bfi i'}, \ell'_{\bfi i'} \right) \\ \vdots
                \end{array} \right),
  \left( \begin{array}{ccc}
          \ddots & \vdots & \bdots \\
          \cdots & p'^{\ell'_{\bfi j'} - \ell'_{\bfi i'}}_{\bfi i'j'} & \cdots \\
          \bdots & \vdots & \ddots
         \end{array} \right) \right)
 \]
 is an object of $\grcmpl(\bbA')$ for every $\bfi \in \bfI$, and for some family
 \[
  \left\{ \underline{\bfp_{\bfi \bfj}}^{\bfl_{\bfj} - \bfl_{\bfi}} \in \Hom_{\grcmpl(\bbA) \sqtimes \grcmpl(\bbA')} \left( \underline{\bfx_{\bfj}},\underline{\bfx_{\bfi}} \right) \Bigm| (\bfi,\bfj) \in \bfI^2 \right\}
 \]
 of the form
 \[
  \underline{\bfp_{\bfi \bfj}}^{\bfl_{\bfj} - \bfl_{\bfi}} := \sum_{\bfn = 1}^{n_{\bfp_{\bfi \bfj}}} \left( p^{\bfl_{\bfj} - \bfl_{\bfi} - \bfl'_{\bfn} + \ell_{\bfj j} - \ell_{\bfi i}}_{\bfi \bfj \bfn ij} \right)_{(i,j) \in I_{\bfi} \times I_{\bfj}} \otimes \left( p'^{\bfl'_{\bfn} + \ell'_{\bfj j'} - \ell'_{\bfi i'}}_{\bfi \bfj \bfn i'j'} \right)_{(i',j') \in I'_{\bfi} \times I'_{\bfj}}
 \]
 satisfying
 \[
  \sum_{\bfh \in \bfI} \underline{\bfp_{\bfi \bfh}}^{\bfl_{\bfh} - \bfl_{\bfi}} \circ \underline{\bfp_{\bfh \bfj}}^{\bfl_{\bfj} - \bfl_{\bfh}} = \underline{\bfp_{\bfi \bfj}}^{\bfl_{\bfj} - \bfl_{\bfi}}
 \]
 for every $(\bfi,\bfj) \in \bfI^2$, where
 \[
  \left( p^{\bfl_{\bfj} - \bfl_{\bfi} - \bfl'_{\bfn} + \ell_{\bfj j} - \ell_{\bfi i}}_{\bfi \bfj \bfn ij} \right)_{(i,j) \in I_{\bfi} \times I_{\bfj}} := 
  \left( \begin{array}{ccc}
          \ddots & \vdots & \bdots \\
          \cdots & p^{\bfl_{\bfj} - \bfl_{\bfi} - \bfl'_{\bfn} + \ell_{\bfj j} - \ell_{\bfi i}}_{\bfi \bfj \bfn ij} & \cdots \\
          \bdots & \vdots & \ddots
         \end{array} \right)
 \]
 is a degree $\bfl_{\bfj} - \bfl_{\bfi} - \bfl'_{\bfn} + \ell_{j \bfj} - \ell_{i \bfi}$ morphism of 
 \[
  \Hom_{\grcmpl(\bbA)} \left( \im \left( \left( p^{\ell_{\bfj j} - \ell_{\bfj i}}_{\bfj ij} \right)_{(i,j) \in I_{\bfj}^2 } \right),\im \left( \left( p^{\ell_{\bfi j} - \ell_{\bfi i}}_{\bfi ij} \right)_{(i,j) \in I_{\bfi}^2 } \right) \right)
 \]
 and 
 \[
  \left( p'^{\bfl'_{\bfn} + \ell'_{\bfj j'} - \ell'_{\bfi i'}}_{\bfi \bfj \bfn i'j'} \right)_{(i',j') \in I'_{\bfi} \times I'_{\bfj}} := 
  \left( \begin{array}{ccc}
          \ddots & \vdots & \bdots \\
          \cdots & p'^{\bfl'_{\bfn} + \ell'_{\bfj j'} - \ell'_{\bfi i'}}_{\bfi \bfj \bfn i'j'} & \cdots \\
          \bdots & \vdots & \ddots
         \end{array} \right)
 \]
 is a degree $\bfl'_{\bfn} + \ell'_{\bfj j'} - \ell'_{\bfi i'}$ morphism of 
 \[ 
  \Hom_{\grcmpl(\bbA')} \left( \im \left( \left( p'^{\ell'_{\bfj j'} - \ell'_{\bfj i'}}_{i'j' \bfj} \right)_{(i',j') \in I'^2_{\bfj} } \right),\im \left( \left( p'^{\ell'_{\bfi j'} - \ell'_{\bfi i'}}_{i'j' \bfi} \right)_{(i',j') \in I'^2_{\bfi} } \right) \right)
 \]
 for every $(\bfi,\bfj) \in \bfI^2$ and $1 \leqslant \bfn \leqslant n_{\bfp_{\bfi \bfj}}$. 
 Analogously, every degree $\bfk$ morphism of the completion $\grcmpl(\bbA) \csqtimes \grcmpl(\bbA')$ from $\im ((\underline{\bfp_{\bfi \bfj}}^{\bfl_{\bfj} - \bfl_{\bfi}})_{(\bfi,\bfj) \in \bfI^2})$ to $\im ((\underline{\bfq_{\bfi \bfj}}^{\bfm_{\bfj} - \bfm_{\bfi}})_{(\bfi,\bfj) \in \bfJ^2})$ has the form
 \[
  \left( \underline{\bff_{\bfi \bfj}}^{\bfk + \bfl_{\bfj} - \bfm_{\bfi}} \right)_{(\bfi,\bfj) \in \bfJ \times \bfI} :=
  \left( \begin{array}{ccc}
          \ddots & \vdots & \bdots \\
          \cdots & \underline{\bff_{\bfi \bfj}}^{\bfk + \bfl_{\bfj} - \bfm_{\bfi}} & \cdots \\
          \bdots & \vdots & \ddots
         \end{array} \right)
 \] 
 for some family
 \[
  \left\{ \underline{\bff_{\bfi \bfj}}^{\bfk + \bfl_{\bfj} - \bfm_{\bfi}} \in \Hom_{\grcmpl(\bbA) \sqtimes \grcmpl(\bbA')} \left( \underline{\bfx_{\bfj}},\underline{\bfy_{\bfi}} \right) \Bigm| (\bfi,\bfj) \in \bfJ \times \bfI \right\}
 \]
 of the form
 \[
  \underline{\bff_{\bfi \bfj}}^{\bfk + \bfl_{\bfj} - \bfm_{\bfi}} := \sum_{\bfn = 1}^{n_{\bff_{\bfi \bfj}}} \left( f^{\bfk + \bfl_{\bfj} - \bfm_{\bfi} - \bfk'_{\bfn} + \ell_{\bfj j} - m_{\bfi i}}_{\bfi \bfj \bfn ij} \right)_{(i,j) \in J_{\bfi} \times I_{\bfj}} \otimes \left( f'^{\bfk'_{\bfn} + \ell'_{\bfj j'} - m'_{\bfi i'}}_{\bfi \bfj \bfn i'j'} \right)_{(i',j') \in J'_{\bfi} \times I'_{\bfj}}
 \]
 satisfying
 \[
  \sum_{\bfh \in \bfI} \underline{\bff_{\bfi \bfh}}^{\bfk + \bfl_{\bfh} - \bfm_{\bfi}} \circ \underline{\bfp_{\bfh \bfj}}^{\bfl_{\bfj} - \bfl_{\bfh}} = \underline{\bff_{\bfi \bfj}}^{\bfk + \bfl_{\bfj} - \bfm_{\bfi}} = \sum_{\bfh \in \bfJ} \underline{\bfq_{\bfi \bfh}}^{\bfm_{\bfh} - \bfm_{\bfi}} \circ \underline{\bff_{\bfh \bfj}}^{\bfk + \bfl_{\bfj} - \bfm_{\bfh}}
 \]
 for every $(\bfi,\bfj) \in \bfJ \times \bfI$, where
 \[
  \left( f^{\bfk + \bfl_{\bfj} - \bfm_{\bfi} - \bfk'_{\bfn} + \ell_{\bfj j} - m_{\bfi i}}_{\bfi \bfj \bfn ij} \right)_{(i,j) \in J_{\bfi} \times I_{\bfj}} :=
  \left( \begin{array}{ccc}
          \ddots & \vdots & \bdots \\
          \cdots & f^{\bfk + \bfl_{\bfj} - \bfm_{\bfi} - \bfk'_{\bfn} + \ell_{\bfj j} - m_{\bfi i}}_{\bfi \bfj \bfn ij} & \cdots \\
          \bdots & \vdots & \ddots
         \end{array} \right)
 \]
 is a degree $\bfk + \bfl_{\bfj} - \bfm_{\bfi} - \bfk'_{\bfn}$ morphism of 
 \[
  \Hom_{\grcmpl(\bbA)} \left( \im \left( \left( p^{\ell_{\bfj j} - \ell_{\bfj i}}_{\bfj ij} \right)_{(i,j) \in I_{\bfj}^2 } \right),\im \left( \left( q^{m_{\bfi j} - m_{\bfi i}}_{\bfi ij} \right)_{(i,j) \in J_{\bfi}^2 } \right) \right)
 \]
 and 
 \[
  \left( f'^{\bfk'_{\bfn} + \ell'_{\bfj j'} - m'_{\bfi i'}}_{\bfi \bfj \bfn i'j'} \right)_{(i',j') \in J'_{\bfi} \times I'_{\bfj}} := 
  \left( \begin{array}{ccc}
          \ddots & \vdots & \bdots \\
          \cdots & f'^{\bfk'_{\bfn} + \ell'_{\bfj j'} - m'_{\bfi i'}}_{\bfi \bfj \bfn i'j'} & \cdots \\
          \bdots & \vdots & \ddots
         \end{array} \right)
 \]
 is a degree $\bfk'_{\bfn}$ morphism of 
 \[ 
  \Hom_{\cmpl(A')} \left( \im \left( \left( p'^{\ell'_{\bfj j'} - \ell'_{\bfj i'}}_{\bfj i'j'} \right)_{(i',j') \in I'^2_{\bfj} } \right),\im \left( \left( q'^{m'_{\bfi j'} - m'_{\bfi i'}}_{\bfi i'j'} \right)_{(i',j') \in J'^2_{\bfi} } \right) \right)
 \]
 for every $(\bfi,\bfj) \in \bfJ \times \bfI$ and $1 \leqslant \bfn \leqslant n_{\bff_{\bfi \bfj}}$.
 Then, we define the image under $\bfchi_{\bbA,\bbA'}$ of an object
 \[
  \im \left( \left( \underline{\bfp_{\bfi \bfj}}^{\bfl_{\bfj} - \bfl_{\bfi}} \right)_{(\bfi,\bfj) \in \bfI^2} \right)
 \]
 of $\grcmpl(\bbA) \csqtimes \grcmpl(\bbA')$ to be the object
 \[
  \im \left( \left( \underline{p_{ij}}^{\ell_j - \ell_i} \right)_{(i,j) \in I^2} \right) :=  
  \left( \left( \begin{array}{c}
                 \vdots \\ \left( \underline{x_i}, \ell_i \right) \\ \vdots
                \end{array} \right),
  \left( \begin{array}{ccc}
          \ddots & \vdots & \bdots \\
          \cdots & \underline{p_{ij}}^{\ell_j - \ell_i} & \cdots \\
          \bdots & \vdots & \ddots
         \end{array} \right) \right)
 \]
 of $\grcmpl(\bbA \sqtimes \bbA')$ given by the finite ordered set
 \[
  I := \bigcup_{\bfi \in \bfI} \left( \{ \bfi \} \times I_{\bfi} \times I'_{\bfi} \right)
 \]
 of indices, by the elements
 \[
  \ell_{(\bfi,i,i')} := \bfl_{\bfi} + \ell_{\bfi i} + \ell'_{\bfi i'}
 \]
 of $Z$, by the objects
 \[
  \underline{x_{(\bfi,i,i')}} := \left( x_{\bfi i},x'_{\bfi i'} \right)
 \]
 of $\bbA \sqtimes \bbA'$, and by the degree $\bfl_{\bfj} - \bfl_{\bfi} + \ell_{\bfj j} - \ell_{\bfi i} + \ell'_{\bfj j'} - \ell'_{\bfi i'}$ morphisms
 \[
  \underline{p_{(\bfi,i,i')(\bfj,j,j')}}^{\ell_{(\bfj,j,j')} - \ell_{(\bfi,i,i')}} := \sum_{\bfn=1}^{n_{\bfp_{\bfi \bfj}}} p^{\bfl_{\bfj} - \bfl_{\bfi} - \bfl'_{\bfn} + \ell_{\bfj j} - \ell_{\bfi i}}_{\bfi \bfj \bfn ij} \otimes p'^{\bfl'_{\bfn} + \ell'_{\bfj j'} - \ell'_{\bfi i'}}_{\bfi \bfj \bfn i'j'}
 \]
 of $\Hom_{\bbA \sqtimes \bbA'} ( \underline{x_{(\bfj,j,j')}},\underline{x_{(\bfi,i,i')}} )$, and we define the image under $\bfchi_{\bbA,\bbA'}$ of a degree $\bfk$ morphism
 \[
  \left( \underline{\bff_{\bfi \bfj}}^{\bfk + \bfl_{\bfj} - \bfm_{\bfi}} \right)_{(\bfi,\bfj) \in \bfJ \times \bfI}
 \]
 of $\grcmpl(\bbA) \csqtimes \grcmpl(\bbA')$ to be the degree $k$ morphism
 \[
  \left( \underline{f_{ij}}^{k + \ell_j - m_i} \right)_{(i,j) \in J \times I} :=  
  \left( \begin{array}{ccc}
          \ddots & \vdots & \bdots \\
          \cdots & \underline{f_{ij}}^{k + \ell_j - m_i} & \cdots \\
          \bdots & \vdots & \ddots
         \end{array} \right)
 \]
 of $\grcmpl(\bbA \sqtimes \bbA')$ given by the element 
 \[
  k := \bfk
 \]
 of $Z$, and by the degree $\bfk + \bfl_{\bfj} - \bfm_{\bfi} + \ell_{\bfj j} - m_{\bfi i} + \ell'_{\bfj j'} - m'_{\bfi i'}$ morphisms
 \[
  \underline{f_{(\bfi,i,i')(\bfj,j,j')}}^{k + \ell_{(\bfj,j,j')} - m_{(\bfi,i,i')}} := \sum_{\bfn=1}^{n_{\bff_{\bfi \bfj}}} f^{\bfk + \bfl_{\bfj} - \bfm_{\bfi} - \bfk'_{\bfn} + \ell_{\bfj j} - m_{\bfi i}}_{\bfi \bfj \bfn ij} \otimes f'^{\bfk'_{\bfn} + \ell'_{\bfj j'} - m'_{\bfi i'}}_{\bfi \bfj \bfn i'j'}
 \]
 of $\Hom_{\grcmpl(\bbA \sqtimes \bbA')} ( \underline{x_{(\bfj,j,j')}},\underline{y_{(\bfi,i,i')}} )$. It can be checked that this assignment actually defines an essentially surjective fully faithful $\PGr$-graded linear functor. Now, remark that for all $\PGr$-graded linear functors $\bbF : \bbA \rightarrow \bbA''$ and $\bbF' : \bbA' \rightarrow \bbA'''$ we have
 \[
  \grcmpl(\bbF \sqtimes \bbF') \circ \bfchi_{\bbA,\bbA'} = \bfchi_{\bbA'',\bbA'''} \circ \left( \grcmpl(\bbF) \csqtimes \grcmpl(\bbF') \right),
 \]
 so we can set $\bfchi_{\bbF,\bbF'}$ to be the identity $\PGr$-graded natural transformation. Then, we only need to specify 2-modifications for the symmetric monoidal structure of $\grcmpl$. In the quasi-strict versions of $\bfCat_{\Bbbk}^{\PGr}$ and $\coCat_{\Bbbk}^{\PGr}$ we have, for all $\PGr$-graded linear categories $\bbA$, $\bbA'$, and $\bbA''$, the equalities
 \begin{gather*}
  \bfchi_{\Bbbk,\bbA} = \id_{\grcmpl(\bbA)}, \quad \bfchi_{\bbA,\Bbbk} = \id_{\grcmpl(\bbA)}, \\
  \bfchi_{\bbA,\bbA' \sqtimes \bbA''} \circ \left( \id_{\grcmpl(\bbA)} \csqtimes \bfchi_{\bbA',\bbA''} \right) = \bfchi_{\bbA \sqtimes \bbA',\bbA''} \circ \left( \bfchi_{\bbA,\bbA'} \csqtimes \id_{\grcmpl(\bbA'')} \right), \\
  \bfchi_{\bbA',\bbA} \circ \hat{\bfc}^{\gamma}_{\grcmpl(\bbA),\grcmpl(\bbA')} = \grcmpl(\gmmbr_{\bbA,\bbA'}) \circ \bfchi_{\bbA,\bbA'},
 \end{gather*}
 so we can set $\bfGamma_{\bbA}$, $\bfDelta_{\bbA}$, $\bfOmega_{\bbA,\bbA',\bbA''}$, and $\bfTheta_{\bbA,\bbA'}$ to be identity $\PGr$-graded natural transformations.
\end{proof}


%% file: bibliography.tex

\bibliographystyle{amsalpha}
